\documentclass[12pt,oneside,a4]{thesis}
\usepackage{mathpazo}

\usepackage{amsthm}
\usepackage{amssymb}
\usepackage{graphicx}
\usepackage{mathtools}
\usepackage{tikz}
\usepackage{dsfont}

\usetikzlibrary{matrix}

\usepackage[british]{babel}
\usepackage[all,arc]{xy}
\usepackage{setspace}

\usepackage{geometry}
\let\ab\allowbreak

\usepackage{color}\newcounter{marginnotes}
\let\oldmarginpar\marginpar
\renewcommand{\marginpar}[1]{\addtocounter{marginnotes}{1}
\oldmarginpar[\raggedleft\footnotesize #1]{\raggedright\footnotesize #1}}
\newcommand{\countnotes}{\ifnum\value{marginnotes}=0 {}
 \else
   \begin{center}
   \color{red}\bf\Huge
   \ifnum\value{marginnotes}=1 {There is 1 margin note in total.}
   \else {There are \arabic{marginnotes} margin notes in total.}
   \fi
    \end{center}
 \fi}

\newtheorem{theorem}{Theorem}[section]
\newtheorem*{theorem*}{Theorem}
\newtheorem{proposition}[theorem]{Proposition}
\newtheorem{lemma}[theorem]{Lemma}
\newtheorem*{lemma*}{Lemma}
\newtheorem{corollary}[theorem]{Corollary}
\newtheorem*{corollary*}{Corollary}

\theoremstyle{definition}
\newtheorem*{notation}{Notation}
\newtheorem{definition}[theorem]{Definition}
\newtheorem*{definition*}{Definition}
\newtheorem{exmp}[theorem]{Example}
\newtheorem{remark}[theorem]{Remark}

\numberwithin{equation}{section} 
\numberwithin{figure}{section}   

\makeatother
\usepackage{babel}
\graphicspath{ {images/} }

\begin{document}
\begin{titlepage}
    \begin{center}

    \vspace*{3cm}
    
        \huge\textbf{Pattern-Equivariant Homology of Finite Local Complexity Patterns}
        
        \vspace{2cm}
        
        \normalsize
        
        Thesis submitted for the degree\\
        of Doctor of Philosophy at the\\
        University of Leicester
        
        \vspace{3cm}
        
        \large \textbf{James Jonathan Walton}
               
        \vspace{1cm}
        \small
        \emph{Department of Mathematics\\
        University of Leicester\\}     
        
        \vspace{3cm}     
        
	    \large\textbf{April 2014}
    \end{center}
\end{titlepage}

\pagenumbering{roman}

\chapter*{Abstract} \singlespacing
	\begin{center} \vspace{-0.5cm} \large{\textbf{Pattern-equivariant homology of \\ finite local complexity patterns}}

\vspace{0.5cm}

\normalsize

James J.\ Walton

\vspace{1cm}
	\end{center}

This thesis establishes a generalised setting with which to unify the study of finite local complexity (FLC) patterns. The abstract notion of a \emph{pattern} is introduced, which may be seen as an analogue of the space group of isometries preserving a tiling but where, instead, one considers partial isometries preserving portions of it. These inverse semigroups of partial transformations are the suitable analogue of the space group for patterns with FLC but few global symmetries. In a similar vein we introduce the notion of a \emph{collage}, a system of equivalence relations on the ambient space of a pattern, which we show is capable of generalising many constructions applicable to the study of FLC tilings and Delone sets, such as the expression of the tiling space as an inverse limit of approximants.

An invariant is constructed for our abstract patterns, the so called \emph{pattern-equivariant} (\emph{PE}) \emph{homology}. These homology groups are defined using infinite singular chains on the ambient space of the pattern, although we show that one may define cellular versions which are isomorphic under suitable conditions. For FLC tilings these cellular PE chains are analogous to the PE cellular cochains \cite{Sadun1}. The PE homology and cohomology groups are shown to be related through Poincar\'{e} duality.

An efficient and highly geometric method for the computation of the PE homology groups for hierarchical tilings is presented. The rotationally invariant PE homology groups are shown not to be a topological invariant for the associated tiling space and seem to retain extra information about global symmetries of tilings in the tiling space. We show how the PE homology groups may be incorporated into a spectral sequence converging to the \v{C}ech cohomology of the rigid hull of a tiling. These methods allow for a simple computation of the \v{C}ech cohomology of the rigid hull of the Penrose tilings.

\chapter*{Acknowledgements}
Foremost, I wish to thank my supervisor, John Hunton. His expert and patient guidance was not only invaluable to me during my time as a PhD student, but also before it in inspiring me to continue studying mathematics. I am also greatly appreciative of his constant support following his departure from Leicester, shortly before the finish of my PhD, which has ensured that his leaving has caused no inconvenience to me. I express my deepest gratitude to Alex Clark for taking John's place as my supervisor in Leicester. I thank him for his excellent advice in the final stages of my PhD, which I am sure must always be far less interesting than those in which the content of a project are still in formation.

I thank my wonderful family for their love and support, in particular my mother and father. I can barely imagine the task of completing a PhD without the company of good friends. So I thank them for their interesting mathematical discussions -- particular mention goes to James Cranch and Dan Rust -- but especially for the non-mathematical ones (I can assure my friends that those in the former class are all also members of the latter).

I would also like to thank the staff of the Department of Mathematics at Leicester, whose efforts make it such a friendly and welcoming place to do mathematics. I am grateful to the EPSRC, whose financial support made this PhD possible.

\chapter*{Dedication}
\centerline{
I dedicate this thesis to Kristina, for finding those extra smiles in me which I could not.
}
\tableofcontents

\doublespacing

\listoffigures

\chapter*{Introduction} \addcontentsline{toc}{section}{Introduction} \markright{Introduction}
When should one consider a pattern of Euclidean space as \emph{ordered}? This is a question with many potential answers, and many of these answers lead to interesting and important mathematics.

An important (although perhaps not interesting) suggestion is that at least \emph{periodic} patterns should be categorised as ordered. Let us consider a pattern of Euclidean space $\mathbb{R}^d$ as some sort of decoration of it, perhaps as a tiling or as a point set (for many of the questions that interest us here, the particular decoration of our original pattern will not be significant). A pattern is \emph{periodic} if there exist $d$ linearly independent vectors $x_1 \ldots, x_d \in \mathbb{R}^d$ for which our original pattern is left invariant under a translation by any of these vectors. A most natural question arises: what kinds of symmetry can such a periodic pattern possess?

A \emph{space group} (sometimes known as a \emph{Bieberbach group}) $G$ is a discrete group of isometries of $\mathbb{R}^d$ with compact fundamental domain (so that $\mathbb{R}^d/G$ is compact; one also says that $G$ acts on $\mathbb{R}^d$ \emph{cocompactly}). Any such space group necessarily contains $d$ linearly independent translations. In answer to part of Hilbert's Eighteenth Problem, Bieberbach showed that there are only finitely many such space groups in each dimension up to conjugation by orientation-preserving affine transformations. This allows for a finite classification of the space groups in any given dimension. For example, for $d=2$ the space groups are sometimes known as the \emph{wallpaper groups} and there are $17$ of them. Of great relevance to crystallographers is that there is a classification of the space groups in dimension $3$, of which there are $230$. In some sense, this classification provides a complete categorisation of the possible geometries of these periodic patterns.

In the context of crystallography it was long believed that periodicity and order were one and the same. The hallmark of a material crystal, and of a high degree of structural order in the material, is the existence of sharp Bragg peaks in its X-ray diffraction pattern. It was believed that this phenomenon only existed for crystals whose atoms lie on a periodic lattice. However, in $1982$ Dan Shectman et al.\ discovered a material whose X-ray diffraction pattern possessed fivefold rotational symmetry. The classification of space groups of $\mathbb{R}^3$ does not allow for periodic crystals possessing such symmetries and so the material discovered, although evidently highly ordered due to the existence of sharp Bragg peaks, was necessarily non-periodic. Such substances have come to be known as \emph{quasicrystals}.

This discovery has forced crystallographers to reconsider the definition of a crystal. And, for the mathematician, new questions emerge: how does one characterise order in a pattern of Euclidean space (or of a more general space)? What mathematical tools should one use to study these patterns and their properties? Is some form of classification of these patterns possible? Are these patterns applicable to other areas of mathematics? These questions form the base of what is now known as the discipline of aperiodic order. See \cite{BGM} for an introduction to the field.

Although ordered patterns are our primary motivation here, the patterns of concern to us in this thesis are those which possess \emph{finite local complexity} (or are \emph{FLC}, for short). Loosely speaking, an FLC pattern is one for which, for any given radius, the number of distinct motifs of the pattern of this radius is finite, up to some agreed notion of equivalence (e.g., by translation or rigid motion). Finite local complexity is not a signifier of order but it does allow for a more combinatorial approach to the study of certain properties of these patterns. In particular, the approach of the classical periodic setting given by studying the space group of isometries preserving a pattern can be mimicked in the general setting. The idea is to consider, instead of groups of isometries, inverse semigroups of \emph{partial} isometries. This is an idea that we shall make precise in Chapter \ref{chap: Patterns of Finite Local Complexity}.

Of particular interest to us here is the approach of studying patterns through the topology of some associated space, for a tiling this space is often called the \emph{tiling space} or \emph{continuous hull} of the tiling (see \cite{Sadun2} for an accessible introduction to these ideas). The approach is a common one in mathematics: one associates to our objects of interest moduli spaces of associated objects. The goal is to study the geometry of these spaces and to infer from this analysis properties of the original objects of interest.

Many of the constructions in this direction are surprisingly intuitive for FLC tilings. For example, in the tiling space two tilings of it are considered as ``close'' if small motions of each are equal to some ``large'' radius. It is perhaps not surprising that this topology has applications to physical systems modelled by these tilings. Indeed, given some quantity which varies continuously on its surroundings in the tiling -- so that it evaluates similarly at two points which look the same to a large radius up to a small perturbation -- then the quantity necessarily varies continuously on the tiling space. There are already results in this direction, relating abstract topological invariants of tiling spaces to areas of physical interest. Most notable here is gap labelling \cite{B}, which relates gaps in the energy spectrum of a Hamiltonian of a particle moving through a solid modelled on the tiling to the $K$-theory of the tiling space. See \cite{KP} for an excellent introduction to the ideas relating the physics of quasicrystals to topological invariants of tiling spaces.

For FLC tilings the tiling space has a simple description as an inverse limit of ``approximants'' (see \cite{Sadun2,BDHS}). A point of each approximant determines how one may lay down a patch of tiles at the origin of a certain radius; a point of the inverse limit is a consistent sequence of such instructions for larger and larger patches. In Chapter \ref{chap: Patterns of Finite Local Complexity} we shall consider a generalised setting in which one may similarly express the hull of a pattern in terms of an inverse limit of approximants.

This description is of greater theoretical value than of computational value in the context of topological invariants of tiling spaces. It allows one to interpret certain topological invariants of the tiling space as invariants ``on the tiling'' which respect the tiling in a certain sense (see \cite{Clark} for an example of where it is useful to consider these invariants in this way). For example, the \emph{pattern-equivariant} (\emph{PE}) \emph{cohomology groups} express the \v{C}ech cohomology of the tiling space in terms of differential forms \cite{Kellendonk} or cellular cochains on the tiling \cite{Sadun1} which respect its structure. In Chapter \ref{chap: Pattern-Equivariant Homology} we shall introduce an invariant for the abstract patterns defined in Chapter \ref{chap: Patterns of Finite Local Complexity}. One considers Borel-Moore chains, or ``infinite'' singular chains, which are invariant with respect to the pattern to a certain radius. In Theorem \ref{thm: sing=cell} we show that these groups can be described by homology groups of certain cellular chains whenever the pattern comes equipped with a suitable CW-decomposition. In the case of patterns arising from tilings, these cellular chain groups correspond precisely to the cellular PE cohomology groups, but where one considers the boundary instead of the coboundary maps.

In Chapter \ref{chap: Poincare Duality for Pattern-Equivariant Homology} we shall link the PE homology and cohomology groups through Poincar\'{e} duality. This means that the PE cohomology groups, which are isomorphic to the \v{C}ech cohomology groups of the tiling space, may be understood using cellular chains. This leads to some rather pleasing geometric descriptions of these rather abstract groups, see for example Figure \ref{fig:Penrose} which illustrates a generator for the degree one rotationally-invariant PE homology of the Penrose kite and dart tilings. The Penrose tilings are hierarchical, the figure shows how an analogous generator of the ``supertiling'' is related to the generator on the tiling through a PE $2$-chain. Asides from being a cute observation, this is part of a step of the general computation of the PE homology groups for hierarchical tilings, which we outline in Chapter \ref{chap: Pattern-Equivariant Homology of Hierarchical Tilings}. The method is highly geometric, the computations are defined in terms of cellular chains on the tiling of interest. It also appears to be an efficient method of computation, indeed, in some sense, the ``minimal'' amount of collaring information is needed at each approximant computation, supposing that one has not determined that the tiling ``forces the border'' (see \cite{AP}). We exhibit further computations of this nature in this chapter. The method may be applied to more general tilings than of Euclidean space, for example, we make computations for the beautiful``regular'' pentagonal tilings of Bowers and Stephenson \cite{Pent}. The ambient spaces of these tilings are not homogeneous, indeed, their geometries reflect precisely the combinatorics of the tilings which live on them, but this presents no difficulties when one takes the view that patches should be compared with \emph{partial} isometries.

Given a tiling of Euclidean space with FLC with respect to rigid motions, there are two spaces that one may associate to it, denoted in \cite{BDHS} by $\Omega^0$ and $\Omega^{\text{rot}}$. The latter space, the \emph{rigid hull}, is the completion of a Euclidean orbit of the tiling whereas the former is a quotient of it by the group action of $SO(d)$ by rotation. This action is not free in general since some tilings in the hull may be fixed by non-trivial rotations. This means that the quotient map is not a fibration, so one may not apply a na\"{\i}ve Serre spectral sequence to the map, incorporating just the cohomology of $\Omega^0$ and of the fibre $SO(d)$. In \cite{BDHS} it was shown how one may define a spectral sequence using the \v{C}ech cohomology of $\Omega^0$ for certain hierarchical tilings of the Euclidean plane. One needs to incorporate extra torsion into the spectral sequence which accommodates for the existence of tilings in the hull fixed by non-trivial rotations. We present here a similar spectral sequence, applicable to any FLC (with respect to rigid motions) tiling of the Euclidean plane. The entries of the spectral sequence are the rotationally invariant PE homology groups of the tiling, some of them in a ``modified'' form. The modified groups are Poincar\'{e} dual to the \v{C}ech cohomology groups of the space $\Omega^0$, whereas the usual PE homology groups are Poincar\'{e} dual to these groups in all degrees except zero (for a two-dimensional tiling).

So the spectral sequence we present incorporates the \v{C}ech cohomology of $\Omega^0$ and a single degree zero PE homology group. In this sense, it is much like the spectral sequence of \cite{BDHS}, since this degree zero group is an extension of $\check{H}^2(\Omega^0)$ over a torsion group defined in terms of the number of rotationally invariant tilings in the hull. Indeed, this degree zero group is not a topological invariant of $\Omega^0$ and appears to retain information about the existence of rotationally invariant tilings. For example, for the triangle and square periodic tilings of the Euclidean plane, their rigid hulls are both homeomorphic to the $2$-sphere but their degree zero PE homology groups have different torsion parts.

We show how to compute the $E^\infty$ page of our spectral sequence, by giving a description of the $d^2$ map. This involves simply determining a single generator of the image. We show that it has a particularly geometric description, defined in terms of winding numbers about vertices of the tiling. This method coalesces conveniently with the methods of computation for hierarchical tilings given in Chapter \ref{chap: Pattern-Equivariant Homology of Hierarchical Tilings} and we show how one may perform a rather painless computation for the \v{C}ech cohomology of the Penrose kite and dart tilings using these techniques. This is an original computation and gives a surprising answer. We find that torsion does not appear in $\check{H}^2(\Omega^{\text{rot}})$. It has been suspected that this cannot be the case due to the existence of ``exceptional fibres'' in the approximants for the rigid hull. We give a direct geometric explanation of why the exceptional fibres are homotopic, showing that their difference in fact does not generate torsion in the \v{C}ech cohomology of the rigid hull.

\chapter{Background} \pagenumbering{arabic}
In this chapter we shall introduce tilings and Delone sets. We consider the study of these objects as a primary motivation for the work of this thesis. Much of the background material of this chapter will no doubt be familiar to a reader already acquainted with the study of the topology of tiling spaces. However, we should draw particular attention to our overview of uniform spaces (in Section \ref{sect: Uniform Spaces}) and of Borel-Moore homology, or homology of ``infinite'' singular chains (in Subsection \ref{subsect: Borel-Moore Homology and Poincare Duality}), whose use in this context is non-standard.

\section{Tilings and Delone Sets} \label{sect: Tilings and Delone Sets}
\subsection{Tilings}

Loosely speaking, a tiling is a decomposition of a space into tiles. Here, the ambient space of a tiling, the space which its tiles decomposes, will always be some metric space $(X,d_X)$. A \emph{prototile} is some subspace $p \subset (X,d_X)$. A \emph{prototile set} $\mathfrak{P}$ will be simply a set of prototiles, which will usually be finite but we will only make this restriction where necessary. It may be the case that we wish to distinguish congruent prototiles with a \emph{label}, that is, we may assume if necessary that there exists some \emph{label set} $L$ along with a bijection $l \colon \mathfrak{P} \rightarrow L$. A \emph{tile} $t$ is then some subspace congruent to one of our prototiles $p \in \mathfrak{P}$, that is, $t \subset (X,d_X)$ for which there exists some (surjective) isometry $\Phi \colon p \rightarrow t$; we assume that the tile $t$ is coupled with the same label as $p$ where we demand that our prototiles are labelled. We shall always assume, unless otherwise stated, that tiles are the closures of their interiors and are bounded in radius, that is, for each tile $t$ there exists some $x \in t$ and $r \in \mathbb{R}_{>0}$ such that $t \subset B_{d_X}(x,r)$.

A \emph{patch} is some union of tiles for which distinct tiles intersect on at most their boundaries. Given two patches $P$ and $Q$ for which each tile of $P$ is a tile of $Q$, we say that $P$ is a \emph{subpatch} of $Q$ and write $P \subset Q$. A \emph{tiling} is a patch for which the tiles cover all of $X$. Given a patch $T$ and some subset $U \subset X$, we define $T(U)$ to be the subpatch of tiles of $T$ which have non-trivial intersection with $U$. This follows the notation of $\cite{AP}$; one may consider patches and tilings as defining (order-preserving) functions $T \colon P(X) \rightarrow P(\{\text{tiles in } X\})$ (where $P(A)$ denotes the power set of a set $A$). In the case that $U=B_{d_X}(x,r)$ is an $r$-ball about $x$, we write $T(U)=T(x,r)$. For a (partial) isometry $\Phi$ of $X$ whose domain contains the patch $P$ we define the patch $\Phi(P)$ to be the set of images of tiles $t \in P$ under $\Phi$.

It is often the case that a tiling has a CW structure or is at least MLD to a tiling (see Example \ref{ex: MLD}) which has a CW structure. A \emph{cellular tiling} $T$ will consist of the following data. Firstly, each of the prototiles should have a CW-decomposition (see Subsection \ref{subsect: CW-Complexes and Cellular Homology}). The tiling $T$ (or, more accurately, $(X,d_X)$) should have a CW-decomposition for which the intersection of its (open) cells with a tile is a subcomplex which is a congruent image of the CW-decomposition of the prototile that the tile is an image of. Notice that for any two subcells of tiles, their intersection is a sub-complex; one often says that the tiles meet \emph{full-face to full-face}. Of course, it is possible that the tiles may be given multiple CW-decompositions if the original prototile possesses non-trivial isometries. We shall consider a CW-decomposition of a tile into cells (along with its possible label) as part of the \emph{decoration} of the tile. One should only consider patches $P$ and $Q$ as being equal if the set of tiles of each are the same when taken with their decoration of label and CW-decomposition.

\subsection{Delone Sets}
Although tilings have an instant geometric appeal, point sets are also important geometric objects. Given some crystal, an obvious approach to studying it is to investigate the properties of the lattice on which its atoms live, which one may assume to extend infinitely in all directions in Euclidean space. Periodic lattices in Euclidean space are well understood, but the discovery of quasicrystals has given impetus to study a wider range of point sets, ones which may be non-periodic but still highly ordered.

We should usually like, at the very least, for our point sets to be \emph{Delone sets}. An \emph{$(r,R)$-Delone set} (on the metric space $(X,d_X)$) is some subset $D \subset X$ which satisfies the following:
\begin{enumerate}
	\item $D$ is \emph{$r$-discrete}, that is, we have that $B_{d_X}(p,r) \cap D = \{p\}$ for all $p \in D$. Given the existence of such an $r$, one says that $D$ is \emph{uniformly discrete}.
	\item $D$ is \emph{$R$-dense}, that is, we have that $B_{d_X}(x,R) \cap D \neq \emptyset$ for all $x \in X$. Given the existence of such an $R$, one says that $D$ is \emph{relatively dense}.
\end{enumerate}
A \emph{Delone set} (on $(X,d_X)$) is of course simply an $(r,R)$-Delone set for some $r,R \in \mathbb{R}_{>0}$, that is some $D \subset X$ which is uniformly discrete and relatively dense in $(X,d_X)$. Just as for tilings we may, if so desired, label the points of a Delone set with some function $l \colon D \rightarrow L$. For a subset $U \subset X$ we define $D(U)$ to be the set of (labelled) points of $D$ which are contained in $U$. For a (partial) isometry $\Phi$ of $X$ whose domain contains $U$ we may define the point set $\Phi(D(U))$ as the set of images of the (labelled points) of $D(U)$ under $\Phi$.

\subsection{Examples of Tilings and Delone Sets} \label{subsect: Examples of Tilings and Delone Sets}

\subsubsection{Voronoi Diagrams and Punctured Tilings}

There is a sense in which the study of tilings and of Delone sets are equivalent ``modulo local redecorations''. Every Delone set $D$ on $(X,d_X)$ induces a \emph{Voronoi tiling} $V(D)$ on $(X,d_X)$. The tiles of $V(D)$ may be identified with elements $x \in D$ and consist of those points of $X$ which do not lie closer to any other $y \in D$ than they do to $x$, such a tile is sometimes called a \emph{Voronoi cell} (of $x$). It is not necessarily true that the Voronoi tiling is a tiling in the sense defined above. It is convenient here to restrict the metric space $(X,d_X)$ somewhat, so as to restrict the topology of the tiles. For example, it is not too hard to see that given a Delone set $D$ of $(\mathbb{R}^d,d_{euc})$, the Voronoi tiles are intersections of a finite number of half-spaces, making them convex polygons and so the Voronoi tiles do define a tiling of $(\mathbb{R}^d,d_{euc})$ with respect to some prototile set.

In the other direction, suppose that we have some tiling $T$ of $(X,d_X)$. A choice of \emph{puncture} (that is, a choice of point) of each prototile $p$ in the prototile set for $T$ induces a puncture of each tile $t \in T$ (we assume for simplicity that there is no ambiguity here, of course the puncture may not be well defined if a tile has non-trivial self-isometries. It is perhaps best to assume from the outset that the punctures are part of the decoration of each prototile). This defines a point pattern $P(T)$, the \emph{set of punctures} of $T$ on $(X,d_X)$, which will be a Delone set given reasonable conditions on the tiling.

Of course, both of these constructions may allow for the embellishment of tiles/points with labels. We see that examples of Delone sets produce examples of tilings and vice versa. Supposing that our punctures were chosen sensibly, these constructions are dual in the following sense: $D$ may be recovered from $V(D)$, and vice versa, ``via locally defined rules'', and similarly for $T$ and $P(T)$. These ideas are made more precise by the notion of MLD equivalence (see Example \ref{ex: MLD}). In the language of Chapter \ref{chap: Patterns of Finite Local Complexity}, the induced patterns of these structures are all equivalent.

\subsubsection{Matching Rules}

There are surprisingly few general constructions of aperiodic tilings or point patterns. However, the constructions that do exist have appeal in and of themselves, possessing a rich (and growing) set of interactions with other areas of mathematics (as well as outside of mathematics). These are the so called matching rule, projection method and substitution (or hierarchical) constructions. We shall give a brief overview of each, although extra attention will be paid to the substitution method in this thesis.

The matching rule tilings are perhaps the most readily accessible, the set-up is rather simple. One starts with some prototile set $\mathfrak{P}$ and a matching rule tiling (for $\mathfrak{P}$) is then simply a tiling of these prototiles. So a matching rule tiling is nothing more than an infinite jigsaw: one has a box of pieces, the prototiles, (although one has access to an infinite supply of each type of piece!) and a matching rule tiling of these pieces is a covering of our space with isometric copies of the original pieces in a way such that distinct pieces do not overlap (on anything more than their boundaries). One often specifies other conditions, for example, one may only wish for the tiles to be, say, translates of the original prototiles. One may also specify extra matching conditions on the boundaries of tiles (or, more generally, only allow certain patches of tiles to be found in the tiling, which may not be a rule which can be forced by conditions on the adjacency of tiles, see \cite{SocTaylor} for such an example).

A famous class of matching rule tilings are the \emph{Wang tilings}. One starts with an initial finite prototile set of \emph{Wang tiles}, whose supports are unit squares of $\mathbb{R}^2$, which possess an assignment of colour to each edge of the square. A Wang tiling of such a prototile set is then a tiling of these prototiles so that the squares meet edge to edge for which, where tiles meet at an edge, they agree on their assigned colours there. One usually only allows for translates of the original prototile set, which of course is more general than allowing all isometries since rotates/reflections of tiles may just be included in the original prototile set if so desired.

Unlike for projection method or hierarchical tilings, meaningful computations for general matching rule tilings, such as computations of topological invariants of the continuous hull of the tiling, are not at the moment possible. There is good reason to expect for this not to change, at least for a class as general as the set of all Wang tilings. The domino problem asks whether it is algorithmically decidable for a given prototile set of Wang tiles to admit a tiling of the plane. It was shown by Berger \cite{Berg} that, in fact, the domino problem is undecidable. If one cannot even in general determine that a given prototile set admits a tiling of the plane, one certainly should not expect for there to be any general method of computation for invariants dependent on the long-range order of such objects!

Nevertheless, matching rule tilings are still of interest in a range of disciplines. The links to logic have already been alluded to above. Given a general Turing machine and some input, one may generate a Wang tile set which admits a tiling of the plane if and only if the Turing machine does not halt \cite{Berg}. In a less abstract setting, one may be interested in the formation of quasicrystals; it is reasonable to expect that they arise, in large part, from very specific sets of local interactions. The question arises of how such local interactions can lead to global structure. Related to this is a connection between matching rule and substitution tilings. The structure of substitution tilings can be forced by some finite set of matching rules \cite{CGS,Moz}. Loosely speaking, this shows how large scale hierarchical structure may be forced by local interactions -- one can not help but feel that there is potential for such a statement to possess rather profound implications.

\subsubsection{Projection Method Patterns}

The projection method is a scheme capable of producing a large class of non-periodic but highly ordered patterns. The idea is to take some periodic structure from a higher dimensional space, take a slice of it and then project this to our ambient space. The hope is that this pattern will inherit the ordered or repetitive nature of the original periodic structure but that the choice of ``slice'' followed by the projection will break the periodicity of it.

It is common to take our periodic structure as a lattice of $\mathbb{R}^N$ e.g., one may consider the integer lattice $\mathbb{Z}^N \subset \mathbb{R}^N$. The ambient space of our resulting point pattern will be $\mathbb{R}^d$, which may be considered as a vector subspace $E \subset \mathbb{R}^N$. The ``slice'' or ``acceptance strip'' will then be some fattening of $E$ in $\mathbb{R}^N$, given as a product of $E$ with some ``window'' of the orthogonal complement of $E$. The points of the lattice $\mathbb{Z}^N$ which fall into this slice may then be projected orthogonally onto $E$ and define a projection method pattern on $\mathbb{R}^d$. Given a nice choice of window, the pattern so defined is repetitive (that is, loosely, finite motifs occur relatively densely throughout the pattern -- see Definition \ref{def: FLC tiling}) and, if $\mathbb{R}^d$ is placed irrationally in $\mathbb{R}^N$, it will also be non-periodic.

We see that the resulting pattern depends on the window and the placement of $\mathbb{R}^d$ within $\mathbb{R}^N$. In some sense, unlike the situation for matching rule tilings, this collection of data gives some handle on the long-range structure of the pattern and computations (e.g., for the cohomology of the hull of the pattern) can be made, see for example \cite{FHK,GHK}. Such computations, along with similar techniques for substitution tilings (see \cite{AP}) show, for example, that the classes of projection method and substitution patterns are distinct: there are projection method patterns that do not arise from substitution rules and vice versa (although some very interesting patterns, such as the Penrose tiling, fall in the intersection of these two classes).

Fixing some ``canonical'' window, the geometry of a projection point pattern is a variable of the placement of the subspace $\mathbb{R}^d$ within $\mathbb{R}^N$. It is then reasonable to expect that number-theoretic issues may to come into play and, indeed, there are already results in this direction, see for example \cite{ArnBerEiIto,HaynesKellyWeiss}.

\subsubsection{Substitution Tilings}

We now turn our attention to substitution tilings. Substitution tilings possess a hierarchical structure, we shall come back to this idea in Section \ref{subsect: Patterns Associated to Tilings} and invariants of these tilings will be considered in Chapter \ref{chap: Pattern-Equivariant Homology of Hierarchical Tilings}.

Let $\mathfrak{P}$ be some finite prototile set of $(\mathbb{R}^d,d_{euc})$ and $\lambda > 1$ (which will be called the \emph{inflation factor}). A \emph{substitution rule} (on $\mathfrak{P}$, with inflation factor $\lambda$) is a rule which associates to each $p \in \mathfrak{P}$ a patch with support equal to the support of $p$ in a way such that this patch inflated by a factor of $\lambda$ is a patch of the original prototiles of $\mathfrak{P}$. Then this rule also acts on patches of tiles by substituting each tile of the patch individually. By following a substitution with an inflation by $\lambda$, we see that repeated iteration of this \emph{inflation rule}, which we shall call $\omega$, grows larger and larger patches of tiles. We shall say that a tiling $T$ of prototiles $\mathfrak{P}$ is an \emph{$\omega$-substitution tiling} if, for each patch $P \subset T$ with bounded support, we have that in fact $P$ is a subpatch of some iteratively inflated tile, that is, $P \subset \omega^n(t)$ where $t$ is an isometric copy of one of our original prototiles $p \in \mathfrak{P}$ (note that it is common to restrict to translates here). We shall temporarily denote the set of such tilings by $\Omega_\omega$. Under reasonable conditions the set $\Omega_\omega$ is non-empty (see \cite{AP}).

Of course, the substitution rule acts on patches by substitution on each tile followed by inflation. Then the inflation rule $\omega$ acts on the tilings of $\Omega_\omega$ and it is not too hard to see that for $T \in \Omega_\omega$ we have that $\omega(T) \in \Omega_\omega$ also. Of great importance for the structure of substitution tilings is the reverse operation: one can show that under reasonable conditions \cite{AP} the inflation map is surjective on $\Omega_\omega$. In other words, for every $T_0 \in \Omega_\omega$ there exists some tiling $T_1$, which is a tiling of $\Omega_\omega$ inflated by a factor of $\lambda$, for which the substitution rule subdivides $T_1$ into $T_0$. The tiles of $T_1$ are based on inflated prototiles of $\mathfrak{P}$ which are called \emph{supertiles}. This process can be iterated, given some tiling $T_0 \in \Omega_\omega$ there is a string of tilings $(T_0,T_1,\ldots )$, where each $T_i$ is a tiling of \emph{super$^i$-tiles}, for which the substitution rule subdivides the tiles of $T_{i+1}$ to $T_i$. This string of tilings may be thought of as specifying a way of grouping the super$^i$-tiles of $T_i$ into super$^{i+1}$-tiles. Loosely speaking, if this grouping can be performed using only local information (so that $T_i$ is forced to be MLD to $T_{i+1}$), then the substitution rule is said to be \emph{recognisable}. The combintorics of how tiles fit together locally is the same at each level of the hierarchy; analogously to the projection method tilings, this hierarchy gives us a handle on the long range structure of the tiling which allows, for example, the computation of the \v{C}ech cohomology groups of the continuous hull of such a tiling, see \cite{AP,BDHS}.

There are several ways in which one may generalise the above setting. It is unnecessarily restrictive that the substitution rule replaces a prototile with a patch which covers the prototile exactly. One only needs that the substituted patch covers the prototile and that the substitution agrees on potential overlaps that may occur near adjacent tiles. See, for example, the famous Penrose kite and dart examples, for which we make computations in \ref{subsect: The Penrose Tiling}. One may also consider hierarchical tilings on spaces other than $(\mathbb{R}^d,d_{euc})$, see, for example, the Pentagonal tilings of Bowers and Stephenson \cite{Pent}, for which we shall calculate invariants for in Subsection \ref{subsect: Pent}. Another direction is to consider not just one, but multiple substitutions on the prototiles set $\mathfrak{P}$, in symbolic dynamics such a set-up is known as an ``$s$-adic system'', see \cite{Ler}, and multi/mixed substitutions of tilings have also been considered \cite{GM}. A general framework for describing a hierarchical structure on a tiling is described in \cite{PF2}. And, despite these interesting generalised settings for hierarchical tilings, there remain many interesting problems for substitution tilings of $\mathbb{R}^1$, the most intensely studied problems those surrounding the Pisot substitution conjecture, see \cite{BS}.

\section{Uniform Spaces} \label{sect: Uniform Spaces}

A well-established approach to studying tilings and Delone sets is to associate to one a certain space (often called the \emph{continuous hull}), which may be constructed as the completion of a metric space defined in terms of the original pattern. We also take this approach here but we find it advantageous to utilise uniform spaces instead of metric spaces for these constructions. We shall provide here the necessary background on uniform spaces required for the next chapter.

\subsection{Uniformities}
The idea behind uniform spaces is rather intuitive. Suppose that we have some set $X$ and we want to define a notion of ``closeness'' on it. A na\"{\i}ve attempt would be the following. A \emph{relation} on $X$ is simply a subset $U \subset X \times X$. When we consider $x$ as being close to $y$ we declare that $(x,y) \in U$, thus defining a relation on $X$. We should clearly have that $x$ is close to $x$ for any $x \in X$, that is, the diagonal $\Delta_X:=\{(x,x) \mid x \in X \} \subset U$ or $U$ is \emph{reflexive}. It also seems intuitive that if $x$ is close to $y$ then $y$ should be close to $x$. That is, $\{(y,x) \mid (x,y) \in U\} =: U^\text{op} = U$ or $U$ is \emph{symmetric}. Given two subsets $U,V \subset X \times X$ we define $U \circ V:=\{ (x,z) \mid (x,y) \in U, (y,z) \in V \text{ for some } y \in X \}$. A relation is said to be \emph{transitive} if $U \circ U \subset U$. It is not clear that our $U$ defining ``closeness'' should be transitive, intuitively it should satisfy some weak form of transitivity. Of course, attempting to define a notion of ``closeness'' on $X$ using just a reflexive and symmetric relation is wholly insufficient. However, the general philosophy is almost right, the idea is to use a \emph{family} of relations $U \subset X \times X$:

\begin{definition} A \emph{uniformity base} on a set $X$ consists of a non-empty collection $\mathcal{U}$ of \emph{entourages} $U \subset X \times X$ which satisfies the following:
\begin{enumerate}
	\item We have that $\Delta_X \subset U$ for all $U \in \mathcal{U}$.
	\item For all $U \in \mathcal{U}$ there exists some $V \in \mathcal{U}$ such that $V^\text{op} \subset U$.
	\item For all $U \in \mathcal{U}$ there exists some $V \in \mathcal{U}$ such that $V \circ V \subset U$.
	\item For all $U,V \in \mathcal{U}$ there exists some $W \in \mathcal{U}$ with $W \subset U \cap V$.
\end{enumerate}
A \emph{uniformity} is a uniformity base which additionally satisfies the following. For all $U \in \mathcal{U}$ and $U \subset V \subset X \times X$ we have that $V \in \mathcal{U}$. A pair $(X,\mathcal{U})$ of a set $X$ with a uniformity on it is called a \emph{uniform space}.
\end{definition} 

We note that there are are three equivalent approaches to defining uniform spaces, via entourages (as above), \emph{uniform covers} and collections of pseudometrics, that is, via \emph{gauge spaces}. We shall concentrate here though only on the definition given above and refer the reader to \cite{Uniform Spaces, HAF} for further details.

Notice that the final condition for a uniformity base to be a uniformity (along with axiom $4$) is simply the requirement that $\mathcal{U}$ is a filter on $(P(X \times X),\subset)$. Of course, it implies a strengthening of axioms $2$, $3$ and $4$, that the given inclusions may be taken as equalities. Naturally, a uniformity base is enough to specify a unique uniformity, by taking the upwards closure. We shall say that a uniformity base $\mathcal{U}$ \emph{generates} and \emph{is a base for} the uniformity $\mathcal{V}$ if $\mathcal{V}$ is the upwards closure of $\mathcal{U}$.

One could have defined the notion of two uniformity bases on a set $X$ as being \emph{equivalent}, that is, one could define $\mathcal{U}_1$ and $\mathcal{U}_2$ to be equivalent if for all $U_1 \in \mathcal{U}_1$ there exists some $U_2 \in \mathcal{U}_2$ with $U_2 \subset U_1$, and vice versa. The filter condition simply replaces the notion of equivalence with that of equality and in that regard the first four axioms for a uniformity are the most important. It is easy to see that any uniformity has a base of \emph{symmetric} entourages, that is, entourages $U$ satisfying $U=U^\text{op}$ (just consider all entourages of the form $U \cap U^\text{op}$). Hence, following the preceding remarks, we could think of a uniform space as being specified by entourages satisfying axioms $1$, $3$ and $4$ above and the stronger axiom in replacement for $2$, that for all $U \in \mathcal{U}$ we have that $U=U^\text{op}$. Then the first axiom says that our ``closeness'' relations are reflexive, the strengthened second one says that they are symmetric, the third says that we have some weak form of transitivity and the final one says that being $U$ and $V$-close should also itself be a ``closeness'' relation for any $U,V \in \mathcal{U}$.

\begin{exmp}\label{ex: met} Any metric space $(X,d_X)$ induces a uniformity on $X$ in the following way. For $r \in \mathbb{R}_{>0}$, define $U_r:=\{(x,y) \mid d_X(x,y)<r\}$. It is easily checked that the collection $\{U_r \mid r \in \mathbb{R}_{>0}\}$ is a uniformity base for $X$, and so we define the uniformity induced by $d_X$ to be the upwards closure of this uniformity base. \end{exmp}

\begin{exmp} \label{ex: top gp} Another important set of examples of uniform spaces are given by topological groups. Given a neighbourhood $N$ of the identity in a topological group $G$, define $U_N:=\{(g,h) \mid g \cdot h^{-1} \in N \}$. The collection of $U_N$ ranging over all neighbourhoods $N$ of the identity generates a uniformity known as the \emph{right uniformity} on $G$. The left uniformity is defined analogously and of course they are equal in the case that $G$ is abelian. Both induce the original topology on $G$ (see below). \end{exmp}

For a uniform space $(X,\mathcal{U}_X)$ and a subset $Y \subset X$, one can define the \emph{induced uniformity} on $Y$ in the obvious way, by taking as entourages for $Y$ restrictions of entourages of $X$ to $Y$, that is, sets $U |_{Y \times Y}$ with $U \in \mathcal{U}_X$.

Given a set $U \subset X \times X$, we define $U(x):=\{y \mid (y,x) \in U\}$. That is, $U$ defines a multi-valued function $U \colon X \rightarrow P(X)$ on $X$, which we may think of as the function which assigns to a point $x$ a ``ball of size $U$ about $x$''. Analogously to metric spaces, we say that $N$ is a \emph{neighbourhood} of $x$ in $(X,\mathcal{U})$ if there exists some entourage $U \in \mathcal{U}$ for which $U(x) \subset N$. One may define a topology in terms of neighbourhoods of points and it is easily checked that the neighbourhoods defined in this way satisfy the necessary properties to define a topology on $X$. It can be shown that every uniform space is completely regular and that, conversely, every completely regular topological space is uniformisable, that is, there exists some uniformity which induces the required topology.

Continuous maps between uniform spaces are defined as those functions which are continuous with respect to the induced topologies. In this way we have full and faithful functors $i \colon \mathbf{Met} \rightarrow \mathbf{Unif}$ and $j \colon \mathbf{Unif} \rightarrow \mathbf{Top}$ between the categories of metric, uniform and topological spaces with continuous maps as morphisms (of course $j \circ i$ is the ``usual'' functor which assigns to a metric space the topological space with neighbourhood topology defined by open balls). These should be thought of as forgetful functors, at each stage some structure is lost. Most importantly here, unlike for general topological spaces, uniform spaces possess the structure necessary to define \emph{uniformly continuous maps}, as well as \emph{Cauchy sequences}, as we shall see shortly.

\begin{notation} Let $f \colon X \rightarrow Y$ be a function between the sets $X$ and $Y$. Given $U \subset X \times X$ we define $f(U):=\{(f(x),f(y)) \mid (x,y) \in U \}$. For $V \subset Y \times Y$ we define $f^{-1}(V):=\{(x,y) \mid (f(x),f(y)) \in V\}$. \end{notation}

We shall, in the next chapter, make frequent use of the simple inclusions and equalities of the lemma below.

\begin{lemma} With the above notation, where $U \subset X \times X$ etc., as appropriate, we have that:
\begin{enumerate}
	\item $f^{-1}(U) \circ f^{-1}(V) \subset f^{-1}(U \circ V)$ (with equality if $f$ is surjective).
	\item $f^{-1}(U \cap V) = f^{-1}(U) \cap f^{-1}(V)$.
	\item $f(U \cap V) \subset f(U) \cap f(V)$.
	\item $(g \circ f)^{-1}(U) = f^{-1}(g^{-1}(U))$.
	\item $(g \circ f)(U) = g(f(U))$.
	\item $U \subset f^{-1}(f(U))$.
	\item For $h=g \circ f$ we have that $f(U) \subset g^{-1}(h(U))$.
\end{enumerate}
\end{lemma}

\begin{definition} Let $(X,\mathcal{U}_X)$ and $(Y,\mathcal{U}_Y)$ be uniform spaces. A function $f \colon X \rightarrow Y$ is said to be \emph{uniformly continuous} if for all $V \in \mathcal{U}_Y$ there exists some $U \in \mathcal{U}_X$ such that $U \subset f^{-1}(V)$. \end{definition}

Of course, just as for continuous maps between topological spaces, it is easy to see that one only needs to check the above condition for elements of some given uniformity bases of $(X,\mathcal{U}_X)$ and $(Y,\mathcal{U}_Y)$ to show that a map is uniformly continuous. It is not hard to show that any uniformly continuous map is also continuous.

\begin{exmp} Let $(X,d_X)$ be a metric space. As in Example \ref{ex: met}, we may generate a uniformity $\mathcal{U}_X$ on $X$ using the sets $U_r$. Notice that $U_r(x)$, the ``ball of size $U_r$ about $x$'', is precisely the open ball $B_{d_X}(x,r)$. Given a metric space $(Y,d_Y)$ (with induced uniformity $\mathcal{U}_Y$), a function $f \colon X \rightarrow Y$ is uniformly continuous with respect to the induced uniformities if and only if for all $V_\epsilon \in \mathcal{U}_Y$ there exists some $U_\delta$ such that $U_\delta \subset f^{-1}(V_\epsilon)$. That is, for any $\epsilon>0$ there exists some $\delta>0$ such that $d_X(x,y)< \delta$ implies that $d_Y(f(x),f(y)) < \epsilon$, which of course is the usual definition of uniform continuity for maps between metric spaces. \end{exmp}

\begin{exmp} Given a topological group $G$ with induced right uniformity $\mathcal{U}_G$ (see Example \ref{ex: top gp}), we have that right multiplication $r_x \colon G \rightarrow G$, defined by $r_x(g):=g \cdot x$ for some chosen $x \in G$, is uniformly continuous. Indeed, given some neighbourhood $N$ of the identity, we have that \[r_x^{-1}(U_N)= \{(g,h) \mid (g \cdot x) \cdot (h \cdot x)^{-1} = g \cdot h^{-1} \in N \} = U_N.\]
\end{exmp}

\subsection{Nets and Completions}
There exists a class of topological spaces, the \emph{sequential spaces}, which, loosely, are such that their topology is completely determined by the specification of which points any given sequence converges to. All first-countable spaces are sequential, in particular all metric spaces are sequential. Familiar concepts such as continuity can be defined in terms of convergence of sequences in this setting. Unfortunately, not all topological spaces are sequential, and indeed there exist non-sequential uniform spaces. However, most of the theory follows through by simply replacing sequences with the more general notion of a \emph{net} \cite{HAF}.

\begin{definition} A \emph{directed set} is a poset $(\Lambda,\leq)$ with the following property: for all $\lambda_1,\lambda_2 \in \Lambda$ there exists some $\mu \in \Lambda$ with $\mu \geq \lambda_1,\lambda_2$. A \emph{net} on a set $X$ is a directed set and a function $x \colon \Lambda \rightarrow X$. We shall often write terms of a net as $x_\lambda:=x(\lambda)$ and the net itself as $(x_\alpha)$.

A \emph{subnet} of a net $(x_\alpha)$ over $(\Lambda,\leq_\Lambda)$ is another net $(y_\beta)$ over some directed set $(M,\leq_M)$ along with a monotone and final function $f \colon M \rightarrow \Lambda$ such that $y_\beta=x_{f(\beta)}$. By $f$ being monotone we mean that $\beta_1 \leq_M \beta_2$ implies that $f(\beta_1) \leq_\Lambda f(\beta_2)$ and by final we mean that for every $\lambda \in \Lambda$ there exists some $\mu \in M$ such that $f(\mu) \geq_\Lambda \lambda$.

A net $(x_\alpha)$ defined on some topological space $X$ is said to \emph{converge} (to $l \in X$) if for all neighbourhoods $N$ of $l$ there exists some $\alpha$ such that $x_\beta \in N$ for all $\beta \geq \alpha$.\end{definition}

\begin{exmp} A sequence is simply a net over the directed set $(\mathbb{N},\leq)$. Subsequences of sequences are subnets, although the converse is far from being true.\end{exmp}

As mentioned, many concepts from topology can be rephrased using nets \cite{HAF}:

\begin{proposition} \begin{enumerate}
	\item A subset $U \subset X$ is open in $X$ if and only if for any convergent net $(x_\alpha) \rightarrow x$ in $X$ with $x \in U$ we have that $(x_\alpha)$ is eventually in $U$, that is, there exists some $\beta$ such that $x_\alpha \in U$ for all $\alpha \geq \beta$.
	\item A subset $C \subset X$ is closed in $X$ if and only if for any convergent net $(x_\alpha) \rightarrow x$, where each $x_\alpha \in C$, we have that $x \in C$.
	\item A function $f \colon X \rightarrow Y$ is continuous at $x$ if and only if, given a convergent net $(x_\alpha) \rightarrow x$ in $X$, we have that $f(x_\alpha) \rightarrow f(x)$ in $Y$.
	\item A topological space $X$ is compact if and only if every net has a convergent subnet.
\end{enumerate}
\end{proposition}

\begin{definition} Let $(X,\mathcal{U})$ be a uniform space. A net $(x_\alpha)$ on $X$ is \emph{Cauchy} if, for every entourage $U \in \mathcal{U}$, there exists some $\gamma$ such that $(x_\alpha,x_\beta) \in U$ for all $\alpha,\beta \geq \gamma$. We shall say that $(X,\mathcal{U})$ is \emph{complete} if every Cauchy net converges. \end{definition}

\begin{exmp} Let $(X,d_X)$ be a metric space, with induced uniformity $\mathcal{U}$, and $(x_\alpha)$ be a net over the directed set $(\mathbb{N},\leq)$ in $X$ (that is, $(x_\alpha)$ is a sequence in $X$). Then $(x_\alpha)$ is Cauchy if for every entourage $U_\epsilon \in \mathcal{U}$ there exists some $\gamma$ such that $(x_\alpha,x_\beta) \in U_\epsilon$ for all $\alpha,\beta \geq \gamma$. In other words, for all $\epsilon > 0$ there exists some $\gamma$ such that $d_X(x_\alpha,x_\beta)<\epsilon$ for all $\alpha,\beta \geq \gamma$, which is the usual definition of a Cauchy sequence in a metric space.\end{exmp}

Analogously to the situation for (pseudo)metric spaces, every uniform space $(X,\mathcal{U}_X)$ has a \emph{Hausdorff completion}, that is, a complete Hausdorff uniform space $\overline{(X,\mathcal{U}_X)}$ along with a uniformly continuous map $i \colon (X,\mathcal{U}_X) \rightarrow \overline{(X,\mathcal{U}_X)}$ which satisfies the following universal property: any uniformly continuous map $f \colon \ab (X,\mathcal{U}_X) \rightarrow (Y,\mathcal{U}_Y)$, where $(Y,\mathcal{U}_Y)$ is a complete Hausdorff uniform space, extends to a unique uniformly continuous map $\overline{f} \colon \overline{(X,\mathcal{U}_X)} \rightarrow (Y,\mathcal{U}_Y)$, that is, a uniformly continuous map with $f=\overline{f} \circ i$. As is usual for objects defined by universal properties, the Hausdorff completion of a uniform space is uniquely defined up to isomorphism in the category of uniform spaces (with uniformly continuous maps as morphisms). The Hausdorff completion defines a functor, by taking completions and extending maps, and is left adjoint to the inclusion of the full subcategory of complete Hausdorff uniform spaces into the category of uniform spaces, that is, the complete Hausdorff uniform spaces are a reflective subcategory.

\subsection{Kolmogorov Quotients}

Given a topological space $X$, say that two points $x,y \in X$ are \emph{topologically indistinguishable} if they have the same set of neighbourhoods. Two points are \emph{topologically distinguishable} if they are not topologically indistinguishable, that is, there exists some $U$ which is a neighbourhood of one point but not the other. A space for which distinct points are topologically indistinguishable is called \emph{$T_0$} or \emph{Kolmogorov} (see, for example, \cite{HAF}).

It is instructive to think of non-$T_0$ spaces simply as $T_0$ spaces for which there are certain places where multiple points ``inhabit the same location''. Given a space $X$, define the equivalence relation $\sim$ on $X$ by letting $x \sim y$ if $x$ and $y$ are topologically indistinguishable. Then the \emph{Kolmogorov quotient} $X^{KQ}:=X / \sim$ is a $T_0$ space. Furthermore, it is easy to show that the space $X^{KQ}$ along with the quotient map $\pi^{KQ} \colon X \rightarrow X^{KQ}$ satisfies the following universal property: for every continuous map $f \colon X \rightarrow Y$ where $Y$ is a $T_0$ space, there exists a unique map $f^{KQ} \colon X^{KQ} \rightarrow Y$ lifting $f$, that is, $f= f^{KQ} \circ \pi^{KQ}$. This makes the $T_0$ spaces with continuous maps a reflective subcategory of $\mathbf{Top}$.

For a uniform space $(X,\mathcal{U})$ two points $x$ and $y$ are topologically indistinguishable if and only if for every entourage $U \in \mathcal{U}$ we have that $y \in U(x)$. We say that $(X,\mathcal{U})$ is \emph{separated} if the intersection of all entourages is the diagonal $\Delta_X$ (which also implies that the uniform space is Hausdorff). The Kolmogorov quotient $(X,\mathcal{U})^{KQ}$ of a uniform space is itself a (separated) uniform space, one takes the uniformity $\mathcal{U}^{KQ}$ defined by the entourages $U^{KQ}:=\{([x]^{KQ},[y]^{KQ}) \mid (x,y) \in U\}$ for $U \in \mathcal{U}$. It is easily checked that this defines a uniformity on $X^{KQ}$ which induces the correct topology. Of course, the map $\pi^{KQ} \colon (X,\mathcal{U}) \rightarrow (X,\mathcal{U})^{KQ}$ is uniformly continuous and given a uniformly continuous map $(X,\mathcal{U}_X) \rightarrow (Y,\mathcal{U}_Y)$ to a separated uniform space, we have that $f^{KQ}$ is uniformly continuous. Taking the Kolmogorov quotient factors through the operation of taking the Hausdorff completion:

\begin{proposition} Let $(X,\mathcal{U}_X)$ be a uniform space and $i \colon (X,\mathcal{U}_X) \rightarrow \overline{(X,\mathcal{U}_X)}$ be a Hausdorff completion of it. Then $i^{KQ} \colon (X,\mathcal{U}_X)^{KQ} \rightarrow \overline{(X,\mathcal{U}_X)}$ is a Hausdorff completion of $(X,\mathcal{U}_X)^{KQ}$ for which $i=i^{KQ} \circ \pi^{KQ}$. \end{proposition}

\begin{proof} This follows easily from the universal properties of the completion and Kolmogorov quotient. Since the Hausdorff completion of a space is $T_0$, by the universal property of the Kolmogorov quotient we have a map $i^{KQ} \colon (X,\mathcal{U}_X)^{KQ} \rightarrow \overline{(X,\mathcal{U}_X)}$ with $i=i^{KQ} \circ \pi^{KQ}$. This defines a uniformly continuous map from the Kolmogorov quotient to a complete Hausdorff uniform space, so one just needs to check the universal property for the completion. Let $f \colon (X,\mathcal{U})^{KQ} \rightarrow (Y,\mathcal{U}_Y)$ be a uniformly continuous map to another complete Hausdorff uniform space. This defines a uniformly continuous map $f \circ \pi^{KQ} \colon (X,\mathcal{U}_X) \rightarrow (Y,\mathcal{U}_Y)$. It follows from the universal property of the completion that there exists a uniformly continuous map $g \colon \overline{(X,\mathcal{U}_X)} \rightarrow (Y,\mathcal{U}_Y)$ with $f \circ \pi^{KQ} = g \circ i = g \circ i^{KQ} \circ \pi^{KQ}$. It follows from the universal property of the Kolmogorov quotient that $f = g \circ i^{KQ}$, as desired. Furthermore, the map $g$ is unique for, given any other such map, there would be two distinct extensions of the map $f \circ \pi^{KQ}$ through the completion map $i$. \end{proof}

It will be helpful to have a slightly more general notion of a completion of a uniform space, which may not necessarily be separated:

\begin{definition} The uniform space $(Y,\mathcal{U}_Y)$ is a \emph{completion} of $(X,\mathcal{U}_X)$ if there exists a uniformly continuous map $i \colon (X,\mathcal{U}_X) \rightarrow (Y,\mathcal{U}_Y)$ such that $\pi^{KQ} \circ i \colon \ab (X,\mathcal{U}_X)  \rightarrow (Y,\mathcal{U}_Y)^{KQ}$ is a Hausdorff completion. \end{definition}

Of course, a Hausdorff completion is always a completion in this sense. Since Hausdorff completions are always uniformly isomorphic, general completions are also uniformly isomorphic ``modulo topologically indistinguishable points''.

\begin{proposition}\label{prop:kolmog compl} We have that $(Y,\mathcal{U}_Y)$ is a completion of $(X,\mathcal{U}_X)$ if and only if there exists a map $i \colon (X,\mathcal{U}_X)\rightarrow (Y,\mathcal{U}_Y)$ such that:
\begin{enumerate}
	\item $(Y,\mathcal{U}_Y)$ is complete.
	\item $i$ has dense image.
	\item $(\pi^{KQ} \circ i)^{KQ}$ is injective.
	\item $(\pi^{KQ} \circ i)^{KQ}$ is a uniform isomorphism onto its image.
\end{enumerate}
\end{proposition}

\begin{proof} If $i$ is a completion, by the above proposition we have that $(\pi^{KQ} \circ i)^{KQ} \colon \ab (X,\mathcal{U}_X)^{KQ} \rightarrow (Y,\mathcal{U}_Y)^{KQ}$ is a Hausdorff completion. This is a uniformly continuous map between separated uniform spaces and in this case it is well known that such a map is a Hausdorff completion map if and only if the above conditions are satisfied. 

Assuming the axiom of choice, it is easy to show that $(Y,\mathcal{U}_Y)$ is complete if and only if its Kolmogorov quotient is. The map $i$ has dense image if and only if $(\pi^{KQ} \circ i)^{KQ}$ does, so the result follows. \end{proof}

\begin{lemma}\label{lem:unif iso} Let $f \colon (X,\mathcal{U}_X) \rightarrow (Y,\mathcal{U}_Y)$ be a function which sends topologically indistinguishable points to topologically indistinguishable points so that $(\pi^{KQ} \circ f)^{KQ}$ is well defined. We have that $(\pi^{KQ} \circ f)^{KQ}$ is injective and if and only if $f$ maps topologically distinguishable points to topologically distinguishable points. In this case $(\pi^{KQ} \circ f)^{KQ}$ is a uniform isomorphism onto its image if and only if:
\begin{enumerate}
	\item for all $V \in \mathcal{U}_Y$ there exists some $U \in \mathcal{U}_X$ such that $f(U) \subset V$.
	\item For all $U \in \mathcal{U}_X$ there exists some $V \in \mathcal{U}_Y$ such that $f^{-1}(V) \subset U$.
\end{enumerate}
\end{lemma}

\begin{proof} It is easy to see that the injectivity of $(\pi^{KQ} \circ f)^{KQ}$ is equivalent to the condition given for $f$. Almost by definition an injective function $f \colon (X,\mathcal{U}_X) \rightarrow (Y,\mathcal{U}_Y)$ is a uniform isomorphism onto its image if and only if the two given conditions hold. The result follows easily from the fact that $(\pi^{KQ} \circ f)^{KQ}([x]^{KQ}) = [f(x)]^{KQ}$ and $(x,y) \in U$ if and only if $([x]^{KQ},[y]^{KQ}) \in U^{KQ}$ for entourages $U$. \end{proof}

\section{Topological Background}

We shall assume (and already have assumed) that the reader is comfortable with basic point-set topology. The tiling spaces to appear will be constructed as \emph{inverse limits}, which we introduce here. A useful cohomology theory to apply to this class of spaces is the \v{C}ech cohomology, for which we shall provide here the minimal necessary details. In Chapter \ref{chap: Poincare Duality for Pattern-Equivariant Homology} we shall link these groups, through a Poincar\'{e} duality result, to certain geometrically motivated homology groups defined on the tiling. These \emph{pattern-equivariant homology groups} are defined in terms of non-compactly supported singular (or cellular) chains, which we review here.

\subsection{Inverse and Direct Limits}

Although we shall only be interested in inverse limits of topological spaces, the notion of an inverse limit of an inverse system makes sense in any category (see e.g., \cite{MacLane}).

\begin{definition} Let $\mathcal{C}$ be some category. An \emph{inverse system} (in $\mathcal{C}$) consists of a directed set $(\Lambda,\leq)$ along with objects $X_\lambda \in Ob(\mathcal{C})$, one for each $\lambda \in \Lambda$, and morphisms $\pi_{\lambda,\mu} \colon X_\mu \rightarrow X_\lambda$ for each relation $\lambda \leq \mu$. The morphisms are required to satisfy $\pi_{\lambda,\lambda} = Id_{X_\lambda}$ and $\pi_{\lambda,\nu}=\pi_{\lambda,\mu} \circ \pi_{\mu,\nu}$ for all $\lambda \leq \mu \leq \nu$ (in short, one could say that an inverse system is a contravariant functor from the obvious category associated to the directed set to $\mathcal{C}$).

An \emph{inverse limit} of an inverse system is an object $X_\infty \in Ob(\mathcal{C})$ along with morphisms $\pi_{\lambda,\infty} \colon X_\infty \rightarrow X_\lambda$ with $\pi_{\lambda,\infty}=\pi_{\lambda,\mu} \circ \pi_{\mu,\infty}$ for all $\lambda \leq \mu$ and satisfying the following universal property: for any other such choice of object $Y_\infty \in Ob(\mathcal{C})$ with corresponding morphisms $\rho_{\lambda,\infty}$, there exists a unique morphism $f \colon Y_\infty \rightarrow X_\infty$ such that $\pi_{\lambda,\infty} \circ f = \rho_{\lambda,\infty}$. \end{definition}

As is usual for universal properties, it is easy to show that an inverse limit of an inverse system, if it exists, is unique up to isomorphism. It turns out that a sufficient (and necessary) condition for all inverse limits to exist in $\mathcal{C}$ is that $\mathcal{C}$ has equalisers and (small) products \cite{MacLane}. The proof is constructive and for the category $\mathbf{Top}$ of topological spaces with continuous maps it produces the following:

\begin{definition} Let $(X_\lambda,\pi_{\lambda,\mu})$ be an inverse system in $\mathbf{Top}$. Then the space \[\varprojlim (X_\lambda,\pi_{\lambda,\mu}):= \{(x_\lambda) \in \prod_{\lambda \in \Lambda} X_\lambda \mid \pi_{\lambda,\mu}(x_\mu)=x_\lambda \},\] equipped with the subspace topology relative to $\prod X_\lambda$, is an inverse limit of $(X_\lambda,\pi_{\lambda,\mu})$, which we shall refer to as \emph{the} inverse limit. \end{definition}

One can similarly define \emph{directed systems} and \emph{colimits} of directed systems in a category. We recall here only the definition of a colimit, often called \emph{the direct limit}, of a directed system of groups:

\begin{definition} Let $((G_\lambda,\cdot_\lambda),f_{\lambda,\mu})$ be a directed system (over $(\Lambda, \leq)$) in the category $\mathbf{Grp}$ of groups (the homomorphisms are indexed here as $f_{\lambda,\mu} \colon (G_\lambda,\cdot_\lambda) \rightarrow (G_\mu,\cdot_\mu)$). Define an equivalence relation $\sim$ on $\coprod_{\lambda \in \Lambda} G_\lambda$ by tail equivalence, that is by setting $g_\lambda \sim g_\mu$ if and only if there exists some $\nu \in \Lambda$ such that $f_{\lambda,\nu}(g_\lambda)=f_{\mu,\nu}(g_\mu)$. One can define a multiplication $\cdot_\infty$ on $\coprod_{\lambda \in \Lambda} G_\lambda / \sim$ in the obvious way, multiplying two equivalence classes by multiplying any two representatives which both belong to a common $G_\lambda$. In this way we define the group \[\varinjlim((G_\lambda,\cdot_\lambda),f_{\lambda,\mu}):= (\coprod_{\lambda \in \Lambda} G_\lambda / \sim,\cdot_\infty),\] which is a direct limit of $((G_\lambda,\cdot_\lambda),f_{\lambda,\mu})$. We call it \emph{the direct limit} of $((G_\lambda,\cdot_\lambda),f_{\lambda,\mu})$ and, abusing notation, write it for short as $\varinjlim (G_\lambda,f_{\lambda,\mu})$. \end{definition}

\begin{exmp} Take the direct system of groups $S:=(\mathbb{Z} \rightarrow_{\times 2} \mathbb{Z} \rightarrow_{\times 2} \ldots)$ over $(\mathbb{N}_0,\leq)$ with connecting maps $f_{n,n+1} \colon \mathbb{Z} \rightarrow \mathbb{Z}$ the times two map $n \mapsto 2n$. Then the direct limit of $S$ is isomorphic to the group $(\mathbb{Z}[1/2],+)$ of dyadic rationals under addition, that is, the subgroup of $\mathbb{Q}$ of numbers of the form $a/2^b$ where $a \in \mathbb{Z}$ and $b \in \mathbb{N}_0$. The isomorphism maps the set $\coprod_{n \in \mathbb{N}_0} \mathbb{Z} / \sim$ into $\mathbb{Z}[1/2]$ by sending the $n$-th copy of $\mathbb{Z}$ into the subgroup $1/2^n \mathbb{Z} \subset \mathbb{Q}$ via the map $k \mapsto k/2^n$. \end{exmp}

\subsection{\v{C}ech Cohomology}

A cohomology theory well suited to studying tiling spaces is the \emph{\v{C}ech cohomology} (see \cite{ES,Sadun2}). One reason for this is that tiling spaces can often be expressed as an inverse limit of CW-complexes, and \v{C}ech cohomology is a \emph{continuous} contravariant functor which agrees with the singular cohomology of spaces which are homotopy-equivalent to CW-complexes:

\begin{proposition} Let $X_\infty$ be the inverse limit of an inverse system of spaces $(X_\lambda,\pi_{\lambda,\mu})$. Then there is an isomorphism $\check{H}^\bullet(X_\infty) \cong \varinjlim (\check{H}^\bullet(X_\lambda),\pi_{\lambda,\mu}^*)$ between the \v{C}ech cohomology of the inverse limit and the direct limit of the induced direct system of the \v{C}ech cohomologies of the spaces $X_\lambda$.

Given a space $X$ which is homotopy-equivalent to a CW-complex, we have a natural isomorphism $\check{H}^\bullet(X) \cong H^\bullet(X)$ between the \v{C}ech cohomology and singular cohomology of $X$. \end{proposition}

\begin{exmp} Let $\mathbb{S}_2^1$ be the one-dimensional dyadic solenoid, that is, the inverse limit of the inverse system $S=(S^1 \leftarrow_{\times 2} S^1 \leftarrow_{\times 2} \ldots )$ of the circle under the doubling map. By the above we have that the degree one \v{C}ech cohomology is isomorphic to \[\check{H}^1(\mathbb{S}_2^1) \cong \check{H}^1(\varprojlim(S)) \cong \varinjlim(\check{H}^1(S^1),f^*) \cong \varinjlim(H^1(S^1),f^*) \cong \varinjlim(\mathbb{Z},\times 2) \cong \mathbb{Z}[1/2].\] \end{exmp}

\subsection{Borel-Moore Homology and Poincar\'{e} Duality} \label{subsect: Borel-Moore Homology and Poincare Duality}

Given a compact oriented $d$-dimensional manifold $M$, there is a famous \emph{Poincar\'{e} duality} between the singular homology and cohomology of $M$. In more detail, there exists a fundamental class $[M] \in H_d(M)$ for which the cap product with $[M]$ induces an isomorphism $H^k(M) \cong H_{d-k}(M)$. If $M$ is non-compact then of course nothing analogous to a fundamental class exists in the singular homology of $M$. By changing the cohomology to cohomology with compact support we obtain duality $H_c^k(M) \cong H_{d-k}(M)$. An alternative is to consider the \emph{Borel-Moore homology} of $M$. In the context to be discussed here, describing this homology using chains, this homology is sometimes referred to as \emph{homology of infinite chains}, \emph{homology with closed support} or \emph{locally finite homology} \cite{Ends}. There then exists a Borel-Moore fundamental class $[M] \in H_d^{\text{BM}}(M)$ for which, analogously to the compact case, one obtains duality $H^k(M) \cong H_{d-k}^{\text{BM}}(M)$ by capping with this fundamental class; see \cite{Bredon} for the sheaf-theoretic point of view on this.

To limit discussion we shall consider here only the case where our (co)chain complexes are taken over $\mathbb{Z}$-coefficients. Let $\Delta^k$ be the standard $k$-simplex, that is the subspace $\{(x_0,\ldots,x_k) \in \mathbb{R}^{k+1} \mid \sum_{i=0}^k x_i = 1 \text{ and } x_i \geq 0 \} \subset\mathbb{R}^{k+1}$. A singular $k$-simplex of a space $X$ is a continuous map $\sigma \colon \Delta^k \rightarrow X$. The sets $\Delta^k(X)$ of singular $k$-simplexes define \emph{singular chain groups}, denoted by $S_k(X)$, as the free abelian groups with basis $\Delta^k(X)$. That is, a singular chain $\sigma \in S_k(X)$ is a formal sum $\sum_{i=1}^n c_i \sigma_i$ of singular $k$-simplexes, where each $c_i \in \mathbb{Z}$. We have boundary operators $\partial_k \colon S_k(X) \rightarrow S_{k-1}(X)$ for each $k \in \mathbb{N}_0$ defined by the well-known formula on each singular simplex by $\partial_k(\sigma) = \sum_{n=0}^k (-1)^n \sigma |_{[e_0,\ldots,\hat{e_n},\ldots,e_k]}$ and extended to $S_k(X)$ linearly. Here, $\sigma |_{[e_0,\ldots,\hat{e_n},\ldots,e_k]}$ is the singular $(k-1)$-simplex defined by restricting to the face of $\Delta^k$ with vertex set $\{e_0,\ldots,e_k \} - \{e_n\}$ which itself may be identified with $\Delta^{k-1}$ in a canonical way. We have that $\partial_{k-1} \circ \partial_k = 0$, that is, $S_\bullet(X):=(0 \leftarrow_{\partial_0} S_0(X) \leftarrow_{\partial_1} S_1(X) \leftarrow_{\partial_2} \cdots)$ is a chain complex. We define the singular homology $H_\bullet(X):=H(S_\bullet(X))$ to be the homology of this chain complex. Given a continuous map $f \colon X \rightarrow Y$ and a singular simplex $\sigma \colon \Delta^n \rightarrow X$ we have that $f_*(\sigma):=f \circ \sigma$ is a singular simplex of $Y$. This defines an induced map $f_\bullet \colon H_\bullet(X) \rightarrow H_\bullet(Y)$ in the obvious way and makes singular homology a functor $H_\bullet(-) \colon \mathbf{Top} \rightarrow \mathbf{Ab}$ to the category of abelian groups. In fact, $H_\bullet(-)$ is a homotopy invariant, that is, we have that $f_* = g_*$ if $f$ and $g$ are homotopy equivalent.

The singular Borel-Moore homology will be defined similarly, but where we allow an infinite sum of chains. Of course, a restriction needs to be made to ensure that the boundary map is well defined. Let $\sigma = \sum_{\sigma_i \in \Delta^k(X)} c_i \sigma_i$ be a (possibly infinite) sum of singular simplexes on $X$, where each $c_i \in \mathbb{Z}$. One may consider $\sigma$ as an element of $\prod_{\Delta^k(X)} \mathbb{Z}$. For a set $S \subset X$, we define $\sigma^S:= \sum c_i^S \sigma_i$ where $c_i^S = c_i$ if $S \cap \sigma_i(\Delta^k) \neq \emptyset$ and $c_i^S=0$ otherwise. We define $\sigma^x:=\sigma^{\{x\}}$ for $x \in X$. A sum of singular simplexes $\sigma = \sum_{\sigma_i \in S_k(X)} c_i \sigma_i$ will be called a \emph{singular Borel-Moore $k$-chain} if the number of non-zero coefficients $c_i^K$ is finite for each compact $K \subset X$.\footnote{A similar approach, see e.g., \cite{Ends}, is to consider chains for which, for all $x \in X$, there exists some open neighbourhood $U \subset X$ of $x$ such that the number of non-zero coefficients of $\sigma^U$ is finite. In general, this is a stronger condition than the one given here. We shall only be interested in Borel-Moore chains of locally compact spaces and of course the two definitions coincide in this case.} For example, a singular Borel-Moore $k$-chain on a proper metric space $(X,d_X)$ corresponds precisely to a (possibly) infinite sum of singular $k$-simplexes for which each bounded subset $B \subset X$ intersects the images of only finitely many singular simplexes of the chain.

The singular Borel-Moore $k$-chains form groups $S_k^{\text{BM}}(X)$ which fit into a chain complex $S_\bullet^{\text{BM}}(X) :=(0 \leftarrow_{\partial_0} S_0^{\text{BM}}(X) \leftarrow_{\partial_1} S_1^{\text{BM}}(X) \leftarrow_{\partial_2} \cdots )$ by extending the usual boundary map to each $S_k^{\text{BM}}(X)$ in the obvious way. Notice that the boundary map is well defined since, given a singular $(k-1)$-simplex $\sigma$, we have that $\sigma(\Delta^{k-1})$ is compact and a singular simplex is the face of only finitely many singular $k$-simplexes of a given singular Borel-Moore $k$-chain.

We thus define the \emph{singular Borel-Moore homology} $H_\bullet^{\text{BM}}(X)$ of a topological space $X$ as the homology of the chain complex $S_\bullet^{\text{BM}}(X)$. It is \emph{not} a functor from $\mathbf{Top}$ to $\mathbf{Ab}$, the complication lies in the fact that, given a map $f \colon X \rightarrow Y$, the ``obvious'' chain $f_*(\sigma)$ may not be a Borel-Moore chain (or even well defined) if $f$ maps infinitely many simplexes with non-zero coefficient to some compact $K$. For example, let $f$ be the unique map $f \colon \mathbb{R} \rightarrow \{\bullet\}$ to the one point space and $\sigma$ be the singular Borel-Moore $0$-chain assigning coefficient $1$ to each point of $\mathbb{Z} \subset \mathbb{R}$. To avoid this situation, one restricts to proper maps. We say that $f \colon X \rightarrow Y$ is \emph{proper} if $f^{-1}(K)$ is compact for every compact $K \subset Y$. This supposed deficiency, of $H_\bullet^{\text{BM}}(-)$ only being functorial over proper maps, is in fact related to its utility: it is a \emph{proper homotopy invariant} which makes it a more useful invariant in the study of certain properties of non-compact spaces. Two continuous maps $f$ and $g$ are \emph{properly homotopic} if there exists some proper $F \colon [0,1] \times X \rightarrow Y$ with $F(0,-)=f$ and $F(1,-)=g$ (which implies that $f$ and $g$ are themselves proper). It can then be shown that singular Borel-Moore homology is a functor $H_\bullet^{\text{BM}}(-) \colon \mathbf{PropTop} \rightarrow \mathbf{Ab}$ from the category of topological spaces with proper maps as morphisms to the category of abelian groups and one has that $f_*=g_*$ if $f$ and $g$ are properly homotopic.

\begin{exmp} Homeomorphisms are proper and so $H_\bullet^{\text{BM}}(-)$ is a homeomorphism invariant. It is easy to see that $H_\bullet^{\text{BM}}(X) \cong H_\bullet(X)$ when $X$ is compact.

Borel-Moore homology is not a homotopy invariant. For example, we have that $H_0^{\text{BM}}(\mathbb{R}) \cong 0$ (which one may interpret as saying that $0$-chains may be ``pushed to infinity'') and $H_1^{\text{BM}}(\mathbb{R}) \cong \mathbb{Z}$ (generated by a fundamental class for $\mathbb{R}$) whereas, for the homotopy equivalent (but not \emph{properly} homotopy equivalent) one point space, we have that $H_0^{\text{BM}}(\{\bullet\}) \cong \mathbb{Z}$ and $H_1^{\text{BM}}(\{\bullet\}) \cong 0$. \end{exmp}

The \emph{singular cochain complex} $S^\bullet(X)$ of a topological space $X$ is defined to be the dual of the singular chain complex. That is, $S^\bullet(X)$ is the cochain complex with cochain groups $S^k(X):= \text{Hom}_\mathbb{Z}(S_k(X),\mathbb{Z})$, so a singular $k$-cochain $\psi$ assigns to each singular simplex $\sigma \colon \Delta^k \rightarrow X$ some coefficient $\psi(\sigma) \in \mathbb{Z}$. The coboundary maps $\delta_k$ are defined on $k$-cochains $\psi$ by $\delta(\psi)(\sigma):=\psi(\partial(\sigma))$. The \emph{singular cohomology} of $X$ is defined to be the cohomology of this cochain complex.

The (co)homology groups introduced here also have relative versions. Let $(X,A)$ be a pair of a topological space $X$ along with a subspace $A \subset X$. Then one may define the \emph{relative singular homology groups} $H_\bullet(X,A)$ by considering the chain complex $S_\bullet(X,A)$ with chain groups $S_k(X)/S_k(A)$ and boundary maps induced by the usual singular boundary maps. One may similarly define the \emph{relative singular cohomology groups} $H^\bullet(X,A)$. It is not necessarily true that a singular Borel-Moore chain $\sigma \in S_k(A)$ is a singular Borel-Moore chain of $X$ unless $A$ is closed in $X$, but in such a case one may define the \emph{relative singular Borel-Moore homology groups} $H_\bullet^{\text{BM}}(X,A)$.

There is an operation, called the \emph{cap product}, which produces a singular chain given a combination of a singular chain and cochain. Loosely, one may think of the cap product as integrating a cochain over a given chain. More formally, it is defined as follows. Let $\sigma \colon \Delta^p \rightarrow X$ be a singular $p$-simplex and $\psi$ be a singular $q$-cochain where $p \geq q$. The sets of vertices $\{e_0,\ldots,e_q\}$ and $\{e_q,\ldots,e_p\}$  span simplexes which may be canonically identified with $\Delta^q$ and $\Delta^{p-q}$, respectively. We have the singular $q$-simplex $\sigma|_{[e_0,\ldots,\ldots,e_q]}$, the \emph{$q$-th front face} of $\sigma$, by restricting $\sigma$ to the former face of $\Delta^n$, and similarly the $(p-q)$-simplex $\sigma|_{[e_q,\ldots,\ldots,e_p]}$, the \emph{$(p-q)$-th back face} of $\sigma$, by restricting to the latter face. We define the singular $(p-q)$-chain $\sigma \frown \psi:=\psi(\sigma|_{[e_0,\ldots,\ldots,e_q]})\sigma|_{[e_q,\ldots,\ldots,e_p]}$. By extending linearly to the whole group $S_p(X)$, we thus define the homomorphism $\frown \colon S_p(X) \times S^q(X) \rightarrow S_{p-q}(X)$. It is routine to check that $\partial(\sigma \frown \psi)=(-1)^q(\partial(\sigma) \frown \psi - \sigma \frown \delta(\psi))$ so that $\frown$ defines a homomorphism $\frown \colon H_p(X) \times H^q(X) \rightarrow H_{p-q}(X)$, called the \emph{cap product}. One may analogously define the cap product for the Borel-Moore homology groups $\frown \colon H_p^{\text{BM}}(X) \times H^q(X) \rightarrow H^{\text{BM}}_{p-q}(X)$. There are also relative versions $\frown \colon H_p(X,A) \times H^q(X) \rightarrow H_{p-q}(X,A)$ and $\frown \colon H_p(X,A) \times H^q(X,A) \rightarrow H_{p-q}(X)$ (and similarly for the Borel-Moore homology) given by extending the above formulas in the obvious way.

Say that a $k$-cochain $\psi \in S^k(X)$ is \emph{compactly supported} if there exists some compact $K \subset X$ such that $\psi(\sigma)=0$ for all singular chains $\sigma$ with support $\sigma(\Delta^k)$ contained in $X-K$. It is easy to see that the coboundary of a compactly supported singular $k$-cochain is also compactly supported and so one may define the cochain complex $S_c^\bullet:=(0 \rightarrow S_c^0(X) \rightarrow_{\delta_1} S_c^1(X) \rightarrow_{\delta_2} \ldots)$ of compactly supported singular cochains and the \emph{compactly supported singular cohomology} $H_c^\bullet(X)$ of $X$ as the cohomology of this cochain complex. A compactly supported cochain is precisely an element of $\varinjlim_K S^\bullet(X,X-K)$ taken over the directed set of compact subspaces $K \subset X$, directed by inclusion (notice that we have canonical inclusions $S^\bullet(X,X-K) \hookrightarrow S^\bullet(X,X-L)$ for $K \subset L$). We have that $\varinjlim_K H^\bullet(X,X-K) \cong H_c^\bullet(X)$. Just as for singular Borel-Moore homology, compactly supported cohomology is not functorial over all continuous maps, only over proper ones. Of course, it is isomorphic to singular cohomology on compact spaces.

\begin{theorem}[Poincar\'{e} Duality] Let $M$ be a $d$-dimensional manifold (without boundary). Then
\begin{enumerate}
	\item $H_c^k(M) \cong H_{d-k}(M)$.
	\item $H^k(M) \cong H_{d-k}^{\text{BM}}(M)$.
\end{enumerate} \end{theorem}

\begin{proof} For item $1$ see \cite{Hat}. The idea is to define a duality map by capping with the fundamental classes $[K] \in H_d(M,M-K)$ for compact $K$ and applying a limiting process over all such compact subspaces $K \subset M$. The duality is easy to show for convex subsets of $\mathbb{R}^n$ and the general result follows from ``stitching'' together the duality isomorphisms using diagrams of Mayer-Vietoris sequences.

The second isomorphism is induced by taking the cap product with a Borel-Moore fundamental class for the manifold. We will not present the details here since most of the main ideas will appear in the proof of a Poincar\'{e} duality result for pattern-equivariant (co)chains in Chapter \ref{chap: Poincare Duality for Pattern-Equivariant Homology}.\end{proof}

\subsection{CW-Complexes and Cellular Homology} \label{subsect: CW-Complexes and Cellular Homology}

In practice, one never computes homology directly from the singular complexes, the chain groups are unwieldy and of huge cardinality. However, many spaces encountered are \emph{CW-complexes}, spaces which possess a nice decomposition into cells, which allows for a more combinatorial calculation of these invariants, that is, through the \emph{cellular homology groups}, which are isomorphic to the singular homology groups. We refer the reader to \cite{Hat} for extra details on CW-complexes and cellular homology.

For $k \in \mathbb{N}$ let $D^k$ be the closed unit disc of $\mathbb{R}^k$ and $S^{k-1}$ be the unit sphere, the boundary of $D^k \subset \mathbb{R}^k$; for dimension $0$ we define $D^0$ to be the one point space $\{\bullet\} \subset \{\bullet\} =: \mathbb{R}^0$ so that $S^{-1}=\emptyset \subset \mathbb{R}^0$. Given a topological space $X$ a \emph{closed $k$-cell} is a subset $e_c \subset X$ together with a continuous (and surjective) map $c \colon D^k \rightarrow e_c$, called the \emph{characteristic map of $e_c$}, such that the restriction of $c$ to $\text{int}(D^k) \subset \mathbb{R}^k$ (that is, to the open unit disc) is a homeomorphism. A subset $\alpha \subset X$ which is homeomorphic to $\mathbb{R}^k$ is called an \emph{open $k$-cell}. A \emph{CW-decomposition} of a Hausdorff space $X$ is a partition of $X$ into open cells for which:
\begin{enumerate}
	\item for each open $k$-cell $\alpha \subset X$ of the partition there exists a closed cell with characteristic map $c$ with $\alpha = c(\text{Int}(D^k))$. We require that $c(S^{k-1})$ is contained in a finite union of open cells of dimension $\leq k-1$.
	\item A subset $C \subset X$ is closed if and only if for all open cells $\alpha$ we have that $C \cap \text{cl}(\alpha)$ is closed.\footnote{Note that $\text{cl}(\alpha)=c(D^k)$ in a CW-decomposition. This follows since, on the one hand, $\text{cl}(\alpha) \supset c(\text{cl}(\text{int}((D^k))) = c(D^k)$ by continuity and on the other hand we have that $c(D^k)$ is a compact and hence closed subspace (since $X$ is Hausdorff).}
\end{enumerate}

A space along with a CW-decomposition of it is called a \emph{CW-complex}. Given a CW-decomposition $X^\bullet$ of $X$ (which, abusing notation, we shall often also refer to by $X$), a \emph{subcomplex} is a subspace $A \subset X$ which is a union of open cells of $X^\bullet$ for which, for each open cell $\alpha \subset A$, we have that $\text{cl}(\alpha) \subset A$ (as subspaces of $X$) or, equivalently, $A$ is closed in $X$. Since we have that $\text{cl}(\alpha)=c(D^n)$ for some given characteristic map $c$, the open cells of $A$ define a CW-decomposition of $A$ with the same characteristic maps for each open cell.

For $n \in \mathbb{N}_0$, the union of all open cells of dimension $\leq n$ is called the \emph{$n$-skeleton} and is denoted, where there can be no confusion of which CW-decomposition of $X$ is in question, by $X^n$. We define $X^{-1}$ to be the ``empty partition'' $\emptyset$ of the subcomplex $\emptyset \subset X$. Of course, we have that $X^n$ is a subcomplex of $X$. The CW-decomposition is said to be \emph{finite-dimensional} if $X^n=X$ for some $n \in \mathbb{N}_0$, that is, if the dimensions of the cells are bounded, and \emph{infinite-dimensional} otherwise. It is \emph{finite} if there are only finitely many cells and \emph{locally finite} if each closed cell intersects only finitely many other closed cells. Of course, these are properties of the space $X$, not just the CW-decomposition i.e., the CW-decomposition is finite if and only if $X$ is compact and is locally finite if and only if $X$ is locally compact.

Given a CW-complex $X$, define $C_k(X) := H_k(X^k,X^{k-1})$. One may show that $C_k(X) \cong \bigoplus_\alpha \mathbb{Z}$, where the direct sum is taken over all $k$-cells $\alpha$. Given $[\sigma] \in H_k(X^k,X^{k-1})$ represented by the singular $k$-chain $\sigma$, we have that $\partial_k(\sigma)$ is an element of $S_{k-1}(X^{k-1},X^{k-2})$. These maps in fact induce boundary maps $\partial_k \colon C_k(X) \rightarrow C_{k-1}(X)$ and hence a chain complex $C_\bullet(X)$. The homology groups $H(C_\bullet(X))$ are called the \emph{cellular homology groups} of $X$. We have that the singular and cellular homology groups of a CW-complex are isomorphic (see e.g., \cite{Hat}).

One may proceed as above with singular Borel-Moore chains instead of singular chains. Let $X$ be a locally finite CW-complex. In this case, we have that $C_k^{\text{BM}}(X) := H_k^{\text{BM}}(X^k,X^{k-1}) \cong \prod_\alpha \mathbb{Z}$, that is, a cellular Borel-Moore $k$-chain corresponds to a choice of (oriented) coefficient for each $k$-cell. Analogously to above, one may define boundary maps $\partial_k \colon C_k^{\text{BM}}(X) \rightarrow C_{k-1}^{\text{BM}}(X)$; these correspond ``locally'' to the usual cellular boundary maps. However, it is not necessarily the case that $H(S_\bullet^{\text{BM}}(X)) \cong H(C_\bullet^{\text{BM}}(X))$ (see below for a counterexample to this isomorphism). It is shown in \cite{Ends}, however, that the singular and cellular groups are isomorphic when the CW-complex is \emph{strongly locally finite}, a notion introduced in \cite{FTW}. In particular, the singular and cellular Borel-Moore homology groups are isomorphic for locally finite, finite dimensional CW-complexes with a countable number of cells.

\begin{exmp} The following simple example appears in \cite{Ends} and \cite{FTW} as a CW-complex which is not strongly locally finite. Define a CW-complex $X$ with one cell of each dimension where the boundary of the unique $k$-cell $\alpha$ is contained in a point of the interior of the $(k-1)$-cell $\beta$. That is $X$, as a topological space, is an infinite string of spheres of increasing dimension. The cellular Borel-Moore chain complex for $X$ is $C^{\text{BM}}_\bullet = (0 \leftarrow \mathbb{Z} \leftarrow \mathbb{Z} \leftarrow \ldots )$ where each boundary map is the zero map. Hence, in particular, we have that $H_0(C_\bullet^{\text{BM}}(X)) \cong \mathbb{Z}$. On the other hand, it is not too hard to see that $H_0(S_\bullet^{\text{BM}}(X)) \cong 0$, one may ``push singular $0$-chains to infinity''. \end{exmp}

\begin{exmp} Let $\mathbb{R}^d$ be given by the CW-decomposition induced from the tiling of unit squares with corners at the lattice points $\mathbb{Z}^d$. This is a locally finite, finite dimensional CW-complex with a countable number of cells. Hence, its singular and cellular Borel-Moore homology groups coincide. By Poincar\'{e} duality we have that $H_i^{\text{BM}}(\mathbb{R}^d) \cong H^{d-i}(\mathbb{R}^d) \cong 0$ for $i \neq d$ and $H_d^{\text{BM}}(\mathbb{R}^d) \cong H^{0}(\mathbb{R}^d) \cong \mathbb{Z}$. The cellular Borel-Moore homology group in degree $d$ is generated by the fundamental class for $\mathbb{R}^d$ by assigning (oriented) coefficient one to each $d$-cube.\end{exmp}

\chapter{Patterns of Finite Local Complexity}
\label{chap: Patterns of Finite Local Complexity}

The tilings (and point patterns) of interest to us in this thesis are those which possess \emph{finite local complexity} (or are \emph{FLC}, for short). Loosely speaking, a tiling is FLC if, for any given radius, the number of motifs of the tiling of that radius is finite, up to some agreed notion of equivalence (by translations or rigid motions, say), see Definition \ref{def: FLC tiling}. Then to determine the local decoration of our geometric object to some agreed radius, one needs only to choose from a finite selection of motifs (although along with some continuous choice of location within one). This rigidity means that the underlying geometric object of interest is well described by a system of partial isometries, which describe the motions preserving portions of it. So we shall firstly in this chapter introduce the abstract notion of a \emph{pattern}. The idea is that we may associate with some FLC pattern the system of underlying partial symmetries that it possesses. These abstract patterns may be thought of as an analogue to the space group of isometries preserving a tiling, but where one considers instead the inverse semigroup of all \emph{partial} isometries preserving the tiling (see Subsection \ref{subsect: Patterns as Inverse Semigroups}). The MLD class of an FLC tiling may be recovered from its underlying pattern of partial isometries, so if one is interested only in the tiling ``modulo local redecorations'', one really can identify with a tiling its underlying abstract pattern. Many other constructions will also only depend on these partial isometries; the pattern-equivariant homology that we introduce in the next chapter will be applied to these abstract patterns.

We shall then introduce the notion of a \emph{collage} which will be a system of equivalence relations on some uniform space. The observation motivating them is that many constructions associated to FLC tilings are actually defined in terms of an underlying system of equivalence relations determined by the tiling (and analogously for Delone sets). Any pattern, in the abstract sense as above, induces a collage. The equivalence relations of patch equivalence to a certain radius induce a tiling metric and an inverse system of ``approximants''. For FLC tilings, the completion of this induced space of tilings is homeomorphic to the inverse limit of approximants (see \cite{BDHS}). Analogously, the collages defined here will determine a collage uniformity and an inverse system of approximants. In cases of interest the inverse limit of the inverse system of approximants can be shown to be a completion of the uniform space induced by the collage.

\section{Patterns} \label{sect: Patterns}

\subsection{Patterns of Partial Isometries} \label{subsect: Patterns of Partial Isometries}

Let $T$ be a periodic tiling of $\mathbb{R}^d$, that is, a tiling for which there exist $d$ linearly-independent vectors $x_1,\ldots,x_d$ such that $T+x_i=T$. Everything one could wish to know about $T$ up to ``local redecorations'', that is, up to MLD equivalence (see Example \ref{ex: MLD}) is stored in the symmetry group $G_T$ of isometries $\Phi\colon \mathbb{R}^d \rightarrow \mathbb{R}^d$ which preserve the tiling. Partial isometries preserving sufficiently large patches of the tiling necessarily extend uniquely to full symmetries of the whole tiling. This is no longer true for non-periodic tilings, indeed, there may not even be any non-trivial isometries preserving the tiling, let alone enough to provide an adequate description of the tiling! However, the tiling may possess a rich supply of \emph{partial} isometries preserving patches of the tiling. For FLC tilings, this system of partial isometries determines the MLD class of the tiling.

We shall always assume that (partial) isometries are surjective onto their codomain. For a partial isometry $\Phi \colon U \rightarrow V$ we define $dom(\Phi):=U$ and $ran(\Phi):=V$. Two partial isometries will always be composed on the largest domain for which the composition makes sense, that is $\Phi_2 \circ \Phi_1 \colon \Phi_1^{-1}(ran(\Phi_1) \cap dom(\Phi_2)) \rightarrow \Phi_2(ran(\Phi_1) \cap dom(\Phi_2))$.

\begin{definition} \label{def: pattern} Let $(X,d_X)$ be a metric space. A \emph{pattern} $\mathcal{P}$ (\emph{on $(X,d_X)$}) is a collection of partial isometries $\Phi \colon U \rightarrow V$ between the open sets of $(X,d_X)$ such that:
\begin{enumerate}
\item for every open set $U$ the identity map $Id_U \in \mathcal{P}$.
\item If $\Phi \in \mathcal{P}$ its inverse $\Phi^{-1} \in \mathcal{P}$.
\item If $\Phi_1,\Phi_2 \in \mathcal{P}$ then $\Phi_2 \circ \Phi_1 \in \mathcal{P}$.
\end{enumerate}
\end{definition}

Note that, by $1$, we should always have the ``empty morphism'' $Id_{\emptyset} \colon \emptyset \rightarrow \emptyset$ as an element of $\mathcal{P}$, which is the unique such morphism with domain and codomain equal to the empty set in the category of sets. The axioms imply that restrictions of elements $\Phi \in \mathcal{P}$ to open sets $U \subset X$ are also elements of $\mathcal{P}$ since ${\Phi |}_U = \Phi \circ Id_U$.

\begin{notation} Let $\mathcal{P}$ be a pattern on $(X,d_X)$. Denote by $\mathcal{P}_{x,y}^r$ the set of partial isometries $\Phi \in \mathcal{P}$ with $\Phi(x)=y$, $dom(\Phi)=B_{d_X}(x,r)$ and $ran(\Phi)=B_{d_X}(y,r)$.\end{notation}

\begin{definition} Let $\mathcal{P}$ be a pattern on $(X,d_X)$. We shall say that\begin{enumerate}
	\item $\mathcal{P}$ has \emph{finite local complexity} (or is \emph{FLC}, for short) if for all $r \in \mathbb{R}_{> 0}$ there exists some compact $K \subset X$ for which, for any $x \in X$, there exists some $y \in K$ with $\mathcal{P}_{x,y}^r \neq \emptyset$.
	\item $\mathcal{P}$ is \emph{repetitive} if for all $r \in \mathbb{R}_{> 0}$ there exists some $R \in \mathbb{R}_{> 0}$ for which, for any $x,y \in X$, there exists some $y' \in B_{d_X}(y,R)$ with $\mathcal{P}_{x,y'}^r \neq \emptyset$.
\end{enumerate}
\end{definition}

Notice that a repetitive pattern on a proper metric space is also FLC. These properties are named from their usage in the context of tilings.

\subsection{Patterns Associated to Tilings} \label{subsect: Patterns Associated to Tilings}

\subsubsection{Tilings on General Metric Spaces}

Let $(X,d_X)$ be some metric space on which a tiling $T$ is defined. We shall define a pattern on $(X,d_X)$ given some collection of partial isometries on it, which allow us to compare local patches of tiles. This collection of partial isometries with which we shall compare patches should at least satisfy the axioms of being a pattern, so we set $\mathcal{S}$ to be some pattern of \emph{allowed partial isometries} of $(X,d_X)$. In most cases of interest to us here, the allowed partial isometries will be restrictions (to the open sets of $(X,d_X)$) of a group of global isometries of $(X,d_X)$.

\begin{exmp} \label{ex: Allowed Partial Isometries} Given a group $S$ of isometries on $(X,d_X)$ we can define a pattern $\mathcal{S}$ of partial isometries on $(X,d_X)$ simply by restricting the isometries of $S$ to each of the open subsets of $X$, it is easy to verify that this defines a pattern on $(X,d_X)$. A common example in the context of tilings is the group $\mathbb{R}^d$ of translations $t_y$ defined by $x \mapsto x+y$ on $\mathbb{R}^d$. More generally, one could consider a Lie group $L$ with left invariant metric $d_L$. Then $L$ acts on itself as a group of isometries through left multiplication $L_g(x)=g \cdot x$.

Another common example for tilings of $\mathbb{R}^d$ is to take the group $E^+(d)$ of orientation-preserving isometries on $\mathbb{R}^d$, which we shall call \emph{rigid motions}. \end{exmp}

See the example of \ref{subsect: Pent} for a situation where it is necessary to consider partial isometries which are not restrictions of global isometries of the space.

Given a tiling along with some collection of allowed partial isometries, we associate a pattern to it:

\begin{definition} Let $T$ be a tiling of $(X,d_X)$ with allowed partial isometries $\mathcal{S}$. Let $\Phi \in \mathcal{S}$ be an element of $\mathcal{T}_\mathcal{S}$ if and only if $\Phi$ is the restriction of some partial isometry $\Psi$ with $dom(\Psi)=T(dom(\Phi))$, $ran(\Psi)=T(ran(\Phi))$ and $\Psi(T(dom(\Phi))) = T(ran(\Phi))$. \end{definition}

That is, the pattern defined by a tiling and collection of allowed partial isometries is given by those partial isometries which preserve patches of the tiling. We will frequently drop the $\mathcal{S}$ from $\mathcal{T}_\mathcal{S}$, where the allowed partial isometries are understood.

\begin{proposition} The collection of partial isometries defining $\mathcal{T}_\mathcal{S}$ forms a pattern. \end{proposition}

\begin{proof} Since $Id_U \in \mathcal{S}$ and $Id_{T(U)}(T(U))=T(U)$, we have that $Id_U \in \mathcal{T}_\mathcal{S}$ for all open $U \subset X$.

Suppose that $\Phi \in \mathcal{T}_\mathcal{S}$. Then $\Phi^{-1} \in \mathcal{S}$ since $\mathcal{S}$ is closed under inverses. Since $\Phi \in \mathcal{T}_\mathcal{S}$ we have that $\Phi$ is the restriction of some isometry $\Psi \colon T(dom(\Phi)) \rightarrow T(ran(\Phi))$ with $\Psi(T(dom(\Phi))) = T(ran(\Phi))$. Then $\Phi^{-1} \in \mathcal{T}_\mathcal{S}$ since it is the restriction of $\Psi^{-1} \colon T(ran(\Phi)) \rightarrow T(dom(\Phi))$, for which one has that $\Psi^{-1}(T(ran(\Phi)))$ \linebreak $= \Psi^{-1}( \Psi (T(dom(\Phi)))) = T(dom(\Phi))$.

Finally, suppose that $\Phi_1, \Phi_2 \in \mathcal{T}_\mathcal{S}$. Then $\Phi_2 \circ \Phi_1 \in \mathcal{S}$ since $\mathcal{S}$ is closed under composition. For $i=1,2$ we have that $\Phi_i$ is a restriction of some isometry $\Psi_i \colon T(dom(\Phi_i)) \rightarrow T(ran(\Phi_i))$ with $\Psi_i (T(dom(\Phi_i)))=T(ran(\Phi_i))$. Then $\Phi_2 \circ \Phi_1 \in \mathcal{T}_\mathcal{S}$ since it is the restriction of ${\Psi_2 \circ \Psi_1 |}_{T(\Phi_1^{-1}(ran(\Phi_1) \cap dom(\Phi_2)))}$ which has domain $T(\Phi_1^{-1}(ran(\Phi_1) \cap dom (\Phi_2)))$ and range $T(\Phi_2(ran(\Phi_1) \cap dom (\Phi_2)))$. We have that $\Psi_2 \circ \Psi_1(T(ran(\Phi_2 \circ \Phi_1))) = \Psi_2 \circ \Psi_1 (T(\Phi_1^{-1}(ran(\Phi_1) \cap dom (\Phi_2)))) = \Psi_2 (T(ran(\Phi_1)) \cap T(dom(\Phi_2))) = T(\Phi_2(ran(\Phi_1) \cap dom(\Phi_2))) = T(ran(\Phi_2 \circ \Phi_1))$. \end{proof}

\begin{remark} \label{rem} Note that it was not necessary for the tiles to cover $X$, for tiles to be equal to the closures of their interiors or for tiles to intersect on at most their boundaries in the above. There will be occasions where dropping these assumptions will be useful. For example, we could define a pattern associated to a Delone set (or in fact any point pattern) analogously, simply by considering the points of the Delone set as ``tiles'', and much of what is said here in the context of tilings will also apply to Delone sets. \end{remark}

Given a tiling $T$ with collection of allowed partial isometries $\mathcal{S}$, write $P \sim_\mathcal{S} Q$ for subpatches $P,Q \subset T$ if there exists some $\Phi \in \mathcal{S}$ with $P \subset dom(\Phi)$ and $\Phi(P)=Q$. This defines an equivalence relation on the subpatches of $T$ (note that $\mathcal{S}$ is closed under partial composition so that $\sim_\mathcal{S}$ is transitive), we denote the equivalence classes by $[P]_T^\mathcal{S}$.

\begin{definition} \label{def: FLC tiling} Let $T$ be a tiling on $(X,d_X)$. We shall say that $T$ is \emph{locally finite} if there exists some $\epsilon > 0$ such that each $\epsilon$-ball $B_{d_X}(x,\epsilon)$ intersects only finitely many tiles and \emph{proper} if each $R$-ball intersects only finitely many tiles for all $R \in \mathbb{R}_{>0}$. For a proper tiling we say that
\begin{enumerate}
	\item $T$ has \emph{finite local complexity} (or is \emph{FLC}, for short) with respect to $\mathcal{S}$ if for each $r \in \mathbb{R}_{>0}$ the number of equivalence classes $[T(x,r)]_T^\mathcal{S}$ for $x \in X$ is finite.
	\item $T$ is \emph{repetitive} with respect to $\mathcal{S}$ if, for each $r \in \mathbb{R}_{>0}$, there exists some $R \in \mathbb{R}_{> 0}$ such that each $R$-ball $B_{d_X}(y,R)$ contains a representative of each equivalence class $[T(x,r)]_T^\mathcal{S}$ for $x \in X$.
\end{enumerate}
\end{definition}

The above definitions of FLC and repetitivity are well known in the literature where they are usually just stated for tilings of $\mathbb{R}^d$ with respect to translations or rigid motions. The FLC condition is also known as the ``finite type condition'' (e.g., in \cite{AP}). Note that FLC or repetitivity is not a property of the tiling $T$, but of the pair of the tiling $T$ and the collection of allowed partial isometries $\mathcal{S}$. For example, the pinwheel tiling \cite{Rad} is a well-known example of a tiling with has FLC with respect to rigid motions but not with respect to translations alone, the tiles appear in infinitely many orientations in the plane.

\begin{proposition} \label{prop: FLC iff FLC} Let $(X,d_X)$ be a proper metric space. Then a tiling $T$ on $(X,d_X)$ is locally finite if and only if it is proper. If $T$ is locally finite then
\begin{enumerate}
	\item if $T$ is repetitive with respect to $\mathcal{S}$ then it is FLC with respect to $\mathcal{S}$.
	\item $T$ is FLC with respect to $\mathcal{S}$ if and only if $\mathcal{T}_\mathcal{S}$ is FLC.
	\item $T$ is repetitive with respect to $\mathcal{S}$ if and only if $\mathcal{T}_\mathcal{S}$ is repetitive.
\end{enumerate}
\end{proposition}

\begin{proof} It is obvious that being proper is a stronger condition than of being locally finite. Suppose that $T$ is locally finite, so there exists some $\epsilon > 0$ for which each $\epsilon$-ball intersects only finitely many tiles. Set $R > 0$ and $x \in X$. Then since $(X,d_X)$ is proper we have that $B_{d_X}(x,R)$ is contained in a compact subset, which thus possesses a finite covering by $\epsilon$-balls. Since each such $\epsilon$-ball only intersects finitely many tiles of $T$, the same is true of $B_{d_X}(x,R)$ and so $T$ is proper.

Suppose that $T$ is repetitive with respect to $\mathcal{S}$ and set $r>0$ and $p \in X$. Then there exists some $R > 0$ such that a representative of each equivalence class $[T(B_{d_X}(x,r))]_T^\mathcal{S}$ is contained in $B_{d_X}(p,R)$ and hence is a subpatch of $P:=T(B_{d_X}(p,R))$. Since $T$ is proper, we have that $P$ is finite. Since a subpatch of $P$ is simply a choice of tiles from it, we have that the cardinality of $\{[T(B_{d_X}(x,r))]_T^\mathcal{S} \mid x \in X\}$ is less than or equal to the cardinality of the power set of $P$, which is finite since the cardinality of $P$ is.

To prove $2$, assume firstly that $T$ is FLC with respect to $\mathcal{S}$ and set $r > 0$ arbitrary. We always assume that the radii of the tiles are bounded by, say, $t > 0$. By FLC the number of equivalence classes $[T(B_{d_X}(x,r))]_T^\mathcal{S}$ is finite. Choose representative patches $P_1 = T(B_{d_X}(x_1,r)),\ldots,P_k = T(B_{d_X}(x_k,r))$ for each such equivalence class. Let $x \in X$. Then there exists some partial isometry $\Phi \in \mathcal{S}$ taking the patch $T(x,r)$ isometrically to $P_i$ for some $i \in \{1,\ldots,k\}$. Since $P_i$ has radius less than $r+2t$ we have that $d_X(x,x_i) < r+2t$. On the other hand we have that $\Phi|_{B_{d_X}(x,r)} \in (\mathcal{T}_\mathcal{S})_{x,\Phi(x)}^r$ so that $B_{d_X}(x_i,r+2t)$ contains a point $y$ with $(\mathcal{T}_\mathcal{S})_{x,y}^r \neq \emptyset$. It follows that for all $x \in X$ there exists some $y \in \bigcup_{i=1}^k B_{d_X}(x_i,r+2t)$ such that $(\mathcal{T}_\mathcal{S})_{x,y}^r \neq \emptyset$. Since $(X,d_X)$ is proper, this finite union is contained in a compact $K \subset X$ and so $\mathcal{T}_\mathcal{S}$ is FLC.

Conversely, suppose that $\mathcal{T}_\mathcal{S}$ is FLC and set $r>0$. Then by FLC there exists some compact $K \subset X$ such that for all $x \in X$ there exists some $y \in K$ with $(\mathcal{T}_\mathcal{S})_{x,y}^{r+2t} \neq \emptyset$. Since $K$ is compact it is bounded, so we have that $K \subset B_{d_X}(p,R)$ for some $R > 0$ and $p \in X$. Let $x \in X$ be arbitrary. Then $T(x,r) \subset B_{d_X}(x,r+2t)$. By the above we know that there exists some $\Phi \in \mathcal{S}$ taking $T(x,r)$ to a patch $T(y,r)$ with $d_X(p,y)<R$. That is, every patch $T(x,r)$ is equivalent, under $\mathcal{S}$, to a patch contained in $B_{d_X}(p,R+r+2t)$. Since $T$ is proper, this patch contains only finitely many tiles and hence subpatches, and so $T$ is FLC with respect to $\mathcal{S}$.

To prove $3$, firstly suppose that $T$ is repetitive with respect to $\mathcal{S}$ and set $r>0$ and $x,y \in X$ arbitrary. It follows that there exists some $R > 0$ for which each $R$-ball $B_{d_X}(y,R)$ contains a representative of each equivalence class $[T(x,r)]_T^\mathcal{S}$. So, in particular, there exists some $\Phi \in \mathcal{S}$ mapping $T(x,r)$ isometrically to a subpatch of $T(y,R)$. Since this subpatch is contained in $B_{d_X}(y,R+2t)$ we have that $d_X(y,\Phi(x)) < R+2t$. On the other hand, we know that $\Phi|_{B_{d_X}(x,r)} \in (\mathcal{T}_\mathcal{S})_{x,\Phi(x)}^r$. Hence, for all $x,y \in X$ we have that $(\mathcal{T}_\mathcal{S})_{x,y'}^r \neq \emptyset$ for some $y' \in B_{d_X}(x,R+2t)$ and so $\mathcal{T}_\mathcal{S}$ is repetitive.

Finally, assume that $\mathcal{T}_\mathcal{S}$ is repetitive and set $r > 0$ arbitrary. Then there exists some $R > 0$ such that for all $x,y \in X$ there exists some $y' \in B_{d_X}(y,R)$ with $(\mathcal{T}_\mathcal{S})_{x,y'}^{r+2t} \neq \emptyset$. Take an arbitrary $r$-patch $T(x,r)$ of $T$. Then $T(x,r) \subset B_{d_X}(x,r+2t)$ and so by the above there exists some $y' \in B_{d_X}(x,R)$ and $\Phi \in \mathcal{S}$ taking the patch $T(x,r)$ to $T(y',r)$. We have that $T(y',r) \subset B_{d_X}(y,R+2t)$. That is, every equivalence class of patch $T(x,r)$ is represented by a subpatch of $B_{d_X}(y,R+2t)$ and so $T$ is repetitive with respect to $\mathcal{S}$. \end{proof}

\subsubsection{Tilings of $(\mathbb{R}^d,d_{euc})$}

Let $T$ be a tiling of $(\mathbb{R}^d,d_{euc})$. We present here three different patterns $\mathcal{T}_1$, $\mathcal{T}_0$ and $\mathcal{T}_{\text{rot}}$ associated to it. The first pattern is usually only suitable for tilings which have FLC with respect to translations, while the latter two are suitable for tilings which have FLC with respect to rigid motions. The spaces $\Omega^{\mathcal{T}_1}$, $\Omega^{\mathcal{T}_0}$ and $\Omega^{\mathcal{T}_{\text{rot}}}$ (see Definition \ref{def: Pattern Space}) correspond to the tiling spaces commonly seen in the literature; they are named $\Omega^1$, $\Omega^0$ and $\Omega^{\text{rot}}$, resp., in \cite{BDHS}, for example.

\begin{definition} \label{def: 0,1 tiling patterns} Given a tiling $T$ of $(\mathbb{R}^d,d_{euc})$, define $\mathcal{T}_1$ ($\mathcal{T}_0$, resp.) to be the pattern $\mathcal{T}_\mathcal{S}$ for $T$, where the allowed partial isometries $\mathcal{S}$ are taken to be restrictions of translations (rigid motions, resp.) to the open sets of $(\mathbb{R}^d,d_{euc})$.\end{definition}

So $(\mathcal{T}_i)_{x,y}^r \neq \emptyset$ for $i=1$ ($i=0$, resp.) if and only if the tilings $T-x$ and $T-y$ have the same patch of tiles within radius $r$ of the origin (after a rotation at the origin, resp.) 

A different way of taking the full orientation-preserving isometry group into account is to define a pattern on the group $E^+(d)$ of rigid motions of $\mathbb{R}^d$. Elements of $E^+(d) \cong \mathbb{R}^d \rtimes SO(d)$ can be uniquely described as a translation $b \in \mathbb{R}^d$ followed by a rotation $A \in SO(d)$. One can define a metric on $E^+(d)$ by setting $d_E(f,g) := \max \{d_{euc}(f(x),g(x)) \mid \|x\| \le 1 \}$. Then $E^+(d)$ acts on itself as a group of isometries by post-composition, so we may take the restrictions of such isometries to the open sets of $E^+(d)$ to be a set of allowed partial isometries. We temporarily allow ourselves to use tiles with empty interior for the following definition:

\begin{definition} \label{def: rot tiling pattern} Let $T$ be a tiling of $(\mathbb{R}^d,d_{euc})$. We define a tiling $T_{\text{rot}}$ on $(E^+(d), \ab d_E)$ in the following way. The prototiles of $T_{\text{rot}}$ are the prototiles of $T$ as subsets of $\mathbb{R}^d$, where we consider $\mathbb{R}^d \subset E^+(d)$ as the subset of translations. A tile $t_{\text{rot}} \in T_{\text{rot}}$ if and only if $t_{\text{rot}}=g(t)$ for some tile $t \in T$ and $g \in SO(d)$. Define $\mathcal{T}_{\text{rot}}$ to be the pattern associated to this tiling with collection of allowed partial isometries given as restrictions of isometries given by the action of $E^+(d)$ on itself, as above.\end{definition}

Note that the tiles of $T_{\text{rot}}$  here are contained in the cosets of $\mathbb{R}^d \trianglelefteq E^+(d)$; one may like to think of the tiling as a union of copies of the tiling $T$, one copy on each fibre $\mathbb{R}^d \hookrightarrow E^+(d) \rightarrow SO(d)$ for each rotation $g \in SO(d)$. We have that $(\mathcal{T}_{\text{rot}})_{f,g}^r \neq \emptyset$ if and only if the patches of tiles within distance $r$ of the origin of $f^{-1}(T_{\text{rot}})$ and $g^{-1}(T_{\text{rot}})$ are the same. This implies that the patch of tiles within distance $r$ of the origin of $f^{-1}(T)$ and $g^{-1}(T)$ are the same, which is also sufficient (for then this is also true of their rotates). That is, $(\mathcal{T}_{\text{rot}})_{f,g}^r \neq \emptyset$ if and only if $f^{-1}(T)$ and $g^{-1}(T)$ have the same patch of tiles within distance $r$ of the origin. Of course, for $(\mathcal{T}_{\text{rot}})_{f,g}^r \neq \emptyset$ we have that in fact $(\mathcal{T}_{\text{rot}})_{f,g}^r = \{g \circ f^{-1}|_{B_{d_E}(f,r)} \}$ so, in general, we have that $|(\mathcal{T}_{\text{rot}})_{f,g}^r| = 0$ or $1$.

\begin{exmp} Let $T$ be the periodic tiling of unit squares of $(\mathbb{R}^2,d_{euc})$ with the vertices of the squares lying on the integer lattice. Then for $i=1$ ($i=0$, resp.) $(\mathcal{T}_i)^r_{x,y}$ consists of (restrictions to $B_{d_X}(x,r)$ of) translations (rigid motions, resp.) taking $x$ to $y$ preserving the integer lattice. So $|(\mathcal{T}_1)^r_{x,y}| \neq 0$ if and only if $y-x \in \mathbb{Z}^2$; for $y-x \in \mathbb{Z}^2$ we have that $(\mathcal{T}_1)^r_{x,y}=\{t_{y-x}|_{B_{d_X}(x,R)}\}$. The set of morphisms $(\mathcal{T}_0)^r_{x,y}$ may have more than one element; when $x,y$ both lie in the centre of an edge then $|(\mathcal{T}_0)_{x,y}^r|=2$ and if $x,y$ both lie in the centre of a square or both lie on a vertex then $|(\mathcal{T}_0)_{x,y}^r|=4$.

For $\mathcal{T}_{\text{rot}}$, we have that $|(\mathcal{T}_{\text{rot}})_{f,g}^r| \neq 0$ if and only if $f^{-1}(T)=g^{-1}(T)$. \end{exmp}

\subsubsection{Hierarchical Tilings}

Given a tiling substitution $\omega$, one may define the notion of an $\omega$-tiling, that is, a tiling admitted by the substitution rule $\omega$; see Subsection \ref{subsect: Examples of Tilings and Delone Sets}. Given such a tiling $T_0$, there exists another tiling $T_1$ (of inflated tiles) admitted by $\omega$ which $\omega$ decomposes into $T_0$. The tiling $T_1$ and the substitution $\omega$ determines the tiling $T_0$. One may continue this process inductively to define a string of tilings $(T_0,T_1,\ldots)$. The patches of each $T_n$ are locally determined by the patches of the tiling $T_{n+1}$ that lies above it in the hierarchy. This notion could be expressed using partial isometries by saying that if an allowed partial isometry preserving patches $\Phi \in \mathcal{T}_{n+1}$ then $\Phi \in \mathcal{T}_n$ also.

\begin{definition} Let $(X,d_X)$ be any metric space and a $\mathcal{S}$ a collection of allowed partial isometries on it. We shall say that a string of tilings $T_\omega=(T_0,T_1,\ldots)$ is a \emph{hierarchical tiling} if for each $\Phi \in \mathcal{T}_{n+1}$ we have that $\Phi \in \mathcal{T}_n$. \end{definition}

Given such a hierarchical tiling, we wish to assign a pattern to it:

\begin{definition}For a partial isometry $\Phi$, temporarily write $rad(\Phi) \ge R$ to mean that there exists some $x$ for which $B_{d_X}(x,R) \subset dom(\Phi)$ and $B_{d_X}(\Phi(x),R) \subset ran(\Phi)$. Let $k \colon \mathbb{R}_{> 0} \rightarrow \mathbb{N}_0$ be a non-decreasing unbounded function. Then define the pattern $\mathcal{T}_\omega$ for $T_\omega$ by setting $\Phi \in \mathcal{T}_\omega$ if and only if $\Phi \in \mathcal{T}_{k(R)}$ for any $R$ with $rad(\Phi) \ge R$. \end{definition}

It follows rather immediately from the definitions that $\mathcal{T}_\omega$ is a pattern (note that $rad(\Phi_2 \circ \Phi_1) \ge R$ implies that $rad(\Phi_1),rad(\Phi_2) \ge R$). For a hierarchical tiling of $(\mathbb{R}^d,d_{euc})$ we have that $(\mathcal{T}_\omega)_{x,y}^R = (\mathcal{T}_{k(R)})_{x,y}^R$. This does not hold for general metric spaces since one may have $B_{d_X}(x,r)=B_{d_X}(y,R)$ for $r<R$ e.g., for the positive real line $\mathbb{R}_{> 0}$. For this reason, for simplicity, we shall only consider hierarchical tilings of metric spaces for which open balls determine their centre and radii. A perhaps more elegant approach would be to give a definition of a pattern which allows one to specify a hierarchy of partial isometries, but we will be content here with hierarchical tilings in this restricted setting.

The idea is that one removes ``large'' (as measured by the radii of the domain and range) isometries which do not preserve patches of the tilings high up the hierarchy. Two points of the hierarchical tiling will be considered as ``close'' in the induced uniformity (see the next section) if they have nearby points which agree on a large patch of a tiling $T_n$ for some large $n$. An interesting feature of these patterns is that they do not necessarily satisfy a ``glueing'' axiom with respect to the allowed partial isometries in a way similar to that of classical pseudogroups. That is ``large''  morphisms of $\mathcal{T}_\omega$ can not necessarily be determined by the collection of allowed isometries and the ``small'' morphisms of $\mathcal{T}_\omega$ alone (see Example \ref{ex: DS}). As described above, tiling substitutions provide examples of hierarchical tilings in this sense, although one could also easily consider multi/mixed-substitutions in this context (see \cite{GM}).

\section{Collages} \label{sect: Collages}

Many constructions associated to finite local complexity tilings may in fact be defined solely in terms of a collection of induced equivalence relations on the underlying space of the tiling. These equivalence relations, for each $R \in \mathbb{R}_{>0}$, codify the relation of two points $x,y$ being such that ``to radius $R$, the tiling centred at $x$ is equal, modulo some type of isometry, to the tiling centred at $y$''. For example, the tiling metric, the notion of MLD equivalence, the BDHS-complexes \cite{BDHS} and pattern-equivariant cohomology may all be defined in terms of these induced equivalence relations.

In this section we shall define the notion of a \emph{collage} on a uniform space, a system of equivalence relations which patch together in some consistent way with respect to the underlying uniform space. They are intended as a simple abstraction of the above equivalence relations associated to an FLC tiling. Analogously to how these equivalence relations of an FLC tiling induce a tiling metric, collages will induce a \emph{collage uniformity}. The tiling space is constructed as a completion with respect to the tiling metric and is homeomorphic to the inverse limit of BDHS-complexes of increasing patch radii \cite{BDHS}. We shall similarly consider completions of the uniform spaces induced by collages and show that, at least for the cases of interest, the inverse limit of approximants defined by the system of equivalence relations is a completion of this uniform space.

\begin{notation} Recall from Section \ref{sect: Uniform Spaces} that we consider equivalence relations (on $X$, say) as subsets $U \subset X \times X$ and so for two equivalence relations $U,V$ on $X$ we have that $U \subset V$ if and only if being $U$-equivalent implies being $V$-equivalent. Given a general relation $U$, we will sometimes write $x U y$ to mean that $(y,x) \in U$. \end{notation}

\begin{definition} \label{def: collage} A \emph{collage} $\tilde{\mathcal{P}}$ (on a uniform space $(X,\mathcal{U}_X)$ over the directed set $(\Lambda,\leq)$) consists of equivalence relations $\lambda_{\tilde{\mathcal{P}}} \subset X \times X$, one for each $\lambda \in \Lambda$, such that:
	\begin{enumerate}
		\item if $\lambda \leq \mu$ then $\mu_{\tilde{\mathcal{P}}} \subset \lambda_{\tilde{\mathcal{P}}}$.
		\item For each $\lambda \in \Lambda$ there exists some $\mu \geq \lambda$ such that, for all $U \in \mathcal{U}_X$, there exists some $V \in \mathcal{U}_X$ with $V \subset U$ and $V \circ \mu_{\tilde{\mathcal{P}}} \subset \lambda_{\tilde{\mathcal{P}}} \circ U$.
	\end{enumerate}
	
We shall say that a collage $\tilde{\mathcal{P}}$ has \emph{finite local complexity} (or is \emph{FLC}, for short) if each quotient space $K_\lambda^{\tilde{\mathcal{P}}}: = X/\lambda_{\tilde{\mathcal{P}}}$ (where $X$ is given the topology induced by $\mathcal{U}_X$) is compact.
\end{definition}

The second axiom for a collage looks rather arbitrarily asymmetric. We could have also demanded that $\mu_{\tilde{\mathcal{P}}} \circ V \subset U \circ \lambda_{\tilde{\mathcal{P}}}$, but doing so would be superfluous. Indeed, given $\lambda \in \Lambda$ and $U \in \mathcal{U}_X$, set $\mu \in \Lambda$ which satisfies $2$. Then there exists some $V \in \mathcal{U}_X$ such that $V \circ \mu_{\tilde{\mathcal{P}}} \subset \lambda_{\tilde{\mathcal{P}}} \circ U^\text{op}$ (since $U^\text{op} \in \mathcal{U}_X$, being a uniformity on $X$). But then $(V \circ \mu_{\tilde{\mathcal{P}}})^\text{op} \subset (\lambda_{\tilde{\mathcal{P}}} \circ U^\text{op})^\text{op}$ which implies that $\mu_{\tilde{\mathcal{P}}} \circ V^\text{op} \subset U \circ \lambda_{\tilde{\mathcal{P}}}$. For $W$ with $W \circ \mu_{\tilde{\mathcal{P}}} \subset \lambda_{\tilde{\mathcal{P}}} \circ U$ we have that $W \cap V^\text{op} \in \mathcal{U}_X$ satisfies both of the required inclusions.

The simple axioms of the definition of a collage may be interpreted geometrically. The equivalence relations $\lambda_{\tilde{\mathcal{P}}}$ should be thought of as codifying the relations ``are equivalent to magnitude $\lambda$'' (for example, our directed set $(\Lambda,\leq)$ will often correspond to $(\mathbb{R}_{>0},\leq)$ whose elements may be thought of directly as radii). The first axiom states that a prerequisite for two points to being equivalent to magnitude $\mu$ is that they are equivalent to any smaller $\lambda$. The second axiom may be interpreted as follows: suppose that when $(x,y) \in U \in \mathcal{U}_X$ for $U$ ``small'' we should consider the geometric object of interest centred at $x$ as a small perturbation of it at $y$. Then the second axiom states that, given some allowance of perturbation $U \in \mathcal{U}_X$ and magnitude $\lambda \in \Lambda$, there is a smaller allowance of perturbation $V \in \mathcal{U}_X$ and larger magnitude $\mu \in \Lambda$ such that, given a point $x$ and a point $y$ which is a $V$-perturbation of some point which is equivalent to magnitude $\mu$ to $x$, we have that in fact $y$ is equivalent to magnitude $\lambda$ to a point which is a $U$-perturbation from $x$. In some sense this establishes a coherence between the equivalence relations and the uniformity on $X$.

Given a tiling $T$ on $(X,d_X)$, one may consider points $x \in (X,d_X)$ as being pointed tilings. The tiling then induces a tiling metric on $X$, two points being considered ``close'' if their pointed tilings are ``close'' (see, e.g., \cite{AP}). Analogously, given a collage on $(X,\mathcal{U}_X)$, we can define an induced uniformity on $X$:

\begin{definition} Let $\tilde{\mathcal{P}}$ be a collage on $(X,\mathcal{U}_X)$. We define the \emph{collage uniformity} $\mathcal{U}_{\tilde{\mathcal{P}}}$ on $X$ by setting as a base of entourages sets of the form $U \circ \lambda_{\tilde{\mathcal{P}}} \circ U$ for some $\lambda \in \Lambda$ and $U \in \mathcal{U}_X$. \end{definition}

\begin{proposition} The collage uniformity is a uniformity on $X$. \end{proposition}

\begin{proof} Each $\lambda_{\tilde{\mathcal{P}}}$ is reflexive and each $U$ is an entourage so we have that the diagonal $\Delta_X \subset U \circ \lambda_{\tilde{\mathcal{P}}} \circ U$.

Since $\mathcal{U}_X$ is a uniformity and each $\lambda_{\tilde{\mathcal{P}}}$ is symmetric we have that $(U \circ \lambda_{\tilde{\mathcal{P}}} \circ U)^{\text{op}} = U^{\text{op}} \circ \lambda_{\tilde{\mathcal{P}}}^{\text{op}} \circ U^{\text{op}} = U^{\text{op}} \circ \lambda_{\tilde{\mathcal{P}}} \circ U^{\text{op}}$ is an element of $\mathcal{U}_{\tilde{\mathcal{P}}}$ so $\mathcal{U}_{\tilde{\mathcal{P}}}$ is closed under inverses.

For $U \circ \lambda_{\tilde{\mathcal{P}}} \circ U$ and $V \circ \mu_{\tilde{\mathcal{P}}} \circ V$, since $\Lambda$ is directed we have that there exists some $\nu \geq \lambda,\mu$. Then $(U \cap V) \circ \nu_{\tilde{\mathcal{P}}} \circ (U \cap V) \subset (U \circ \lambda_{\tilde{\mathcal{P}}} \circ U) \cap (V \circ \mu_{\tilde{\mathcal{P}}} \circ V)$ so $\mathcal{U}_{\tilde{\mathcal{P}}}$ is closed under finite intersections.

Finally, given $U \circ \lambda_{\tilde{\mathcal{P}}} \circ U$, since $\mathcal{U}_X$ is a uniformity we have that there exists some $V$ with $V \circ V \subset U$. By the definition of a collage (and the subsequent discussion we gave following it), there exists $\mu \geq \lambda$ and $W \subset V$ such that $W \circ \mu_{\tilde{\mathcal{P}}} \subset \lambda_{\tilde{\mathcal{P}}} \circ V$ and $\mu_{\tilde{\mathcal{P}}} \circ W \subset V \circ \lambda_{\tilde{\mathcal{P}}}$. It follows that \[(W \circ \mu_{\tilde{\mathcal{P}}} \circ W)^2 \subset (W \circ V \circ \lambda_{\tilde{\mathcal{P}}} ) \circ (\lambda_{\tilde{\mathcal{P}}} \circ V \circ W) = \] \[(W \circ V) \circ \lambda_{\tilde{\mathcal{P}}} \circ (V \circ W) \subset U \circ \lambda_{\tilde{\mathcal{P}}} \circ U.\] \end{proof}

The basis elements $V=U \circ \lambda_{\tilde{\mathcal{P}}} \circ U$ for the collage uniformity $\mathcal{U}_{\tilde{\mathcal{P}}}$ may be interpreted as follows. One may consider such an entourage as a function, which assigns to a point $x$ the points $V(x)$ considered as $V$-close to $x$. Reading from right to left, using function composition convention, $U \circ \lambda_{\tilde{\mathcal{P}}} \circ U(x)$ takes a point $x$ to all points which are a $U$-perturbation from it in $X$, takes all points which are equivalent to magnitude $\lambda$ to these points, and then allows for a further $U$-perturbation of these points. So the ``small'' entourages $V$ correspond to the entourages where $U$ is ``small'' and $\lambda$ is ``large''.

Consider the case of a collage induced from an FLC tiling pattern $\mathcal{T}$. We have that $(y,x) \in U \circ R_\mathcal{T} \circ U$ if there are points $x'$ and $y'$, with $x'$ $U$-close to $x$ and $y$ $U$-close to $y'$, such that $x' R_\mathcal{T} y'$. So points are ``close'' in the collage uniformity if there are nearby points where the patches centred there are equal to a large radius up to an allowed partial isometry. In some regards the uniformity approach is more natural than that given by the tiling metric (see Subsection \ref{subsect: Metric Collages}). One may not describe a tiling metric as ``the'' metric which considers two tilings as sufficiently close if they agree to a large radius up to a small perturbation, there are many choices for such a metric which depend on associating a specific number to distances between tilings. The tiling uniformity, however, is not so dependent on such arbitrary choices.

Note that we could have generated the uniformity instead by entourages of the form $\lambda_{\tilde{\mathcal{P}}} \circ U$ (or $U \circ \lambda_{\tilde{\mathcal{P}}}$). Indeed, given such a set, pick $V \in \mathcal{U}_X$ with $V \circ V \subset U$, $\mu \in \Lambda$ and $W \in \mathcal{U}_X$ with $W \subset V$ such that $W \circ \mu_{\tilde{\mathcal{P}}} \subset \lambda_{\tilde{\mathcal{P}}} \circ V$. Then we have that $W \circ \mu_{\tilde{\mathcal{P}}} \circ W \subset \lambda_{\tilde{\mathcal{P}}} \circ (V \circ W) \subset \lambda_{\tilde{\mathcal{P}}} \circ U$. So points $x,y$ are ``close'' if $y$ is equivalent to a ``large'' magnitude to some ``small'' perturbation of $x$ (or vice versa).

The notation $\tilde{\mathcal{P}}$ is intended to be indicative of the fact that a collage is a system of equivalence relations, but also that they will often be induced by patterns (and hence, perhaps, from tilings or Delone sets):

\begin{definition} Let $\mathcal{P}$ be a pattern on the metric space $(X,d_X)$ with corresponding uniform space $(X,\mathcal{U}_X)$. We define the \emph{induced collage} $\tilde{\mathcal{P}}$ on $(X,\mathcal{U}_X)$ over the directed set $(\mathbb{R}_{>0},\leq)$ by letting $(y,x) \in r_{\tilde{\mathcal{P}}}$ if and only if $\mathcal{P}_{x,y}^r \neq \emptyset$ (recall this notation from Subsection \ref{subsect: Patterns of Partial Isometries}).\end{definition}

\begin{proposition} The induced collage $\tilde{\mathcal{P}}$ of a pattern $\mathcal{P}$ is a collage.\end{proposition}

\begin{proof} Firstly, each relation $R_{\tilde{\mathcal{P}}}$ is an equivalence relation since each is reflexive (as $Id_{B_{d_X}(x,R)} \in \mathcal{P}$), symmetric (since $\mathcal{P}$ is closed under inverses) and transitive (since $\mathcal{P}$ is closed under (partial) composition).

Given $r \leq R$, for an isometry $\Phi \in \mathcal{P}_{x,y}^R$ we have that $Id_{B_{d_X}}(y,r) \circ \Phi \circ Id_{B_{d_X}}(x,r)$ is an element of $\mathcal{P}_{x,y}^r$ and so $R_{\tilde{\mathcal{P}}} \subset r_{\tilde{\mathcal{P}}}$.

Finally, let $R$ be given and $\epsilon>0$ be arbitrary. Given $U \in \mathcal{U}_X$ we set $V=U_a \subset U$ for some $a \leq \epsilon$. We wish to show that $V \circ (R+\epsilon)_{\tilde{\mathcal{P}}} \subset R_{\tilde{\mathcal{P}}} \circ U$. So, take some $(z,x) \in V \circ (R+\epsilon)_{\tilde{\mathcal{P}}}$. That is, there exists some $y \in X$ with $(z,y) \in V$ and $(y,x) \in (R+\epsilon)_{\tilde{\mathcal{P}}}$. We have that $z \in B_{d_X}(y,a)$ and $z \in ran(\Phi)$ for some $\Phi \in \mathcal{P}_{x,y}^{R+\epsilon}$; define $z':= \Phi^{-1}(z)$. Since $\Phi$ (and its inverse) is an isometry, we have that $d_X(x,z') < a$, $B_{d_X}(z',R) \subset dom(\Phi)$ and $B_{d_X}(z,R) \subset ran(\Phi)$. It follows that $Id_{B_{d_X}(z,R)} \circ \Phi \circ Id_{B_{d_X}(z',R)} \in \mathcal{P}_{z',z}^R$ so $(z,z') \in R_{\tilde{\mathcal{P}}}$ and $(z',x) \in U_a$, that is, $(z,x) \in R_{\tilde{\mathcal{P}}} \circ U_a \subset R_{\tilde{\mathcal{P}}} \circ U$.\end{proof}

It follows that constructions associated to collages may also be associated to patterns, so we shall often drop the tilde from the notation of an induced collage e.g., the space $\Omega^{\tilde{\mathcal{P}}}$ (see the following subsection) associated to the induced collage $\tilde{\mathcal{P}}$ of $\mathcal{P}$ will be written as $\Omega^\mathcal{P}$.

\begin{proposition} Let $(X,d_X)$ be a proper metric space and $\mathcal{P}$ be a pattern on it. Then $\mathcal{P}$ is FLC if and only if its induced collage $\tilde{\mathcal{P}}$ is FLC. \end{proposition}

\begin{proof} Suppose that $\mathcal{P}$ is FLC. Then, for each $r > 0$, there exists some compact $K \subset X$ which contains a representative of each equivalence class of $r_\mathcal{P}$. That means that the approximant $K_r^\mathcal{P}:=X/r_\mathcal{P}$ is the continuous image of a compact space and so is compact and hence $\tilde{\mathcal{P}}$ is FLC.

Conversely, suppose that $\tilde{\mathcal{P}}$ is FLC and $r > 0$. Then there exists some $R > 0$ such that, for each $U \in \mathcal{U}$, there exists some $V \in \mathcal{U}$ with $V \circ R \subset r \circ U$. Let $\delta$ be such that each $\delta$-ball is relatively compact and $\epsilon < \delta$ be such that $U_\epsilon \circ R \subset r \circ U_\delta$ (of course, by properness any $\delta > 0$ will do so in particular such a $\delta$ exists). For every $[x]_R \in K_R^\mathcal{P}$, we have that $\pi_{R,X}(U_\epsilon([x]_R))$ is an open neighbourhood of $[x]_R$ since, given $x \in X$ with $\pi_{R,X}(x)=[x]_R$, we have that $U_\epsilon(x) \subset R \circ U_\epsilon \circ R(x) = \pi_{R,X}^{-1}(\pi_{R,X}(U_\epsilon)([x]_R))$. Then $\{\pi_{R,X}(U)[x]_R \mid x \in X \}$ is such that the interiors cover $K_R^\mathcal{P}$ and, by FLC, there must exist $x_1, \ldots, x_k$ with $\bigcup_{i=1}^k \pi_{R,X}(U)([x_i]) = K_R^\mathcal{P}$. That is, we have that $X=R \circ U_\epsilon \circ R(\{x_1,\ldots,x_k\}) \subset R \circ r \circ U_\delta(\{x_1,\ldots,x_k\}) = r \circ U_\delta(\{x_1,\ldots,x_k\})$. This means that a representative of each $r_{\tilde{\mathcal{P}}}$-equivalence class is contained in $U_\delta(\{x_1,\ldots,x_k\})$. Since each $U_\delta(x_i)$ is contained in a compact subset $K_i \subset X$, we have that a representative of each $r_{\tilde{\mathcal{P}}}$-equivalence class is contained in the compact set $\bigcup_{i=1}^k K_i$ so $\mathcal{P}$ is FLC. \end{proof}

Note that it follows, by the above and Proposition \ref{prop: FLC iff FLC}, that a locally finite tiling $T$ of a proper metric space $(X,d_X)$ is FLC with respect to allowed partial isometries $\mathcal{S}$ if and only if its induced collage is FLC.

\subsection{Completions of Uniform Spaces Induced from Collages}

We have seen that, given a collage, one may associate a uniform space to it. The definition of a collage also suggests the construction of a certain topological space, an ``inverse limit of approximants'':

\begin{definition} \label{def: Pattern Space} Given a collage $\tilde{\mathcal{P}}$ define the \emph{approximants} $K_\lambda^{\tilde{\mathcal{P}}}$ as the quotient spaces of $X$ (equipped with the topology compatible with $\mathcal{U}_X$) by the equivalence relations $\lambda_{\tilde{\mathcal{P}}}$, denoting the quotient maps by $\pi^{\tilde{\mathcal{P}}}_{\lambda,X} \colon X \rightarrow K_\lambda^{\tilde{\mathcal{P}}}$. Since for each $\lambda \leq \mu$ we have that $\mu_{\tilde{\mathcal{P}}} \subset \lambda_{\tilde{\mathcal{P}}}$ we have \emph{connecting maps} given by the induced quotients $\pi^{\tilde{\mathcal{P}}}_{\lambda,\mu} \colon K_\mu^{\tilde{\mathcal{P}}} \rightarrow K_\lambda^{\tilde{\mathcal{P}}}$. As $\Lambda$ is directed the connecting maps between the approximants form an inverse system, for which we denote the inverse limit as $\Omega^{\tilde{\mathcal{P}}}:= \varprojlim (K_\mu^{\tilde{\mathcal{P}}},\pi^{\tilde{\mathcal{P}}}_{\lambda,\mu})$. Of course, there are projection maps from the inverse limit to the approximants, which we denote by $\pi^{\tilde{\mathcal{P}}}_{\lambda,\infty} \colon \Omega^{\tilde{\mathcal{P}}} \rightarrow K_\lambda^{\tilde{\mathcal{P}}}$ and there is a map from $X$ to the inverse limit, which we denote by $\pi^{\tilde{\mathcal{P}}}_{\infty,X} \colon X \rightarrow \Omega^{\tilde{\mathcal{P}}}$. \end{definition}

It shall frequently be the case that only one collage will be in question, so we shall often drop its notation from the indexing of the connecting maps etc.

\subsubsection{Collages of Euclidean Tilings}

We shall make a brief detour to see how these definitions apply to Euclidean tilings. By design, the resulting constructions coincide with those of the literature, see, e.g., \cite{Sadun2,BDHS}.

Let $T$ be a tiling of $(\mathbb{R}^d,d_{euc})$. Recall the definitions of the patterns $\mathcal{T}_1$, $\mathcal{T}_0$ and $\mathcal{T}_{\text{rot}}$ as defined in the previous subsection. We have that $x R_{\mathcal{T}_1} y$ if and only if the tilings $T-x$ and $T-y$ agree at the origin to radius $R$. The induced collage uniformity for $\mathcal{T}_1$ sets points $x$ and $y$ of $\mathbb{R}^d$ to be ``close'' if and only if they both have nearby points where the tilings centred there agree to some large radius. This uniformity will usually only be suitable for tilings which have FLC with respect to translations. We have that the approximants $K_R$ are precisely the BDHS approximants \cite{BDHS}. For an FLC (with respect to translations) tiling, these approximants are branched $d$-manifolds, the branching occurs at points for which there is a point of the boundary of a tile at distance $R$ to the represented patch where there is a choice of extension of the patch. The connecting maps $\pi_{r,R} \colon K_R \rightarrow K_r$ are the quotient maps between approximants which ``forget'' patch information from distance $r$ onwards.

The uniformity $\mathcal{U}_{\mathcal{T}_1}$ on $\mathbb{R}^d$ agrees with the uniformity induced by the usual tiling metric \cite{AP} on $\mathbb{R}^d$, where we consider a point $x \in \mathbb{R}^d$ as representing the tiling $T-x$. The inverse limit of approximants $\Omega^{\mathcal{T}_1}$ is homeomorphic to the completion of this uniform space. This space is often called the \emph{translational hull} or, when we are being less specific, the \emph{tiling space} of $T$. For non-periodic tilings this space is connected but not locally connected and is globally a fibre bundle with base space the $d$-torus and fibre a totally disconnected set \cite{SW} (which, for a repetitive tiling, is the Cantor set). Points of $\Omega^{\mathcal{T}_1}$ may themselves be considered as tilings, they are those tilings whose bounded patches are translates of patches from the tiling $T$. We have that $(\Omega^{\mathcal{T}_1},\mathbb{R}^d)$ is a dynamical system, where $\mathbb{R}^d$ acts on the tiling space $\Omega^{\mathcal{T}_1}$ by translations.

For $\mathcal{T}_{\text{rot}}$ defined on $E^+(d)$ we have that $f R_{\mathcal{T}_{\text{rot}}} g$ if and only if $f^{-1}(T)$ and $g^{-1}(T)$ agree at the origin to radius $R$. So we have that the induced collage uniformity for $\mathcal{T}_{\text{rot}}$ sets $f$ and $g$ as ``close'' if and only if there are two tilings which are small rigid motions from $f^{-1}(T)$ and $g^{-1}(T)$ which agree at the origin to a large radius. This uniformity will usually only be suitable for tilings with FLC with respect to rigid motions. Again, in the FLC case, the \emph{rigid hull} (or \emph{tiling space}) is the completion of the uniform space $(E^+(d),\mathcal{U}_{\mathcal{T}_{\text{rot}}})$ and is homeomorphic to the inverse limit of approximants $K_R$, which are branched $(d(d+1)/2)$-manifolds. We have that $(\Omega^{\mathcal{T}_{\text{rot}}},E^+(d))$ is a dynamical system, where the group $E^+(d)$ acts on $\Omega^{\mathcal{T}_{\text{rot}}}$ by rigid motions of tilings.

The final pattern that we associated to a tiling of $(\mathbb{R}^d,d_{euc})$ was $\mathcal{T}_0$. We have that $x R_{\mathcal{T}_0} y$ if and only if $T-x$ and $T-y$ agree at the origin to radius $R$, perhaps after a further rotation at origin. The uniformity $\mathcal{U}_{\mathcal{T}_0}$ considers $x$ and $y$ to be ``close'' if and only if each have nearby points such that the tilings centred there agree to a large radius, perhaps after some rotation. This uniformity is usually only suitable for tilings which have FLC with respect to rigid motions. In this case, the tiling space as the completion of this uniform space is homeomorphic to the inverse limit of approximants. We have that $\Omega^{\mathcal{T}_0} \cong \Omega^{\mathcal{T}_{\text{rot}}}/SO(d)$ where $SO(d)$ acts on $\Omega^{\mathcal{T}_{\text{rot}}}$ by rotations. In case the tiling $T$ was FLC with respect to translations as well, we have that $\Omega^{\mathcal{T}_0}$ is a quotient of $\Omega^{\mathcal{T}_1}$ given by identifying tilings of the hull which are rotates of each other.

None of the constructions above required our tiling to possess FLC. However, it is usually the case that for non-FLC tilings the induced uniformities will not be suitable. For example, we may wish for two tilings to be ``close'' if they agree to a large radius up to some small perturbation which is not just a rigid motion. A metric which is capable of dealing with a rather general class of non-FLC tilings of $\mathbb{R}^d$ is given in \cite{PF}.

Our motivations here are not to provide a good setting with which to study non-FLC tilings. Constructions such as the pattern-equivariant cohomology rely on the tiling space being given as an inverse limit of approximants which are quotients of the ambient space of the tiling. So one certainly should not expect such techniques to be applicable to non-FLC tilings for which the tiling space possesses extra dimensions to the ambient space of the tiling, arising from small perturbations of the tiling. There may in certain specific settings, however, be ad hoc techniques with which one may represent these extra dimensions on an ambient space for the tiling c.f., the pattern $\mathcal{T}_{\text{rot}}$ defined on the Euclidean group $E^+(d)$.

\subsubsection{Collages of Hierarchical Tilings}

Recall our definition of a hierarchical tiling $T_\omega = (T_0,T_1,\ldots)$ as given in the previous subsection. We have that $x R_{\mathcal{T}_\omega} y$ if and only if (given our condition on the ambient metric space of the hierarchical tiling) there is an allowed isometry $\Phi \in (\mathcal{T}_{k(R)})_{x,y}^R$ between $x$ and $y$ of radius $R$ preserving patches of $T_{k(R)}$. So two points are ``close'' if there is a tiling $T_n$ which has large patches about $x$ and $y$ which agree up to an allowed partial isometry to some large radius for some large $n$.

Note that the collage uniformity and inverse limit of approximants $\Omega^{\mathcal{T}_\omega}$ are independent of the function $k$ (in the language to follow, the induced collages for different choices of $k$ are equivalent). Let $T_\omega$ be some hierarchical tiling coming from a substitution rule on $(\mathbb{R}^d,d_{euc})$ with inflation factor $\lambda$. Then taking some $\epsilon>0$ (typically small relative to the size of the tiles) it is convenient to pick $k(t):=\max \{\lceil \log_\lambda (t /\epsilon) \rceil,0\}$; that is $k(t)$ is the smallest $n \in \mathbb{N}_0$ for which $t/\lambda^n \le \epsilon$. We have that $(\mathcal{T}_\omega)_{x,y}^{\lambda^n\epsilon} \neq \emptyset$ if and only if the patches of tiles within distance $\lambda^n \epsilon$ of $x$ and $y$ agree up to an isometry taking $x$ to $y$ in the $n$-th supertile decomposition $T_n$. Then (when the allowed partial isometries are given as translations or rigid motions) each of the approximants $K_{\lambda^n\epsilon}$ for $n \in \mathbb{N}_0$ are homeomorphic and in a way which makes all of the connecting maps between them the same -- these are just the BDHS-complexes for the tiling with a fixed parameter and the connecting maps are those induced by substitution \cite{BDHS}.

Suppose that we allowed, and then defined $k$ to be the constant function $t \mapsto 0$. In this case $\Omega^{\mathcal{T}_\omega} \cong \Omega^\mathcal{T}$, the tiling space of $T_0$. For any function $k$ there will always exist a (surjective) continuous map $f \colon \Omega^{\mathcal{T}_\omega} \rightarrow \Omega^\mathcal{T}$. Under certain circumstances one can also find an inverse. One can define a continuous map in the other direction if one can determine the local configuration of supertiles given enough information about the local configuration of tiles; the notion to consider here is that of ``recognisability''. In such a situation, the hierarchical tiling is determined by $T_0$. But where the substitution is not invertible, the spaces $\Omega^{\mathcal{T}_\omega}$ and $\Omega^\mathcal{T}$ will differ. One may like to view the pattern $\mathcal{T}_\omega$ in such a case as instead a tiling which ``knows'' where its tiles are as well as how to combine them into supertiles, how to combine those into super$^2$-tiles and so on. We illustrate this in the following example.

\begin{exmp}\label{ex: DS} The $d$-dimensional dyadic solenoid $\mathbb{S}_2^d$ is the inverse limit of the $d$-torus $\mathbb{R}^d/ \mathbb{Z}^d$ under the times two map induced from the map $x \mapsto 2x$ in $\mathbb{R}^d$. Although the dyadic solenoid is not the translational hull for any tiling of $\mathbb{R}^d$ because of its equicontinuous $\mathbb{R}^d$-action \cite{CH}, it can be realised as the translational hull of a hierarchical tiling. That is, we can define a hierarchical tiling $T_\omega$ such that, with the collection of allowed partial isometries given as translations, the space $\Omega^{\mathcal{T}_\omega}$ is homeomorphic to the dyadic solenoid. It comes from the non-recognisable substitution of a single $d$-cube prototile, which subdivides to $2^d$ $d$-cubes in the obvious way. Such a hierarchical tiling consists of a sequence $T_\omega= (T_0,T_1,\ldots )$ of periodic tilings $T_i$ of $d$-cubes with sides length $2^i$ such that $T_i$ subdivides to $T_{i-1}$. For example, one may pick each $T_i$ as the periodic tiling with vertex of a cube at the origin.

To explain our comment on the failure of the ``glueing'' of partial isometries, consider the pattern associated to the above hierarchical tiling with $k(t) := \max \{\lceil \log_2 (t) \rceil,0\}$. We see that the partial translations $x \mapsto x + y$ restricted to open balls of radius $2^n$ with $y \in 2^n \mathbb{Z}^d$ are elements of $\mathcal{T}_n$ for all $n \in \mathbb{N}_0$ and hence are also elements of $\mathcal{T}_\omega$. However, translation by such a $y$ defined on the whole of $\mathbb{R}^d$ is not an element of $\mathcal{T}_\omega$, since the larger domain and range of this global isometry of $\mathbb{R}^d$ ``sees'' that the translation does not preserve the tilings of $T_i$ for $i$ large enough.\end{exmp}

\subsubsection{Inverse Limits of Approximants for General Collages}

We shall now consider general collages and show that, under relatively general conditions, we have that the inverse limit of approximants $\Omega^{\tilde{\mathcal{P}}}$ is homeomorphic to a completion of $(X,\mathcal{U}_{\tilde{\mathcal{P}}})$.

\begin{lemma} For a collage $\tilde{\mathcal{P}}$ we have that $\pi^{\tilde{\mathcal{P}}}_{\lambda,X}(U) = \pi^{\tilde{\mathcal{P}}}_{\lambda,X}(\lambda_{\tilde{\mathcal{P}}} \circ U) = \pi^{\tilde{\mathcal{P}}}_{\lambda,X}(U \circ \lambda_{\tilde{\mathcal{P}}})$. We have that $\pi^{\tilde{\mathcal{P}}}_{\lambda,X}(U)^2 = \pi^{\tilde{\mathcal{P}}}_{\lambda,X}(U \circ \lambda_{\tilde{\mathcal{P}}} \circ U)$.\end{lemma}

\begin{proof} Since $\lambda$ is reflexive we have that $\pi_{\lambda,X}(U) \subset \pi_{\lambda,X}(\lambda \circ U)$. For the other inclusion, take $(x',z') \in \pi_{\lambda,X}(\lambda \circ U)$. That is, there exists $x,y,z \in X$ with $(x,y) \in \lambda$, $(y,z) \in U$, $\pi_{\lambda,X}(x)=x'$ and $\pi_{\lambda,X}(z)=z'$. But then $\pi_{\lambda,X}(x)=\pi_{\lambda,X}(y)$ so $\pi_{\lambda,X}(y,z)=(x',z') \in \pi_{\lambda,X}(U)$. The other equality is proved analogously.

It follows that $\pi_{\lambda,X}(U)^2 = \pi_{\lambda,X}(U) \circ \pi_{\lambda,X}(\lambda \circ U)$. This latter set is equal to $\pi_{\lambda,X}(U \circ \lambda \circ U)$. Take $(x',z') \in \pi_{\lambda,X}(U) \circ \pi_{\lambda,X}(\lambda \circ U)$. That is, there exists $(x,y_1) \in U$ and $(y_2,z) \in \lambda \circ U$ with $\pi_{\lambda,X}(x)=x'$, $\pi_{\lambda,X}(y_1)=\pi_{\lambda,X}(y_2)$ and $\pi_{\lambda,X}(z)=z'$. But then $y_1 \lambda y_2$ so $(y_1,z) \in \lambda \circ U$ and hence $(x,z) \in U \circ \lambda \circ U$ so $(x',z') \in \pi_{\lambda,X}(U \circ \lambda \circ U)$. For the other inclusion, take $(x',z') \in \pi_{\lambda,X}(U \circ \lambda \circ U)$. Then there exists $x,y,z$ such that $(x,y) \in U$, $(y,z) \in \lambda \circ U$, $\pi_{\lambda,X}(x)=x'$ and $\pi_{\lambda,X}(y)=y'$. So then $(x',z') \in \pi_{\lambda,X}(U) \circ \pi_{\lambda,X}(\lambda \circ U)$. \end{proof}

The above may be seen as giving an obstruction to the pushforwards of entourages defining uniformities on the approximants. Given some $\lambda$, to ensure approximate transitivity for the pushforward of entourages from $X$ to $K_\lambda$, we would want to find, for each entourage $U \in \mathcal{U}_X$, some entourage $V \in \mathcal{U}_X$ such that $V \circ \lambda \circ V \subset \lambda \circ U \circ \lambda$, for then by the above we have that $\pi_{\lambda,X}(V)^2 = \pi_{\lambda,X}(V \circ \lambda \circ V) \subset \pi_{\lambda,X}(\lambda \circ U \circ \lambda) = \pi_{\lambda,X}(U)$. Since $\mathcal{U}_X$ is a uniformity there exists some $W$ such that $W \circ W \subset U$ and we could then use any $V \subset W$ such that $V \circ \lambda \subset \lambda \circ W$, if it existed. Axiom $2$ of a collage forces something very much like this condition, and indeed the problem disappears when we push forward to the whole inverse limit instead of just the approximants:

\begin{proposition} Define $\mathcal{U}_{\Omega^{\tilde{\mathcal{P}}}}$ to be the upwards closure of the collection of sets of the form $\pi_{\lambda,\infty}^{-1}(\pi_{\lambda,X}(U)) \subset \Omega^{\tilde{\mathcal{P}}} \times \Omega^{\tilde{\mathcal{P}}}$ where $U \in \mathcal{U}_X$. Then $\mathcal{U}_{\Omega^{\tilde{\mathcal{P}}}}$ is a uniformity on $\Omega^{\tilde{\mathcal{P}}}$. \end{proposition}

\begin{proof} Since the diagonal $\Delta_X \subset U$ for each $U \in \mathcal{U}_X$ we have that $\Delta \subset \pi_{\lambda,X}(U)$ so $\Delta \subset \pi_{\lambda,\infty}^{-1}(\pi_{\lambda,X}(U))$.

For each $U \in \mathcal{U}_X$ we have that $U^{\text{op}} \in \mathcal{U}_X$ and $\pi_{\lambda,\infty}^{-1}(\pi_{\lambda,X}(U))^{\text{op}} = $ \linebreak $\pi_{\lambda,\infty}^{-1}(\pi_{\lambda,X}(U)^{\text{op}})$ $= \pi_{\lambda,\infty}^{-1}(\pi_{\lambda,X}(U^{\text{op}}))$ and hence $\mathcal{U}_{\Omega^{\tilde{\mathcal{P}}}}$ is closed under inverses.

Let $U,V \in \mathcal{U}_X$ and $\lambda,\mu \in \Lambda$. Since $\Lambda$ is directed there exists some $\nu \geq \lambda,\mu$. We see that \[\pi_{\nu,\infty}^{-1}(\pi_{\nu,X}(U \cap V)) \subset \pi_{\nu,\infty}^{-1}(\pi_{\nu,X}(U) \cap \pi_{\nu,X}(V)) =\] \[\pi_{\nu,\infty}^{-1}(\pi_{\nu,X}(U)) \cap \pi_{\nu,\infty}^{-1}(\pi_{\nu,X}(V)) \subset \pi_{\lambda,\infty}^{-1}(\pi_{\lambda,X}(U)) \cap \pi_{\mu,\infty}^{-1}(\pi_{\mu,X}(V))\] so $\mathcal{U}_{\Omega^{\tilde{\mathcal{P}}}}$ is closed under finite intersections.

Finally, given $\pi_{\lambda,\infty}^{-1}(\pi_{\lambda,X}(U))$, we need to show that there exists some $\mu$ and $W$ with $\pi_{\mu,\infty}^{-1}(\pi_{\mu,X}(W))^2 \subset \pi_{\lambda,\infty}^{-1}(\pi_{\lambda,X}(U))$. Since $\mathcal{U}_X$ is a uniformity there exists some $V \in \mathcal{U}_X$ with $V \circ V \subset U$. By the definition of a collage, there exists $\mu \geq \lambda$ and $W \in \mathcal{U}_X$ with $W \subset V$ such that $W \circ \mu \subset \lambda \circ V$. Hence \[\pi_{\mu,\infty}^{-1}(\pi_{\mu,X}(W))^2 = \pi_{\mu,\infty}^{-1}(\pi_{\mu,X}(W)^2) = \pi_{\mu,\infty}^{-1}(\pi_{\mu,X}(W \circ \mu \circ W)) \subset\] \[\pi_{\mu,\infty}^{-1}(\pi_{\mu,X}(\lambda \circ V \circ V)) \subset \pi_{\mu,\infty}^{-1}(\pi_{\mu,X}(\lambda \circ U))\subset \pi_{\lambda,\infty}^{-1}(\pi_{\lambda,X}(U)).\] \end{proof}

\begin{proposition} The induced topology on the uniform space $(\Omega^{\tilde{\mathcal{P}}},\mathcal{U}_{\Omega^{\tilde{\mathcal{P}}}})$ agrees with the original topology on $\Omega^{\tilde{\mathcal{P}}}$ as an inverse limit of quotient spaces. \end{proposition}

\begin{proof} A neighbourhood base for the topology of $\Omega$ at $\underleftarrow{x} \in \Omega$ consists of subsets of the form $\pi_{\lambda,\infty}^{-1}(N)$ where $N$ is a neighbourhood of $\pi_{\lambda,\infty}(\underleftarrow{x})$ in $K_\lambda$. So let $\lambda$ and $N$ be so given; we need to show that there exists some $\mu$ and $V \in \mathcal{U}_X$ such that $\pi_{\mu,\infty}^{-1}(\pi_{\mu,X}(V))(\underleftarrow{x}) \subset \pi^{-1}_{\lambda,\infty}(N)$.

By the definition of a collage, there exists some $\mu \geq \lambda$ such that for all $U \in \mathcal{U}_X$ there is some $V \in \mathcal{U}_X$ with $V \circ \mu \subset \lambda \circ U$. We have that $N$ is a neighbourhood of $\pi_{\lambda,\infty}(\underleftarrow{x})$ if and only if $\pi_{\lambda,X}^{-1}(N)$ is a neighbourhood of each $x \in X$ with $\pi_{\lambda,X}(x) = \pi_{\lambda,\infty}(\underleftarrow{x})$. Take such an $x$, then there exists some entourage $U \in \mathcal{U}_X$ such that $U(x) \subset \pi^{-1}_{\lambda,X}(N)$. Set $V \in \mathcal{U}_X$ and $\mu$ so that $V \circ \mu \subset \lambda \circ U$. We claim that $\pi_{\mu,\infty}^{-1}(\pi_{\mu,X}(V))(\underleftarrow{x}) \subset \pi_{\lambda,\infty}^{-1}(N)$. Indeed, \[\pi_{\mu,\infty}^{-1}(\pi_{\mu,X}(V)) (\underleftarrow{x}) = \pi_{\mu,\infty}^{-1}(\pi_{\mu,X}(V \circ \mu))(\underleftarrow{x}) = \pi_{\mu,\infty}^{-1}(\pi_{\mu,X} (V \circ \mu)(\pi_{\mu,\infty} (\underleftarrow{x}))) = \] \[\pi_{\mu,\infty}^{-1}(\pi_{\mu,X} (V \circ \mu (x))) \subset \pi_{\mu,\infty}^{-1}(\pi_{\mu,X} (\lambda \circ U(x))) \subset \pi_{\mu,\infty}^{-1}(\pi_{\mu,X} (\pi_{\lambda,X}^{-1}(N))).\] This final set is of course equal to $\pi_{\lambda,\infty}^{-1}(N)$ since it consists of all sequences of the inverse limit whose projections to $K_\mu$ are in the set $\pi_{\mu,X}(\pi_{\lambda,X}^{-1}(N)) = \pi_{\mu,X} ((\pi_{\lambda,\mu} \circ \pi_{\mu,X})^{-1}(N)) = \pi_{\mu,X}(\pi_{\mu,X}^{-1}(\pi_{\lambda,\mu}^{-1}(N))) = \pi_{\lambda,\mu}^{-1}(N)$, which are precisely those sequences which project to $N$.

Conversely, suppose that $N$ is an open neighbourhood of $\underleftarrow{x}$ in the induced topology from $(\Omega,\mathcal{U}_\Omega)$. That is, there exists some entourage $U \in \mathcal{U}_X$ and $\lambda \in \Lambda$ with $\pi_{\lambda,\infty}^{-1}(\pi_{\lambda,X}(U))(\underleftarrow{x}) \subset N$. We claim that $\pi_{\lambda,\infty}^{-1}(\pi_{\lambda,X}(U))(\underleftarrow{x})$ is an open neighbourhood of $\underleftarrow{x}$ in the original topology of $\Omega$. Indeed, we have that \\* $\pi_{\lambda,\infty}^{-1}(\pi_{\lambda,X}(U))(\underleftarrow{x})$ $=$ $\pi_{\lambda,\infty}^{-1}(\pi_{\lambda,X}(U)(\pi_{\lambda,\infty}(\underleftarrow{x})))$ so we are done if \\* $\pi_{\lambda,X}(U)(\pi_{\lambda,\infty}(\underleftarrow{x}))$ is a neighbourhood of $\pi_{\lambda,\infty}(\underleftarrow{x})$ in $K_\lambda$. Given $x \in X$ with $\pi_{\lambda,X}(x) = \pi_{\lambda,\infty}(\underleftarrow{x})$ we have that $U(x)$ is a neighbourhood of $x$ and \[U(x) \subset U \circ \lambda(x) \subset \pi_{\lambda,X}^{-1}(\pi_{\lambda,X}(U \circ \lambda(x))) =\] \[ \pi_{\lambda,X}^{-1}(\pi_{\lambda,X}(U \circ \lambda)(\pi_{\lambda,\infty}(\underleftarrow{x}))) = \pi_{\lambda,X}^{-1}(\pi_{\lambda,X}(U)(\pi_{\lambda,\infty}(\underleftarrow{x}))).\] \end{proof}

\begin{lemma} The induced map between Kolmogorov quotients \[(\pi^{KQ} \circ \pi_{\infty,X})^{KQ} \colon (X,\mathcal{U}_{\tilde{\mathcal{P}}})^{KQ} \rightarrow (\Omega^{\tilde{\mathcal{P}}},\mathcal{U}_{\Omega^{\tilde{\mathcal{P}}}})^{KQ}\] is injective and a uniform isomorphism onto its image. \end{lemma}

\begin{proof} To show injectivity we must show that two topologically distinguishable points $x,y \in X$ are mapped to topologically distinguishable points. So suppose (without loss of generality) that there exists some uniformity $U \circ \lambda \circ U \in \mathcal{U}$ with $y \not\in U \circ \lambda \circ U (x)$. There exists some $\mu$ and $V \in \mathcal{U}_X$ with $V \subset U$ such that $V \circ \mu \subset \lambda \circ U$ and $\mu \circ V \subset U \circ \lambda$.

We claim that $\pi_{\mu,X}(y) \not\in \pi_{\mu,X}(V)(\pi_{\mu,X}(x))$. Indeed, \[V \subset \mu \circ V \circ \mu  \subset \mu \circ V \circ V \circ \mu \subset U \circ \lambda \circ U.\] So we would have otherwise that there exists $(y',x') \in \mu \circ V \circ \mu$ with $\pi_{\mu,X}(x')=\pi_{\mu,X}(x)$ and $\pi_{\mu,X}(y')=\pi_{\mu,X}(y)$. But then $(y,x) \in \mu \circ V \circ \mu$ also so $(y,x) \in U \circ \lambda \circ U$, which we assumed not to be the case. Hence $\pi_{\mu,X}(y) \not\in \pi_{\mu,X}(V)(\pi_{\mu,X}(x))$ but that means that $\pi_{\infty,X}(y) \not\in \pi_{\infty,\mu}^{-1}(\pi_{\mu,X}(V)(\pi_{\mu,X}(x))) = \pi_{\mu,\infty}^{-1}(\pi_{\mu,X}(V))(\pi_{\infty,X}(x))$. It follows that $\pi_{\infty,X}(x)$ and $\pi_{\infty,X}(y)$ are topologically distinguishable in $(\Omega,\mathcal{U}_\Omega)$ so the map $(\pi^{KQ} \circ \pi_{\infty,X})^{KQ}$ is injective.

Finally, we must show that $(\pi^{KQ} \circ \pi_{\infty,X})^{KQ}$ is a uniform isomorphism onto its image which, by Lemma \ref{lem:unif iso}, amounts to showing that arbitrary entourages contain images or preimages of entourages under $\pi_{\infty,X}$. So let $\pi_{\lambda,\infty}^{-1}(\pi_{\lambda,X}(U)) \in \mathcal{U}_\Omega$. We can find $\mu$ and $V \in \mathcal{U}_X$ such that $V \circ \mu \circ V \subset \lambda \circ U$. But then $\pi_{\lambda,X}(V \circ \mu \circ V) \subset \pi_{\lambda,X}(\lambda \circ U) = \pi_{\lambda,X}(U)$. Hence $\pi_{\infty,X}(V \circ \mu \circ V) \subset \pi_{\lambda,\infty}^{-1}(\pi_{\lambda,X}(U))$. It follows that $(\pi^{KQ} \circ \pi_{\infty,X})^{KQ}$ is uniformly continuous. In the other direction, take an entourage $U \circ \lambda \circ U \in \mathcal{U}_{\tilde{\mathcal{P}}}$. Then there exists some $\mu$ and $V \in \mathcal{U}_X$ such that $\mu \circ V \circ \mu \subset U \circ \lambda \circ U$. Then we have that $\pi_{\infty,X}^{-1}(\pi_{\mu,\infty}^{-1}(\pi_{\mu,X}(V))) = \pi_{\mu,X}^{-1}(\pi_{\mu,X}(V)) = \mu \circ V \circ \mu \subset U \circ \lambda \circ U$ and so the inverse of $(\pi^{KQ} \circ \pi_{\infty,X})^{KQ}$ is uniformly continuous. \end{proof}

\begin{corollary}\label{cor:comp=>completion} If $(\Omega^{\tilde{\mathcal{P}}},\mathcal{U}_{\Omega^{\tilde{\mathcal{P}}}})$ is complete then it is a completion of $(X,\mathcal{U}_{\tilde{\mathcal{P}}})$.\end{corollary}

\begin{proof} By Proposition \ref{prop:kolmog compl} we have that $(Y,\mathcal{U}_Y)$ is a completion of $(X,\mathcal{U}_X)$ if $(Y,\mathcal{U}_Y)$ is complete and there exists a map $i \colon (X,\mathcal{U}_X) \rightarrow (Y,\mathcal{U}_Y)$ with dense image for which $(\pi^{KQ} \circ i)^{KQ}$ is injective and a uniform isomorphism onto its image. Since $\pi_{\infty,X}$ is surjective onto each $K_\lambda$ it has dense image and so by the above we have that $(\Omega,\mathcal{U}_\Omega)$ is a completion of $(X,\mathcal{U}_{\tilde{\mathcal{P}}})$. \end{proof}

So, in the cases where $(\Omega,\mathcal{U}_\Omega)$ can be shown to be complete, we have arrived at a generalisation of the BDHS description of the tiling space of an FLC tiling as an inverse limit of approximants. We consider now the question of when $(\Omega,\mathcal{U}_\Omega)$ is complete. The following definition and proposition show that, in fact, we need only consider this question up to a notion of equivalence of collages.

\begin{definition} \label{def: Equivalent} Let $\mathcal{P}$ and $\mathcal{Q}$ be two patterns on the metric space $(X,d_X)$. We shall say that $\mathcal{P}$ and $\mathcal{Q}$ are \emph{equivalent} if, for each $R_1 > 0$, there exists some $R_2 > 0$ such that $\mathcal{P}_{x,y}^{R_2} \subset \mathcal{Q}_{x,y}^{R_1}$ and $\mathcal{Q}_{x,y}^{R_2} \subset \mathcal{P}_{x,y}^{R_1}$ for all $x,y \in X$.

We shall say that two collages $\tilde{\mathcal{P}}$ and $\tilde{\mathcal{Q}}$ on $(X,\mathcal{U}_X)$ over directed sets $\Lambda$ and $M$, respectively, are \emph{equivalent} if for each $\lambda \in \Lambda$ there exists some $\mu \in M$ such that $\mu_{\tilde{\mathcal{Q}}} \subset \lambda_{\tilde{\mathcal{P}}}$ and for each $\mu \in M$ there exists some $\lambda \in \Lambda$ such that $\lambda_{\tilde{\mathcal{P}}} \subset \mu_{\tilde{\mathcal{Q}}}$. \end{definition}

\begin{exmp} It is easy to see that any two equivalent patterns have equivalent induced collages. A collage $\tilde{\mathcal{P}}$ over $(\mathbb{R}_{>0},\leq)$ restricted to a sub-poset $(S,\leq)$ where $S$ is unbounded is equivalent to $\tilde{\mathcal{P}}$ (for example, we may replace the uncountable set $\mathbb{R}_{>0}$ with the countable one $S=\mathbb{N}$). \end{exmp}

\begin{exmp} \label{ex: MLD} There is a well known MLD equivalence relation on tilings of $\mathbb{R}^d$. Loosely speaking, given two such tilings $T$ and $T'$, we say that $T'$ is \emph{locally derivable} from $T$ if the tiles of $T'$ intersecting a point $x$ are determined by the tiles of $T$ intersecting $B_{d_{euc}}(x,R)$ for some $R>0$ (see, e.g., \cite{Sadun2} for full details). We see that arbitrary finite patches of $T'$ are determined, to some finite radius, by the nearby tiles of $T$. If $T$ is locally derivable to $T'$ and vice versa, we say that $T$ and $T'$ are \emph{mutually locally derivable} or \emph{MLD}, for short. Equivalently, here, we have that $T$ is MLD to $T'$ if and only if the patterns $\mathcal{T}_1$ and $\mathcal{T}'_1$ (that is, the patterns taken with respect to translations) or their induced collages are equivalent.

MLD equivalence preserves many important properties of an FLC (with respect to translations) tiling. In essence, properties of tilings which only depend on MLD classes only depend on tilings ``up to local redecorations''. These local redecorations, however, need not respect rotational symmetries. For example, one may break symmetries of the square tiling of $\mathbb{R}^2$ by marking some (non-central) point of the square prototile, and yet the two tilings are MLD. This observation was made in \cite{BSJ} where the stricter notion of $S$-MLD equivalence was also introduced. In the language presented here, two tilings $T$ and $T'$ are $S$-MLD equivalent if and only if the patterns $\mathcal{T}_0$ and $\mathcal{T}_0'$, or $\mathcal{T}_{\text{rot}}$ and $\mathcal{T}_{\text{rot}}'$ (or their induced collages) are equivalent. Of course, one could easily allow for orientation-reversing isometries here by making the obvious modifications.

Given a tiling $T$ of $\mathbb{R}^d$, define $G^+_T$ to be the group of orientation-preserving isometries of $\mathbb{R}^d$ which preserve $T$. If $T$ and $T'$ are periodic, it is not too hard to see that $G_T^+ = G_{T'}^+$ if and only if $\mathcal{T}_0$ and $\mathcal{T}_0'$ are equivalent. Again, one makes the obvious modifications to accommodate orientation-reversing isometries. \end{exmp}

\begin{exmp} Given a pattern $\mathcal{T}_\omega$ associated to some hierarchical tiling $\mathcal{T}=(T_0,T_1,\ldots)$, we shall say that $T$ is \emph{recognisable} if $\mathcal{T}_\omega$ is equivalent to the pattern $\mathcal{T}$ associated to the tiling $T_0$ with the same collection of allowed partial isometries. Notice that one always has a continuous map $f \colon \Omega^{{\mathcal{T}_\omega}} \rightarrow \Omega^{\mathcal{T}}$, in the recognisable case this map is a homeomorphism. Hierarchical tilings are recognisable if, for any $t$ and radius $R_1$, there exists some (perhaps larger) radius $R_2$ such that an isometry preserving a patch of radius $R_2$ of $T_0$ also preserves a patch of radius $R_1$ of $T_t$. \end{exmp}

\begin{exmp} \label{ex: GC} The most natural approximants to associate to a tiling in our context are the BDHS approximants \cite{BDHS}, but one may also define the G\"{a}hler approximants as the approximants associated to some collage. Given a tile $t \in T$, define the \emph{$0$-corona} of $t$ to be the singleton patch $\{t\}$ and, for $N \in \mathbb{N}$, inductively define the \emph{$N$-corona} of $t$ to be the patch of tiles which intersect the $(N-1)$-corona of $t$. Say that $x$ and $y$ are \emph{$N$-collared equivalent} (with respect to allowed partial isometries $\mathcal{S}$) for $N \in \mathbb{N}_0$ if and only if $x$ and $y$ belong to tiles $t_x$ and $t_y$ for which their $N$-corona agree up to an allowed partial isometry, that is, there exists some $\Phi \in \mathcal{S}$ whose domain contains the $N$-corona of $t_x$ and takes it to the $N$-corona of $t_y$. Being $N$-collared equivalent is not necessarily an equivalence relation on $X$, transitivity may fail at points which lie at the intersection of tiles. Nevertheless, by taking the transitive closure of these equivalence relations, one defines equivalence relations $N_G$ on $X$ for each $N \in \mathbb{N}_0$. Under reasonable conditions the equivalence relations $N_G$ define a collage on $(X,\mathcal{U}_X)$ over the directed set $(\mathbb{N}_0,\leq)$ which is equivalent to $\tilde{\mathcal{T}_\mathcal{S}}$. The approximants $K_N:=X/N_G$ are the well-known \emph{G\"{a}hler approximants}, see for example \cite{Sadun2}.

Take for example an FLC tiling of $(\mathbb{R}^d,d_{euc})$ with $\mathcal{S}$ given by (partial) translations. The approximants are given by glueing representatives of $N$-collared tiles (that is, tiles along with the information of their $N$-corona) along their boundaries according to how these $N$-collared tiles are found in the tiling. See \cite{Sadun2} for a discussion of the G\"{a}hler approximants. We may similarly define a collage whose approximants are the G\"{a}hler approximants for $\mathcal{T}_{\text{rot}}$ by letting $N_G$ by the transitive closure of the relation given by setting $f \sim g$ whenever there exist tiles $t_f$ and $t_g$ at the origin of $f^{-1}(T)$ and $g^{-1}(T)$ with identical $N$-corona. The approximants are glued together from quotients of the spaces $t \times SO(d)$, where $t$ is an $N$-collared tile of $T$, where one quotients by the action of $SO(d)$ preserving the $N$-collared patch defined by the collared tile $t$.
\end{exmp}

\begin{proposition} Equivalent collages $\tilde{\mathcal{P}}$ and $\tilde{\mathcal{Q}}$ induce the same collage uniformities on $X$. Further, the uniform spaces $(\Omega^{\tilde{\mathcal{P}}},\mathcal{U}_{\Omega^{\tilde{\mathcal{P}}}})$ and $(\Omega^{\tilde{\mathcal{Q}}},\mathcal{U}_{\Omega^{\tilde{\mathcal{Q}}}})$ are uniformly isomorphic. \end{proposition}

\begin{proof} That two equivalent collages induce the same uniformities is clear from the definitions. To define a map $f \colon \Omega^{\tilde{\mathcal{P}}} \rightarrow \Omega^{\tilde{\mathcal{Q}}}$ proceed as follows. We define $f_{\mu}$ for $\mu \in M$ by taking $\lambda \in \Lambda$ such that $\lambda_{\tilde{\mathcal{P}}} \subset \mu_{\tilde{\mathcal{Q}}}$ and define $f_{\mu} = \pi_{\mu,\lambda} \circ \pi^{\tilde{\mathcal{P}}}_{\lambda,\infty}$, where $\pi_{\mu,\lambda}$ is the quotient map taking $(x)_{\lambda_{\tilde{\mathcal{P}}}}$ to the unique equivalence class $(x)_{\mu_{\tilde{\mathcal{Q}}}}$ containing it. Notice that for $\mu_1 \leq \mu_2$ we have that $\pi^{\tilde{\mathcal{Q}}}_{\mu_1,\mu_2} \circ f_{\mu_2} = f_{\mu_1}$. Indeed, since $\Lambda$ is directed there exists some $\lambda \geq \lambda_1,\lambda_2$. Then \[\pi^{\tilde{\mathcal{Q}}}_{\mu_1,\mu_2} \circ f_{\mu_2} = \pi^{\tilde{\mathcal{Q}}}_{\mu_1,\mu_2} \circ \pi_{\mu_2,\lambda} \circ \pi^{\tilde{\mathcal{P}}}_{\lambda,\infty} = \pi_{\mu_1,\lambda} \circ \pi^{\tilde{\mathcal{P}}}_{\lambda,\infty} = f_{\mu_1}.\] It follows from universal properties that the map sending $x \in \Omega^{\tilde{\mathcal{P}}}$ to $\prod_{\mu \in M} f_\mu(x)$ is a continuous map to $\Omega^{\tilde{\mathcal{Q}}}$. In fact, it is also uniformly continuous since, given $U \in \mathcal{U}_X$, we see that $f^{-1}(\pi_{\lambda,\infty}^{-1}(\pi_{\lambda,X}(U))) \subset \pi_{\mu,\infty}^{-1}(\pi_{\mu,X}(U))$. We may analogously construct a uniformly continuous map $g \colon \Omega^{\tilde{\mathcal{Q}}} \rightarrow \Omega^{\tilde{\mathcal{P}}}$.

Finally, we must show that $f$ and $g$ are inverses of each other. Notice that to define the $\lambda$ component of $g \circ f$, one finds $\mu \in M$ and $\nu \in \Lambda$ such that $\lambda_{\tilde{\mathcal{P}}} \subset \mu_{\tilde{\mathcal{Q}}} \subset \nu_{\tilde{\mathcal{P}}}$ and projects an element of the inverse limit to $K_\nu$ and then to $K_\lambda$. Given such $\lambda$ and $\nu$, since $\Lambda$ is directed, there exists some $\xi \geq \lambda,\nu$. We have that the projection of $g \circ f(x)$ to $K_\lambda$ is equal to $\pi_{\lambda,\nu}(\pi_{\nu,\infty}(x)) = \pi_{\lambda,\nu}(\pi_{\nu,\xi}(\pi_{\xi,\infty}(x))) = \pi_{\lambda,\xi}(\pi_{\xi,\infty}(x)) = \pi_{\lambda,\infty}(x)$. Hence $g \circ f(x)$ preserves $\pi_{\lambda,\infty}(x)$ for each approximant and so is the identity map on the inverse limit. Analogously, $f \circ g = Id_{\Omega^{\tilde{\mathcal{Q}}}}$. \end{proof}

\begin{proposition}\label{prop:FLC+Haus} Suppose that $\tilde{\mathcal{P}}$ is FLC and has Hausdorff approximants. Then $(\Omega^{\tilde{\mathcal{P}}},\mathcal{U}_{\Omega^{\tilde{\mathcal{P}}}})$ is a completion of $(X,\mathcal{U}_{\tilde{\mathcal{P}}})$.\end{proposition}

\begin{proof} It is well known that an inverse limit of compact Hausdorff spaces is compact, and hence is complete for any uniformity, and so the result follows.

For reference, we give a proof of the above claim that an inverse limit of compact Hausdorff spaces is compact. Suppose that $L:=\varprojlim(X_i,f_{i,j})$ where each $X_i$ is compact and Hausdorff. By Tychonoff's theorem we have that $\prod_{i \in \mathcal{I}} X_i$ is compact. Since closed subspaces of compact spaces are compact, we just need to show that $L$ is closed in $\prod X_i$. To see this, note that $L = \bigcap_{i \leq j} L_{i,j}$ where $L_{i,j} \subset \prod X_i$ consists of those elements for which $f_{i,j}(x_j) = x_i$ (that is, an element of the inverse limit is precisely an element which satisfies these consistency relations for each $i \leq j$). We have that $L_{i,j}$ is the preimage of the diagonal $\Delta_i \subset X_i \times X_i$ under the continuous map $(Id_{X_i} \times f_{i,j}) \circ (\pi_i \times \pi_j)$. But since $X_i$ is Hausdorff, we have that $\Delta_i$ is closed and so each $L_{i,j}$, and in particular their intersection $L$, is closed.\end{proof}

\begin{exmp} For a locally finite tiling $T$ of $(\mathbb{R}^d,d_{euc})$ with FLC with respect to translations or rigid motions its approximants are Hausdorff and so we have that $\Omega^{\tilde{\mathcal{T}}}$ homeomorphic to the (Hausdorff) completion of $(\mathbb{R}^d,\mathcal{U}_{\tilde{\mathcal{T}}})$, which is simply a restatement of the result that the tiling space is homemorphic to the inverse limit of its BDHS approximants of increasing patch radii \cite{BDHS}.

Suppose that in the definition of the tiling collage, one instead set $xRy$ if the pointed tilings agreed on their \emph{closed} balls of radius $R$. In the aperiodic case the approximants are non-Hausdorff at the branch points, which correspond to points where the boundaries of $R$-balls intersect the boundary of the local patch of radius $R$. However, this modified collage is equivalent to the usual one of Hausdorff approximants and so the non-Hausdorffness of the approximants is not an issue. \end{exmp}

That the approximants were Hausdorff was essential in the above proof for showing that the inverse limit is closed in the product of approximants. Although for the most part we shall only be interested in collages with Hausdorff approximants, we shall consider here the general case in some more detail.

It is easy to construct inverse limits of compact but non-Hausdorff spaces which are non-compact. Determining whether or not an inverse limit of compact but non-Hausdorff spaces is compact (or even non-empty) can be subtle, for some positive results see \cite{Stone}. For example, if the inverse limit of compact spaces is taken over $(\mathbb{N},\leq)$ and the connecting maps are closed then the inverse limit is compact. Our goal here was to show that $(\Omega^{\tilde{\mathcal{P}}},\mathcal{U}_{\Omega^{\tilde{\mathcal{P}}}})$ is complete, not necessarily compact. The following theorem shows that the inverse limit being closed in the product of approximants is sufficient (in fact, the theorem only requires a weaker condition, see the example of \ref{ex: non-Haus} below). Note that $A \subset X$ is closed if any only if for all nets $(x_\alpha)$ of $A$ with $(x_\alpha) \rightarrow x$ in $X$ then in fact $x \in A$.

\begin{theorem} \label{thm: inverse limit of approximants} Let $\tilde{\mathcal{P}}$ be a collage on $(X,\mathcal{U}_X)$. Suppose that there exists some $U \in \mathcal{U}_X$ for which $U(x)$ is precompact for each $x \in X$. Then we have that $(\Omega^{\tilde{\mathcal{P}}},\mathcal{U}_{\tilde{\mathcal{P}}})$ is a completion of $(X,\mathcal{U}_{\tilde{\mathcal{P}}})$ if, for each net $(x_\alpha)$ with $x_\alpha \in \Omega^{\tilde{\mathcal{P}}}$ and $(x_\alpha) \rightarrow x$ for some $x \in \prod_{\lambda \in \Lambda} K_\lambda$, we have that $(x_\alpha) \rightarrow l$ for some $l \in \Omega^{\tilde{\mathcal{P}}}$. In particular, if the approximants of $\tilde{\mathcal{P}}$ are Hausdorff then $(\Omega^{\tilde{\mathcal{P}}},\mathcal{U}_{\Omega^{\tilde{\mathcal{P}}}})$ is a completion of $(X,\mathcal{U}_{\tilde{\mathcal{P}}})$. \end{theorem}

\begin{proof} By Corollary \ref{cor:comp=>completion} we just need to show that $(\Omega,\mathcal{U}_\Omega)$ is complete. So let $(x_\alpha)$ be a Cauchy net in $(\Omega,\mathcal{U}_\Omega)$. We shall show that a subnet of $(x_\alpha)$, projected to $K_\lambda$, converges to a point of $K_\lambda$.

Let $U \in \mathcal{U}_X$ be such that $U(x)$ is precompact for every $x \in X$, that is, each $U(x)$ has compact closure. There exists some $\mu \in \Lambda$ and $V \in \mathcal{U}_X$ such that $V \circ \mu \subset \lambda \circ U$. Since $\pi_{\mu,\infty}^{-1}(\pi_{\mu,X}(V)) \in \mathcal{U}_\Omega$ and $(x_\alpha)$ is Cauchy, there exists some $\gamma$ for which $(x_\alpha,x_\beta) \in \pi_{\mu,\infty}^{-1}(\pi_{\mu,X}(V))$ for all $\alpha, \beta \geq \gamma$. Take $y \in X$ with $\pi_{\mu,X}(y) = \pi_{\mu,\infty}(x_\gamma)$. Then we have that \[\pi_{\mu,X}(V)(\pi_{\mu,\infty}(x_\gamma)) = \pi_{\mu,X}(V \circ \mu)(\pi_{\mu,\infty}(x_\gamma)) = \pi_{\mu,X}(V \circ \mu(y)) \subset \pi_{\mu,X}(\lambda \circ U(y)).\] It follows that $\pi_{\lambda,\mu}(\pi_{\mu,X}(V)(\pi_{\mu,\infty}(x_\gamma))) \subset \pi_{\lambda,\mu}(\pi_{\mu,X}(\lambda \circ U(y))) = \pi_{\lambda,X}(U(y)) \subset \pi_{\lambda,X}(\overline{U(y)})$. This last space is the image of a compact space and so is compact. It follows that projections of the Cauchy sequence to $K_\lambda$ are contained in a compact subspace and so there exists a subnet of $(y_\beta)$ which converges to some point $x_\lambda \in K_\lambda$.

It now follows from being Cauchy that the net, not just a subnet, converges in $K_\lambda$. Take $\lambda \leq \mu \leq \nu$, as in axiom $2$ of a collage. As above, there exists a subnet $(y_\beta)$ which converges to some $x_\nu$ in $K_\nu$. Define $x_\mu:=\pi_{\mu,\nu}(x_\nu)$ and $x_\lambda:=\pi_{\lambda,\nu}(x_\nu)=\pi_{\lambda,\mu}(x_\mu)$. We claim that $(\pi_{\lambda,\infty}(x_\alpha))$ converges to $x_\lambda$. Let $N$ be some neighbourhood of $x_\lambda$ in $K_\lambda$. One may find $V,W \in \mathcal{U}_X$ such that $\pi_{\nu,\infty}^{-1}(\pi_{\nu,X}(W))^2 \subset \pi_{\mu,\infty}^{-1}(\pi_{\mu,X}(V))$ and $\pi_{\mu,\infty}^{-1}(\pi_{\mu,X}(V)(x_\mu)) \subset \pi_{\lambda,\infty}^{-1}(N)$.

Now, since $\pi_{\nu,\infty}^{-1}(\pi_{\nu,X}(W))$ is an entourage of $(\Omega,\mathcal{U}_\Omega)$, there exists some $\gamma$ such that $(x_\alpha,x_\beta) \in\pi_{\nu,\infty}^{-1}(\pi_{\nu,X}(W))$ for all $\alpha,\beta \geq \gamma$. Notice that the set $\pi_{\nu,X}(W)(x_\nu)$ is a neighbourhood of $x_\nu$ in $K_\nu$. Indeed, we have that $W(x) \subset W \circ \nu(x) \subset \pi_{\nu,X}^{-1}(\pi_{\nu,X}(W \circ \nu(x))) = \pi_{\nu,X}^{-1}(\pi_{\nu,X}(W)(x_\nu))$ for any $x \in X$ with $\pi_{\nu,X}(x)=x_\nu$. Since $(\pi_{\nu,\infty}(x_\alpha))$ has a subnet converging to $x_\nu$, there exists some $\beta \geq \gamma$ such that $(\pi_{\nu,\infty}(x_\beta),x_\nu) \in \pi_{\nu,X}(W)$. That is, we have that $(x_\beta,\underleftarrow{x}) \in \pi_{\infty,\nu}^{-1}(\pi_{\nu,X}(W))$ for all $\underleftarrow{x}$ with $\pi_{\nu,\infty}(\underleftarrow{x})=x_\nu$. It follows that $(x_\alpha,x_\beta), (x_\beta,\underleftarrow{x}) \in \pi_{\nu,X}^{-1}(\pi_{\nu,X}(W))$ for $\alpha \geq \gamma$. That is, we have that $(x_\alpha,\underleftarrow{x}) \in \pi_{\nu,X}^{-1}(\pi_{\nu,X}(W))^2 \subset \pi_{\mu,\infty}(\pi_{\mu,X}(V))$ so that $\pi_{\mu,\infty}(x_\alpha) \in \pi_{\mu,\infty}^{-1}(\pi_{\mu,X}(V)(x_\mu)) \subset \pi_{\lambda,\infty}^{-1}(N)$. Hence, we have that $\pi_{\lambda,\infty}(x_\alpha) \in N$, so $(x_\alpha)$ converges to $x_\lambda$.

It follows that $(x_\alpha)$ converges in each approximant $K_\lambda$ and so converges in the product $\prod K_\lambda$. By assumption, this means that $(x_\alpha) \rightarrow l$ for some $l \in \Omega$ and so $(\Omega,\mathcal{U}_\Omega)$ is complete and hence, by Corollary \ref{cor:comp=>completion}, is a completion of $(X,\mathcal{U}_{\tilde{\mathcal{P}}})$. Note that the required property holds if $\Omega$ is closed in $\prod K_\lambda$, which is the case when $K_\lambda$ are Hausdorff, as in the proof of Proposition \ref{prop:FLC+Haus}.
\end{proof}

\begin{exmp} \label{ex: non-Haus} Let $(X,\mathcal{U}_X)$ satisfy the properness condition of the above theorem. It is clear that the condition on $\tilde{\mathcal{P}}$ is necessary for $\Omega^{\tilde{\mathcal{P}}}$ to be a completion of $(X,\mathcal{U}_{\tilde{\mathcal{P}}})$ (for, given a net $(x_\alpha)$ in $\Omega$ which converges in the product of approximants, it must be Cauchy and so should have a limit in $\Omega$ for $\Omega$ complete). It is not obvious to the author whether or not all collages satisfy the condition of the theorem (although we conjecture that counter-examples exist). It is certainly not necessary for the approximants to be Hausdorff. For example, consider the (non uniformly discrete) point pattern $\mathbb{Q} \subset \mathbb{R}$. Then each approximant is $\mathbb{R}/\mathbb{Q}$ since $x$ and $y$ are equivalent to any given radius if and only if $x-y \in \mathbb{Q}$. The connecting maps are the identity maps. It follows that $\Omega$ is homeomorphic to $\mathbb{R}/\mathbb{Q}$ which is compact and hence also complete (this is a topological space with the indiscrete topology). On the other hand, the induced collage uniformity on $\mathbb{R}$ is clearly the indiscrete one, consisting only of the entourage $\mathbb{R} \times \mathbb{R}$. Its Hausdorff completion is the one point space, which is in agreement with the fact that the Kolmogorov quotient and Hausdorff completion of $\mathbb{R}/\mathbb{Q}$ is also the one point space. \end{exmp}

\section{Patterns and Collages Revisited}
\subsection{Metric Collages} \label{subsect: Metric Collages}

The collages defined in the previous section have as ambient space some uniform space $(X,\mathcal{U}_X)$ and we have shown how one may associate a uniform space to one. The idea is based on the process of passing from a tiling to a tiling space defined by a certain metric induced by the tiling. We shall show here how one may restrict the notion of a collage to more closely mirror the classical set-up. For this we define a metric version of the collages introduced in Section \ref{sect: Collages}:

\begin{definition} A \emph{metric collage} $\tilde{\mathcal{P}}$ (on $(X,d_X)$) is a set of equivalence relations $R_{\tilde{\mathcal{P}}}$ on $X$, one for each $R \in \mathbb{R}_{> 0}$, satisfying the following properties: \begin{enumerate}

\item if $0 < r < R$ then $R_{\tilde{\mathcal{P}}} \subset r_{\tilde{\mathcal{P}}}$.

\item For all $0 < r < R$, if $x R_{\tilde{\mathcal{P}}} y$ and $x' \in B_{d_X}(x,r)$ then $x' (R-r)_{\tilde{\mathcal{P}}} y'$ for some $y'\in B_{d_X}(y,r)$.
\end{enumerate}
\end{definition}

It is not hard to show that the induced collage of a pattern is also a metric collage in the above sense. One may now attempt to associate a metric space to any metric collage. It turns out that one needs to introduce an extra restriction on the metric space $(X,d_X)$ to ensure that the triangle inequality is satisfied.

\begin{definition}\label{metric} Given a metric collage $\tilde{\mathcal{P}}$ of $(X,d_{X})$, define the function $d_{\tilde{\mathcal{P}}} \colon X\times X\rightarrow \mathbb{R}$ by $$d_{\tilde{\mathcal{P}}}(x,y):=\inf \left\{ \frac{1}{\sqrt{2}} \right\} \bigcup \left\{ \epsilon>0 \mid x'(\epsilon^{-1})_{\tilde{\mathcal{P}}} y' \text{ for } x'\in B_{d_X}(x,\epsilon),y'\in B_{d_X}(y,\epsilon) \right\}.$$ \end{definition}

\begin{exmp} One may define a metric collage associated to a tiling of $(X,d_X)$ analogously to the approach in the uniform space setting. Assuming the function $d_{\tilde{\mathcal{T}}}$ is a (pseudo)metric (see below) we see that two points $x,y$ of the tiling are considered as ``close'' under this metric if each have nearby points which agree on a large patch of tiles up to an allowed partial isometry. The function $d_{\tilde{\mathcal{T}}}$ is precisely the tiling metric defined in \cite{AP}. \end{exmp}

\begin{proposition} Suppose that the metric space $(X,d_X)$ satisfies the following: for all $x,z \in X$ with $d_X(x,z) < \epsilon_1+\epsilon_2 < 1/\sqrt{2}$ for $\epsilon_1,\epsilon_2>0$ there exists some $y \in X$ such that $d_X(x,y) < \epsilon_1$ and $d_X(y,z) < \epsilon_2$. Then for a metric collage $\tilde{\mathcal{P}}$ on $(X,d_X)$ the function $d_{\tilde{\mathcal{P}}}$ is a pseudometric on $X$.\end{proposition}

\begin{proof} That $d_{\tilde{\mathcal{P}}}(x,x)=0$ for all $x \in X$ follows from the fact that each $R_{\tilde{\mathcal{P}}}$ is reflexive and that $d_{\tilde{\mathcal{P}}}$ is symmetric follows from the fact that each $R_{\tilde{\mathcal{P}}}$ is symmetric.

Let $x,y,z\in X$. We intend to show that $d_{\tilde{\mathcal{P}}}(x,y)+d_{\tilde{\mathcal{P}}}(y,z)\ge d_{\tilde{\mathcal{P}}}(x,z)$. If one of $d_{\tilde{\mathcal{P}}}(x,y)$ or $d_{\tilde{\mathcal{P}}}(y,z)$ are $\frac{1}{\sqrt{2}}$ then we are done. So suppose not. Let $x_1 (\epsilon_1^{-1})_{\tilde{\mathcal{P}}} y_1$ and $y_2 (\epsilon_2^{-1})_{\tilde{\mathcal{P}}} z_2$ for some $\epsilon_1,\epsilon_2 < \frac1{\sqrt{2}}$ and $x_1 \in B_{d_X}(x,\epsilon_1),y_1 \in B_{d_X}(y,\epsilon_1)$, $y_2\in B_{d_X}(y,\epsilon_2),z_2 \in B_{d_X}(z,\epsilon_2)$. Then, by our assumption on $(X,d_X)$, there exists some $y'$ for which $d_X(y_1,y')<\epsilon_2$ and $d_X(y_2,y')<\epsilon_1$. Hence, there exist $x',z'$ for which $x' (\epsilon_1^{-1}-\epsilon_2)_{\tilde{\mathcal{P}}} y' (\epsilon_2^{-1}-\epsilon_1)_{\tilde{\mathcal{P}}} z'$ for $x'\in B_{d_X}(x,\epsilon_1+\epsilon_2)$, $z' \in B_{d_X}(z,\epsilon_1+\epsilon_2)$ since $\tilde{\mathcal{P}}$ is a metric collage.

We have $x' (\epsilon^{-1})_{\tilde{\mathcal{P}}} z'$ where
\[ \epsilon=min(\epsilon_{1}^{-1}-\epsilon_{2},\epsilon_{2}^{-1}-\epsilon_{1}).\]
For any numbers $0<a,b\le\frac{1}{\sqrt{2}}$, we have that $a^{-1}-b\ge(a+b)^{-1}$ so, in particular, $\epsilon^{-1} \ge (\epsilon_1+\epsilon_2)^{-1}$ and hence $x' ((\epsilon_1+\epsilon_2)^{-1})_{\tilde{\mathcal{P}}} z'$, proving the triangle inequality. \end{proof}

A metric space $(X,d_X)$ satisfying the condition of the above proposition (removing the condition $\epsilon_1+\epsilon_2 < 1/\sqrt{2}$) has \emph{approximate midpoints}. Without this condition, the triangle inequality is not necessarily satisfied for $d_{\tilde{\mathcal{P}}}$. It follows from the Hopf-Rinow theorem (see \cite{Roe}) that a complete and locally compact metric space with approximate midpoints is in fact a geodesic space, that is, the distance between points may be realised exactly as the length of a geodesic path between the two points. The following shows that our previous treatment of collages using uniformities is indeed a generalisation of the approach using metrics:

\begin{proposition} Let $\tilde{\mathcal{P}}$ be a metric collage on a metric space $(X,d_X)$ which satisfies the condition of the above proposition. Then the equivalence relations $R_{\tilde{\mathcal{P}}}$ defining $\tilde{\mathcal{P}}$ also define a collage on $(X,\mathcal{U}_X)$ over the directed set $(\mathbb{R}_{>0},\leq)$. The collage uniformity of this collage is equal to the uniformity induced by the metric $d_{\tilde{\mathcal{P}}}$. \end{proposition}

\begin{proof} It is clear that the induced collage on $(X,\mathcal{U}_X)$ satisfies axiom $1$ to be a collage. For axiom $2$, take some $r > 0$ and $U \in \mathcal{U}_X$. Since $U \in \mathcal{U}_X$, there exists some $\epsilon$ such that $U_\epsilon \subset U$ (recall the definition of $U_\epsilon$ from Example \ref{ex: met}), without loss of generality assume that $\epsilon < 1$ and set $V=U_\epsilon$. We claim that $V \circ (r+1)_{\tilde{\mathcal{P}}} \subset r_{\tilde{\mathcal{P}}} \circ U$. Indeed, given $(x,y) \in V \circ (r+1)_{\tilde{\mathcal{P}}}$ there exists some $x'$ with $d_X(x',x)<\epsilon$ and $x' (r+1)_{\tilde{\mathcal{P}}} y$. By the axioms for a metric collage, there exists some $y'$ such that $x r_{\tilde{\mathcal{P}}} y'$ and $d_X(y',y)<\epsilon$ so that $(y',y) \in U_\epsilon \subset U$. It follows that $(x,y) \in r_{\tilde{\mathcal{P}}} \circ U$ and so $\tilde{\mathcal{P}}$ is a collage over $(\mathbb{R}_{>0},\leq)$.

Let $0 < \epsilon < 1/\sqrt{2}$ and set $U = U_\epsilon$, $r=\epsilon^{-1}$. Then if $(x,y) \in U \circ r_{\tilde{\mathcal{P}}} \circ U$ we have that $d_{\tilde{\mathcal{P}}}(x,y) < \epsilon$. It follows that the collage uniformity is contained in the uniformity generated by the metric $d_{\tilde{\mathcal{P}}}$. Conversely, let $U \circ r_{\tilde{\mathcal{P}}} \circ U \in \mathcal{U}_{\tilde{\mathcal{P}}}$. Then there exists some $\delta$ with $U_\delta \subset U$. Set $\epsilon < \min \{\delta,r^{-1},1/\sqrt{2} \}$. Then if $d_{\tilde{\mathcal{P}}}(x,y) < \epsilon$ if follows that $(x,y) \in U \circ r_{\tilde{\mathcal{P}}} \circ U$. Hence, the collage uniformity is equal to the uniformity induced by the metric $d_{\tilde{\mathcal{P}}}$. \end{proof}

\subsection{Patterns of Uniform Spaces}

In the previous subsection, we considered a restriction of the notion of a collage from the setting of uniform spaces to that of metric spaces. We saw that complications arise in doing this, mostly from the need to satisfy the triangle inequality. For many applications, at least those of interest to us here, there is little advantage in having a metric to having a uniformity and so we find the uniformity approach more natural. This motivates a question in the opposite direction to that considered above: can one define the notion of a pattern on a uniform space? Such an object should be defined in terms of partial homeomorphisms on the underlying space and should in some way induce a collage of equivalence relations.

A notion of repetitivity was absent for collages of uniform spaces since a notion of relative density of the equivalence classes is not defined. If one were to induce a collage from a pattern on some uniform space, which directed set would the collage be indexed over? These problems are related: for a general uniform space there is no notion of a uniformly bounded covering and there is no way to say that a particular morphism has a ``large'' (co)domain. The correct structure with which to enrich the uniform space here is a \emph{coarse structure}, see \cite{Roe}. The axioms for a coarse structure are, in some sense, dual to those for a uniform structure and a coarse structure indeed has the dual job to a uniform structure, it describes the \emph{large}-scale properties of a space. This sort of notion is needed in our context to say that ``two points agree to some large magnitude''.

\subsection{Patterns as Inverse Semigroups} \label{subsect: Patterns as Inverse Semigroups}

The definition of a pattern is reminiscent to that of a \emph{pseudogroup}. A (classical) pseudogroup on a topological space consists of partial homeomorphisms on the space which satisfy some natural properties (such as having partial identities, restrictions, compositions and a form of ``glueing''). A more general definition is given in \cite{paro}, the patterns here are examples of such pseudogroups. Central to each of these structures is the idea of a partial mapping. If a group is the correct algebraic gadget with which to study the \emph{global} isomorphisms of an object of interest, then its counterpart for studying \emph{partial} isomorphisms is an inverse semigroup:

\begin{definition} A \emph{semigroup} is a set $S$ along with an associative binary operation $\cdot$ on $S$. An element $0 \in S$ is said to be a \emph{zero} for the semigroup if $0 \cdot s=s \cdot 0 = 0$ for all $s \in S$. An element $1 \in S$ is said to be an \emph{identity} in $S$ if $1 \cdot s = s \cdot 1 =s$ for all $s \in S$. An element $s \in S$ is \emph{regular} if it has at least one \emph{inverse}, that is, some $t \in S$ such that $s \cdot t \cdot s = s$ and $t \cdot s \cdot t = t$. A semigroup is \emph{regular} if every element of it is regular and \emph{inverse} if every element has a unique inverse. \end{definition}

Cayley's Theorem states that every group faithfully embeds into some symmetric group (that is, the group of bijections on some set) in a way which respects the group operation. This fact in some sense makes precise the notion that groups are suitable for the study of global isomorphisms. The analogous theorem for inverse semigroups is the Wagner-Preston representation theorem (see \cite{Law}):

\begin{theorem} Let $\mathcal{I}_X$ be the symmetric inverse semigroup on $X$, that is, the inverse semigroup of all partial bijections on $X$ with (partial) composition as binary operation. Then, given any inverse semigroup $S$, there exists some set $X$ and an injective map $f \colon S \rightarrow \mathcal{I}_X$ which respects the semigroup operation, that is, $f(s \cdot t)=f(t) \circ f(s)$. \end{theorem}

The patterns defined in Section \ref{sect: Patterns} can be seen as inverse semigroups in a natural way. Indeed, patterns on $(X,d_X)$ are, by definition, a collection of elements of $\mathcal{I}_X$. The second and third axioms for a pattern ensure that a pattern $\mathcal{P}$ is a \emph{sub-inverse semigroup} of $\mathcal{I}_X$, that is, $\mathcal{P}$ is an inverse semigroup with respect to the operation $\cdot$ defined by (partial) composition $\Phi_1 \cdot \Phi_2 := \Phi_2 \circ \Phi_1$. Denote by $S(\mathcal{P})$ the inverse semigroup thus defined. We have that $S(\mathcal{P})$ has a zero element (the element $Id_\emptyset$) and an identity (the element $Id_X$).

We note that inverse semigroups have been considered before in the context of tilings, see \cite{KL}, although our approach here is quite different. For a tiling $T$ of $\mathbb{R}^d$ and $\mathcal{S}$ the collection of all (partial) isometries of $\mathbb{R}^d$, one may consider $S(\mathcal{T}_\mathcal{S})$ as being a generalisation of the symmetry group $G_T$ of the tiling (sometimes known as the \emph{space group} of $T$). Indeed, we have that $G_T$ corresponds precisely to the morphisms of $S(\mathcal{T}_\mathcal{S})$ whose (co)domains are all of $\mathbb{R}^d$ (this is a property that can be checked algebraically, such elements being those which satisfy $\Phi \cdot \Phi^{-1}=Id_{\mathbb{R}^d}$, the unique identity in $S(\mathcal{T}_\mathcal{S})$). For an aperiodic tiling $T$, the space group $G_T$ is practically useless since there will be very few (if any) global isometries preserving the tiling, but for an aperiodic FLC tiling there is a rich supply of partial isometries preserving patches of the tiling. The inverse semigroups $S(\mathcal{T}_\mathcal{S})$ are the suitable generalisation of the space group $G_T$ to the aperiodic (FLC) setting.

For a semigroup $S$ say that $s \in S$ is an \emph{idempotent} in $S$ if $s^2 = s$. Notice that the idempotents of $\mathcal{I}_X$ are precisely the partial identities $Id_U$ for subsets $U \subset X$. For an inverse semigroup $S$ we have that the idempotents correspond to elements of the form $s \cdot s^{-1}$. To see this, note that for an idempotent $s$ we have that $s \cdot s \cdot s = s \cdot s = s$ and so is its own inverse and satisfies $s=s \cdot s^{-1}$. Conversely, each $s \cdot s^{-1}$ is an idempotent since $(s \cdot s^{-1}) \cdot (s \cdot s^{-1}) = (s \cdot s^{-1} \cdot s) \cdot s^{-1} = s \cdot s^{-1}$. Every inverse semigroup has a \emph{natural partial order} defined by setting $s \leq t$ if and only if there exists some idempotent $e \in S$ such that $s=e \cdot t$ (one could equivalently use the convention that $s=t \cdot e$). For $\mathcal{I}_X$, this partial order sets $f \leq g$ if and only if $f = g \circ Id_U = {g|}_U$ for some subset $U \subset X$, that is, if and only if $f$ is some restriction of $g$.

This same partial order applies to the inverse semigroups $S(\mathcal{P})$, we have that $\Phi_1 \leq \Phi_2$ in $S(\mathcal{P})$ if and only if $\Phi_1$ is a restriction of $\Phi_2$ to some open subset $U \subset X$. Notice that on the idempotents, that is, the morphisms $Id_U \in S(\mathcal{P})$ for open $U \subset X$, we have that $Id_U \leq Id_V$ if and only if $U \subset V$. So the partial order on the idempotents of $S(\mathcal{P})$ corresponds precisely to the partial order of inclusion on the topology of $X$, that is, the idempotents with the natural partial order form a \emph{locale}. In essence, a locale is a generalisation of a topological space where one considers the open sets, instead of the points of $X$, as fundamental to a space (see \cite{Loc} for a good introduction to locales). Foundational examples of locales are the lattices of open sets of topological spaces, with partial order given by inclusion, and in this regard one may consider a locale as a generalised topological space. Inverse semigroups whose idempotents with the natural partial order are locales (such as our pattern inverse semigroups $S(\mathcal{P})$) are sometimes referred to as \emph{abstract pseudogroups} (in the terminology of \cite{MR}, since we have partial identities on the open sets, they are \emph{full pseudogroups}).

We briefly mention an alternative way of viewing an inverse semigroup $S$ that may aid the reader in visualising the structure of inverse semigroups. One may consider a general element of an inverse semigroup $s \in S$ as mapping from some idempotent $dom(s):=s \cdot s^{-1}$ to another $ran(s):=s^{-1} \cdot s$. For example, for a partial bijection $f \colon U \rightarrow V$ of $\mathcal{I}_X$ we have that $dom(f)=Id_U$ and $ran(f)=Id_V$. So one may think of an inverse semigroup $S$ as a \emph{groupoid} (a small category with invertible morphisms), with object set the idempotents of $S$ and morphisms from idempotents $e$ to $f$ those $s \in S$ with $dom(s)=e$ and $ran(s)=f$. Elements $s$ and $t$, thought of as morphisms in the groupoid, are composable whenever $ran(s)=dom(t)$. Just as in the definition of partial composition, the semigroup operation on $S$ may be conducted by, firstly, taking an appropriate restriction of the morphisms, making the elements composable,  and then composing these elements. The idea of an \emph{inductive groupoid} makes this precise. Loosely, an inductive groupoid is a groupoid with a partial ordering on its morphisms (which ones may think of as corresponding to the natural partial order), these structures satisfying some natural compatibility conditions, such that certain unique (co)restrictions of the morphisms exist. This viewpoint of inverse semigroups, through inductive groupoids, is in a certain sense equivalent, this being made precise by the Ehresmann-Schein-Nambooripad theorem, which says that the categories of inverse semigroups and inductive groupoids (equipped with naturally defined morphisms) are isomorphic, see \cite{Law} for further details.

One can now ask some natural questions: does the algebraic data of the inverse semigroup $S(\mathcal{P})$ determine the pattern $\mathcal{P}$, the uniform space $(X,\mathcal{U}_{\mathcal{P}})$ or the inverse limit space $\Omega^{\mathcal{P}}$? What is meant here is whether in fact the \emph{isomorphism class} of $S(\mathcal{P})$ determines these structures. In fact, this is too much to hope for. Consider an FLC (with respect to translations) aperiodic tiling $T$ of $\mathbb{R}^d$ with respect to allowed partial isometries given by translations. We have that $\mathcal{T}_1$ is a pattern on $(\mathbb{R}^d,d_{euc})$. It follows easily from the axioms, however, that it also defines a pattern on $(\mathbb{R}^d,d_c)$ where $d_c$ is the equivalent ``capped'' Euclidean metric defined by $d_c(x,y):=\min \{d_{euc}(x,y),1 \}$ (note that a translation is still an isometry with respect to $d_c$). However, the induced collage $\tilde{\mathcal{T}_1}$ (and hence the uniform space $(\mathbb{R}^d,\mathcal{U}_{\tilde{\mathcal{T}_1}})$ and space $\Omega^{\tilde{\mathcal{T}_1}}$) depends vitally on the metric: for $x (2)_{\tilde{\mathcal{T}_1}} y$ with respect to the Euclidean metric, we require that there is a translation taking the patch of tiles intersecting $B_{d_{euc}}(x,2)$ to the similar such patch at $y$, whereas for the capped metric there should be such a translation preserving the entire tiling, since $B_{d_c}(x,2)=\mathbb{R}^d$. So $2_{\tilde{\mathcal{T}_1}}$ is the trivial equivalence relation with respect to $d_c$, whereas we know that for an FLC tiling that this equivalence relation is far from trivial with respect to $d_{euc}$. The issue here is that the equivalence relations of the induced collage are dramatically altered by a change in the metric and so we would like to enrich the structure of $S(\mathcal{P})$ using the metric $d_X$. We present now one approach to this.

We have a natural partial order $\leq$ on $S(\mathcal{P})$ defined by setting $\Phi_1 \leq \Phi_2$ if and only if $\Phi_1$ is a restriction of $\Phi_2$. Extending this idea, for $R \in \mathbb{R}_{\geq 0}$, let $\Phi_1 \leq_R \Phi_2$ if and only if $\Phi_1$ is a restriction of $\Phi_2$ for which $dom(\Phi_2)$ contains all points of distance less than or equal to $R$ of a point of $dom(\Phi_1)$ and, similarly, $ran(\Phi_2)$ contains all points of distance less than or equal to $R$ of a point of $ran(\Phi_1)$. We have the following properties for $\leq_R$:

\begin{proposition} \label{prop: grad poset}
The inverse semigroup $S(\mathcal{P})$ equipped with $\leq_R$ for $R \in \mathbb{R}_{\geq 0}$ as above satisfies:
\begin{enumerate}
	\item $\Phi_1 \leq_0 \Phi_2$ if and only if $\Phi_1$ is a restriction of $\Phi_2$, that is, $\leq_0$ corresponds to the natural partial order on $S(\mathcal{P})$.
	\item The idempotents of $S(\mathcal{P})$ are precisely the morphisms $Id_U$ for open $U \subset X$ and $Id_U \leq_R Id_V$ if and only if every point of distance less than or equal to $R$ of a point of $U$ is contained in $V$.
	\item $\Phi_1 \leq_R \Phi_2$ if and only if $\Phi_1=Id_{dom(\Phi_1)}\cdot \Phi_2 = \Phi_2 \cdot Id_{ran(\Phi_1)}$ with $Id_{dom(\Phi_1)} \leq_R Id_{dom(\Phi_2)}$ and $Id_{ran(\Phi_1)} \leq_R Id_{ran(\Phi_2)}$.
	\item If $\Phi_1 \leq_R \Phi_2$ then $\Phi_1 \leq_r \Phi_2$ for all $r \leq R$.
	\item If $\Phi_1 \leq_R \Psi_1$ and $\Phi_2 \leq_R \Psi_2$ then $\Phi_1 \cdot \Phi_2 \leq_R \Psi_1 \cdot \Psi_2$.
	\item If $\Phi_1 \leq_R \Phi_2$ then $\Phi_1^{-1} \leq_R \Phi_2^{-1}$.
	\item If $\Phi_1 \leq_{R_1} \Phi_2 \leq_{R_2} \Phi_3$ then $\Phi_1 \leq_{\max \{R_1,R_2\}} \Phi_3$.
\end{enumerate}
\end{proposition}

\begin{proof} Parts $1$, $2$, $3$, $4$, $6$ and $7$ follow easily from the definitions and our previous discussions on inverse semigroups. For part $5$, take $\Phi_1 \leq_R \Psi_1$ and $\Phi_2 \leq_R \Psi_2$. We must show that every point within distance less than or equal to $R$ of $dom(\Phi_2 \circ \Phi_1) = \Phi_1^{-1}(ran(\Phi_1) \cap dom(\Phi_2))$ is contained within $dom(\Psi_2 \circ \Psi_1) = \Psi_1^{-1}(ran(\Psi_1) \cap dom(\Psi_2))$ (the discussion for the ranges of the functions is analogous). So let $x$ be within distance $R$ of $y \in \Phi_1^{-1}(ran(\Phi_1) \cap dom(\Phi_2))$. It follows that $x,y \in dom(\Psi_1)$ and so $\Psi_1(x)$ is within distance $R$ of $\Psi_1(y) \in ran(\Phi_1) \cap dom(\Phi_2)$. Hence $\Phi_1(x)$ is in $ran(\Psi_1) \cap dom(\Psi_2)$, that is, $x \in \Phi_1^{-1}(ran(\Psi_1) \cap dom(\Psi_2))$. \end{proof}

The above proposition shows that the relations $\leq_R$ possess similar properties to the natural partial order on $S(\mathcal{P})$, that is, they respect multiplication and inversion (note that they are not themselves partial orders, since usually $\Phi \nleq_R \Phi$). They get ``finer'' as $R$ increases by item $4$ of the above proposition, that is, $\leq_R \subset \leq_r$ for $r \leq R$. 

\begin{definition} We shall say that the patterns $\mathcal{P}_X$, defined on $(X,d_X)$, and $\mathcal{P}_Y$, defined on $(Y,d_Y)$, are \emph{isomorphic} if and only if there exists some surjective isometry $f \colon (X,d_X) \rightarrow (Y,d_Y)$ such that $f$ and its inverse preserve the semigroup structures of $\mathcal{P}_X$ and $\mathcal{P}_Y$. That is, we have that for all $\Phi \in \mathcal{P}_X$ the partial isometry $f(\Phi) \colon f(dom(\Phi)) \rightarrow f(ran(\Phi))$ defined by $f(\Phi)(y) := (f \circ \Phi \circ f^{-1})(y)$ is an element of $\mathcal{P}_Y$ and, similarly, for all $\Psi \in \mathcal{P}_Y$, we have that $f^{-1} \circ \Psi \circ f \in \mathcal{P}_X$. \end{definition}

In effect, isomorphic patterns are equal ``modulo a relabelling of points of $(X,d_X)$''. The inverse semigroup $S(\mathcal{P})$ equipped with the relations $\leq_R$ is able to recover $\mathcal{P}$ up to isomorphism:

\begin{proposition} Let $\mathcal{P}$ be a pattern. Then the pair $(S(\mathcal{P}),\leq_R)$ determines the pattern $\mathcal{P}$ up to isomorphism. \end{proposition}

\begin{proof} Above, we saw how the lattice of open sets of $(X,d_X)$ may be recovered from $S(\mathcal{P})$ algebraically as the set of idempotents of $S(\mathcal{P})$ with partial order given by the natural partial order $\leq_0$. The map sending an open set $U$ to the morphism $Id_U$ is an isomorphism between the locale associated to $(X,d_X)$ and the lattice of idempotents of $S(\mathcal{P})$ with partial order $\leq_0$. Given $\Phi \in S(\mathcal{P})$, we have that $dom(\Phi)$ corresponds to the map $Id_{dom(\Phi)}=\Phi \cdot \Phi^{-1}$, $ran(\Phi)$ corresponds to the map $Id_{ran(\Phi)}=\Phi^{-1} \cdot \Phi$ and the preimage of an open set $V$ in the range of $\Phi$ is the open set $U$ satisfying $Id_U= \Phi \cdot Id_V \cdot \Phi^{-1}$. Since all of the spaces involved are Hausdorff and hence sober \cite{Loc}, there is a natural homeomorphism from $(X,d_X)$ to its space of points, so we may recover the topological space $(X,d_X)$ and the partial isometries of $\mathcal{P}$ defined on it up to homeomorphism preserving this structure of partial isometries. Hence, we just need to determine the metric on $X$.

Points of $(X,d_X)$ may be identified as the completely prime filters of the locale associated to $(X,d_X)$ (see \cite{Loc}). These correspond precisely to filters $F_x$ of the form $U \in F_x$ if and only if $x \in U$ for some choice of $x \in X$. For $r \in \mathbb{R}_{>0}$ consider the open set $U_x^r$ which is defined to be the interior of $\bigcap U$ where the intersection is taken over all open sets $U$ with $Id_{V_x} \leq_r Id_U$ for some $V_x \in F_x$. That is, $U_x^r$ is the meet of these open sets in the locale of $(X,d_X)$. It is easy to see that $B_{d_X}(x,r) \subset U_x^r$ and if $y \in U_x^r$ then $d_X(x,y) \leq r$. It follows that $\bigcup_{r < R} U_x^r=B_{d_X}(x,R)$. Since $d_X(x,y)=\inf \{R \mid y \in B_{d_X}(x,R)\}$, we may recover the metric on $(X,d_X)$. \end{proof}

We shall conclude with one final structure on $S(\mathcal{P})$ which allows one to recover the space $\Omega^{\mathcal{P}}$ more directly. Recall that two points $x,y \in X$ are equivalent to a ``large'' radius in the pattern if and only if there is some partial isometry $\Phi \in \mathcal{P}$ with ``large'' domain and range taking $x$ to $y$. Restricting $\Phi$ to some small open set about $x$, we see that $x$ and $y$ are equivalent to a ``large'' radius if and only if there exists some $\Phi \in \mathcal{P}$ taking $x$ to $y$ which extends to some $\Phi' \in \mathcal{P}$ which has much larger domain about $x$ and range about $y$ than $\Phi$. This motivates the following: given $\Phi \in S(\mathcal{P})$ let $\Phi \in S^R(\mathcal{P})$ if and only if there exists some element $\Phi^R \in \mathcal{P}$ with $\Phi \leq \Phi^R$. By Proposition \ref{prop: grad poset} we have that the collection of elements of each $S^R(\mathcal{P})$ is a pattern on $(X,d_X)$, that $S^0(\mathcal{P})=S(\mathcal{P})$ and that if $r \leq R$ then $S^R(\mathcal{P})$ is a sub-inverse semigroup of $S^r(\mathcal{P})$.

Let us write $x R_{[\mathcal{P}]} y$ if there exists some $\Phi \in S^R(\mathcal{P})$ with $\Phi(x)=y$. Notice that if $x R_{[\mathcal{P}]} y$ then $x R_{\tilde{\mathcal{P}}} y$ in the induced collage of $\mathcal{P}$. Conversely, it is not too hard to see that if $x (R+\epsilon)_{\tilde{\mathcal{P}}} y$ then $x R_{[\mathcal{P}]} y$ for any $\epsilon > 0$. As in the proof of the above proposition, the equivalence relations $R_{[\mathcal{P}]}$ on $X$ may be recovered from the sub-inverse semigroups $S^R(\mathcal{P})$. We see that the inverse limit space $\Omega^{\tilde{\mathcal{P}}}$ may be recovered from the purely algebraic data of this decreasing sequence of sub-inverse semigroups. Note that in fact one only needs a cofinal subsequence of values of $R$ here. Such a sequence of sub-inverse semigroups recovers the induced collage of $\mathcal{P}$ up to equivalence.

\chapter{Pattern-Equivariant Homology}

\label{chap: Pattern-Equivariant Homology}

Given an FLC tiling of $\mathbb{R}^d$ with associated continuous hull $\Omega$ one may ask how topological invariants of this space may be interpreted with respect to the properties of the original tiling. A useful topological invariant in this context is the \v{C}ech cohomology, which may be effectively computed for substitution tilings \cite{AP,BDHS} or projection tilings \cite{FHK}. Ian Putnam and Johannes Kellendonk developed a way of visualising these topological invariants using cochains which ``live on the tiling'', reference to the space $\Omega$ is not necessary in describing them. In more detail, one considers differential forms of $\mathbb{R}^d$ (the space in which the tiling ``lives'') which respect the tiling, that is, differential forms which are equal at any two points of the tiling whenever there is agreement in the local patches of tiles about these points to some given radius, see \cite{Kellendonk} for more details. This cochain complex is a subcomplex of the usual de Rham complex of $\mathbb{R}^d$ and its cohomology is isomorphic to the (real-valued) \v{C}ech cohomology of the continuous hull $\Omega$. This shows how rather abstract invariants of a rather abstract space may be described in a more familiar and visual setting.

If the tiling has a CW-structure, one may consider a similar approach by taking cellular cochains of the tiling which respect the tiling, that is, cellular cochains which agree on any two cells which have identical patches of tiles about them to some given radius. It is easy to show \cite{Sadun1} that the cohomology of the cochain complex of pattern-equivariant cochains is isomorphic to the \v{C}ech cohomology of the tiling space (over general coefficients). We consider here, instead of cochains, Borel-Moore chains which respect the pattern. These \emph{pattern-equivariant homology groups} are defined for any pattern $\mathcal{P}$ as introduced in the previous chapter. They are defined in terms of singular chains but we shall show that, when our pattern may be equipped with some ``compatible'' CW-structure, one may also define cellular groups which agree with the singular versions. The cellular groups, in the case of a pattern induced by a tiling, are analogous to the cellular pattern-equivariant cohomology groups defined in \cite{Sadun1} but where one uses the cellular boundary instead of coboundary maps. In Chapter \ref{chap: Poincare Duality for Pattern-Equivariant Homology} we shall link these groups to the pattern-equivariant cohomology, and hence \v{C}ech cohomology groups of the continuous hull $\Omega$, via a Poincar\'{e} duality result.

\section{Singular PE Homology} Given a pattern $\mathcal{P}$ a Borel-Moore chain $\sigma=\sum_{i}c_{i}\sigma_{i}\in C_{n}^{\text{BM}}(X)$ will be called \emph{$\mathcal{P}$-equivariant} ($\mathcal{P}$E for short or \emph{pattern-equivariant} when the pattern is understood) \emph{to radius $R$} if each $\sigma_{i}$ has radius of support bounded by\footnote{We say that $A \subset (X,d_X)$ has radius bounded by $\alpha$ if there exists some $x$ with $A \subset B_{d_X}(x,\alpha)$. A function $f \colon Y \rightarrow (X,d_X)$ has radius of support bounded by radius $\alpha$ if its image has radius bounded by $\alpha$.} $R/2$ and, for each pair of elements $x,y\in X$ and each element $\Phi \in \mathcal{P}_{x,y}^R$, we have that $\Phi_{*}(\sigma^x)=\sigma^y$ (recall the notation of a chain restricted to some subset of $X$ defined in Subsection \ref{subsect: Borel-Moore Homology and Poincare Duality}). One may like to say that such a chain \emph{looks the same locally at $x$ and $y$ via the isometries $\Phi \in \mathcal{P}_{x,y}^R$}.

We denote by $C_n^{\mathcal{P};R}(X)$ the group of all $n$-chains with $\mathcal{P}$E radius $R$. A Borel-Moore chain is said to be \emph{$\mathcal{P}$-equivariant} if it is $\mathcal{P}$-equivariant to some radius $R$ and we denote the set of all $\mathcal{P}$E $n$-chains by $C_n^\mathcal{P}(X)$. The lemma below will establish that
\[C_{\bullet}^{\mathcal{P}}(X)=(0\underset{\partial_{0}}{\longleftarrow}C_{0}^{\mathcal{P}}(X)\underset{\partial_{1}}{\longleftarrow}C_{1}^{\mathcal{P}}(X)\underset{\partial_{2}}{\longleftarrow}C_{2}^{\mathcal{P}}(X)\underset{\partial_{3}}{\longleftarrow}\cdots) \]

\noindent
is a subchain complex of $C_{\bullet}^{\text{BM}}(X)$, the homology of which, called the \emph{$\mathcal{P}$-equivariant homology}, we shall denote by $H_\bullet^\mathcal{P}(X)$ (and similarly for $H_\bullet^{\mathcal{P};R}(X)$, the \emph{$\mathcal{P}$E homology to radius $R$}).

It will be useful in our consideration of a cellular $\mathcal{P}$E homology to restrict attention to $\mathcal{P}$E chains contained in some subset. Given a pattern $\mathcal{P}$ and any subset $A\subset X$, one can form the chain complex
\[C_{\bullet}^{\mathcal{P}}(A)=(0\underset{\partial_{0}}{\longleftarrow}C_{0}^{\mathcal{P}}(A)\underset{\partial_{1}}{\longleftarrow}C_{1}^{\mathcal{P}}(A)\underset{\partial_{2}}{\longleftarrow}C_{2}^{\mathcal{P}}(A)\underset{\partial_{3}}{\longleftarrow}\cdots)\]

\noindent
of the chain groups $C_i^\mathcal{P}(A)$ whose elements are those $\mathcal{P}$E chains of $X$ for which the support of the singular simplexes are contained in the subset $A$ (and similarly for $C_i^{\mathcal{P};R}(A)$). The boundary maps are the restrictions of the usual boundary maps to these chain groups. The homology of this chain complex is denoted by $H_{\bullet}^{\mathcal{P}}(A)$.

For $B \subset A \subset X$, one can form the quotient complex $C_\bullet^\mathcal{P}(A,B) :=$ $C_\bullet^\mathcal{P}(A)/C_\bullet^\mathcal{P}(B)$, for which we denote the homology $H_\bullet^\mathcal{P}(A,B)$. Notice that two $\mathcal{P}$E chains of $A$ represent the same element of $C_\bullet^\mathcal{P}(A,B)$ if and only if their difference belongs to $B$. Of course, we have the usual short exact sequence of chain complexes and long exact sequence in homology
\[0 \leftarrow H_0^\mathcal{P}(A,B) \leftarrow H_0^\mathcal{P}(A) \leftarrow H_0^\mathcal{P}(B) \leftarrow H_1^\mathcal{P}(A,B) \leftarrow \cdots.\]

\noindent
We emphasise that $C_\bullet^\mathcal{P}(A)$ is defined as a subgroup of $C_\bullet^\mathcal{P}(X)$. When $A$ is not closed in $X$, it is not necessarily true that Borel-Moore chains of $A$ are Borel-Moore chains of $X$; in our notation, elements of $C_\bullet^\mathcal{P}(A)$ are both.

\begin{lemma}\label{boundary} Given any $\sigma\in C_n^{\mathcal{P};r}(A)$, the following holds: \begin{itemize}

\item $\sigma\in C_{n}^{\mathcal{P};R}(A)$ for any $R \geq r$,
\item $\partial_n(\sigma)\in C_n^{\mathcal{P};r}(A)$.\end{itemize} \end{lemma}

\begin{proof} The first item follows from the fact that restrictions of elements of $\mathcal{P}_{x,y}^R$ to $B_{d_X}(x,r)$ are elements of $\mathcal{P}_{x,y}^r$, so the requirements for $\sigma$ to be $\mathcal{P}$E to radius $R$ are weaker than for it to be $\mathcal{P}$E to radius $r$. For the second note that, since the transformations $\Phi \in \mathcal{P}_{x,y}^r$ are injective, $\Phi_{*}(\sigma^x)=(\Phi_{*}(\sigma))^y$ for chains $\sigma$ with support contained in $dom(\Phi)$. We also have that $\partial_n(\sigma)^x=\partial_n(\sigma^x)^x$ and so $\Phi_*(\partial_n(\sigma)^x) = \Phi_*(\partial_n(\sigma^x)^x) = \Phi_{*}(\partial_n(\sigma^x))^y = \partial_n (\Phi_*(\sigma^x))^y = \partial_n(\sigma^y)^y = \partial_n(\sigma)^y$. \end{proof}

\begin{remark} Note that for two equivalent patterns $\mathcal{P}$ and $\mathcal{Q}$ (see Definition \ref{def: Equivalent}) we have that a chain $\sigma$ is $\mathcal{P}$-equivariant if and only if it is $\mathcal{Q}$-equivariant, so equivalent patterns have isomorphic pattern-equivariant homology. In particular, MLD tilings have isomorphic $\mathcal{T}_1$E homology (and similarly $S$-MLD tilings have isomorphic $\mathcal{T}_0$ and $\mathcal{T}_{\text{rot}}$E homology) and for a recognisable hierarchical tiling $T_\omega$ we have that the $\mathcal{T}$E homology of $T_0$ is isomorphic to the $\mathcal{T}_\omega$E homology as a hierarchical tiling. \end{remark}

By a \emph{$\mathcal{P}$-set}, we shall mean a set $U=\pi_{R,X}^{-1}(S)$ for some $R \ge 0$, $S\subset K_R$. A \emph{collection of $\mathcal{P}$-sets $\mathcal{U}=\{U_i \mid i \in \mathcal{I}\}$} of $X$ is simply a set of $\mathcal{P}$-sets but with the restriction that each set of the collection is a pullback from $K_R$ for some fixed $R$ for the collection (of course this restriction is not needed when the collection is finite), we shall say that the \emph{$\mathcal{P}$E radius of $\mathcal{U}$} is $R$. Note that finite unions, finite intersections and complements of $\mathcal{P}$-sets are also $\mathcal{P}$-sets and the same is true for arbitrary unions and intersections when the $\mathcal{P}$-sets are from some collection of $\mathcal{P}$-sets with $\mathcal{P}$E radius $R$.

For $D>0$, temporarily denote by $C_{\bullet}^{\mathcal{P}|D}$ the chain complex of $\mathcal{P}$E chains for which the singular simplexes have radius of support bounded by $D$ and, for an open cover $\mathcal{U}$ of $X$, denote by $C_{\bullet}^{\mathcal{P}^\mathcal{U}}$ the chain complex of $\mathcal{P}$E chains for which the singular simplexes are contained in the cover $\mathcal{U}$.

\begin{lemma}\label{mv} For any pattern $\mathcal{P}$ on $X$ and $Y\subset X$ we have: \begin{enumerate}
\item $C_{\bullet}^{\mathcal{P}}(Y)$ is a subchain complex of $C_{\bullet}^{\text{BM}}(Y)$.

\item $C_{\bullet}^{\mathcal{P};R}(Y)$ is a subchain complex of $C_{\bullet}^{\mathcal{P}}(Y)=\varinjlim_{R}C_{\bullet}^{\mathcal{P};R}(Y)$.

\item The inclusion $\iota \colon C_{\bullet}^{\mathcal{P}|D}(Y)\rightarrow C_{\bullet}^{\mathcal{P}}(Y)$ is a quasi-isomorphism.

\item For a collection of $\mathcal{P}$-sets $\mathcal{U}$ whose interiors cover $Y$ (with the subspace topology, relative to $Y$), the inclusion $\iota \colon C_{\bullet}^{\mathcal{P}^{\mathcal{U}}}(Y)\rightarrow C_{\bullet}^{\mathcal{P}}(Y)$ is a quasi-isomorphism.

\item Under the same conditions above, with $\mathcal{U}=\{A,B\}$, the obvious inclusion of chain complexes $\iota \colon C_{\bullet}^{\mathcal{P}}(B,A\cap B)\rightarrow C_{\bullet}^{\mathcal{P}}(Y,A)$ is a quasi-isomorphism. Hence, $H_n^{\mathcal{P}}(Y-Z,A-Z)\cong H_{n}^{\mathcal{P}}(Y,A)$ for $\mathcal{P}$-sets $Z\subset A\subset Y$ with (in the subspace topology relative to $X$) the closure of $Z$ contained in the interior of $Y$.

\item Under the same conditions as above, we have the following long exact sequence:
\[ \cdots \leftarrow H_{n-1}^\mathcal{P}(A\cap B) \leftarrow H_n^\mathcal{P}(Y) \leftarrow  H_n^\mathcal{P}(A)\oplus H_{n}^{\mathcal{P}}(B)\leftarrow H_{n}^{\mathcal{P}}(A\cap B) \leftarrow \cdots\]\end{enumerate} \end{lemma}

\begin{proof} The first and second items follow trivially from the above lemma.

To prove the third, we can follow the classical approach (see \cite{Hat}). One defines an operator $\rho$ on the singular chain complex, defined in terms of the barycentric subdivision operator and the number $m(\sigma_i)$ which, for each singular simplex $\sigma_i$, is the smallest natural number such that subdividing the simplex $m(\sigma_i)$ times yields a chain whose singular simplexes have radius of support bounded by $D$. Going through all of the details again would be tedious, but note that $m(\sigma_i)$ is invariant under the isometries $\mathcal{P}_{x,y}^R$ which are applied to $\sigma_i$. So $\rho$ ``commutes'' with the maps $\Phi_*$ and, since $\rho(\sigma_i^x)^x=\rho(\sigma_i)^x$, $\rho$ will send a $\mathcal{P}$E chain of radius $R$ to a $\mathcal{P}$E chain of radius $R$, mimicking the proof of the Lemma \ref{boundary}.

For the fourth item, one may show that given a $\mathcal{P}$E chain $\sigma$ of $\mathcal{P}$E radius $R$ (pick $R \geq R_\mathcal{U}$, the $\mathcal{P}$E radius of the cover $\mathcal{U}$), it has a subdivision $\rho(\sigma)$ contained in the cover with $\mathcal{P}$E radius $R+\epsilon$ for any $\epsilon>0$. The proof is similar to the above. If the cover has a Lebesgue number, one can of course subdivide the chains as above so that their radius of supports are smaller than the Lebesgue number of the cover by $3$. If it does not, pick $\epsilon>0$ and first subdivide to a $\mathcal{P}$E chain with simplexes of support bounded by radius $\epsilon/2$ as described above. Let $m(\sigma_i)$ be the smallest number such that subdividing $\sigma_i$ $m(\sigma_i)$ times produces a singular chain contained in the cover. Then $m(\sigma_i) = m(\Phi(\sigma_i))$ for every $\Phi \in \mathcal{P}_{x,y}^{R+\epsilon}$ since $\pi_{R,X}=\pi_{R,X} \circ \Phi$ on $B_{d_X}(x,\epsilon)$, which contains the support of $\sigma_i$. It follows that the chain given by subdividing each singular simplex $\sigma_i$ $m(\sigma_i)$ times is $\mathcal{P}$E to radius $R+\epsilon$.

The proof of excision, again, follows from a simple adaptation of the classical case. We note that the constructions above allow us to canonically define a quasi-isomorphism $C_\bullet^{\mathcal{P}^\mathcal{U}}(Y)/C_\bullet^\mathcal{P}(A) \rightarrow C_\bullet^\mathcal{P}(Y)/C_\bullet^\mathcal{P}(A)$. There is an obvious isomorphism $C_\bullet^\mathcal{P}(B)/C_\bullet^\mathcal{P}(A\cap B) \rightarrow C_\bullet^{\mathcal{P}^\mathcal{U}} (Y)/C_\bullet^\mathcal{P}(A)$ (elements of both can be identified with $\mathcal{P}$E chains contained in $B$ but not $A$), so we obtain the quasi-isomorphism as desired.

Given $\mathcal{P}$-sets $Z\subset A\subset Y$, we have that $Y-Z$ is also a $\mathcal{P}$-set, so we obtain the result that $H_n^\mathcal{P}(Y-Z,A-Z) \cong H_n^\mathcal{P}(Y,A)$ by setting $B=Y-Z$ and applying the above.

Finally, the Mayer-Vietoris sequence follows, since the above gives us that the inclusion $C_\bullet^{\mathcal{P}^\mathcal{U}}(Y) \hookrightarrow C_\bullet^\mathcal{P}(Y)$ is a quasi-isomorphism, which implies the result from long exact sequence associated to the short exact sequence of chain complexes
\[0\rightarrow C_{\bullet}^{\mathcal{P}}(A\cap B)\xrightarrow{\phi}C_{\bullet}^{\mathcal{P}}(A)\oplus C_{\bullet}^{\mathcal{P}}(B)\xrightarrow{\psi}C_{\bullet}^{\mathcal{P}^{\mathcal{U}}}(X)\rightarrow0\]

\noindent
where $\phi(\sigma)=(\sigma,-\sigma)$ and $\psi(\sigma,\tau)=\sigma+\tau$.\end{proof}

It follows that the pattern-equivariant homology groups $H_\bullet^\mathcal{P}(X)$ are \linebreak well-defined and are isomorphic to the direct limits $\varinjlim_R H_\bullet^{\mathcal{P};R}(X)$. The Mayer-Vietoris exact sequence of the above lemma will be useful when showing the equivalence of the singular and cellular versions of $\mathcal{P}$E homology, as well as in Chapter \ref{chap: Poincare Duality for Pattern-Equivariant Homology} for Poincar\'{e} duality.

\section{Cellular PE Homology} \label{sect: Cellular PE Homology} To make explicit computations, we require a cellular version of $\mathcal{P}$E homology. For this, we need the notion of a cellular pattern, which may represent a given pattern in a certain sense. Much like for PE cohomology, one must be careful when dealing with rotational symmetry. We show that, with certain conditions on either the divisibility of the coefficient group or the CW-decomposition with respect to the symmetry of the pattern, singular and cellular $\mathcal{P}$E homology coincide. We shall often identify a closed cell of $(X,d_X)$ with the image of its characteristic map from the closed $d$-disc $c \colon D^d \rightarrow X$.

\begin{definition}\label{CPatDef} A \emph{cellular pattern} $\mathcal{C}$ (on $(X,d_X)$) consists of the following data: \begin{itemize}
	\item A finite-dimensional and locally finite CW-decomposition of $(X,d_X)$ for which the cells are bounded in radius.
	\item For each $R>0$ a groupoid $\mathcal{C}^R$. 
\end{itemize}

The groupoids $\mathcal{C}^R$ are required to satisfy the following:

\begin{enumerate}
	\item The set of objects of each $\mathcal{C}^R$ is the set of (closed) cells of the CW-\linebreak decomposition. The set of morphisms between $k$-cells $c$ and $d$, denoted $\mathcal{C}^R_{c,d}$, is a finite set of cellular isometries of $c$ onto $d$, where composition is given by the usual composition of isometries. Further, these sets of isometries should be \emph{tame}: for each $c$ there exists some closed neighbourhood in $c$ of $\partial c$, the boundary of $c$, which deformation retracts to $\partial c$ equivariantly with respect to each of the isometries of $\mathcal{C}_{c,c}^R$. That is, for a $k$-cell $c$, $k>0$, there exists a deformation retraction $F \colon [0,1] \times U \rightarrow U$ of some closed neighbourhood $U \subset c$ of $\partial c$ with $F(t,x)= x$ for all $x \in \partial c$, $F(0,-)=Id_U$, $F(1,x) \subset \partial c$ and $F(t,-) \circ \Phi = \Phi \circ F(t,-)$ for all $\Phi \in \mathcal{C}_{c,c}^R$ and $t \in [0,1]$.
	\item For $r \le R$, if $\Phi \in \mathcal{C}^R_{c,d}$ then $\Phi \in \mathcal{C}^r_{c,d}$. One may like to say that $\mathcal{C}^R \subset \mathcal{C}^r$.
	\item For every $R>0$ there exists some $R'>0$ satisfying the following: for every $\Phi \in \mathcal{C}_{c,d}^{R'}$ between $k$-cells $c$ and $d$ there exists a bijection $f$ between the $(k+1)$-cells containing $c$ and the $(k+1)$-cells containing $d$ for which, for each $(k+1)$-cell $c'$ bounding $c$, there exists some $\Phi' \in \mathcal{C}_{c',f(c')}^R$ which restricts to $\Phi$.
\end{enumerate}
\end{definition}

For the ``tameness'' of the isometries it is sufficient, for example, that there exists some such neighbourhood $U$ of $\partial c$ which is fibred as $U \cong \partial c \times [0,\epsilon]$ in a way such that $\Phi(x,t)=(\Phi(x),t)$ for all $\Phi \in \mathcal{C}_{c,c}^R$. Whilst the tameness condition seems reasonable, it is not obvious to the author that this condition must always hold. Group actions of closed disks, even of finite groups, can have properties which depart from intuition sufficiently to cause some doubt; for example, there exist finite group actions of disks which have no fixed points for every element of the group \cite{FR}. In most situations of interest here, the tameness of the isometries will not be an issue. For example, a sufficient condition is that the open cells are isometric to star domains of Euclidean space. It is not too hard to show that the centre of rotation $p$ for a star domain is a star-centre, that is, a point for which the line segments between $p$ and any other point of the domain are contained in the domain. Given such a point, the required homotopy is given by ``flowing away from $p$''.

Choosing an orientation for each $k$-cell $c$ with boundary $\partial c$, we may identify the singular homology $H_i(c,\partial c;G)$ as $G$ for $i=k$ and as being trivial otherwise. A map $\Phi \in \mathcal{C}_{c,d}^R$ induces an isomorphism between the coefficient groups of the cells $c$ and $d$, the isomorphism being determined by the choice of orientation of the cells and whether or not $\Phi$ preserves or reverses orientations. We may then define the cellular pattern-equivariant homology chains as those chains which are ``equivariant'' with respect to these isomorphisms of cellular groups:

\begin{definition} Let $\mathcal{C}$ be a cellular pattern on $(X,d_X)$. A cellular Borel-Moore $k$-chain $\sigma$ on $(X,d_X)$ is called \emph{pattern-equivariant} (\emph{with PE radius $R$}, with respect to $\mathcal{C}$) if there exists some $R>0$ such that for any $k$-cell $c$ and $\Phi \in \mathcal{C}_{c,d}^R$ then $\Phi_*(\sigma)(c)=\sigma(d)$. The homology of this chain complex, the \emph{cellular pattern-equivariant homology of $\mathcal{C}$}, will be denoted by $H_\bullet^\mathcal{C}(X;G)$. \end{definition}

It is easy to see that condition $3$ in the definition of a cellular pattern ensures that the usual cellular boundary map of cellular Borel-Moore chains restricts to a well-defined boundary map on cellular PE chains.

\begin{definition} Let $\mathcal{P}$ be a pattern and $\mathcal{C}$ be a cellular pattern of $(X,d_X)$. We shall say that $\mathcal{C}$ is a \emph{cellular representation of $\mathcal{P}$} if: \begin{itemize}
\item For all $R>0$ there exists some $R'>0$ such that if $\Phi \in \mathcal{P}^{R'}_{x,y}$ then $\Phi|_{c_x} \in \mathcal{C}^R_{c_x,\Phi(c_x)}$ for each cell $c_x$ containing $x$.
\item For all $R>0$ (with $R/2$ greater than the radii of the cells) there exists some $R'>0$ such that if $\Phi \in \mathcal{C}_{c,d}^{R'}$ then there exists some $\Phi' \in \mathcal{P}_{x,y}^R$ with $x$ contained in $c$ and $\Phi'|_c=\Phi$.
\end{itemize}
\end{definition}

We may loosely summarise the above by saying that for arbitrarily large $R$ there exists some $R'$ with $\mathcal{P}^{R'} \subset \mathcal{C}^R$ and $\mathcal{C}^{R'} \subset \mathcal{P}^R$.

\begin{exmp} Let $S$ be a (discrete) group which acts on $(X,d_X)$ by isometries. As described in Example \ref{ex: Allowed Partial Isometries} we can define a pattern $\mathcal{S}$ on $(X,d_X)$ by setting $\Phi \in \mathcal{S}$ if and only if $\Phi = {\Phi_g|}_U \colon U \rightarrow \Phi_g(U)$ for some isometry $\Phi_g$ of the action and open set $U \subset X$. The $S$-action on $X$ induces an $S$-action on the chain complex of (bounded) singular Borel-Moore chains on $X$. It is not too hard to see that the singular $\mathcal{S}$E chain complex will correspond precisely to the complex of bounded $S$-invariant singular Borel-Moore chains on $X$.

Suppose that $X$ is equipped with a CW-decomposition of cells of bounded radius which is invariant under the action of $S$ (that is, if $c$ is a cell of $X$ then so is $\Phi_g(c)$). Let $\mathcal{C}_{c,d}^R$ be the set of isometries between $c$ and $d$ which are restrictions of the isometries $\Phi_g$. Then (given that the isometries are tame on each cell, and each $\mathcal{C}_{c,d}^R$ is finite) $\mathcal{C}$ is a cellular representation of $\mathcal{S}$. Of course, the cellular $\mathcal{S}$E chains are precisely cellular Borel-Moore chains which are invariant under the action of $S$. This example indicates that ``pattern-\emph{invariant}'' would be a justifiable adjective for the homology groups defined here. The name ``pattern-equivariant'' was chosen due to the pre-established term as used in the context of tilings, which shall be discussed in Section \ref{sect: Pattern-Equivariant Homology of Tilings}. \end{exmp}

\begin{exmp} \label{ex: Cellular Tiling} Let $T$ be a cellular tiling of $\mathbb{R}^d$. The CW-decomposition of $\mathbb{R}^d$ given by the tiling can easily be incorporated into cellular patterns that represent $\mathcal{T}_1$ or $\mathcal{T}_0$ (see the following section). With a little more effort, one can also find a CW-decomposition of $E^+(d)$ and make it a cellular pattern representing $\mathcal{T}_{\text{rot}}$ (this CW-decomposition is described in \cite{Sadun2}). The $\mathcal{T}_i$E cellular homology groups (for $i=0,1$) are discussed in detail below in Section \ref{sect: Pattern-Equivariant Homology of Tilings}. \end{exmp}

\begin{exmp} Let $P \subset \mathbb{R}^d$ be a Delone set of $(\mathbb{R}^d,d_{euc})$. One may consider a pattern $\mathcal{P}$ associated to $P$ (see Remark \ref{rem}). There is a CW-decomposition associated to $P$ given by the Voronoi tiling of $P$ (see Subsection \ref{subsect: Examples of Tilings and Delone Sets}); naturally it can be made into a cellular pattern which represents $\mathcal{P}$. \end{exmp}

Let $G$ be a commutative ring with identity $1 \in G$. The group structure on $G$ is a $\mathbb{Z}$-module in the usual way; we shall say that $G$ has \emph{division by $n$} if $n.1$ is invertible in $G$. This defines, for each $g \in G$, a unique element $g/n:=n^{-1}g$ for which $n(g/n)=g$. \footnote{More generally, one may define the notion of an abelian group $G$ having division by $n$ whenever $nG=G$. Of course, elements $g/n$ which satisfy $n(g/n)=g$ are not necessarily unique, but with care the arguments here could be adapted to this setting.}
We may now state the main theorem of this section:

\begin{theorem}\label{thm: sing=cell} Let $\mathcal{P}$ be a pattern on $(X,d_X)$ and $\mathcal{C}$ be a cellular pattern that represents it. Suppose further that $G$ has division by $|\mathcal{C}^R_{c,d}|$ (for $|\mathcal{C}^R_{c,d}| \neq 0$) for each $c,d$ for sufficiently large $R$. Then $H_\bullet^\mathcal{P}(X,G) \cong H_\bullet^\mathcal{C}(X,G)$. \end{theorem}

One should draw attention to the requirement in the above of having an appropriate coefficient group or CW-decomposition if the pattern has local symmetries. To see the need for these restrictions in a simple example, consider the unit disc $D^2$ together with its boundary $S^1$ and rotation $\tau$ about its centre by $\pi$ radians. This defines a group action of $\mathbb{Z}_2$ on the relative singular chain group of $(D^2,S^1)$ which commutes with the boundary map and so one may construct a chain complex consisting of those chains which are invariant under the induced action of $\tau$.

Over $\mathbb{Z}$-coefficients, there is two-torsion in the degree zero relative homology group -- one may only ``remove'' singular $0$-simplexes at the origin in multiples of two. The rotation invariant homology calculations will be incorrect unless the coefficient groups have division by two or the CW-decomposition has the origin as a $0$-cell.

\subsubsection{Proof of Theorem \ref{thm: sing=cell}}

The remainder of this section will be devoted to proving the above theorem. The following technical lemma is used to show that the pairs $(X^k,X^{k-1})$ of the $k$ and $(k-1)$-skeleta are ``good'' in a $\mathcal{P}$-equivariant sense:

\begin{lemma}\label{def} Let $\mathcal{P}$ be a pattern. Suppose that $F \colon [0,1]\times A \rightarrow A$ is a proper homotopy with $A \subset X$ for which:

\begin{itemize}

\item $F_0=Id_X$.

\item $\sup\{d_X(F_t(x),F_{t'}(x)) \mid t,t'\in[0,1]\}<D_F$ for some $D_F>0$.

\item $\Phi \circ F_t=F_t \circ \Phi$ on $B_{d_X}(x,D_F+\epsilon) \cap A$ for all $x \in A$, $\Phi \in \mathcal{P}_{x,y}^{R_F}$ for some $\epsilon,R_F>0$ (in particular, $R_F$ should be chosen so that these maps are well defined on $B_{d_X}(x,D_F+\epsilon) \cap A$). \end{itemize}

Then, for all $n$-chains $\sigma$ with $\mathcal{P}$E radius $R \geq R_F$ with radius of support bounded by $\epsilon /2$ and contained in $A$ we have that ${F_1}_*(\sigma)$ is $\mathcal{P}$E to radius $R+D_F+\epsilon$ and ${F_1}_*(\sigma)$ is homologous to $\sigma$ in $C_n^\mathcal{P}(X)$.\end{lemma}

\begin{proof} We have that $({Id_A})_* \simeq {F_1}_{*}$ on Borel-Moore chains. We need to show that the singular chain resulting from the prism operator \cite{Hat} given by the homotopy is $\mathcal{P}$E.

The homotopy gives, for each singular simplex $\omega \colon \Delta^n\rightarrow A$, a singular prism $F(\omega) \colon [0,1] \times \Delta^n \rightarrow A$. We could consider, given some chain $\sigma$, the ``singular prism chain'' $F(\sigma)$, where each singular prism inherits a coefficient from $\sigma$ and the homotopy $F$ in the obvious way. Using analogous notation as for singular chains, if $\Phi_*(F(\sigma)^x)=F(\sigma)^y$ for a $\mathcal{P}$E chain $\sigma$ the result follows since the same will hold for the corresponding singular chain after we apply the prism operator.

Let $\sigma$ have $\mathcal{P}$E radius $R \geq R_F$ and singular simplexes with radius of support bounded by $\epsilon/2$. Set $S=B_{d_X}(x,D_F)$. Then $\Phi_* (\sigma^S)=\sigma^{\Phi(S)}$ for all $\Phi \in \mathcal{P}_{x,y}^{R+D_F+\epsilon}$. To see this, note that $\Phi |_{B_{d_X}(x',R)} \in \mathcal{P}^R_{x',\Phi(x')}$ with $x' \in B_{d_X}(x,D_F+\epsilon)$. If $\Phi_*(\sigma^S) \neq \sigma^{\Phi(S)}$ then there would exist some singular simplex $\tau$ with non-zero coefficient in $\Phi_*(\sigma^S) - \sigma^{\Phi(S)}$ with support contained in $B_{d_X}(y,D_F+\epsilon)$. But then $\Phi_*(\sigma^{x'}) \neq \sigma^{\Phi(x')}$ for $x'$ in the support of $\Phi_*^{-1}(\tau)$.

The second assumption in the above lemma implies that $(F(\sigma^S))^x=(F(\sigma))^x$. The third implies that $\Phi_*(F(\sigma^S))=F(\Phi_*(\sigma^S))$ since the singular simplexes of $\sigma^S$ with non-zero coefficients are contained in $B_{d_X}(x,D_F+\epsilon)$, on which $F$ and $\Phi$ ``commute''.

Hence, we have that $\Phi_*(F(\sigma)^x) = \Phi_*(F(\sigma^S)^x) = (\Phi_*(F(\sigma^{S})))^y =$ \linebreak $(F(\Phi_* (\sigma^S)))^y = (F(\sigma^{\Phi(S)}))^y = (F(\sigma))^y$, so we get the required result after an application of the prism operator. \end{proof}

For a CW-complex, the pair $(X^k,X^{k-1})$ is good, that is, there exists a deformation retraction of some closed neighbourhood $U \subset X^k$ of $X^{k-1}$ onto $X^{k-1}$. It follows that $H_\bullet(X^k,X^{k-1}) \cong H_\bullet(X^k-X^{k-1},U-X^{k-1})$ by excision. Cellular patterns will have the same property with respect to the pattern:

\begin{lemma} Let $\mathcal{P}$ be a pattern on $X$ which has some cellular representation $\mathcal{C}$ with CW-decomposition $X^\bullet$. Then $H_\bullet^\mathcal{P}(X^k,X^{k-1}) \cong H_\bullet^\mathcal{P}(X^k-X^{k-1},U^{k-1}-X^{k-1})$ for $U^{k-1} \subset X^k$ some neighbourhood of $X^{k-1}$.\end{lemma}

\begin{proof} Since $\mathcal{P}$ represents $\mathcal{C}$ we have that elements $\Phi_{x,y} \in \mathcal{P}_{x,y}^R$ induce isometries between the cells intersecting $B_{d_X}(x,D+\epsilon)$ and $B_{d_X}(y,D+\epsilon)$ for sufficiently large $R$, these isometries on each cell being elements of $\mathcal{C}^r$ for some $r > 0$. By the ``tameness'' of the isometries of $\mathcal{C}$ we have deformation retractions from some closed neighbourhood $U \subset c$ of the boundary of each closed cell $c$ onto the boundary of $c$ which commute with the maps $\mathcal{C}_{c,c}^r$.

We define our deformation retraction by choosing, for each class of $k$-cell (under the equivalence relations defined by $\mathcal{C}_{c,d}^r$), a deformation retraction as above. Performing the deformation retraction on each cell induces a deformation retraction $F \colon [0,1] \times U^{k-1} \rightarrow U^{k-1}$ of some closed neighbourhood $U \subset X^k$ of $X^{k-1}$ onto $X^{k-1}$, which we claim satisfies the properties of the above lemma. It is easy to check that the homotopy is proper. Since each cell is bounded, points are moved only by some bounded distance. It also commutes with the maps $\Phi \in \mathcal{P}_{x,y}^R$ for sufficiently large $R$ since, by construction, it restricts to elements of $\mathcal{C}_{c,d}^r$ on each cell, each of which commute with the homotopy.

Every $\mathcal{P}$E chain is homologous to one with radius of support of arbitrarily small size by barycentric subdivision (see Lemma \ref{mv}). It follows, by an application of the homotopy constructed above, that every $\mathcal{P}$E chain of $U^{k-1}$ is homologous to one contained in $X^{k-1}$, so the inclusion $X^{k-1} \hookrightarrow U^{k-1}$ induces an isomorphism $H_\bullet^\mathcal{P}(X^k,X^{k-1}) \cong H_\bullet^\mathcal{P}(X^k,U^{k-1})$. By excision (Lemma \ref{mv}), we have that $H_\bullet^\mathcal{P}(X^k,X^{k-1}) \cong H_\bullet^\mathcal{P}(X^k-X^{k-1},U^{k-1}-X^{k-1})$. \end{proof}

\begin{lemma}\label{avg} Let $\mathcal{P}$ be a pattern with cellular representation $\mathcal{C}$. Suppose that $G$ has division by $|\mathcal{C}_{c,d}^R|$ (for $|\mathcal{C}^R_{c,d}| \neq 0$) for sufficiently large $R$. Then $H_i^\mathcal{P}(X^k-X^{k-1},U^{k-1}-X^{k-1};G)$ is trivial for $i \neq k$ and for $i=k$ is canonically isomorphic to the group of cellular PE $k$-chains of $\mathcal{C}$, that is, the group given by assigning coefficients to each oriented $k$-cell in a way which is equivariant with respect to the isometries $\Phi \in \mathcal{C}^R$ for some sufficiently large $R$. \end{lemma}

\begin{proof} We shall show that the inclusion of chain complexes $i \colon C_\bullet^\mathcal{P}(X-X^{k-1},U^{k-1}-X^{k-1};G) \rightarrow C_\bullet^{\text{BM}}(X-X^{k-1},U^{k-1}-X^{k-1};G)$ induces inclusions on homology. That is, if $\sigma$ is a $\mathcal{P}$E relative chain and $\sigma=\partial(\tau)$ for some relative Borel-Moore chain $\tau$, then $\tau$ can in fact be chosen to be $\mathcal{P}$E.

Let $\sigma$ be $\mathcal{P}$E to radius $R_1$, with $R_1$ large enough so that $G$ has division by $|\mathcal{C}_{c,d}^{R_1}|$, and $\sigma=\partial(\tau)$ with $\tau$ not necessarily $\mathcal{P}$E. Let the cells of $\mathcal{C}$ have radius of support bounded by $D$. Since $\mathcal{C}$ is a cellular representation of $\mathcal{P}$, there exists some $R_2$ such that whenever $\Phi \in \mathcal{C}^{R_2}_{c,d}$ then there exists some $\Phi' \in \mathcal{P}^{R_1+2D}_{x,y}$ which restricts to $\Phi$. We replace $\tau$ by an ``averaged'' version $\tau_A$. First, for each orbit of $k$-cell in $\mathcal{C}^{R_2}$ pick some representative cell $c$. Restrict $\tau$ to it and replace each coefficient $g$ of singular chain on it by $g/|\mathcal{C}_{c,c}^{R_2}|$. Now transport it to each equivalent $k$-cell by each of the maps $\Phi \in \mathcal{C}_{c,d}^{R_2}$. Repeating for each orbit of $k$-cell, we denote the resulting chain by $\tau_A$. Of course, $\tau_A$ is defined by its restrictions $\tau_A^d$ to each $k$-cell $d$; for $c$ the representative of the orbit of $d$, we have that $\tau_A^d = \sum_{\Phi \in \mathcal{C}_{c,d}^{R_2}} \Phi_*(\tau^c/|\mathcal{C}_{c,c}^{R_2}|)$. We must check that $\tau_A$ is $\mathcal{P}$E and that $\sigma=\partial(\tau)=\partial(\tau_A)$.

Since $\mathcal{C}$ is a cellular representation of $\mathcal{P}$, there exists some $R_3$ such that whenever $\Phi \in \mathcal{P}^{R_3}_{x,y}$ then $\Phi|_{c_x} \in \mathcal{C}_{c_x,\Phi(c_x)}^{R_2}$ for $c_x \ni x$. We claim that $\tau_A$ is $\mathcal{P}$E to radius $R_3$. To see this, let $\Phi \in \mathcal{P}_{x,y}^{R_3}$ for $x \in X^k-X^{k-1}$ and $x \in c_x$, $y \in c_y$ and $c$ the representative of $c_x$. We shall show that $\Phi_*(\tau_A^{c_x})=\tau_A^{c_y}$, from which it follows that $\tau_A$ is $\mathcal{P}$E. Since (the restriction of) $\Phi$ is in $\mathcal{C}_{c_x,c_y}^{R_2}$, we have that the set of isometries $\Phi \circ \mathcal{C}_{c,c_x}^{R_2}=\mathcal{C}_{c,c_y}^{R_2}$. Hence \[\Phi_*(\tau_A^{c_x})=\Phi_*\sum_{\Psi \in \mathcal{C}_{c,c_x}^{R_2}} \Psi_*(\tau^c /|\mathcal{C}_{c,c}^{R_2}|)=\sum_{\Psi \in \mathcal{C}_{c,c_x}^{R_2}} (\Phi \circ \Psi)_*(\tau^c /|\mathcal{C}_{c,c}^{R_2}|)= \] \[\sum_{\Psi \in \mathcal{C}_{c,c_y}^{R_2}} \Psi_*(\tau^c /|\mathcal{C}_{c,c}^{R_2}|)=\tau_A^{c_y}.\]

To show that $\partial(\tau_A)=\sigma$, it is sufficient to prove that $\partial(\tau_A)^d=\sigma^d$ for each $k$-cell $d$. Let $\Phi \in \mathcal{C}_{c,d}^{R_2}$. Since $\Phi=\Phi'|_c$ for some $\Phi' \in \mathcal{P}_{x,y}^{R_1+2D}$ with $x \in c$, it follows that $\Phi$ is a restriction of some $\Phi' \in \mathcal{P}_{x,\Phi(x)}^{R_1}$ for each $x \in c$. Hence, $\Phi_*(\sigma^c)=\sigma^d$ for each $\Phi \in \mathcal{C}_{c,d}^{R_2}$ since $\sigma$ is $\mathcal{P}$E to radius $R_1$. It follows that \[\partial(\tau_A)^d=\partial \sum_{\Phi \in \mathcal{C}_{c,d}^{R_2}}\Phi_*(\tau^c/|\mathcal{C}_{c,c}^{R_2}|)=\sum_{\Phi \in \mathcal{C}_{c,d}^{R_2}} \Phi_*(\sigma^c)/ |\mathcal{C}_{c,c}^{R_2}| = \] \[\sum_{\Phi \in \mathcal{C}_{c,d}^{R_2}} \sigma^d/ |\mathcal{C}_{c,c}^{R_2}|=(|\mathcal{C}_{c,c}^{R_2}|\sigma^d)/|\mathcal{C}_{c,c}^{R_2}|=\sigma^d.\]

Hence, we have an inclusion $i_* \colon H_\bullet^\mathcal{P}(X-X^{k-1},U-X^{k-1};G) \rightarrow H_\bullet^{\text{BM}}(X-X^{k-1},U-X^{k-1};G)$. It follows that $H_i^\mathcal{P}(X-X^{k-1},U-X^{k-1};G) \cong 0$ for $i \neq k$. For $i=k$, it is not too hard to see that the image of $i_*$ consists precisely of those chains of the form described in the lemma (one can construct arbitrary such equivariant relative singular chains similarly to the construction of $\tau_A$).\end{proof}

The remainder of the proof that the singular and cellular PE homology groups are isomorphic is now an algebraic affair. One uses the usual diagram chases, associated to the long exact sequences of pairs $(X^k,X^{k-1})$, as one usually uses to show the equivalence of singular and cellular homology (see, for example, \cite{Hat} pg.\ 139). Crucially, the vanishing of $H_i^\mathcal{P}(X^k,X^{k-1})$ for $i \neq k$ allows one to show that there is an isomorphism between the singular homology and the homology groups induced by the pairs $H_k^\mathcal{P}(X^k,X^{k-1})$. On the other hand, the above lemma also shows that this chain complex corresponds precisely to the cellular chain complex defining $H_k^\mathcal{C}$, which completes the proof.

\section{Pattern-Equivariant Homology of Tilings} \label{sect: Pattern-Equivariant Homology of Tilings} We postponed giving an explicit description of a cellular representation $\mathcal{C}_i$ for the pattern $\mathcal{T}_i$ ($i=0,1$) of a cellular tiling of $\mathbb{R}^d$ in Example \ref{ex: Cellular Tiling}; we shall provide one here.

For two $k$-cells $c_1,c_2$ of the tiling $T$, let $\Phi \in (\mathcal{C}_1)_{c_1,c_2}^R$ if and only if $\Phi$ is a translation taking $c_1$ to $c_2$ which also takes all of the tiles within distance $R$ of every point of $c_1$ to the equivalent patch at $c_2$ (of course preserving decorations of the tiles, if necessary). It is easy to see that $\mathcal{C}_1$ is a cellular representation of $\mathcal{T}_1$. Since $|(\mathcal{C}_1)_{c_1,c_2}^R|=0$ or $1$, it follows from the previous section that $H_\bullet^{\mathcal{T}_1}(\mathbb{R}^d;G) \cong H_\bullet^{\mathcal{C}_1}(\mathbb{R}^d;G)$ for any coefficient group $G$. Let two cells of the tiling $c_1 \sim^1_R c_2$ if $|(\mathcal{C}_1)^R_{c_1,c_2}| \neq 0$. A cellular $k$-chain $\sigma \in \mathcal{C}_1$ is a cellular Borel-Moore $k$-chain of the tiling with the following property: there exists some $R$ such that if $c_1 \sim_R^1 c_2$ then $\sigma(c_1) = \sigma(c_2)$. We shall say that $\sigma$ is a chain of $\mathcal{C}_1$E radius $R$. Then the group of $k$-chains with $\mathcal{C}_1$E radius $R$, denoted $C_k^{\mathcal{C}_1;R}(\mathbb{R}^d;G)$, can be identified with an assignment of coefficient from $G$ to each equivalence class of $k$-cell $[c]_R^1$.

A useful feature of this choice of $\mathcal{C}_1$ is that if $\sigma \in C_k^{\mathcal{C}_1;R}$ then $\partial(\sigma) \in C_{k-1}^{\mathcal{C}_1;R}$, so we may write $H_\bullet^{\mathcal{C}_1} \cong \varinjlim_R H_\bullet^{\mathcal{C}_1;R}$. We see from the above that these chain groups have the analogous description to the PE cohomology groups of the tiling as described in \cite{Sadun1}, except with the cellular coboundary maps replaced with the cellular boundary maps.

Of course, we may analogously define a pattern $\mathcal{C}_0$, replacing translations by rigid motions. It is not necessarily the case that $|(\mathcal{C}_0)_{c,d}^R|=0$ or $1$ since some cells may have rotational symmetries preserving their patches of tiles. Also note that for $\sigma \in C_k^{\mathcal{C}_0;R}$ here, we require that $\sigma(c)=-\sigma(c)$ if there exists some $\Phi \in (\mathcal{C}_0)^R_{c,c}$ reversing the orientation of $c$. If $|(\mathcal{C}_0)_{c,d}^R|$ may have order larger than $1$ then it is not necessarily true that $H_\bullet^{\mathcal{T}_1}(\mathbb{R}^d;G) \cong H_\bullet^{\mathcal{C}_1}(\mathbb{R}^d;G)$ for a general coefficient group $G$. To obtain this isomorphism for a general coefficient group, one should use a finer CW-decomposition for the tiling. For example, for a tiling of $\mathbb{R}^2$, one should take a CW-decomposition for the tiling so that points of local rotational symmetry are contained in the $0$-skeleton.

An alternative way of defining a cellular representation for a cellular tiling is to consider the patches of all tiles within an $R$-neighbourhood of a cell. Let $T$ be a cellular tiling on $(X,d_X)$ with allowed partial isometries $\mathcal{S}$. Then $\mathcal{T}_\mathcal{S}$ has a cellular representation, with CW-decomposition given by the original tiling, where we set $\Phi \in \mathcal{C}_{c,d}^R$ if and only if $\Phi$ is a restriction of some $\Phi' \in (\mathcal{T}_\mathcal{S})_{x,y}^R$ for every $x \in c$. Provided that the isometries fixing a cell are tame (and finite), it is easily checked that $\mathcal{C}$ as so defined is a cellular representation for $\mathcal{T}_\mathcal{S}$.

\subsection{Rotationally Invariant PE chains} Let $T$ be an FLC (with respect to translations) cellular tiling of $\mathbb{R}^d$. Suppose that a rotation group $\Theta < SO(d)$ acts on the translation classes of patches of the tiling. That is, for any finite patch of the tiling, the rotate of that patch by $\theta \in \Theta$ is a (translate of) a patch of the tiling. Then $\Theta$ canonically acts on the chain groups $C_k^{\mathcal{C}_1;R}(\mathbb{R}^d;G)$ in a way which commutes with the boundary maps, and similarly induces a well-defined action on $C_\bullet^{\mathcal{C}_1}(\mathbb{R}^d;G)$. We assume further that if two sufficiently large patches are equivalent up to a rigid motion, then they are so for some rotation in $\Theta$ followed by a translation. Then the cellular $\mathcal{T}_0$E chain groups can be identified with those $\mathcal{T}_1$E chains which are invariant with respect to the action of $\Theta$. Just as for PE cohomology, the $\mathcal{T}_0$E homology groups can be identified as the subgroups of the $\mathcal{T}_1$E homology groups of elements represented by rotationally invariant elements:

\begin{lemma} Let $G$ have division by $|\Theta|$. Then the inclusion of chain complexes $i \colon C_\bullet^{\mathcal{C}_0}(\mathbb{R}^d;G) \rightarrow C_\bullet^{\mathcal{C}_1}(\mathbb{R}^d;G)$ induces an injective map on homology. \end{lemma}

\begin{proof} Suppose that $\tau$ is $\mathcal{C}_1$E and $\sigma=\partial(\tau)$ is $\mathcal{C}_0$E. Similarly as to in the proof of Lemma \ref{avg}, by ``averaging'' $\tau$ one can show that there is a $\mathcal{T}_0$E chain $\tau_A$ with the same boundary as $\tau$.

In more detail, since $\tau$ is $\mathcal{T}_1$E, there is some $R$ such that $\tau \in C_k^{\mathcal{C}_1;R}$. Replace $\tau$ by the chain $\tau_A$, which assigns to the equivalence class $[c]_R^1$ the average $\sum_{\theta \in \Theta} \tau(\theta[c]_R^1)/|\Theta|$. Then $\tau_A$ is $\mathcal{T}_0$E since it assigns the same coefficient to any two cells $c \sim_R^0 d$. Further, since $\theta_*\partial(\tau)=\partial(\tau)$ as $\partial(\tau)$ is $\mathcal{T}_0$E, we have that \[\partial(\tau_A)=\partial \sum_{\theta \in \Theta} \theta_* \tau /|\Theta|  = \sum_{\theta \in \Theta}\theta_*\partial(\tau)/|\Theta|=\sum_{\theta \in \Theta}\partial(\tau) /|\Theta|=(|\Theta|\sigma)/|\Theta|=\sigma.\]

It follows that if $i(\tau)$ is the boundary of a $\mathcal{T}_1$E chain, then it is also the boundary of a $\mathcal{T}_0$E chain, and so the induced map on homology is injective.\end{proof}

\begin{exmp}\label{ex: Periodic} We shall discuss a simple periodic example, the periodic tiling of equilateral triangles of the plane. For periodic tilings, the (singular and cellular) $\mathcal{T}_1$ and $\mathcal{T}_0$-equivariant homology groups will be the homology of the chain complexes of (singular and cellular) chains invariant under the action of translations and rigid motions, resp., preserving the tiling.

The cellular chain complex associated to the pattern $\mathcal{T}_1$ for the periodic tiling of equilateral triangles (with the obvious CW-decomposition of $\mathbb{R}^2$) is $0 \leftarrow \mathbb{Z} \leftarrow \mathbb{Z}^3 \leftarrow \mathbb{Z}^2 \leftarrow 0$, as one has one vertex-type, three edge-types and two face-types, up to translation. One may compute that $H_0^{\mathcal{T}_1}(\mathbb{R}^2) \cong \mathbb{Z}$, $H_1^{\mathcal{T}_1}(\mathbb{R}^2) \cong \mathbb{Z}^2$ and $H_2^{\mathcal{T}_1}(\mathbb{R}^2) \cong \mathbb{Z}$. In the next chapter, we shall see that we have Poincar\'{e} duality between the PE homology and cohomology of this pattern. The PE cohomology agrees with cohomology of the tiling space of $\mathcal{T}_1$ which in this case is the $2$-torus, agreeing with the (regraded) PE homology groups here.

Now consider the pattern $\mathcal{T}_0$, for which chains should be invariant with respect to rigid motions. Due to complications with rotationally symmetric points, one must either choose to use coefficients with division by $6$, or a finer CW-decomposition of $\mathbb{R}^2$ for which the rotationally symmetric points are contained in $0$-cells. For the former case, using $\mathbb{R}$-coefficients, there is only one vertex and face-type. There are no generators associated to $1$-cells; for such a cellular $1$-chain $\sigma$, the rotational symmetry of the $1$-cells means that $\sigma=-\sigma$, and hence is zero over $\mathbb{R}$-coefficients. It follows that $H_0^{\mathcal{T}_0}(\mathbb{R}^2; \mathbb{R}) \cong \mathbb{R}$, $H_1^{\mathcal{T}_0}(\mathbb{R}^2;\mathbb{R}) \cong 0$ and $H_2^{\mathcal{T}_0}(\mathbb{R}^2;\mathbb{R}) \cong \mathbb{R}$, which are precisely the rotationally invariant parts of $H_{\bullet}^{\mathcal{T}_1}(\mathbb{R}^2;\mathbb{R})$.

To calculate the correct homology with $\mathbb{Z}$-coefficients, one must choose a finer CW-decomposition. A barycentric subdivision of the original complex is sufficient. One has three vertex and edge types and two face types. We calculate $H_0^{\mathcal{T}_0}(\mathbb{R}^2) \cong \mathbb{Z} \oplus \mathbb{Z}_6$, $H_1^{\mathcal{T}_0}(\mathbb{R}^2) \cong0$ and $H_2^{\mathcal{T}_0}(\mathbb{R}^2) \cong \mathbb{Z}$. To compute $H_0$, consider the map $H_0 \rightarrow \mathbb{Z}$ which takes the value $S=A+2B+3C$ on a class represented by assigning coefficients $A,B,C$ to the vertices of rotational symmetry order $6,3$ and $2$, respectively. It is easy to see that this is a well defined homomorphism (in fact, $S$ can be seen as the (rescaled) ``density'' of the cycle, which must be invariant under boundaries). One may always represent such a class by one for which $0 \le B < 3$ and $0 \le C < 2$ and, given such a class, since $A$ is determined by $S,B$ and $C$, one sees that the map sending this class to $(S,[C]_2,[B]_3) \in \mathbb{Z} \oplus \mathbb{Z}_2 \oplus \mathbb{Z}_3 \cong \mathbb{Z} \oplus \mathbb{Z}_6$ is an isomorphism. The generator of $\mathbb{Z}$ may be given as $(A,B,C)=(1,0,0)$, of $\mathbb{Z}_2$ as $(-3,0,1)$ and of $\mathbb{Z}_3$ as $(-2,1,0)$. Note that $H_0^{\mathcal{T}_0}$ does not correspond to a subgroup of $H_0^{\mathcal{T}_1}$ generated by rotationally invariant elements, the torsion elements are nullhomologous only by non-rotationally-invariant $\mathcal{T}_1$E boundaries.

This confirms our above calculation over $\mathbb{R}$-coefficients using the universal coefficient theorem. It seems common that torsion groups are picked up when one has two or more points of rotational symmetry in the pattern. Repeating the above calculations for the square tiling, one obtains $H_0^{\mathcal{T}_0} \cong \mathbb{Z} \oplus \mathbb{Z}_2 \oplus \mathbb{Z}_4$, $H_1^{\mathcal{T}_0} \cong 0$ and $H_0^{\mathcal{T}_0} \cong \mathbb{Z}$. Note that both of the spaces $\Omega^{\mathcal{T}_0}$ for the square and triangle tilings are homeomorphic to the $2$-sphere, despite the different $\mathcal{T}_0$E homologies (over $\mathbb{Z}$-coefficients) which seem to retain extra information about the different symmetries of the two tilings. \end{exmp}

\subsection{Uniformly Finite Homology} Much of the motivation for the definition of the PE chains is derived from an effort to describe topological invariants of tiling spaces in a highly geometric way. As an aside, we shall briefly mention here a possible alternative route. In \cite{BW}, Block and Weinberger construct aperiodic tilings on non-amenable spaces (see also \cite{MN}). The idea is to construct so called ``unbalanced tile sets''. This is made possible through the vanishing of certain homology groups, the (degree zero) uniformly finite homology groups. These homology groups are coarse invariants: if two metric spaces are coarsely quasi-isometric then these metric spaces have isomorphic uniformly finite homology groups.

It should be possible to construct the PE homology groups here from uniformly finite homology chains, instead of singular Borel-Moore chains. Indeed, given a tiling $T$ and a $\mathcal{T}$E homology chain on it, under reasonable conditions, this chain may be considered as a uniformly finite homology chain by mapping a simplex to the image of its vertices and extending linearly. This defines a chain map into the uniformly finite homology chain complex, with image those chains with coefficients assigned in a way which ``only depends on the local tile arrangement to some radius''. In fact, many of the uniformly finite chains constructed in \cite{BW} are of this nature: the ``impurity'' of a tiling determines a uniformly finite $0$-chain which is defined by the local tiles, and this impurity is homologous to zero via the boundary of a uniformly finite $1$-chain which is locally determined by the way pairs of tiles meet on their boundaries.

\chapter{Poincar\'{e} Duality for Pattern-Equivariant Homology}
\label{chap: Poincare Duality for Pattern-Equivariant Homology}

In this chapter we shall relate the PE homology and cohomology groups of an FLC tiling of $\mathbb{R}^d$ through Poincar\'{e} duality. There are two approaches to this, which correspond to two classical approaches to proving Poincar\'{e} duality, both of which we shall present here. The first is to consider cellular (co)chain complexes associated to a cellular tiling and a dual tiling, which the tiling is MLD to. The second approach is to consider the singular complexes and induce duality by capping with a PE fundamental class for the tiling.

These methods induce Poincar\'{e} duality isomorphisms for the patterns associated to the translational and rigid hull of an FLC tiling. However, we do not have duality for the pattern $\mathcal{T}_0$ associated to the space $\Omega^{\mathcal{T}_0}$, since it fails at points of local rotational symmetry. We consider a modified $\mathcal{T}_0$E chain complex which allows only multiples of chains at points of rotational symmetry. One regains duality by restricting to this subcomplex of PE chains.

\section{The Dual Cell Approach}

Let $T$ be an FLC tiling of $\mathbb{R}^d$. For our considerations, we may consider $T$ to be a tiling of polygonal tiles such that cells have trivial local isometries in the tiling, since every FLC tiling is $S$-MLD to such a tiling via the Voronoi method (see Subsection \ref{subsect: Examples of Tilings and Delone Sets}).

One may consider a dual CW-complex of this tiling, which is $S$-MLD to the original tiling. Each $k$-cell of the tiling is identified with a $(d-k)$-cell of the dual tiling and this identification induces an isomorphism between the chain complex of $\mathcal{T}_1$E cellular chains of the tiling and the $\mathcal{T}_1$E cellular cochains of the dual tiling. Hence, we have Poincar\'{e} duality $H^k_{\mathcal{T}_1}(\mathbb{R}^d) \cong H_{d-k}^{\mathcal{T}_1}(\mathbb{R}^d)$ between the cellular PE (co)homology groups.

We may use a similar method to deduce duality between the $\mathcal{T}_{\text{rot}}$E (co)homology groups, for which one needs to consider dual-cell decompositions of $E^+(d)$. However, we do not obtain duality for the $\mathcal{T}_0$E groups. The problem is that one may not use the dual complex of the tiling to evaluate the PE (co)homology groups. As we shall see, one may repair the duality isomorphism by only considering cellular chains which assign coefficients to cells in multiples of the order of symmetry of the cells in the tiling.

\section{Singular Pattern-Equivariant Cohomology}

\begin{definition} Let $C_\mathcal{P}^\bullet(X)$ be the subcomplex of $C^\bullet(X)$ (of singular cochains on $X$) of pullback cochains i.e., $C_\mathcal{P}^i(X):=\varinjlim_R \pi_{R,X}^i (C^i(K_R))$, taken with the usual coboundary map. The cohomology of this cochain complex, $H_\mathcal{P}^\bullet(X)$, is called the (\emph{singular}) \emph{pattern-equivariant cohomology} of $\mathcal{P}$.\end{definition}

Since $\pi_{R,X}^* \circ \delta (\psi)= \delta \circ \pi_{R,X}^*(\psi)$ we have that coboundaries of pullback cochains are pullbacks and so $C_\mathcal{P}^\bullet(X)$ is a well-defined cochain complex. Defining $C_{\mathcal{P};R}^\bullet(X) = \pi_{R,X}^\bullet(C^\bullet(K_R))$, we have that $C_\mathcal{P}^\bullet(X) = \varinjlim C_{\mathcal{P};R}^\bullet(X)$.

Let $U$ be a $\mathcal{P}$-set, that is $U=\pi_{R,X}^{-1}(S)$ for some set $S \subset K_R$. Then we define $H_\mathcal{P}^\bullet(U)$ by letting $C_\mathcal{P}^\bullet(U):= \varinjlim_{R' \ge R} \pi_{R',X}^\bullet(C^\bullet(\pi_{R,R'}^{-1}(S))$ and taking the coboundary map to be the usual singular coboundary map in $U$. As above, the coboundary of such a cochain is still a pullback and so this is a well-defined coboundary map on this subcomplex of the singular cochain complex of $U$. 

We may also define relative complexes. For $\mathcal{P}$-sets $U$ and $V$ define the \emph{relative PE cochain groups} $C^\bullet_\mathcal{P}(U,V)$ by letting $\psi \in C^\bullet_\mathcal{P}(U,V)$ if and only if $\psi$ is a cochain of $C^\bullet_\mathcal{P}(U)$ which evaluates to zero on any singular simplex contained in $V$. The coboundary maps are defined to be the usual singular coboundary maps.

\begin{lemma} \label{lem: Singular PE Cohomology} Let $\mathcal{P}$ be a pattern and suppose that, for some unbounded sequence $R_1 < R_2 < \ldots$, each approximant $K_{R_i}$ is a CW-complex for which singular simplexes on $K_R$ with support contained in some cell lift to $(X,d_X)$. Then there is a canonical isomorphism $\check{H}^\bullet(\Omega^\mathcal{P}) \cong H^\bullet_\mathcal{P}(X)$. \end{lemma}

\begin{proof} Denote by $\overline{C_\bullet(K_R)}$ the chain complex of singular chains on $K_R$ which lift to $X$. Since every singular $k$-chain on $K_{R_i}$ is homologous to a chain with simplexes contained in $k$-cells, we have that the inclusion of free abelian chain complexes $\iota_\bullet \colon \overline{C_\bullet(K_{R_i})} \rightarrow C_\bullet(K_{R_i})$ is a quasi-isomorphism. By a result of homological algebra (see Corollary 3.4 of \cite{Hat}), the dual of this map $\iota^\bullet \colon C^\bullet(K_{R_i}) \rightarrow \overline{C^\bullet(K_{R_i})}$ is also a quasi-isomorphism. Notice that this quasi-isomorphism is natural with respect to the connecting maps, in that $\iota^* \circ \pi_{R_i,R_{i+1}}^* = \pi_{R_i,R_{i+1}}^* \circ \iota^*$.

It is easily inspected that the restriction $\pi_{R_i,X}^\bullet \colon \overline{C^\bullet(K_{R_i})} \rightarrow C_{\mathcal{P};R}^\bullet(X)$ of the usual pullback is a cochain map. It is also a cochain isomorphism: it is surjective (by definition) and is injective since each chain of $\overline{C_\bullet(K_{R_i})}$ lifts. It follows that $\pi_{R_i,X}^\bullet=\pi_{R_i,X}^\bullet \circ \iota^\bullet \colon C^\bullet(K_{R_i}) \rightarrow C^\bullet_{\mathcal{P};R}(X)$ is a quasi-isomorphism.

Since each $K_{R_i}$ is a CW-complex, we have that $\check{H}^\bullet(K_{R_i})$ is naturally isomorphic to $H^\bullet(K_{R_i})$. Since \v{C}ech cohomology is a continuous contravariant functor, we have that $\check{H}^\bullet(\varprojlim K_R) \cong \varinjlim \check{H}^\bullet(K_R)$. It follows that \[\check{H}^\bullet(\Omega^\mathcal{P}) \cong \varinjlim_i H^\bullet(K_{R_i}) \cong \varinjlim_i H^\bullet_{\mathcal{P};R_i}(X) \cong H^\bullet_{\mathcal{P}}(X).\]\end{proof}

Notice that the above isomorphism will extend to the ring structures on $\check{H}^\bullet (\Omega^\mathcal{P})$ and $H^\bullet_\mathcal{P}(X)$ since the cup product commutes with the pullback maps.

\begin{corollary} Let $T$ be a tiling of $\mathbb{R}^d$ with FLC with respect to rigid motions. We have that $\check{H}^\bullet(\Omega^{\mathcal{T}_i}) \cong H^\bullet_{\mathcal{T}_i}(X)$ for $i=0,rot$ (with $X= \mathbb{R}^d,E^+(d)$, respectively). If additionally $T$ has FLC with respect to translations then $\check{H}^\bullet(\Omega^{\mathcal{T}_1}) \cong H^\bullet_{\mathcal{T}_1}(\mathbb{R}^d)$. \end{corollary}

\begin{proof} We may assume that our tiling is cellular and is such at each cell has trivial rotational symmetry in the tiling (since an FLC tiling of $\mathbb{R}^d$ is always MLD to such a tiling). Each $K_R$ is a BDHS complex \cite{BDHS} which satisfies the above. Alternatively, one may use the G\"{a}hler approximants to define the inverse limit space $\Omega^\mathcal{P}$ (see \cite{Sadun2} and Example \ref{ex: GC}). Since the collage of equivalence relations defining the G\"{a}hler complexes is equivalent to the original induced collage of $\mathcal{T}_i$, pullback cochains are unchanged by using these approximants. It is trivial to see that these approximants satisfy the required properties of the $K_R$ of the above lemma. \end{proof}

\begin{lemma} The cap product of a $\mathcal{P}$E chain (with singular simplexes of support bounded by $\epsilon/2$) and cochain with $\mathcal{P}$E radius $R$ is a $\mathcal{P}E$ chain (with $\mathcal{P}$E radius $R+\epsilon$, for any $\epsilon>0$). \end{lemma}

\begin{proof} Since $\mathcal{P}$ is a pattern we have that $\pi_{R,X} = \pi_{R,X} \circ \Phi$ on the open ball of radius $\epsilon$ at $x$ for all $\Phi \in \mathcal{P}_{x,y}^{R+\epsilon}$. Hence, if $x \sim_{R+\epsilon} y$ then $\Phi$ sends the front and back faces of $\sigma$ at $x$ to the front and back faces of $\sigma$ at $y$, resp., and since these chains have support bounded by $\epsilon/2$ the cochain will evaluate to the same value on each. Hence $\Phi$ sends the cap product at $x$ to the cap product at $y$. \end{proof}

The above shows that the cap product defined at the level of chains and cochains $\frown \colon C_p^{\text{BM}}(X) \times C^q(X) \rightarrow C_{p-q}^{\text{BM}}(X)$ restricts to a cap product on the pattern-equivariant complexes $\frown \colon C_p^\mathcal{P}(X) \times C^q_\mathcal{P}(X) \rightarrow C_{p-q}^\mathcal{P}(X)$. Of course, the usual formula $\partial(\sigma \frown \psi) = (-1)^q(\partial \sigma \frown \psi - \phi \frown \delta \psi)$ holds, so the cap product defines a homomorphism $\frown \colon H_p^\mathcal{P}(X) \times H^q_\mathcal{P}(X) \rightarrow H_{p-q}^\mathcal{P}(X)$ (and in fact makes $H_\bullet^\mathcal{P}(X)$ a right $H^\bullet_\mathcal{P}(X)$-module).

\section{Pattern-Equivariant Poincar\'{e} Duality} For $D\subset X$, let $H^\bullet(X | D)=H^\bullet(X,X-D)$. The compactly supported cohomology of a space is defined by the cochain complex with cochain groups $C_c^\bullet(X)=\varinjlim C^\bullet(X|K)$ over all compact $K \subset X$, with inclusions as connecting maps. In other words, one restricts to those cochains which vanish outside of some compact set.

The classical proof of Poincar\'{e} duality (see, for example, \cite{Hat}) is to show that capping with a fundamental class induces an isomorphism between compactly supported cohomology and homology, firstly for convex sets of $\mathbb{R}^d$. One then uses commutative diagrams of Mayer-Vietoris sequences to ``stitch together'' these isomorphisms to deduce duality for general manifolds. Analogous arguments will often follow in the pattern-equivariant setting.

Let $M$ be an oriented manifold. For $K \subset L \subset M$ we have the following diagram:

\vspace{5mm}
\centerline{
\xymatrix@C=0.1em@R=0.3em{H_d(M|L) \ar[dd]_{i_*} & \times & H^k(M|L) \ar[drr]_\frown & & & & \\
& & & & H_{d-k}(M) \\ H_d(M|K) & \times & H^k(M|K) \ar[uu]_{i^*} \ar[urr]_\frown & & & & }
}
\vspace{5mm}

\noindent
For each compact $L\subset M$, one can show that there is a unique element $\mu_{L}\in H_{d}(M|L)$ which restricts to some given orientation of $M$ so that $i_{*}(\mu_{L})=\mu_{K}$. One has, by naturality of the cap product, that $\mu_{K}\frown x=\mu_{L}\frown i^{*}(x)$. By varying $K$ over the compact sets of $M$, one has that the maps $-\frown\mu_{K}$ induce in the limit the duality homomorphism $D_M \colon H_c^k(M)\rightarrow H_{n-k}(M)$.

The main technicality of the proof of classical Poincar\'{e} duality is to show that there exists a commutative diagram of the form (see \cite{Hat}, Lemma 3.36):

\vspace{0.5mm}
\footnotesize
\xymatrix{\cdots \ar[r] & H_c^k(U \cap V) \ar[r] \ar[d]^{D_{U \cap V}} & H_c^k(U) \oplus H_c^k(V) \ar[r] \ar[d]^{D_U \oplus -D_V} &H_c^k(M) \ar[r] \ar[d]^{D_M} & H_c^{k+1}(U \cap V) \ar[r] \ar[d]^{D_{U \cap V}} & \cdots\\
\cdots \ar[r] & H_{d-k}(U \cap V) \ar[r] & H_{d-k}(U) \oplus H_{d-k}(V) \ar[r] &H_{d-k}(M) \ar[r] & H_{d-k-1}(U \cap V) \ar[r] &\cdots}
\normalsize
\vspace{5mm}

\noindent
where $M=U\cup V$ are manifolds. In more detail, one in fact has a diagram:

\vspace{5mm}

\xymatrix{\cdots \ar[r] & H^k(M | K \cap L) \ar[r] \ar[d] & H^k(M | K) \oplus H^k(M | L) \ar[r] \ar[d] &H^k(M | K \cup L) \ar[r] \ar[d] & \cdots\\
\cdots \ar[r] & H_{d-k}(U \cap V) \ar[r] & H_{d-k}(U) \oplus H_{d-k}(V) \ar[r] &H_{d-k}(M) \ar[r] & \cdots}

\vspace{5mm}

\noindent
for compact sets $K\subset U$ and $L\subset V$ which in the limit produces the required diagram. 

In cases where the maps to the approximants are well behaved we will have a similar diagram for the $\mathcal{P}$E homology and cohomology groups. In fact, to check the commutativity of the diagram it is easier to take cochains which only evaluate on chains with support bounded by the Lebesgue number of various covers, in which case the commutativity can be easily checked at the level of chains and cochains. As long as the pullbacks of cochains are well-behaved, we may restrict attention to such cochains without changing the $\mathcal{P}$E cohomology:

\begin{lemma} \label{lem: bounded cochains} Let $\epsilon>0$ and $\mathcal{P}$ be a pattern. Suppose that there exists some unbounded sequence $R_1 < R_2 < \ldots$ for which every cycle of each approximant $K_{R_i}$ has a barycentric subdivision for which the pullback chains have radius of support bounded by $\epsilon$. Suppose further that any chain of any $K_{R_i}$ which lifts to a chain of support bounded by $\epsilon$ does not also lift to a chain of support which is not bounded by $\epsilon$.

Then define $C \epsilon^\bullet_\mathcal{P}(X)$ to be the cochain complex of $\mathcal{P}$E cochains which evaluate to zero on any singular chain with radius of support not bounded by $\epsilon$, with coboundary map defined by the usual coboundary map followed by evaluating to zero on all chains with radius of support not bounded by $\epsilon$. Then this is a well-defined cochain complex and we have a natural cochain map $\iota \colon C_\mathcal{P}^\bullet \rightarrow C \epsilon_\mathcal{P}^\bullet(X)$ which is a quasi-isomorphism. \end{lemma}

\begin{proof} We must firstly check that the coboundary map of $C \epsilon^\bullet_\mathcal{P}(X)$ is well-defined. By definition the coboundary of a cochain is non-trivial only on chains bounded by $\epsilon$. It is also $\mathcal{P}$E. Indeed, before restricting to $\epsilon$-bounded chains the cochain was a pullback cochain $\psi = \pi_{R_i,X}^*(\psi')$. The required cochain is given by pulling back a modification of $\psi'$ which ignores those chains which pull back to chains with radius of support not bounded by $\epsilon$. By assumption, this does not alter the coefficients assigned to chains with support bounded by $\epsilon$. Hence, $C \epsilon^\bullet_\mathcal{P}(X)$ is a subcomplex of the dualisation of the singular complex of chains bounded by $\epsilon$.

We have a canonical map $\iota \colon C^\bullet_\mathcal{P}(X) \rightarrow C \epsilon^\bullet_\mathcal{P}(X)$ given by restricting cochains to singular simplexes bounded by $\epsilon$. By the above, this is a well-defined cochain map since it is simply a restriction of the dualisation of the map between singular chain complexes. We wish to show that it is a quasi-isomorphism. Similarly to in the proof of Lemma \ref{lem: Singular PE Cohomology}, we have that $C^\bullet_\mathcal{P}(X)$ is canonically isomorphic to the direct limit of singular cochain complexes of cochains of $K_R$ which lift. But this complex, in a similar manner, is quasi-isomorphic to the singular complex of chains which lift to singular chains bounded by $\epsilon$, since we may use barycentric subdivision and dualise. This cochain group is canonically isomorphic to $C \epsilon_\mathcal{P}^\bullet(X)$. The composition of these quasi-isomorphisms is $\iota$, and so the result follows. \end{proof}

So, in the cases where the approximants are well-behaved, we may replace the singular $\mathcal{P}$E cohomology groups with those based on chains with support bounded by some $\epsilon$. We may similarly define the complexes restricted to subsets and relative complexes. Using bounded cochains makes it simpler to prove the following:

\begin{lemma} Let $M$ be an orientable manifold and $\mathcal{P}$ be a pattern on it. Let $U$, $V$ be open $\mathcal{P}$-sets of $M$ and $K \subset U$, $L \subset V$ be closed $\mathcal{P}$-sets of $M$. Suppose that there exists some $\epsilon > 0$ such that every $\epsilon$-ball of $M$ is either contained in $U \cap V$ or doesn't intersect $K \cap L$, and similarly for the pairs $(U,K)$ and $(V,L)$ and $(U \cup V, K \cup L)$.

Suppose that $\epsilon$ and $\mathcal{P}$ satisfy the above lemma. Then the cap products with a fundamental class, of singular simplexes bounded by $\epsilon$, induces a diagram of Mayer-Vietoris sequences, commutative up to sign:

\vspace{5mm}

\xymatrix{\cdots \ar[r] & H_\mathcal{P}^k(M | K \cap L) \ar[r] \ar[d] & H_\mathcal{P}^k(M | K) \oplus H_\mathcal{P}^k(M | L) \ar[r] \ar[d] &H_\mathcal{P}^k(M | K \cup L) \ar[r] \ar[d] & \cdots\\
\cdots \ar[r] & H_{d-k}^\mathcal{P}(U \cap V) \ar[r] & H_{d-k}^\mathcal{P}(U) \oplus H_{d-k}^\mathcal{P}(V) \ar[r] &H_{d-k}^\mathcal{P}(M) \ar[r] & \cdots} 

\vspace{5mm}

\noindent
where the singular $\mathcal{P}$E homology and cohomology groups are based upon singular chains with radius of support bounded by $\epsilon$.
\end{lemma}

\begin{proof} By Lemma \ref{mv}, by barycentric subdivision, we may restrict attention to singular chains in the $\mathcal{P}$E homology whose chains have radius of support bounded by $\epsilon$. 

The usual proof of the commutativity of the above diagram relies on representing cochains as cochains which evaluate to zero on appropriate singular chains based on the various open covers. By restricting to $\mathcal{P}$E cochains based upon singular chains with radius of support smaller than the required Lebesgue number $\epsilon$, the commutativity of the diagram may be checked on the nose, analogously to in the proof of Lemma 3.36 of \cite{Hat}. \end{proof}

Since the duality here is induced by the cap product, and is therefore defined by a restriction of the classical Poincar\'{e} duality isomorphism to certain subcomplexes, it should be possible to extend the duality to ring structures on these invariants. That is, one should have that the cup product on $\check{H}^\bullet(\Omega_{\mathcal{T}_1})$ corresponds via Poincar\'{e} duality to an intersection product on $H_{d-\bullet}^{\mathcal{T}_1}(\mathbb{R}^d)$.

\subsection{Poincar\'{e} Duality for Translational Tiling Spaces} As discussed in the introduction to this chapter, there is a rather direct argument for Poincar\'{e} duality between $\mathcal{T}_1$E homology and cohomology, by considering dual tilings. We present here an alternative argument, by inducing the isomorphism by capping with a $\mathcal{T}_1$E fundamental class, which is analogous to the usual general proof of Poincar\'{e} duality. This approach makes clear the reason for Poincar\'{e} duality failing for the pattern $\mathcal{T}_0$, since it fails locally at points of rotational symmetry. This suggests an alternative complex, for which we do obtain duality.

\begin{theorem} Let $T$ be an FLC tiling of $(\mathbb{R}^d,d_{euc})$ with respect to translations. Then we have Poincar\'{e} duality $$H_{\mathcal{T}_1}^{d-\bullet}(\mathbb{R}^d) \cong H_\bullet^{\mathcal{T}_1}(\mathbb{R}^d)$$ induced by taking the cap product with a $\mathcal{T}_1$E fundamental class for $\mathbb{R}^d$. \end{theorem}

\begin{proof} By FLC, we may assume that our tiling of $\mathbb{R}^d$ has polygonal tiles. By consideration of the cellular groups, we see that there exists a singular $\mathcal{T}_1$E fundamental class for $\mathbb{R}^d$ (whose simplexes may be taken to be arbitrarily small in diameter). We consider here the G\"{a}hler approximants for the tiling, which does not change the pullback cochains. Due to the simple nature of the quotient maps from $\mathbb{R}^d$ to the approximants, it is easily verified that the condition of Lemma \ref{lem: bounded cochains} holds for small $\epsilon > 0$ for these approximants.

We then consider a covering of the first G\"{a}hler approximant with well-behaved contractible open sets $U_{[c]}$ associated to each cell $[c]$ of the approximant, where $[c]$ is an equivalence class of cell $c$ of the tiling. So, about the $0$-cells we take small disjoint and contractible open neighbourhoods. Similarly we have disjoint contractible open sets, each associated to some $1$-cell, which together with those open sets associated to the $0$-cells cover the $1$-skeleton. Continuing in this manner, one obtains a covering of the approximant. Since the tiles were polygonal, we may choose these regions so that their intersections are also contractible.

The pullbacks of these regions $U_{[c]}$ to $\mathbb{R}^d$ give a disjoint union of open regions $\coprod_{c \in [c]} V_c$, each $V_c$ associated to some cell $c$ of the tiling. One considers the $\mathcal{P}$E cohomology, based on small chains, relative to a small neighbourhood of the boundary of these regions. Then it is easy to see that the pullback cochains restricted to these regions correspond precisely to those cochains which are equal at any two such regions which correspond to cells which agree, up to a translation, on a sufficiently large patch of tiles. Using classical techniques, one may show that the relative singular cohomology on each region $V_c$ is trivial in all degrees except $d$ and in degree $d$ is isomorphic to $\mathbb{Z}$. Hence, the $\mathcal{T}_1$E relative cohomology of the preimage of $U_{[c]}$ relative to a small neighbourhood of the boundary is trivial in all degrees except $d$ where it corresponds to a PE choice of integer for each region $V_c$.

On the other hand, the $\mathcal{T}_1$E homology is trivial in all degrees except zero. The singular Borel-Moore homology of the pullback of $U_{[c]}$ corresponds to a choice of integer for each region $V_c$. The $\mathcal{T}_1$E homology, on the other hand, will make a PE choice of integer in each such region, so that two which agree on a sufficiently large radius of tiles are assigned the same integer. Classically, we have a Poincar\'{e} duality isomorphism at each such region $V_c$ and so the cap product with the Borel-Moore chain associated to the pullback of $U_{[c]}$ induces the desired Poincar\'{e} duality isomorphism for the region $U_{[c]}$.

The same analysis follows on the intersections of these regions. The result now follows by an application of the commutative diagram of the previous lemma. That is, we first choose a suitable Lebesgue number $\epsilon$ of the covering associated to our covering of the approximant and small neighbourhoods of the boundaries of the regions on it. The result follows from Poincar\'{e} duality holding for the simple contractible regions and an application of the five lemma. \end{proof}

\subsection{Poincar\'{e} Duality for Rotational Tiling Spaces} \label{Duality for Rotational Tiling Spaces}

An analogous argument to the above can also be used for the pattern $\mathcal{T}_{\text{rot}}$ defined on the Euclidean group $E^+(d)$. Example \ref{ex: Periodic} (and aperiodic examples in the next chapter), however, show that the argument cannot extend to the pattern $\mathcal{T}_0$ (for integral coefficients). Consider an FLC tiling $T$ of $\mathbb{R}^2$ with $n$-fold symmetry at the origin. Then (for sufficiently large $R$) the map $\pi_{R,X} \colon U \rightarrow K_R$ from a small open disk $U$ of $x$ into the approximant $K_R$ corresponds to the map $re^{i \theta} \mapsto re^{n i \theta}$ onto its image, with singular point at the origin. Hence, one may only realise $n$ times the usual generator of the compactly supported cohomology at $U$ in $\mathcal{T}_0$E cohomology. For the $\mathcal{T}_0$E homology, however, we may still realise the usual generator, by marking those points which are equivalent to radius $R$ to $x$. Hence, the duality map here will have cokernel $\mathbb{Z}/n\mathbb{Z}$ over $\mathbb{Z}$-coefficients.

The above discussion shows how we may retain duality by modifying the $\mathcal{T}_0$E homology groups. When the tiling has symmetry to radius $R$ at a point $x$ of order $n$, one should allow only rotationally invariant chains of $\mathcal{T}_0$E radius $R$ there to have coefficient divisible by $n$.

\begin{definition} Let $\sigma$ be an element of the \emph{modified singular pattern-equivariant chain group} $C_k^{\hat{\mathcal{P}}}(X)$ if and only if $\sigma$ is a $\mathcal{P}$E singular $k$-chain with $\mathcal{P}$E radius $R$ for which, for each $x \in X$, we have that $\sigma^x=\sum_{\Phi \in \mathcal{P}_{x,x}^R}\Phi_*(\tau)$ for some singular chain $\tau$. With the standard boundary map we thus define the \emph{modified singular pattern-equivariant homology groups} as the homology groups $H_\bullet^{\hat{\mathcal{P}}}(X)$ of this chain complex. \end{definition}

Equivalently, we allow a $\mathcal{P}$E chain $\sigma$ to be in the modified $\mathcal{P}$E chain group if the coefficient of some singular simplex $\tau$ of $\sigma$ is divisible by the order of the subgroup $F \le \mathcal{P}_{x,x}^R$ fixing that singular simplex. One may similarly define the \emph{modified cellular pattern-equivariant chain groups} $C_\bullet^{\hat{\mathcal{P}}}$ by taking those $\mathcal{P}$E chains of radius $R$ such that the coefficients assigned to cells $c$ are divisible by $|\mathcal{C}_{c,c}^R|$. One may verify that the results of Chapter \ref{chap: Pattern-Equivariant Homology} follow in their modified form, with only minor modifications to the proofs. For example, the deformation retractions used for showing the equivalence of singular and cellular PE homology apply to the modified complexes.

Returning to the example of a pattern $\mathcal{T}_0$ associated to an FLC tiling of $\mathbb{R}^2$, by considering the \emph{modified} $\mathcal{T}_0$ homology groups, we see that one removes the ``singular behaviour'' of chains at rotationally symmetric points. Following the above proof of Poincar\'{e} duality, we now see that we have Poincar\'{e} duality $H_{\mathcal{T}_0}^{d-\bullet}(\mathbb{R}^2) \cong {H_\bullet^{\hat{\mathcal{T}_0}}(\mathbb{R}^2)}$.

We may also see the Poincar\'{e} duality isomorphism directly on the cellular complexes. Let $T$ be a tiling of $\mathbb{R}^2$ with polygonal tiles such that points of local rotational symmetry are contained in the $0$-skeleton. One may consider a dual tiling which is $S$-MLD to the original. The dual-cell method for Poincar\'{e} duality does not work for the pattern $\mathcal{T}_0$ since to evaluate the $\mathcal{T}_0$E (co)homology we require that points of local rotational symmetry are contained in the $0$-skeleton, which will not be true of the dual complex. For example, applying the duality isomorphism to the $\mathcal{T}_0$E homology we may have coefficient one assigned to $2$-cells in the $\mathcal{T}_0$E cohomology which are fixed by non-trivial rotations. This fault is repaired by only allowing integer multiples of the order of rotational symmetry at the $0$-cells in the $\mathcal{T}_0$E cellular homology.

We have an inclusion of cellular chain complexes $C_\bullet^{\hat{\mathcal{T}_0}} \rightarrow C_\bullet^{\mathcal{T}_0}$. For a two-dimensional tiling the modified cellular chain complex differs from the unmodified one only in degree zero: a cellular $0$-chain in the modified complex should be a $\mathcal{T}_0$E cellular $0$-chain for which there exists some $R$ such that, if there is some rotation of order $n$ preserving some $0$-cell and its patch of tiles to radius $R$, then the coefficient assigned to that cell should be divisible by $n$. Consider the exact sequence of cellular chain complexes $$0 \rightarrow C_\bullet^{\hat{\mathcal{T}_0}} \rightarrow C_\bullet^{\mathcal{T}_0} \rightarrow C_\bullet^{\mathcal{T}_0}/C_\bullet^{\hat{\mathcal{T}_0}} \rightarrow 0.$$
\noindent
The quotient homology groups have a simple description:

\begin{proposition} The above quotient homology groups are trivial in degrees not equal to zero and the degree zero group is isomorphic to $\bigoplus_p \mathbb{Z}_p^{T(p)}$ where the direct sum is taken over all primes $p \in \mathbb{N}$ with $T(p)$ equal to the sum $\sum_{T_i} n_i$ over all tilings $T_i \in \Omega^{\mathcal{T}_{\text{rot}}}$ with rotational symmetry $k.p^{n_i}$ ($p \nmid k$) in case this sum is finite and $T(p)=|\mathbb{N}|$ otherwise.
\end{proposition}

\begin{proof} Since the groups $C_k^{\mathcal{T}_0}$ and $C_k^{\hat{\mathcal{T}_0}}$ have the same elements for $k \neq 0$, it is clear that the quotient homology groups are trivial for $k \neq 0$ and that the degree zero homology group is naturally isomorphic to $C_0^{\mathcal{T}_0} / C_0^{\hat{\mathcal{T}_0}}$.

The rigid hull $\Omega^{\mathcal{T}_{\text{rot}}}$ consists of tilings (modulo rotation) which are \emph{locally isomorphic} to $T$, that is, they are those tilings which have identical finite patches to $T$, up to rigid motion. If there are finitely many such tilings with rotational symmetry, the proposition is clear: each rotationally invariant tiling $T_i$ defines a generator of $C_0^{\mathcal{T}_0} / C_0^{\hat{\mathcal{T}_0}}$ which is given by assigning coefficient one to each vertex of $T$ whose patch of tiles agrees with the central patch of $T_i$ to radius $R$ (and every other vertex is assigned coefficient zero), where $R$ is chosen so that the patches of tiles of each $T_i$ to radius $R$ are pairwise distinct up to rotation.

Where there are infinitely many rotationally invariant tilings in the rigid hull, such an isomorphism is less obvious. By FLC, there are only finitely many integers $n > 1$ for which there are tilings $T_i \in \Omega^{\mathcal{T}_{\text{rot}}}$ with $n$-fold symmetry. We see that $C_0^{\mathcal{T}_0} / C_0^{\hat{\mathcal{T}_0}}$ is a countable abelian group of bounded exponent. It follows from Pr\"{u}fer's first theorem \cite{Fuchs} that $C_0^{\mathcal{T}_0} / C_0^{\hat{\mathcal{T}_0}}$ is isomorphic to a countable direct sum of cyclic groups $\bigoplus_p \mathbb{Z}_p^{r(p)}$. This group is isomorphic to $\bigoplus_p \mathbb{Z}_p^{T(p)}$ if and only if $r(p)=T(p)$ for each prime, which will follow from direct consideration of the group $C_0^{\mathcal{T}_0} / C_0^{\hat{\mathcal{T}_0}}$.

Let $f_1,\ldots,f_k$ be the list of primes appearing as a divisor of the order of rotational symmetry for only finitely many tilings in the hull. Let $R$ be such that the tilings having such symmetries are distinguishable to radius $R$. Then the subgroup generated by elements which assign coefficient $n.c$ with $n \in \mathbb{Z}$ to a patch (of radius larger than $R$) with rotational symmetry $c \cdot f_1^{a_1} \cdots  f_k^{a_k}$ (with $f_i \nmid c$) corresponds precisely to the direct sum of the $f_i$ parts of the group, which confirms that $r(f_i)=T(f_i)$. For the other primes, we have that $r(p) = |\mathbb{N}|$ when there are infinitely many tilings with rotational symmetry $p.k$ since one may easily construct infinitely many distinct elements of order $p$. Clearly there are no elements of order $p$ in $C_0^{\mathcal{T}_0} / C_0^{\hat{\mathcal{T}_0}}$ whenever $T(p)=0$, so the result follows. \end{proof}

It follows from the above that $H_i^{\mathcal{T}_0}(\mathbb{R}^d) \cong \check{H}^{2-i}(\Omega^{\mathcal{T}_0})$ for $i=1,2$, and in degree zero the $\mathcal{T}_0$E homology fits into the short exact sequence: $$0 \rightarrow \check{H}^2(\Omega^{\mathcal{T}_0}) \rightarrow H_0^{\mathcal{T}_0}(\mathbb{R}^d) \rightarrow \bigoplus_p \mathbb{Z}_p^{T(p)} \rightarrow 0.$$

\noindent
A similar statement will hold for two-dimensional hierarchical tilings using analogous arguments. For higher dimensional tilings, one would expect that it is possible for the homology groups of the quotient complex $C_\bullet^{\mathcal{T}_0} / C_\bullet^{\hat{\mathcal{T}}}$ to be non-trivial in degrees larger than zero. For example, a tiling of $\mathbb{R}^3$ may be invariant under rotations which fix a one-dimensional line, which will lead to non-trivial elements of $C_1^{\mathcal{T}_0}/C_1^{\hat{\mathcal{T}_0}}$. Hence, for tilings of $\mathbb{R}^d$ with $d > 2$, the torsion parts of $H_i^{\mathcal{T}_0}(\mathbb{R}^d)$ and $\check{H}^{d-i}(\Omega^{\mathcal{T}_0})$ may differ for degrees $i$ larger than zero.

\chapter{Pattern-Equivariant Homology of Hierarchical Tilings}

\label{chap: Pattern-Equivariant Homology of Hierarchical Tilings}

Given an FLC substitution tiling $T$ with substitution map $\omega$, its PE cohomology groups can be calculated using the fact that the tiling space is the inverse limit of the BDHS \cite{BDHS} or (collared) Anderson-Putnam  \cite{AP} complexes along with their self-maps induced by $\omega$. We give here an alternative method for carrying out these computations using PE homology.

\section{Computation of the Pattern-Equivariant Homology Groups}

We shall set a coefficient group $G$ and suppress its notation in the homology groups in this section. We assume throughout that $T_\omega = (T_0,T_1, \ldots )$ is an FLC hierarchical tiling of $(X,d_X)$, with respect to the allowed partial isometries $\mathcal{S}$, for which each $T_i$ is a cellular tiling for some common locally finite CW-decomposition $X^\bullet$ of $(X,d_X)$ with top dimension of cell equal to $d$.

Each of the tilings $T_i$ share the same CW-decomposition of $(X,d_X)$. To apply the main theorem of this section, we shall need the cells of $X^\bullet$ to get relatively small compared to the tiles of $T_i$ as $i \rightarrow \infty$ (c.f., the tiles of a substitution tiling relative to the supertiles). We assume that for each $i$ the tiles of $T_i$ are given an alternative CW-decomposition, making each $T_i$ a cellular tiling with this new CW-decomposition of the tiles, and in a way such that the $k$-skeleton of the induced CW-decomposition $T_i^k$ is contained in the original $k$-skeleton $X^k$ (for practical computations we shall usually want CW-decompositions which are ``combinatorially equivalent'' in a certain sense for each $T_i$).

\begin{definition} For a cell $c$ of $T_i^k$ let $P(c)$, the \emph{incident patch of $c$}, be the patch of tiles in $T_i$ which intersect the open cell of $c$. Given two $k$-cells $c_1,c_2$ of $T_i^k$, let them be considered as \emph{equivalent} (via the \emph{equivalence} $\Phi$) if $\Phi$ is an isometry from $P(c_1)$ to $P(c_2)$ sending $c_1$ to $c_2$ (and preserving decorations of tiles) which when restricted to the interior of the incident patch of $c_1$ is an element of $\mathcal{S}$. We call a class of $k$-cell under these equivalences a \emph{$k$-cell type}. For $k=0,1,2$ we call $k$-cell types \emph{vertex}, \emph{edge} and \emph{face types}, respectively.\end{definition}

For example, in $(\mathbb{R}^2,d_{euc})$ with $\mathcal{S}$ given by (restrictions of) translations, a vertex type consists of a $0$-cell along with all of the tiles which contain it, taken up to translation. An edge type consists of a $1$-cell along with those tiles intersecting its interior, taken up to translation. A face type is a $2$-cell of the tiling along with the tile it is contained in, taken up to translation. Notice that the composition of an equivalence $\Phi_1 \colon P(c_1) \rightarrow P(c_2)$ from $c_1$ to $c_2$ with an equivalence $\Phi_2 \colon P(c_2) \rightarrow P(c_3)$ from $c_2$ to $c_3$ is an equivalence from $c_1$ to $c_3$. For each $k$-cell type $[c]$ there is a well-defined number $n_{[c]}$, which is defined as the number of distinct equivalences of $c$ with itself when restricted to $c$, where $c$ is any representative of $[c]$. For example, $n_{[v]}=1$ for any vertex type $[v]$. We shall usually want $n_{[c]}=1$ for each $k$-cell type $[c]$; for example, for a tiling of $(\mathbb{R}^2,d_{euc})$ with $\mathcal{S}$ given as (restrictions of) rigid motions, this is satisfied if all points of local rotational symmetry in the tiling are contained in the $0$-skeleton.

The $k$-cell types induce a subcomplex of the cellular $\mathcal{T}_\omega$E chain complex, the complex of cellular chains which are determined locally by the nearby patch up to allowed partial isometry to some finite radius on some $T_i$. We denote by $C_\bullet^{\mathcal{A}_i}$, the \emph{$i^{th}$ approximant chain complex}, the chain complexes of cellular Borel-Moore chains of $T_i^\bullet$ which are invariant under equivalences of $k$-cell types. One may naturally identify a cellular $k$-chain of $T_i^\bullet$ with the obvious cellular chain of the (usually) finer CW-decomposition $X^\bullet$, the cellular chains here will always be assumed to be of the original CW decomposition $X^\bullet$. Given $\sigma \in H_k^{\mathcal{A}_i}$, we would like to say that it only depends on the local $k$-cell types of $T_{i+1}$ also. Unfortunately, the simple inclusion map is not necessarily cellular since the CW-decompositions $T_i^k$ vary as $i$ increases. However, we can always take a ``nice'' boundary which takes $\sigma$ to a cellular chain of $H_k^{\mathcal{A}_{i+1}}$ (see Figure \ref{fig: sub}).

\begin{figure}

\includegraphics{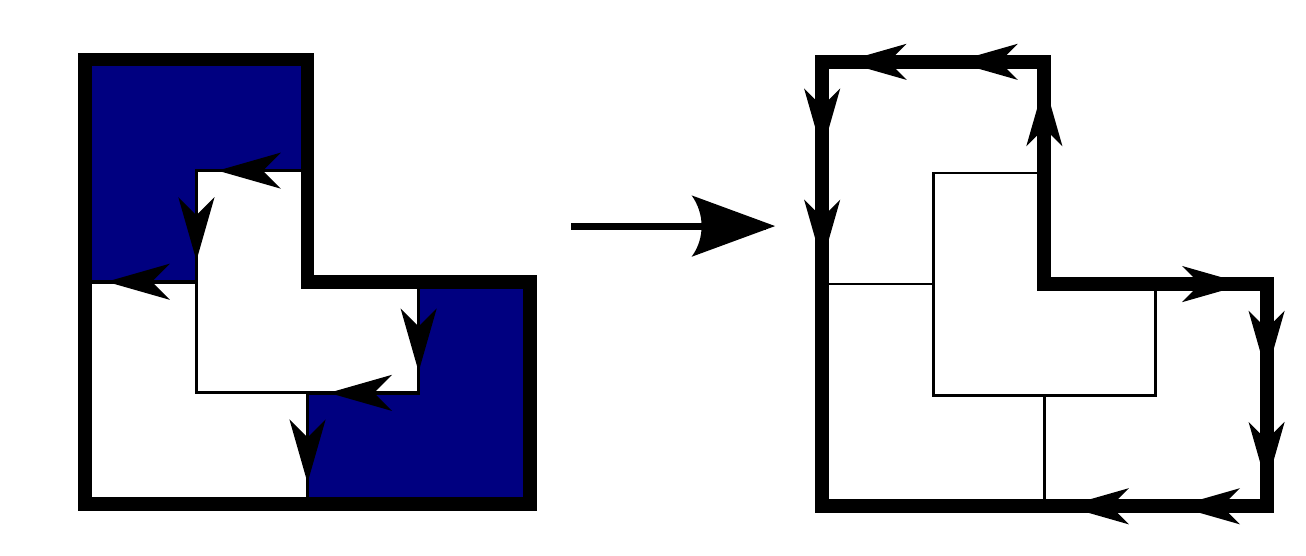} \caption{The substitution map between approximant complexes}

\label{fig: sub}

\end{figure}

\begin{proposition} Suppose that the coefficient group $G$ has division by $n_{[c]}$ for every $k$-cell type $[c]$ of each $T_i$. Let $[\sigma] \in H_k^{\mathcal{A}_i}$. Then there exists some cellular Borel-Moore $(k+1)$-chain $\tau$ such that $[\sigma + \delta(\tau)] \in H_k^{\mathcal{A}_{i+1}}$. The chain $\tau$ may be chosen so that $\tau(c_1)=\tau(c_2)$ for any two $(k+1)$-cells which are isometric via a restriction of an equivalence of $(k+1)$-cells of $T_{i+1}^{k+1}$. Further, given any other such choice $\tau'$ in replacement for $\tau$ satisfying this property, we have that $[\sigma+\tau]=[\sigma+\tau'] \in H_k^{\mathcal{A}_{i+1}}$. \end{proposition}

\begin{proof} If $\sigma \in H_d^{\mathcal{A}_i}$ then, since $X^\bullet$ is $d$-dimensional, each $d$-cell of $T_i$ is contained in a $d$-cell of $T_{i+1}$ and so the claim is immediate. So let $\sigma \in H_k^{\mathcal{A}_i}$ with $k<d$. Away from $T_{i+1}^{d-1}$, the coefficients assigned to $k$-cells by $\sigma$ are invariant under equivalences of $d$-cell types. Hence, we may choose a cellular $(k+1)$-chain which takes the support of $\sigma$ to $T_{i+1}^{d-1}$ by making such a choice for each $d$-cell type of $T_{i+1}$. If a $d$-cell type possesses non-trivial symmetries, we can ensure that the boundary chain chosen is invariant with respect to them by ``averaging'' the chain over the equivalences (c.f, the proof of Lemma \ref{avg}). We may continue this process until $\sigma$ is made homologous to a chain contained in $T_{i+1}^k$. By construction, $\tau$ assigns coefficients in a way which only depends on the local $(k+1)$-cell types up to equivalence and $\sigma + \delta(\tau) \in C_k^{\mathcal{A}_{i+1}}$.

Take a possibly different such choice $\tau'$ of $(k+1)$-chain satisfying the above and consider the chain $\tau - \tau'$. Of course, $\tau-\tau'$ has trivial boundary and, as above, $\tau-\tau'$ is homologous to a $(k+1)$-chain $\rho$ contained in $T_{i+1}^{k+1}$ which is invariant under equivalences of $(k+1)$-cell types of $T_{i+1}$. But then $((\sigma-\partial(\tau'))-(\sigma-\partial(\tau))= \partial(\tau-\tau')=\partial(\rho)$ and so is homologous to $0$ in $H_k^{\mathcal{A}_{i+1}}$ \end{proof}

The above proposition induces a well-defined homomorphism $\omega_k^i: H_k^{\mathcal{A}_i} \rightarrow H_k^{\mathcal{A}_{i+1}}$, which we shall name the \emph{substitution map}.

It will be useful in the proof of the theorem below to consider singular and cellular PE chains. We define cellular PE chains for $T_i$ (with respect to the allowed partial isometries $\mathcal{S}$ and CW-decomposition $X^\bullet$) analogously as to in the discussion in Section \ref{sect: Pattern-Equivariant Homology of Tilings}. A cellular Borel-Moore $k$-chain will be considered to be $\mathcal{T}_\omega$E if it is $(\mathcal{T}_{k(R)})_\mathcal{S}$E to radius $R$ where $k: \mathbb{R}_{>0} \rightarrow \mathbb{N}_0$ is a non-decreasing unbounded function (the choice of such a $k$ is unimportant). By the results of Chapter \ref{chap: Pattern-Equivariant Homology} (given conditions on the divisibility of the coefficient group $G$ with respect to the local symmetries of cells), there is a canonical isomorphism between the singular and cellular $\mathcal{T}_\omega$E homology groups.

Now, note that each element of $C_\bullet^{\mathcal{A}_i}$ may be considered as a cellular $\mathcal{T}_\omega$E chain, since it has as elements chains which depend only on their incident patches of tiles in $T_i$. Since each of the substitution maps is defined by an inclusion followed by a $\mathcal{T}_\omega$E boundary, and each chain of $C_\bullet^{\mathcal{A}_i}$ is a $\mathcal{T}_\omega$E cellular chain, we have a homomorphism $f \colon \varinjlim (H_\bullet^{\mathcal{A}_i},\omega_\bullet^i) \rightarrow H_\bullet^{\mathcal{T}_\omega}(X)$. Under reasonable conditions, this map is in fact an isomorphism:

\begin{theorem} Let $T_\omega$ be a hierarchical tiling. Suppose that the coefficient group $G$ has division by $n_{[c]}$ for each $k$-cell type $[c]$ of each $T_i$ and that, for each $R>0$, for sufficiently large $i$ there exist deformation retractions $F_i^p \colon [0,1] \times A_i^p \rightarrow A_i^p$ for $0 \leq p \leq d-1$ of a neighbourhood $A_i^p \subset T_i^{p+1}$ of the $p$-skeleton $T_i^p$ onto $T_i^p$ satisfying the following:

\begin{enumerate}
	\item Each $F^p_i$ is equivariant with respect to equivalences of $p$-cell types. That is, $A^p_i$ and $F^p_i$ are defined cell-wise, for each $p$-cell type, and are invariant under self equivalences of a $p$-cell type onto itself.
	\item Points of $T_i^{p+1} - A_i^p$ have $R_p$-neighbourhoods intersecting only those tiles in the incident patch of the cell containing the point. We have that $R_{d-1}=R$ (for $R_p$ with $p \neq d$ see below).
	\item For each $F_i^p$ we have that $\sup \{d_X(F^p_i(t,x),F^p_i(t',x)) \mid t,t' \in [0,1] \} < D_p$ for some $D_p>0$.
	\item $D_p \geq R_p$ and $R_{p-1} \geq 3D_p+2 \epsilon$ for some $\epsilon > 0$.
\end{enumerate}

Then there exists an isomorphism: $$\varinjlim (H_\bullet^{\mathcal{A}_i},\omega_\bullet^i) \cong H_\bullet^{\mathcal{T}_\omega}(X).$$
\end{theorem}

\begin{proof} By our assumptions on the divisibility of the coefficient group $G$ and the results of Section \ref{sect: Cellular PE Homology}, we have canonical isomorphisms between the singular and cellular PE homology groups of $\mathcal{T}_i$ and $\mathcal{T}_\omega$. Let $[\sigma] \in H_k^{\mathcal{T}_\omega}(X)$ be represented by the singular $\mathcal{T}_\omega$E chain $\sigma$. Then there is some $T_j$ for which $\sigma$ is $\mathcal{T}_j$E to radius $R$. Pick $i \ge j$ for which $T_i$ has the above deformation retractions for value $R$.

By assumption $2$, away from $A_i^{d-1}$ we have that $\sigma$ is invariant up to equivalences of $d$-cell types since the chain is $\mathcal{T}_i$E to radius $R$. Suppose that $\sigma$ is a $k$-chain with $k<d$ (and, without loss of generality due to Lemma \ref{mv}, has singular simplexes of small radius of support). Then one may choose a boundary $\rho$ with $\sigma+\partial(\rho) \subset A_i^{d-1}$ for which $\rho$ is invariant under equivalences of $d$-cell types (if necessary we average the chain over self-equivalences of $d$-cell types, c.f., the proof of Lemma \ref{avg}). Again, by Lemma \ref{mv}, we may assume that the resulting cycle in $A_i^{d-1}$ has singular simplexes with radius of support bounded by $\epsilon/2$. As in Lemma \ref{def}, the deformation retraction induces a boundary taking $\sigma$ to a chain contained in $T_i^{d-1}$. We have that $\Psi \circ F_i^{d-1} = F_i^{d-1} \circ \Psi$ for equivalences $\Psi$ between $d$-cell types, and so for any $\Phi \colon B_{d_X}(x,2D+\epsilon) \rightarrow B_{d_X}(x,2D+\epsilon)$ which is a restriction of equivalences of $d$-cell types we have that $\Phi \circ F_i^{d-1} = F_i^{d-1} \circ \Phi$ on $B_{d_X}(x,D+\epsilon)$. As in Lemma \ref{def}, we see that $\sigma$ is homologous to a chain contained in $T_i^{d-1}$ which only depends on the $d$-cell types to radius $3D_{d-1}+ 2 \epsilon$.

By assumption $4$ we may inductively repeat the above process until $[\sigma]$ is represented by a chain contained in $T_i^k$ which only depends on the local tiles to radius $3D_k+2 \epsilon$. Since, away from $A_i^{k-1}$, the chain is locally determined by $k$-cell types, it follows that $[\sigma]$ is represented by a cellular chain of $T_i^k$ which is invariant under equivalences of $k$-cell types. Hence, the map $f$ is surjective.

Showing injectivity of the map is similar to the above. Let $\sigma$ be a $k$-cycle of the direct limit for which $\sigma=\partial(\tau')$ for some $\mathcal{T}_\omega$E $(k+1)$-chain $\tau'$. Then $\tau'$ may be represented as a singular chain which is $\mathcal{T}_j$E to radius $R$. Pick $i \ge j$ for which $T_i$ has the above deformation retractions for value $R$.

As in the construction of the substitution map, we may find a boundary $\rho$ which only depends on incident patches of $(k+1)$-cells and makes $\sigma$ homologous to a chain contained in $T_i^k$. Similarly to above, one can show that $\tau:=\tau'+\rho$ is represented by a chain of the direct limit; one inductively pushes $\tau$ to the $(k+1)$-skeleton ending with a $\mathcal{T}_i$E chain contained in $T_i^{k+1}$ which, away from $A_i^{k+2}$, only depends on local $(k+1)$-cells up to equivalences, and so the result follows. \end{proof}

The assumption of the existence of the deformation retractions in the above theorem is analogous to the assumption of the tameness of the isometries in Definition \ref{CPatDef} for cellular patterns. We require here, though, that the deformation retractions can be taken so as to retract arbitrarily large neighbourhoods of the skeleta. In particular the cells of $T_i^k$ should grow arbitrarily large in diameter as $i \rightarrow \infty$. This is a key assumption for the above theorem: chains whose coefficients may depend on a large diameter of tiles in some $T_j$ may be seen as chains whose coefficients only depend on a small diameter of tiles in some $T_i$, relative to the size of the tiles.

For a cellular substitution tiling of $(\mathbb{R}^d,d_{euc})$, with allowed partial isometries given by translations or rigid motions, the above theorem makes the $\mathcal{T}_\omega$E homology groups computable since the approximant homology groups are then all isomorphic and in a way such that the induced substitution maps between them are the same; see the examples below. Hence, by Poincar\'{e} duality, this gives an alternative way of computing the \v{C}ech cohomology groups $\check{H}^{d-\bullet}(\Omega^{\mathcal{T}_\omega})  \cong H_{\mathcal{T}_\omega}^{d-\bullet}(\mathbb{R}^d) \cong H_\bullet^{\mathcal{T}_\omega}(\mathbb{R}^d)$ when we take (restrictions of) translations as the allowed partial isometries. When we take our allowed partial isometries to be given by (restrictions of) rigid motions of $\mathbb{R}^d$, we do not in general obtain duality with the \v{C}ech cohomology $\check{H}^\bullet(\Omega^{\mathcal{T}_\omega})$ (over non-divisible coefficients). One may however use the modified homology groups, as in the discussion at the end of Chapter \ref{chap: Poincare Duality for Pattern-Equivariant Homology}, using only multiples of chains at rotationally invariant cells. The above methods may be adjusted in the obvious way so as to compute the modified homology groups, which are Poincar\'{e} dual to the PE cohomology groups.

It is interesting to note that the duals of the approximant complexes have been considered in the work of Gon\c{c}alves \cite{Gon}. There they were incorporated into a method for computation of the $K$-theory of the $C^*$-algebra associated to the stable equivalence relation of a substitution tiling. The calculations indicate that there is a duality present between the $K$-theory of the stable and unstable equivalence relations (which for low dimension tilings correspond simply to direct sums of the \v{C}ech cohomology groups of the tiling space). The appearance of these approximant complexes here indicate that this is indeed the case, although extra analysis of the substitution maps is needed to confirm the relationship.

\section{Examples}
\subsection{The Fibonacci tiling}

The Fibonacci tiling of $\mathbb{R}^1$ is given as a substitution on two intervals, named here $0$ and $1$ of lengths the golden ratio $\phi$ and one, resp., and the substitution $0 \mapsto 01$, $1 \mapsto 0$ (see Figure \ref{fig:Fibonacci}). We have $2$ edge types, corresponding to the two types of tiles, and $3$ vertex types (denoted by $0.1$, $1.0$ and $0.0$) up to translation. We set an orientation of our tiles all pointing to the right, so that the boundary of the $0$ tile is the formal sum $0.1-1.0$ (since the $0.1$ vertex can by found to the right of the $0$ tile, $1.0$ to the left and for a $0.0$ vertex type the contributions from the left and right tile cancel) and similarly the boundary operator maps the $1$ tile to the formal sum $1.0-0.1$.

\begin{figure}

\includegraphics{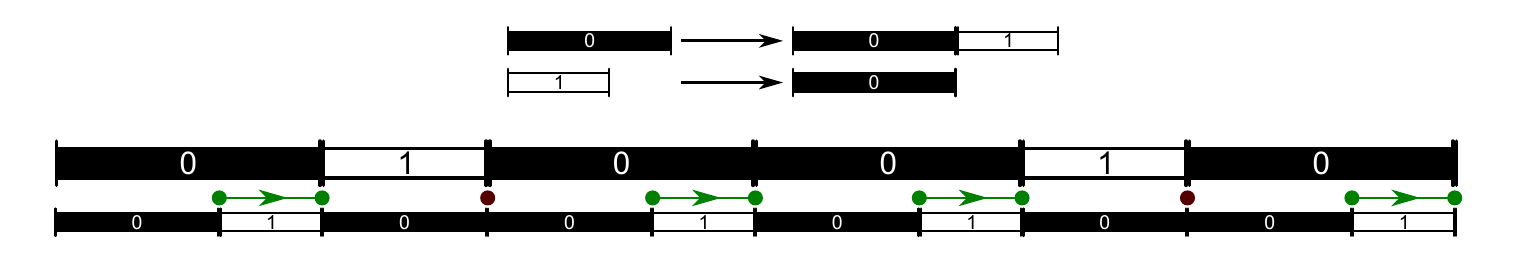} \caption{Fibonacci Tiling}

\label{fig:Fibonacci}

\end{figure}

Hence, the approximant homology is $H_0^\mathcal{A}=\mathbb{Z}^2$, $H_1^\mathcal{A}=\mathbb{Z}$, where $H_0^\mathcal{A}$ is generated by $a=0.1$ and $b=0.0$ ($0.1$ and $1.0$ are identified in homology) and $H_1^\mathcal{A}$ (as usual) is generated by the fundamental class of $\mathbb{R}^1$ by assigning coefficient one to each edge type.

For the substitution map $\omega_0 \colon H_0^\mathcal{A} \rightarrow H_0^\mathcal{A}$, note that all $0.1$ vertices are found in the interior of $1$ tiles (distance $\phi$ from the left endpoint of the interval) and all $0.0$ vertices lie on a $1.0$ vertex in the supertile decomposition of the tiling. The PE chain marking the interior point of the $0$ tiles is homologous to the formal sum $0.0+0.1$ (i.e., the sum of all vertex types with the $0$ tile to the left, see Figure \ref{fig:Fibonacci}) and the $1.0$ vertex type is homologous to $0.1$ as above. Hence, $\omega_0(a,b)=(a+b,a)$, which is an isomorphism so $H_0^{\mathcal{T}_\omega} = \varinjlim (\mathbb{Z}^2,\omega_0) = \mathbb{Z}^{2}$ and, as is usual, $H_1^{\mathcal{T}_\omega} = \mathbb{Z}$ is generated by the fundamental class for $\mathbb{R}^1$ given by assigning coefficient one to each tile.

\subsection{The Thue-Morse Tiling}

The Thue-Morse tiling of $\mathbb{R}^1$ is given as a substitution on two intervals $0$ and $1$ each of unit length, and substitution given by $0 \mapsto 01$, $1 \mapsto 10$. There are $2$ edge types, $0$ and $1$, and $4$ vertex types $0.0$, $0.1$, $1.0$ and $1.1$ up to translation. The boundary map will again send the 0 tile to the formal sum 0.1-1.0 and the 1 tile to 1.0-0.1, so that $H_0^\mathcal{A}=\mathbb{Z}^{3}$.

In the supertile decomposition, vertices of tiles can be found either on the vertices of the supertile decomposition, or in the centre of the supertiles. Applying the substitution map to the vertex types generating $H_0^\mathcal{A}$, we see that $\omega_0(0.1)=1.1+\dot{0}$, $\omega_0(0.0)=1.0$ and $\omega_0(1.1)=0.1$ where $\dot{0}$ is the PE chain marking the centres of the $0$ tiles. Since $\dot{0} \sim 0.1+0.0$ and $1.0 \sim 0.1$, we have the $\omega_0(a,b,c)=(a+b+c,a,a)$. This linear map has eigenvectors with eigenvalues $0$, $-1$ and $2$, although they do not span $\mathbb{Z}^3$. With some further careful analysis, the direct limit can be identified as $H_0^{\mathcal{T}_\omega}=\varinjlim (\mathbb{Z}^{3},\omega_0)=\mathbb{Z}\oplus\mathbb{Z}[1/2]$.

\subsection{The Dyadic Solenoid} \label{subsect: dyadic solenoid} One can put a hierarchical structure onto the periodic triangle tiling analogously to Example \ref{ex: DS}, from the substitution of the triangle into four of half the size. The approximant homology with respect to translations has already been computed and is the homology of the periodic tiling without the added hierarchical structure. It is easy to to compute the induced substitutions $\omega_\bullet$ as the times 4,2 and 1 maps in degree 0,1 and 2, respectively. Hence, we have that $H_0 \cong\mathbb{Z}[1/4]$, $H_1 \cong\mathbb{Z}[1/2]\oplus\mathbb{Z}[1/2]$ and $H_2 \cong \mathbb{Z}$ which, via Poincar\'{e} duality, agrees with the observation that $\Omega^{\mathcal{T}_\omega} \cong \mathbb{S}_2^2$, the two-dimensional dyadic solenoid.

One may also compute the homology with respect to rigid motions. Again, the approximant homology will be as for the homology of the triangle tiling, $H_0 \cong \mathbb{Z} \oplus \mathbb{Z}_2 \oplus \mathbb{Z}_3$, $H_1 \cong 0$ and $H_2 \cong \mathbb{Z}$. One finds that, with the generators as described in Example \ref{ex: Periodic}, $\omega_0 (a,b,c)=(4a,a+b,c)$. This map is conjugate via an automorphism of $\mathbb{Z} \oplus \mathbb{Z}_2 \oplus \mathbb{Z}_3$ to the map sending $(a,b,c) \mapsto (4a,b,c)$ and hence the PE homology is $H_0 \cong \mathbb{Z}[1/4] \oplus \mathbb{Z}_2 \oplus \mathbb{Z}_3$, $H_1 \cong 0$ and $H_2 \cong \mathbb{Z}$. One may similarly compute the homology with respect to rigid motions for the square model of the dyadic solenoid; the substitution map $\omega_0 \colon \mathbb{Z} \oplus \mathbb{Z}_2 \oplus \mathbb{Z}_4 \rightarrow \mathbb{Z} \oplus \mathbb{Z}_2 \oplus \mathbb{Z}_4$ maps $(a,b,c) \mapsto (4a,c,c)$, so $H_0 \cong \mathbb{Z}[1/4] \oplus \mathbb{Z}_4$.

\subsection{The Penrose Kite and Dart Tilings} \label{subsect: The Penrose Tiling}

We shall compute the PE homology of the Penrose kite and dart tilings with respect to rigid motions. There are $7$ vertex types (sun, star, ace, deuce, jack, queen and king, in Conway's notation) and $7$ edge types ($E1-7$) up to rigid motion, see Figure \ref{fig:VaETypes} where each of the vertex and edge types are listed in their respective orders. The two face types correspond to the kite and dart tiles on which there is defined the well-known kite and dart substitution.

\begin{figure}

\includegraphics[width=\linewidth]{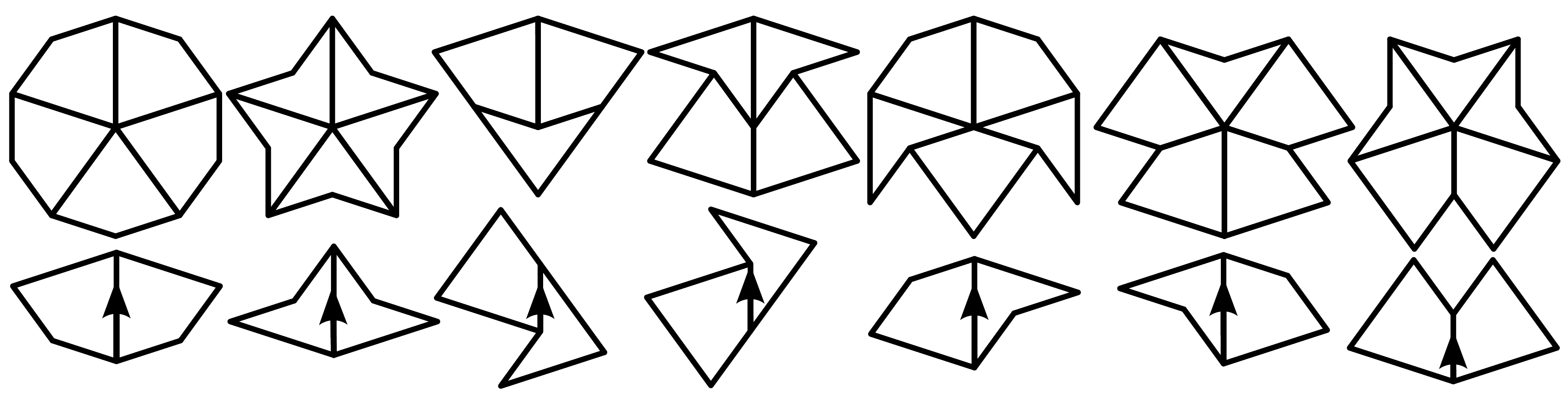} \caption{Vertex and edge types}

\label{fig:VaETypes}

\end{figure}

The $\partial_1$ boundary map, with ordered bases as listed above, is represented as a matrix by: \[ \left( \begin{array}{ccccccc}
5  & 0  & 0  & 0  & 0  & 0  & 0\\
0  & -5 & 0  & 0  & 0  & 0  & 0\\
-1 & 0  & -1 & 1  & 0  & 0  & 0\\
0  & 1  & 1  & -1 & 0  & 0  & 1\\
1  & 0  & 1  & -1 & -1 & -1 & 0\\
-1 & 0  & 0  & 0  & 1  & 1  & -2\\
0  & -2 & 0  & 0  & 1  & 1  & -1   \end{array} \right)\]

\noindent
The first column, for example, is $(5,0,-1,0,1,-1,0)^T$ since at any sun vertex there are $5$ incoming $E1$ edges, at an ace there is an outgoing $E1$, at a jack there is an incoming $E1$ and at a queen there is an outgoing $E1$.

Some simple calculations using the Smith normal form show that $H_0^{\mathcal{A}} \cong \mathbb{Z}^2 \oplus \mathbb{Z}_5$, with basis of the free abelian part given by $e_1=\text{sun}$ and $e_2=\text{star}$ and the torsion part generated by $t=\text{sun}+\text{star}-\text{queen}$. The torsion element is illustrated in Figure \ref{fig:S0TorBdy}; we have that $5t+\partial_1(-E1+E2-E4-2E7)=0$.

For the $\omega_0$ substitution map, we first note that each sun vertex lies on one of a star, queen or king vertex in the supertiling, each star lies on a sun and each queen lies on a deuce. Hence, up to homology, we compute that $\omega_0(e_1) \simeq 3e_1-e_2 +2t$, $\omega_0(e_2) \simeq e_1$ and $\omega_0(t)=t$. We see that the $\omega_0$ is an isomorphism on $H_0^\mathcal{A}$ and hence $H_0^{\mathcal{T}_\omega} \cong \mathbb{Z}^2 \oplus \mathbb{Z}_5$.

The degree one group $H_1^{\mathcal{A}}$ is generated by $E3+E4$, that is, the cycle which trails the bottom of the dart tiles (with appropriate orientation). This is illustrated in Figure \ref{fig:Penrose} in red. The analogous $1$-cycle trailing the bottoms of the superdarts in the supertiling is illustrated in green. It is evident from the figure that the two cycles (with opposite orientations) are homologous via the locally defined $2$-chain given as the boundary of the dart tiles. We see that $\omega_1$ is an isomorphism on $H_1^{\mathcal{A}}$ and so $H_1^{\mathcal{T}_\omega} \cong \mathbb{Z}$. As usual, $H_2^{\mathcal{T}_\omega} \cong \mathbb{Z}$ is generated by a fundamental class for $\mathbb{R}^2$ by assigning coefficient one to each $2$-cell.

Every star and sun vertex has five-fold rotational symmetry. One may consider the modified chain complexes, where one allows only multiples of $5$ for the coefficients of the star and sun vertices. One replaces the boundary map $\partial_1$ with the modified version which has entries $1$ and $-1$ instead of $5$ and $-5$ for the edge types terminating and initiating from the star and sun vertex types, respectively. One may repeat the above calculations for these modified groups. The resulting homology groups are Poincar\'{e} dual to the \v{C}ech cohomology groups $\check{H}^\bullet(\Omega^0)$ and, indeed, we calculate that $H_0^{\hat{\mathcal{T}_\omega}} \cong \mathbb{Z}^2$, $H_1^{\hat{\mathcal{T}_\omega}} \cong \mathbb{Z}$ and $H_2^{\hat{\mathcal{T}_\omega}} \cong \mathbb{Z}$. As in the discussion at the end of Chapter \ref{chap: Poincare Duality for Pattern-Equivariant Homology}, the modified and unmodified groups fit into an exact sequence which at degree zero is \[0 \rightarrow \mathbb{Z}^2 \underset{f}{\rightarrow} \mathbb{Z}^2 \oplus \mathbb{Z}_5 \underset{g}\rightarrow \mathbb{Z}_5^2 \rightarrow 0\]

\noindent
where $f(a,b)=(a+2b,a-3b,[4a+3b]_5)$ and $g(a,b,[c]_5)=([a+c]_5,[b+c]_5)$.

\begin{figure}
\centering

\includegraphics[width=4in]{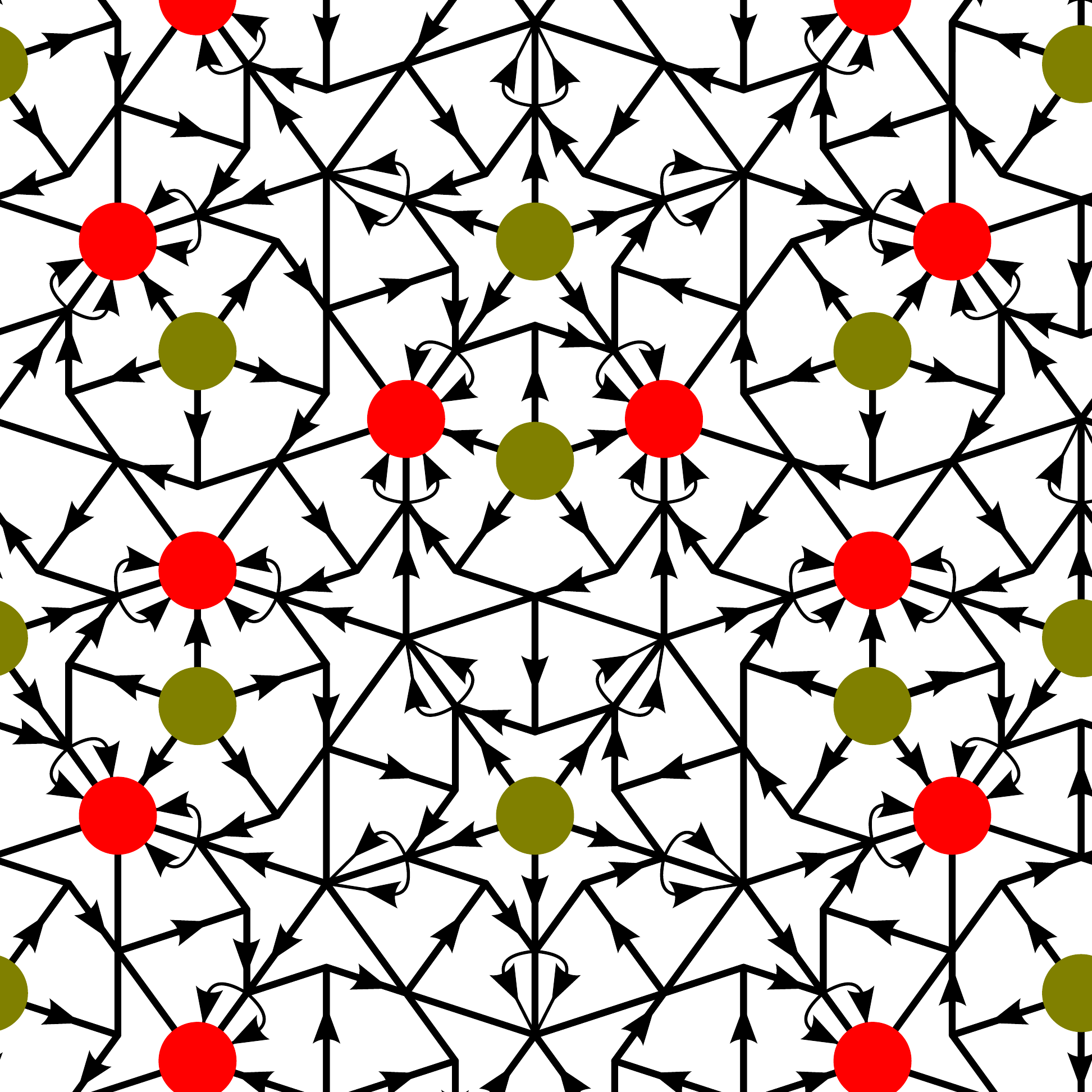}
\caption{Degree $0$ torsion element $t$ with $5t+\partial_1(-E1+E2-E4-2E7)=0$}
\label{fig:S0TorBdy}

\includegraphics[width=4in]{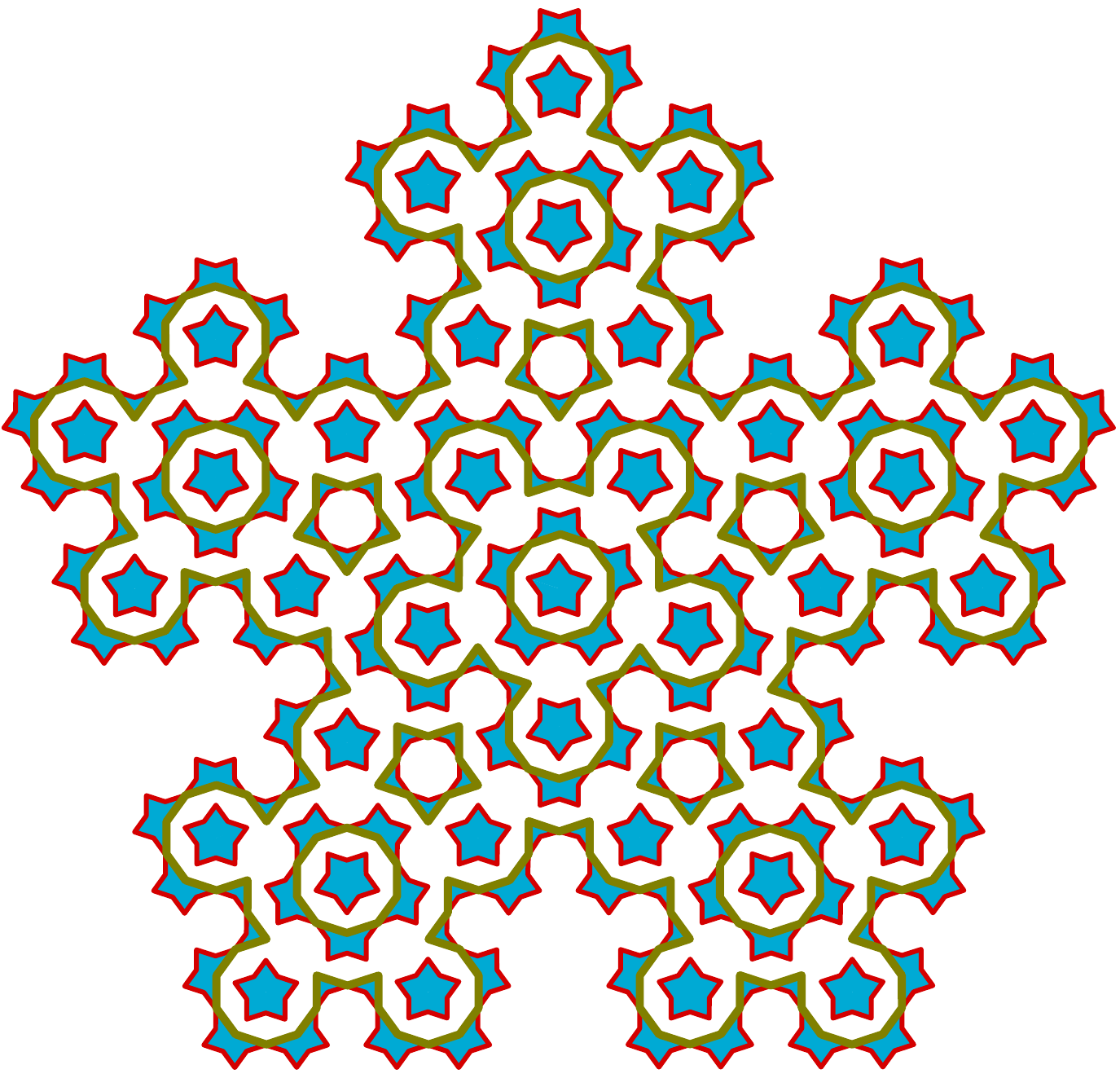}
\caption{Generating elements of $H_1^\mathcal{A}$ for the tiling (red) and supertiling (green) which are homologous, with opposite orientations, by the $2$-chain defined by the dart tiles (blue)}
\label{fig:Penrose}
\end{figure}

\subsection{Point Patterns on a Free Group} Consider the free group on two generators $F(a,b)$ generated by $a,b$. We may consider its associated Cayley graph $C$ (with these generators), which may be given a metric by extending the word metric on the $0$-cells to the $1$-skeleton (giving each edge unit length). The group $F(a,b)$ naturally acts on this space as a group of isometries.

One may construct ``periodic'' point patterns of this space in the following way. Consider a group $G$ along with a homomorphism $f \colon F(a,b) \rightarrow G$. Then the kernel of this map determines a point pattern of the Cayley graph (contained in the vertex set). If the group $G$ is finite then $ker(f)$ is of finite index and so the group (which of course is $ker(f)$) of isometries preserving the point pattern acts cocompactly on $C$.

We shall construct a hierarchical point pattern of a similar nature to the dyadic solenoid construction of Example \ref{ex: DS}. For each $i \in \mathbb{N}_0$ consider the map $f_i \colon F(a,b) \rightarrow \mathbb{Z}_{2^i}$ defined by sending $a,b$ to $[1]_{2^i}$. This defines point patterns $P_i$ on $C$ by marking points of the kernel. For each $i$, define a tile of $C$ as a ``cross'' of $1$-cells connecting the centre of the cross to $a^i$, $b^i$, $a^{-i}$ and $b^{-i}$. By placing such a tile with centre at each point of the point pattern, we form a tiling $T_i$; each $1$-cell is covered by precisely one tile, although tiles may overlap on more than just their boundaries. Each of the point patterns locally determines the corresponding tiling and vice versa. Note that the patch of tiles at a point of $T_i$ is determined by the patch of tiles at that point of $T_{i+1}$, so we may consider the collection $T_\omega=(T_0,T_1,\ldots )$ as a (non-recognisable) hierarchical tiling.

We have that the approximant groups $H_0^{\mathcal{A}_i}(C)$ are isomorphic to $\mathbb{Z}$, as represented by assigning coefficient $1$ to the vertices of $ker(f_i)$. Other $0$-chains whose coefficients are locally determined by the tiling are homologous to these by moving the non-zero coefficients of the $0$-chain locally, for example, to the nearest element of $ker(f_i)$ to their right-hand side (which, again, will be a $1$-chain whose coefficients are locally determined by the tiling). The connecting maps $\omega_0^i \colon \mathbb{Z} \rightarrow \mathbb{Z}$ are the times two maps, since a generator for $H_0^{\mathcal{A}_i}$ is mapped to the chain assigning coefficient one to the centres and boundaries of each cross in $T_{i+1}$, hence $H_0^{\mathcal{T}_\omega} \cong \mathbb{Z}[1/2]$. The group $H_1^{\mathcal{T}_\omega}$ will be free abelian, and it is not too hard to see that it has countable many generators. The space $\overline{(C,d_{\mathcal{T}_\omega})}$ is connected but not locally path-connected, and is such that the Cayley graph $C$ maps bijectively and continuously onto each path component (that is, the space has leaves homeomorphic to $C$). We do not obtain Poincar\'{e} duality here, since we do not have Poincar\'{e} duality for the space $C$; Poincar\'{e} duality fails locally at the vertices where the compactly supported cohomology is isomorphic to $\mathbb{Z}^3$.

\subsection{A ``Regular'' Pentagonal Tiling of the Plane} \label{subsect: Pent}

In \cite{Pent} Bowers and Stephenson define an interesting tiling coming from a combinatorial substitution of pentagons. Of course, one may not tile $(\mathbb{R}^2,d_{euc})$ by regular pentagons. However, the combinatorial substitution defines a CW-decomposition of $\mathbb{R}^2$ for which each $2$-cell has $5$ bounding $0$-cells and $1$-cells. One may assign the unit edge metric on the $1$-skeleton and extend this metric to the faces so that each face is isometric to a regular pentagon. Defining the distance between two points as the length of the shortest path connecting them, the resulting metric space (which is homeomorphic to $\mathbb{R}^d$) can be tiled by the regular pentagon. The tiling shares many properties of usual substitution tilings of Euclidean space, and in fact the techniques as described here are directly applicable to this tiling.

An interesting feature of this example is that the group of global isometries of the metric space $(P,d_P)$ constructed above is rather small. Of course, the metric space is locally non-Euclidean at the vertices of the tiling, and as one can easily imagine this allows the tiling, that is, the locations of the regular pentagons in $P$, to be determined by the metric. One sees that isometries of $K$ correspond precisely to combinatorial automorphisms preserving the CW-decomposition of $K$. However, the tiling is aperiodic in the following sense: $P/Iso(T)=P/Iso(P)$ is non-compact. Indeed, the substitution is recognisable, one may invert the subdivision process using only local information. Any isometry of $P$ must preserve this hierarchy of tilings and so there can be no cocompact group action preserving the space. Hence, one may not define a sensible tiling metric in terms of preservation of large patches by global isometries of the space $P$ as is usual, for example, for tilings of $(\mathbb{R}^d,d_{euc})$. However, this issue does not cause any complications here, we may define the tiling metric (or uniformity) in terms of partial isometries.

We shall let $\mathcal{S}$ be the collection of partial isometries between open sets of $P$ which preserve orientation. We define a hierarchical tiling on $(P,d_P)$ by setting $T_0=T$ and $T_{i+1}$ to be the unique tiling of $(P,d_P)$ of super$^{i+1}$-tiles (which are formed from glueing together $6^{i+1}$ regular pentagons, as described by the substitution rule) which subdivides to $T_i$. For $\mathbb{R}$-coefficients, since the coefficient group is divisible, one may take the CW-decomposition of $T_0$ to be that of the original tiling and the analogous such CW-decompositions for each subsequent $T_i$. We see that there are two vertex types, those which lie at the meeting of three or of four pentagons, one edge type and one face type. Since there is an equivalence of the edge type with itself reversing orientations, the approximant chain complex is $0 \leftarrow \mathbb{R}^2 \leftarrow 0 \leftarrow \mathbb{R} \leftarrow 0$. One may compute that the substitution maps induce isomorphisms in each degree and so $H_0^{\mathcal{T}_\omega} \cong \mathbb{R}^2$, $H_1^{\mathcal{T}_\omega} \cong 0$ and $H_2^{\mathcal{T}_\omega} \cong\mathbb{R}$.

To compute the groups over $\mathbb{Z}$-coefficients, one needs to take a finer CW-decomposition of $(P,d_P)$ for each tiling. Of course, each such CW-decomposition will be analogous for each tiling. One should include the centres of pentagons and edges in the $0$-skeleton of $T_0$ since the tiling has local symmetry at these points. We compute that the approximant homologies are $H_0^{\mathcal{A}} \cong \mathbb{Z}^2$, $H_1^{\mathcal{A}} \cong 0$ and $H_2^{\mathcal{A}} \cong \mathbb{Z}$. As usual, $H_2^{\mathcal{T}_\omega}$ is generated by the fundamental class of $P$, so we just need to compute the substitution map on $H_0^{\mathcal{T}_\omega}$. 

The matrix has eigenvalues $1$ and $6$ with eigenvectors which span $\mathbb{Z}^2$. Hence, we have that $H_0^{\mathcal{T}_\omega} \cong \mathbb{Z} \oplus \mathbb{Z}[1/6]$, $H_1^{\mathcal{T}_\omega} \cong 0$ and $H_2^{\mathcal{T}_\omega} \cong \mathbb{Z}$.

One may repeat these calculations for the modified homology groups. We calculate them as being isomorphic to the unmodified groups. As a result of the substitution being recognisable, we have that the inverse limit space $\Omega^{\mathcal{T}_\omega}$ is homeomorphic to the tiling space $\Omega^\mathcal{T} \cong \overline{(P,d_\mathcal{T})}$ of $T$. Hence, by Poincar\'{e} duality, we have that $\check{H}^0(\overline{(P,d_\mathcal{T})}) \cong \mathbb{Z}$, $\check{H}^1(\overline{(P,d_\mathcal{T})}) \cong 0$ and $\check{H}^2(\overline{(P,d_\mathcal{T})}) \cong \mathbb{Z} \oplus \mathbb{Z}[1/6]$.

We note that these topological invariants may be thought of as invariants for the metric space $(P,d_P)$. One may consider a metric on the pairs $(P,x)$ where $x \in P$ using a Gromov-Hausdorff type metric, see \cite{Calegari}. This metric will be equivalent to that of the tiling metric for tilings centred at points $x \in P$.

\chapter{A Spectral Sequence for the Rigid Hull of a Two-Dimensional Tiling}
\label{A Spectral Sequence for the Rigid Hull of a Two-Dimensional Tiling}

Let $T$ be a tiling of $\mathbb{R}^d$ with FLC with respect to rigid motions. To simplify notation we shall write $\Omega^{\text{rot}}:=\Omega^{\mathcal{T}_{\text{rot}}}$ for the rigid hull. Then $SO(d)$ acts on $\Omega^{\text{rot}}$ by rotations. Let $f \colon \Omega^{\text{rot}} \rightarrow \Omega^0 := \Omega^{\mathcal{T}_0}$ denote the quotient map. It turns out that $f$ is ``nearly'' a fibration with fibres $SO(d)$. The problem lies in the fact that this action is not always free: there may be tilings in the hull $\Omega^{\text{rot}}$ which are left fixed by some non-trivial rotations and at these points the map fails to be a fibration.

It was shown in \cite{BDHS} how one may modify the Serre spectral sequence for $f$ so as to obtain a spectral sequence converging to the cohomology of $\Omega^{\text{rot}}$ given the cohomology of $\Omega^0$ and the number of rotationally invariant tilings in the hull. In more detail, the spectral sequences applies to any FLC (with respect to rigid motions) tiling of $(\mathbb{R}^2,d_{euc})$ coming from a recognisable substitution for which $\Omega^0$ contains exactly $m$ tilings with $n$-fold rotational symmetry and no other rotationally symmetric tilings. A ``na\"{\i}ve Serre spectral sequence'', since the fibre is $SO(2) \cong S^1$, would have $E_2$ terms $E_2^{p,q} \cong \check{H}^p(\Omega^0;\check{H}^q(S^1)) \cong \check{H}^p(\Omega^0)$ for $q=0,1$ and be trivial otherwise. The spectral sequence of \cite{BDHS} modifies and repairs this with the addition of extra torsion $\mathbb{Z}_n^{m-1}$ in $E_2^{1,1}$.

We shall present here a similar spectral sequence applicable to any FLC (with respect to rigid motions) tiling of $(\mathbb{R}^2,d_{euc})$. In some sense, it is similar to that of \cite{BDHS} in that the inputs of the $E^2$ page are the \v{C}ech cohomology groups of $\Omega^0$ with the exception of a group altered with the addition of extra torsion which seems to reflect the existence of rotationally invariant tilings in the hull.

The $E^2$ page is not the end of the story, however, there is one potentially non-trivial homomorphism $d^2_{2,0}$ after the second page of the spectral sequence. We shall show that this map has a particularly nice description. In fact, the approach dovetails excellently with the methods outlined in Chapter \ref{chap: Poincare Duality for Pattern-Equivariant Homology} and we shall show how one may easily compute the \v{C}ech cohomology of the rigid hull of the Penrose tiling given the calculations conducted in \ref{subsect: The Penrose Tiling}.

\section{A Basic Introduction to Spectral Sequences}

Although spectral sequences are notoriously complicated gadgets, the ideas motivating them are rather elementary. Suppose that we have some topological space $X$ and a subspace $A \subset X$. We may consider this as a \emph{filtration} of the space $X$, with the space $A$ at an intermediate level. There is a well-known long exact sequence which one may use to compute the homology of $X$ in terms of the homology of $A$ and the pair $(X,A)$. The maps of this long exact sequence are determined by the way in which the subspace $A$ sits inside of $X$. Then, up to some extension problems, the homology of $X$ is determined by this information.

Suppose that we have a more complicated filtration with several choices of subspace. Then keeping track of the interactions of the pieces becomes more complicated and a more sophisticated gadget than a long exact sequence is required. A spectral sequence is an algebraic tool designed to do this job, described by J.\ F.\ Adams as ``\ldots like an exact sequence, but more complicated''. It should be emphasised that, like short exact sequences, they are a tool of homological algebra and need not be derived from topological considerations, despite this being their original motivation. They may be defined (for example) given a filtration of an algebraic object of interest; loosely speaking, when something is known about how the elements of this filtration ``fit together'', the spectral sequence induced by the filtration may provide information about the algebraic object itself (although, occasionally, information about the filtration or its elements can be derived from knowledge of the algebraic object in question).

We do not have the space here to give a thorough introduction to spectral sequences. We refer the reader to \cite{McCleary} for more information, or to \cite{HatcherSpSe} for a particularly accessible introduction to them. We shall at least describe the general form of a spectral sequence: it consists of a sequence of \emph{pages} $E^r_{\bullet,\bullet}$, one for each $r \in \mathbb{N}_0$, each such page given as a two-dimensional array of, say, $R$-modules $E^r_{p,q}$, one for each $(p,q) \in \mathbb{Z} \times \mathbb{Z}$. In many cases the non-trivial elements of $E^r_{\bullet, \bullet}$ are contained in the upper right quadrant $\mathbb{Z}_{\geq 0} \times \mathbb{Z}_{\geq 0}$. There are \emph{differentials} $d^r_{p,q} \colon E^r_{p,q} \rightarrow E^r_{p-r,q+r-1}$ for each page (our spectral sequence is of \emph{homological type}) and are required to satisfy $d^r \circ d^r =0$. So one may form homology groups at each $(p,q) \in \mathbb{Z} \times \mathbb{Z}$ and, indeed, these homology groups should correspond to the entries of the next page. In many cases, at any coordinate, the incoming and outgoing arrows $d^r$ are eventually zero; it is easy to see that this happens, for example, when the non-trivial modules are contained in the upper right quadrant. In this case the groups $E_{p,q}^r$ are eventually stable at each $(p,q) \in \mathbb{Z} \times \mathbb{Z}$ for sufficiently large $r$ (possibly depending on $(p,q)$) and so one has groups $E^\infty_{p,q}$. These stable groups are of primary interest: up to extension problems, these are what identify the algebraic object of interest.

Take, for example, a chain complex of $R$-modules $C_\bullet$ which one may like to compute the homology of. There are many instances where such a complex comes with the extra structure of a \emph{bounded filtration}, that is, subcomplexes $0 = F_{-1} C_\bullet \leq F_0 C_\bullet \ldots \leq F_n C_\bullet = C_\bullet$. Define $G_p C_\bullet:= F_p C_\bullet / F_{p-1} C_\bullet$, which is a chain complex with the obvious differential induced from that of $C_\bullet$. Then there is a spectral sequence associated to the filtered complex with first page given by $E_{p,q}^1:= H_{p+q}(G_p C_\bullet)$. The differentials of the first page are defined as follows. Given $[\sigma] \in H_{p+q}(G_p C_\bullet)$, it may be represented by some $\sigma \in F_p C_{p+q}$ and is such that $\partial(\sigma) \in F_{p-1}C_{p+q-1}$. One defines $\partial^1([\sigma]) = [\partial (\sigma)]$ (there is also an explicit description of the higher differentials, which we shall ignore for now). We have that the filtration of $C_\bullet$ induces filtrations $F_p H_k(C_\bullet)$ of the homology groups of $C_\bullet$, where $F_p H_k(C_\bullet)$ is given by taking those homology classes represented by an element of $F_p(C_k)$. Define $G_p H_k(C_\bullet) := F_p H_k(C_\bullet) / F_{p-1} H_k(C_\bullet)$. With some meticulous book-keeping, it is possible to show that the above spectral sequence converges to $E_{p,q}^\infty = G_p H_{p+q}(C_\bullet)$. That is, the diagonals of the $E^\infty$ page determine each of the homology groups up to a sequence of extension problems.

A famous example is given by the Serre spectral sequence. Let $f \colon X \rightarrow B$ be a fibration where $B$ is some (for simplicity) finite-dimensional CW-complex. The goal is to compute topological invariants of the total space $X$ in terms of invariants of the base space $B$ and fibre $F$ (which is uniquely determined up to homotopy equivalence for a path connected space), although, as mentioned above, occasionally information can be derived in the other direction. The CW-decomposition of $B$ induces a filtration of $X$. That is, one may define $X_p := f^{-1}(B^p)$, the pullback to the total space of the $p$-skeleton of $B$. This induces a filtration of the singular chain complex $S_\bullet(X)$ defined by $F_p(S_\bullet(X)) := S_\bullet(X_p)$ so that the first page of the spectral sequence associated to this filtration is $E_{p,q}^1 : = H_{p+q}(X_p,X_{p-1})$. It turns out that $H_{p+q}(X_p,X_{p-1}) \cong H_p(B^p,B^{p-1}) \otimes H_q(F)$, which are the cellular chain groups of $B$ over coefficients $H_q(F)$. In fact, the maps $d^1 \colon H_{p+q}(X_p,X_{p-1}) \rightarrow H_{p+q-1}(X_{p-1},X_{p-2})$ which define the second page of the spectral sequence correspond, at least ``locally'', to the cellular boundary maps for $B$. It follows from these considerations that there is a formula for the $E^2$ page of this \emph{Serre spectral sequence}, ``converging'' to $H_\bullet(X)$, whose $E^2$ page is given by $E^2_{p,q} \cong H_p(B,H_q(F))$, although the coefficients $H_q(F)$ need to be interpreted as a ``local-coefficient system'' over $B$ when $B$ is not simply connected.

\section{The Spectral Sequence for Two-Dimensional Tilings}

\begin{theorem} \label{thm: Spectral Sequence} Let $T$ be a cellular tiling of $(\mathbb{R}^2,d_{euc})$ with finite local complexity with respect to rigid motions. Then there is a spectral sequence (of homological type) converging to $H^{\mathcal{T}_{\text{rot}}}_\bullet(E^+(2)) \cong \check{H}^{3-\bullet}(\Omega^{\text{rot}})$ whose $E^2$ term is

\begin{tikzpicture}
  \matrix (m) [matrix of math nodes,
    nodes in empty cells,nodes={minimum width=5ex,
    minimum height=5ex,outer sep=-5pt},
    column sep=1ex,row sep=1ex]{
                &      &     &     & \\
          1     &  H_0^{\mathcal{T}_0}(\mathbb{R}^2) &  \check{H}^1(\Omega^0)  & \check{H}^0(\Omega^0) & \\
          0     &  \check{H}^2(\Omega^0)  & \check{H}^1(\Omega^0) &  \check{H}^0(\Omega^0)  & \\
    \quad\strut &   0  &  1  &  2  & \strut \\};
  \draw[thick] (m-1-1.east) -- (m-4-1.east) ;
\draw[thick] (m-4-1.north) -- (m-4-5.north) ;
\end{tikzpicture}

\end{theorem}

Note that the above spectral sequence is of homological type. The description of the $E^2$ page will follow from considerations of PE \emph{homology} groups, the \v{C}ech cohomology groups appear in the above theorem simply as Poincar\'{e} duals to these groups. The group in position $(0,1)$, as we saw in Chapter \ref{chap: Poincare Duality for Pattern-Equivariant Homology}, is an extension of $\check{H}^2(\Omega^0)$ over a torsion group defined in terms of the number of rotationally invariant tilings of the hull.

\begin{proof} As we have noted, the map $f \colon \Omega^{\text{rot}} \rightarrow \Omega^0$ is not a fibration. However, the spectral sequence will result from consideration of a certain fibration. Let $\pi \colon E^+(2) \rightarrow \mathbb{R}^2$ denote the fibre bundle with total space the Lie group $E^+(2)$ of rigid motions of the plane, $\pi$ is defined by projecting an orientation preserving isometry of the plane to its translation part. Topologically, this is a trivial fibre bundle with fibre $S^1$.

We shall assume that our tiling $T$ has a CW structure for which points of local rotational symmetry are contained in the $0$-skeleton. Any FLC (with respect to rigid motions) tiling is MLD to such a tiling, so this causes no loss in generality.

Recall the construction of the pattern $\mathcal{T}_{\text{rot}}$ on $E^+(2)$ (see Definition \ref{def: rot tiling pattern}) which defines the $\mathcal{T}_{\text{rot}}$E homology groups $H_\bullet^{\mathcal{T}_{\text{rot}}}(E^+(2))$ and the hull $\Omega^{\text{rot}}$. The spectral sequence will be defined similarly to the Serre spectral sequence. We consider a filtration of $E^+(2)$ into pieces $X_p:=\pi^{-1}(T^p)$ where $T^p$ denotes the $p$-skeleton of $T$. The idea is to restrict the $\mathcal{T}_{\text{rot}}$E homology groups to these subspaces. We have that $X_0$ is a disjoint union of circles, each associated to some vertex of the tiling, $X_1$ is a union of cylinders, one for each edge of the tiling and glued at the circles lying over the vertices of the tiling, and $X_2$ is all of $E^+(2)$. This simple description allows for simple descriptions of the relative $\mathcal{T}_{\text{rot}}$E homology groups and the boundary maps between them.

The pairs of subspaces $(X_p,X_{p-1})$ are ``good'' in a PE sense, see Lemma \ref{def} and the subsequent discussions. Although not necessary, one could argue using cellular groups here, by adding additional cells to $E^+(2)$, as described in \cite{Sadun2}. Either way, this allows us to analyse the homology groups $H_\bullet^{\mathcal{T}_{\text{rot}}}(X_p,X_{p-1})$ easily by excision to each $c \rtimes S^1$ for a cell $c$ and considering the relative homology at each such region.

So we shall consider the spectral sequence which converges to the $\mathcal{T}_{\text{rot}}$E homology $H_\bullet^{\mathcal{T}_{\text{rot}}}(E^+(2))$ with the filtration induced by the filtration $X_p$ of $E^+(2)$. So one must consider the relative groups $H_{p+q}^{\mathcal{T}_{\text{rot}}}(X_p,X_{p-1})$ which give the entries $E_{p,q}^1$ of the first page. The boundary maps $d^1$ of the first page take a chain of the relative group $H_{p+q}^{\mathcal{T}_{\text{rot}}}(X_p,X_{p-1})$ and applies the usual boundary operator to it, yielding a chain of $H_{p+q-1}^{\mathcal{T}_{\text{rot}}}(X_{p-1},X_{p-2})$.

Our claim is that the $E^2$ page, given as the homology of the $E^1$ page by $d^1$, has bottom row corresponding to the modified PE cellular chain complex for $\mathcal{T}_0$ and has top row the PE cellular chain complex for $\mathcal{T}_0$. That is, we have a chain isomorphism for the bottom row $E^1_{p,0}$:
\vspace{5mm}

\noindent
\xymatrix{0 & H_0^{\mathcal{T}_{\text{rot}}}(X_0) \ar[l] \ar[d]& H_1^{\mathcal{T}_{\text{rot}}}(X_1,X_0) \ar[l] \ar[d]& H_2^{\mathcal{T}_{\text{rot}}}(X_2,X_1) \ar[l] \ar[d]& 0 \ar[l] \\
0 & C^{\hat{\mathcal{T}_0}}_0(\mathbb{R}^2)  \ar[l] & C^{\hat{\mathcal{T}_0}}_1(\mathbb{R}^2) \ar[l] & C^{\hat{\mathcal{T}_0}}_2(\mathbb{R}^2) \ar[l] & 0 \ar[l] }

\vspace{5mm}

\noindent
and top row $E^1_{p,1}$:

\xymatrix{0 & H_1^{\mathcal{T}_{\text{rot}}}(X_0) \ar[l] \ar[d]& H_2^{\mathcal{T}_{\text{rot}}}(X_1,X_0) \ar[l] \ar[d]& H_3^{\mathcal{T}_{\text{rot}}}(X_2,X_1) \ar[l] \ar[d]& 0 \ar[l] \\
0  & C^{\mathcal{T}_0}_0(\mathbb{R}^2) \ar[l] & C^{\mathcal{T}_0}_1(\mathbb{R}^2) \ar[l]  & C^{\mathcal{T}_0}_2(\mathbb{R}^2) \ar[l] & 0 \ar[l]}
\vspace{5mm}

\noindent
where $C^{\mathcal{T}_0}_\bullet(\mathbb{R}^2)$ is the cellular $\mathcal{T}_0$E chain complex and $C^{\hat{\mathcal{T}_0}}_\bullet(\mathbb{R}^2)$ is the modified cellular $\mathcal{T}_0$E cellular complex (see Subsection \ref{Duality for Rotational Tiling Spaces}). The result then follows by Poincar\'{e} duality.

Let us first examine the lower $E^1_{p,0}$ row. The group $H_0^{\mathcal{T}_{\text{rot}}}(X_0)$ is generated by $0$-chains lying in the circular fibres above each vertex of the tiling. They must be such that there exists some radius $R$ for which, if two vertices have the same patch of tiles about them to radius $R$ up to rigid motion, this rigid motion induces an equivalence between the $0$-chains lying in these fibres. Note that, for example, if a vertex has rotational symmetry of order $n$ to radius $R$, then the $0$-chain at its fibre $S$ must be a multiple of $n$ times a usual generator of $H_0(S)$. We have an isomorphism $H_0^{\mathcal{T}_{\text{rot}}}(X_0) \cong C_0^{\hat{\mathcal{T}_0}}(\mathbb{R}^2)$ given by sending a $0$-chain to the cellular chain which assigns to a vertex of a tiling the sum of the coefficients of singular chains on its fibre.

The elements of $H_1^{\mathcal{T}_{\text{rot}}}(X_1,X_0)$ may be similarly described as a choice of oriented edge for each $1$-cell of the tiling in a way such that this assignment is PE to some radius. In more detail, for each $1$-cell $c$ of the tiling, we have the pair $(c \rtimes S^1, \partial c \rtimes S^1)$. The relative singular homology of this pair is generated by a path travelling from one end of the cylinder along $c$ to the other. For $\mathcal{T}_{\text{rot}}$E relative chains, we must choose such paths in each cylinder in a way which is equivariant with respect to the pattern to some radius. The boundary map simply takes the boundary of the $1$-chain in $X_0$. With the above identifications, this is precisely the same boundary map as in the modified complex.

The group $H_2^{\mathcal{T}_{\text{rot}}}(X_2,X_1)$ may be described as the group of PE assignments of oriented coefficients to the $2$-cells of $T$. The boundary of a $2$-chain has boundary in $(X_1,X_0)$ which, with the identifications above, corresponds to the usual cellular boundary map for $\mathcal{T}_0$E homology. We see that the bottom row of the spectral sequence with boundary map $d^1$ corresponds precisely to the modified $\mathcal{T}_0$E chain complex.

We now similarly geometrically examine the top row of $E^1$. We shall pay extra attention to the homology in degree $0$, which displays a distinct difference to the lower row, see Figure \ref{fig: Spectral Sequence}. Only a multiple of the usual degree-$0$ generator could be given for the fibres lying over the $0$-skeleton of rotationally symmetric points in the bottom row. For $H_1^{\mathcal{T}_{\text{rot}}}(X_0)$, however, note that one may realise a generator of $H_1(S^1)$ as a singular chain which is invariant under rotation of order $n$. We see that $H_1^{\mathcal{T}_{\text{rot}}}(X_0)$ may be described as the group of assignments of integers to each $0$-cell of $T$ for which there exists some $R$ such that, if two $0$-cells have the same patch of tiles to radius $R$, then they are assigned the same coefficient. There is no requirement here for $0$-cells of $n$-fold rotational symmetry to be given an integer some multiple of $n$ and we have an isomorphism $H_1^{\mathcal{T}_{\text{rot}}}(X_0) \cong C_0^{\mathcal{T}_0}(\mathbb{R}^2)$ given by simply calculating the winding number of the $1$-chain at the fibre of each vertex. The chains of $H_2^{\mathcal{T}_{\text{rot}}}(X_1,X_0)$ correspond to ``PE oriented cylinders'' and may be described as a PE assignment of integers to each $1$-cell of the tiling; the boundary map corresponds to the usual boundary map on $1$-cells. Similarly, $H_3^{\mathcal{T}_{\text{rot}}}(X_2,X_1)$ corresponds to the group of PE assignments of integers to each $2$-cell of $T$, again with the usual boundary map. Hence, the top row of the spectral sequence is given as the chain complex of $\mathcal{T}_0$E chains. The result now follows by Poincar\'{e} duality. \end{proof}

\begin{figure}
\begin{center}
\includegraphics[width = \textwidth]{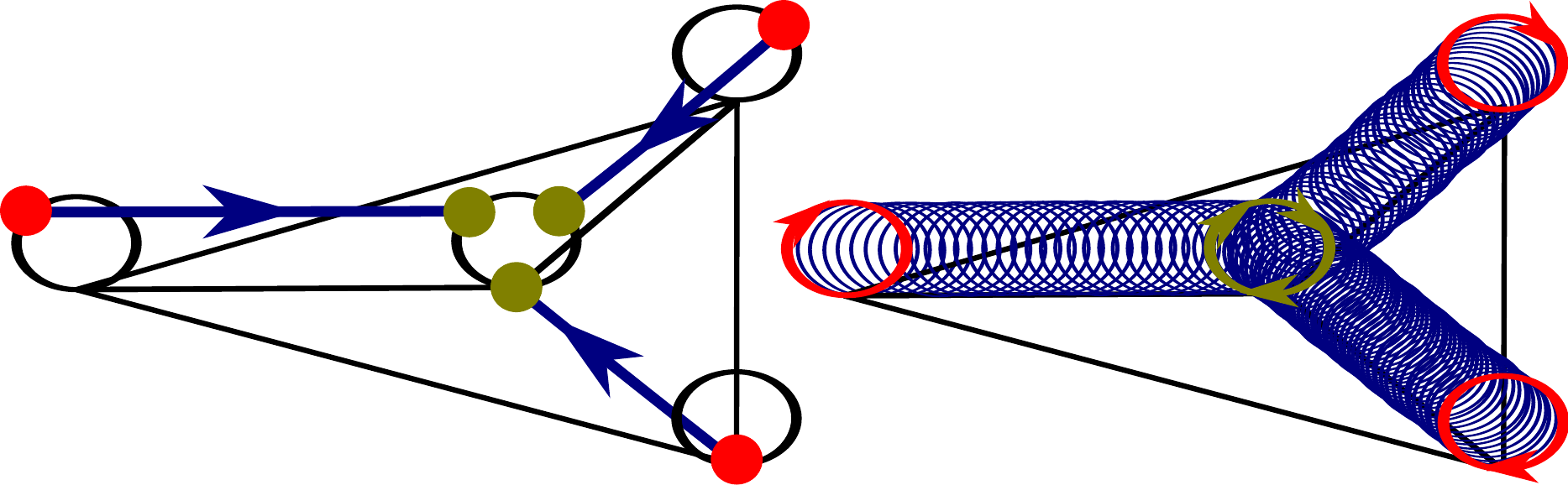}
\caption{Degree $0$ homology at $E^1_{0,0}$ and $E^1_{0,1}$}
\label{fig: Spectral Sequence}
\end{center}
\end{figure}

\subsection{The $d^2$ Map}

We now describe a method for evaluating the map $d^2$. Notice that the only (potentially) non-trivial map is $d^2_{2,0} \colon ker(d^1_{2,0}) \rightarrow coker(d^1_{1,1})$. As described in the above proof, we have that $ker(d^1_{2,0}) \cong H_2^{\hat{\mathcal{T}_0}}(\mathbb{R}^2) \cong \mathbb{Z}$ and $coker(d^1_{1,1}) \cong H_0^{\mathcal{T}_0}(\mathbb{R}^2)$, so we just need to find the image under $d^2$ of the element of $ker(d^1_{2,0})$ corresponding to the fundamental class for $\mathbb{R}^2$.

We assume throughout that our tiling is cellular and is such that points of local rotational symmetry are contained in the $0$-skeleton (which, as mentioned in the proof of Theorem \ref{thm: Spectral Sequence}, causes no loss of generality to us here). We may, in fact, assume that our $1$-cells are straight line segments and our $2$-cells are polygons. As in the previous chapter, we let a vertex type be an equivalence class (up to rigid motion) of a vertex of the tiling, along with each of the tiles which contain it in their boundary. An edge type is described similarly, where the $1$-cell is taken with those tiles intersecting the interior of the $1$-cell. A face type is an equivalence class of $2$-cell, taken with the tile which it is contained in.

For each face type of the tiling, pick some direction (that is, some embedding of it in $\mathbb{R}^2$, up to translation). Now, for each edge type, pick an ordering of the $2$-cells which bound it and a rotation which rotates the patch from an orientation in which the face type of the first chosen $2$-cell is in its chosen direction to an orientation for which the second $2$-cell agrees with its chosen direction. Notice that this number is well defined modulo $2 \pi$ since each $2$-cell of the tiling has trivial rotational symmetry.

We may now describe the image of $d^2$ in $H_1^{\mathcal{T}_{\text{rot}}}(X_0) \cong H_0^{\mathcal{T}_0}(\mathbb{R}^2)$. For each vertex type, pick some $2$-cell of it and traverse the $0$-cell in a clockwise fashion, crossing each $1$-cell with boundary containing the $0$-cell once. Let each crossing of a $1$-cell contribute its chosen rotation (or its negative, if it is traversed in the opposite direction) as described above to a sum of a single lap about the $0$-cell. The sum of these contributions provides a winding number for that vertex type (see Figure \ref{fig: Winding Numbers}). Indeed, the final orientation of the patch after performing the required rotations at each edge should be unchanged. Assign the vertex type the corresponding integer coefficient, with clockwise winding numbers contributing positive terms and anticlockwise ones negative terms. Name the corresponding pattern equivariant $0$-chain $\sigma$.

Perhaps an appealing alternative description is the following: define a non-zero vector field on the tiling away from the vertices of the tiling. It should only depend on the local configuration of tiles up to rigid motion; for example, we could ask that it should be determined at a point by the location of that point within the tile there up to rigid motion, or the patch of two nearest tiles when the point is close to a tile boundary. Then the coefficient assigned to each vertex is the index of the vector field at that vertex.

\begin{proposition} The PE $0$-chain $\sigma$ as constructed above is a generator of the image of $d^2$. \end{proposition}

\begin{proof} The proof of this proposition follows from our identification of the $E^1$ page of the spectral sequence to the $\mathcal{T}_0$E chain complexes, as in the above proof, and an examination of the definition of the map $d^2$.

To define $\sigma$, a direction of each face type is chosen. This defines a generator $\psi$ of the kernel of $d^1_{2,0} \colon H_2^{\mathcal{T}_{\text{rot}}}(X_2,X_1) \rightarrow H_1^{\mathcal{T}_{\text{rot}}}(X_1,X_0)$ as follows. Firstly, for each face type $c$, pick a generator of $H_2(c,\partial c) \cong \mathbb{Z}$ agreeing with the standard orientation of $\mathbb{R}^2$. For a $2$-cell $t \subset \mathbb{R}^2$, let the $2$-chain on $t \rtimes S^1$ (with boundary in $\partial t \rtimes S^1$) be given by pulling back the $2$-chain from the face type to $(t,-\theta)$, where the $2$-cell there is found in direction $\theta$ relative to its chosen direction. Then with our identification of $d^1_{2,0} \colon H_2^{\mathcal{T}_{\text{rot}}}(X_2,X_1) \rightarrow H_1^{\mathcal{T}_{\text{rot}}}(X_1,X_0)$ with the cellular boundary map $\partial \colon C_2^{\mathcal{T}_0}(\mathbb{R}^2) \rightarrow C_1^{\mathcal{T}_0}(\mathbb{R}^2)$, we see that $\psi$ is $\mathcal{T}_{\text{rot}}$E and corresponds to a fundamental class for $\mathbb{R}^2$, that is, it is a generator of the kernel of $d^1_{2,0}$.

Now, the boundary of this $2$-chain may not be zero, but it certainly lies in $X_1$ and is homologous to a chain lying in $X_0$. Indeed, at some $1$-cell $c$ of the tiling, the boundary of $\psi$ in $c \rtimes S^1$ traverses $c \rtimes S^1$ between the bounding circles in two opposed directions. These two $1$-chains, running in opposed directions along the cylinder, are related by the rotation assigned to the edge type of the $1$-cell in the construction of $\sigma$. We see that there is a $2$-chain contained in $c \rtimes S^1$ whose boundary cancels the boundary of $\psi$ in $int(c) \rtimes S^1$ and leaves $1$-chains in $\partial c \rtimes S^1$ which correspond, at the fibre of each vertex, to the rotation contributed by the corresponding edge. It follows that this chain is homologous to $\sigma$. That is, $\partial(\psi)$ is homologous to the representation of $\sigma$ in the chain group $C_1^{\mathcal{T}_{\text{rot}}}(X_1)$. But by construction this element corresponds precisely to $d^2(\psi)$, where $\psi$ is the generator of $E_{2,0}^1$, and so the result follows.
\end{proof}

\subsection{Examples}

\subsubsection{Periodic Examples} We shall present computations for two periodic tilings, the periodic tiling of triangles and of squares. The corresponding tiling spaces $\Omega^{\text{rot}}$, as quotients of free and proper actions of the corresponding wallpaper groups (without reflections) on $E^+(2)$, are closed manifolds. The \v{C}ech cohomology groups $\check{H}^k(\Omega^{\text{rot}})$ computed are then simply the singular cohomology groups of these quotient manifolds.

Firstly, the triangle tiling. In Example \ref{ex: Periodic}, we computed that $H_0^{\mathcal{T}_0} \cong \mathbb{Z} \oplus \mathbb{Z}_2 \oplus \mathbb{Z}_3$. For the triangle (and square) tiling we have that $\Omega^0 \cong S^2$ and so the modified homology group in degree zero is $H_0^{\hat{\mathcal{T}_0}} \cong \mathbb{Z}$. We compute that the image of the generator of the $d^2$ map of the spectral sequence is $(1,1) \in \mathbb{Z}_2 \oplus \mathbb{Z}_3$, and so the torsion is killed on the $E^2$ page. Hence, $H^2(\Omega^{\text{rot}})  \cong \mathbb{Z}$, the remaining (co)homology groups being determined by Poincar\'{e} duality.

For the square tiling we have that $H_0^{\mathcal{T}_0} \cong \mathbb{Z} \oplus \mathbb{Z}_2 \oplus \mathbb{Z}_4$. In this case the image of the generator under $d^2$ is $(1,1) \in \mathbb{Z}_2 \oplus \mathbb{Z}_4$ so $H^2(\Omega^{\text{rot}}) \cong \mathbb{Z} \oplus \mathbb{Z}_2$.

\subsubsection{The Penrose Kite and Dart Tilings} Let $T$ be a Penrose kite and dart tiling of $\mathbb{R}^2$. In \ref{subsect: The Penrose Tiling} we calculated that $H^{\mathcal{T}_0}_0 \cong \mathbb{Z}^2 \oplus \mathbb{Z}_5$. The generators of the free parts here are given by the sun and star vertex types, and torsion group $\mathbb{Z}_5$ is generated by sun$+$star$-$queen. The spectral sequence at the $E^2$ page is as follows: 

\begin{tikzpicture}
  \matrix (m) [matrix of math nodes,
    nodes in empty cells,nodes={minimum width=5ex,
    minimum height=5ex,outer sep=-5pt},
    column sep=1ex,row sep=1ex]{
                &      &     &     & \\
          1     &  \mathbb{Z}^2 \oplus \mathbb{Z}_5 &  \mathbb{Z}  & \mathbb{Z} & \\
          0     &  \mathbb{Z}^2  & \mathbb{Z} &  \mathbb{Z}  & \\
    \quad\strut &   0  &  1  &  2  & \strut \\};
  \draw[thick] (m-1-1.east) -- (m-4-1.east) ;
\draw[thick] (m-4-1.north) -- (m-4-5.north) ;
\end{tikzpicture}

\begin{figure}

\includegraphics[width=\linewidth]{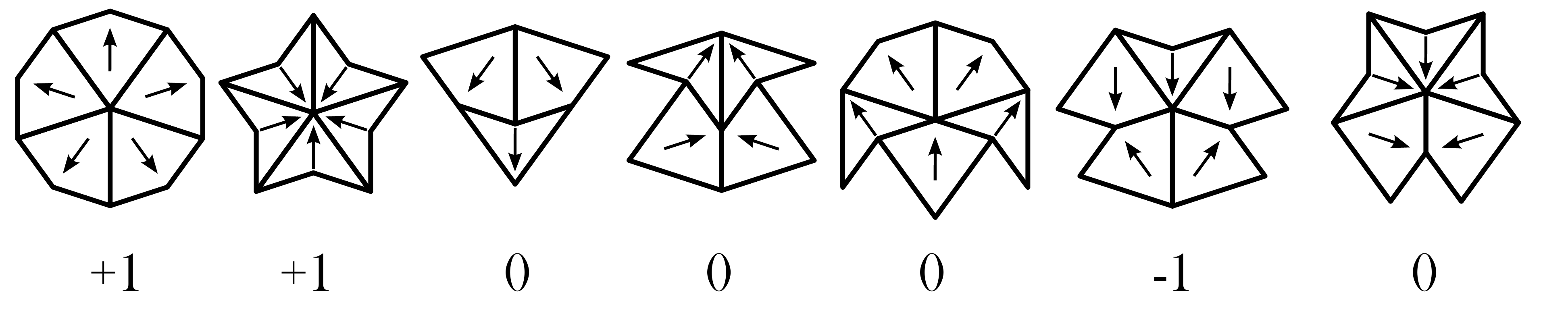} \caption{Winding numbers for the Penrose kite and dart tiling}

\label{fig: Winding Numbers}

\end{figure}

Let us now compute the image of $d^2$. We shall place the kite and dart face types so that they are facing ``upwards'', as indicated by the arrows of Figure \ref{fig: Winding Numbers}. We choose the rotations on each edge type going from the left to right tile, where the edge is pointing vertically upwards, by picking the smallest rotation taking the orientation of the left-hand tile to that of the right-hand tile. One may easily compute winding numbers, we see that the $\mathcal{T}_0$E chain that it represents is given by sun$+$star$-$queen. That is, the image of the generator under $d^2$ is precisely the generator of the torsion part of $H_0^{\mathcal{T}_0}$.

Hence, we see that the $E^\infty$ page loses the torsion group. Since the remaining groups are free-abelian there are no extension problems. It follows that $\check{H}^i(\Omega^{\text{rot}}) \cong \mathbb{Z}, \mathbb{Z}^2,\mathbb{Z}^3,\mathbb{Z}^2$ for $i=0,1,2,3$ respectively.

This is a new result for the cohomology of the Penrose tiling and is quite surprising. There is some justification to suggest that, when there are two tilings of $n$-fold symmetry in the hull of the tiling, there must necessarily be $n$-torsion in the degree $2$ \v{C}ech cohomology of the rigid hull. The argument goes as follows (see \cite{Sadun2}). There is a radius $R$ for which these two rotation invariant tilings are distinguishable. One may consider the approximant $K_R$ of $\Omega^{\text{rot}}$, which is a branched $3$-manifold defined as a quotient of $E^+(2)$ by identifying $f$ and $g$ whenever $f^{-1}(T)$ and $g^{-1}(T)$ agree at the origin to radius $R$. Then there are circular fibres of $K_R$ which represent the rotates of the central patches of these two tilings. These fibres are \emph{exceptional}, that is, they represent only $1/n$ of a full rotation of the nearby \emph{generic} fibres. So it seems that, by looping around one of the fibres clockwise and then around the other anticlockwise, this should represent an element of order $n$ in $H_1(K_R)$, and hence $H^2(K_R)$ has $n$-torsion by the universal coefficient theorem. Indeed, multiplying this chain by $n$ we may ``free'' the loops so that they may travel along the generic fibres and cancel each other out.

\begin{figure}
\begin{center}
\includegraphics[height=4in]{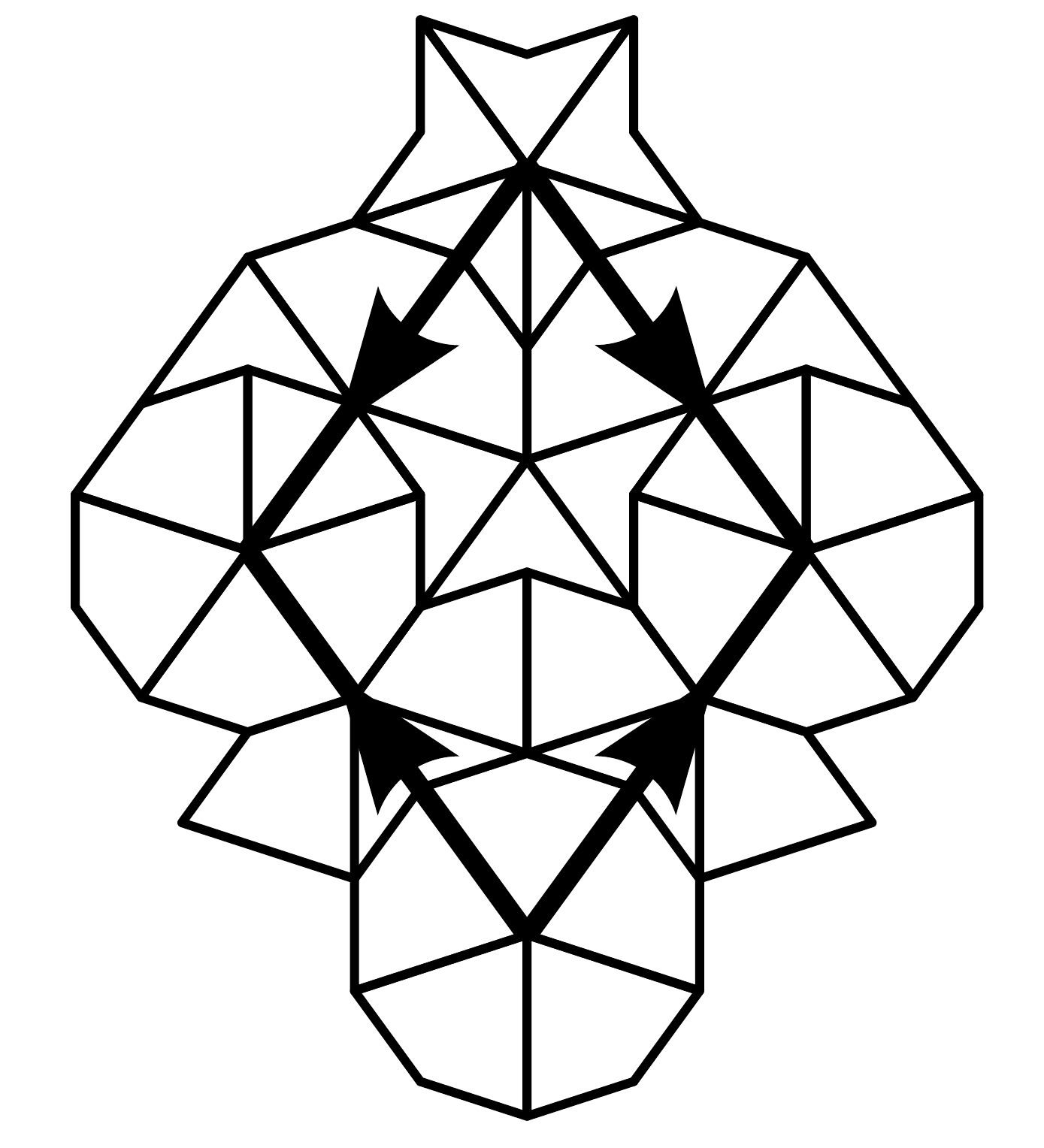} \caption{Homotopic exceptional fibres} \label{fig: Homotopic Exceptional Fibres}
\end{center}

\end{figure}

However, this argument only shows that it may be reasonable to expect torsion in many cases and is not a proof for it in general. The problem is that there is more freedom in homotoping loops around $K_R$ than by just moving the loops of a homotopy ``fibre-wise''. We shall show how this can be done for the Penrose tiling. The patch of tiles of Figure \ref{fig: Homotopic Exceptional Fibres} is a patch from a Penrose tiling, so a path on it (with $R$ small so that an $R$-neighbourhood of our path is contained within the patch) induces a path on $K_R$. We first consider the loop $a$ which winds anticlockwise once around the exception fibre at the sun vertex. We visualise this on the diagram by ``staying'' at the bottom sun vertex and rotating $2 \pi /5$ anticlockwise in the fibre direction. We shall now homotope our loop so that its start and end points, as represented on the diagram, are the left and right sun vertices at the centre of the patch, respectively. Each loop of the homotopy starts at a point of the lower left-hand arrowed edge. One firstly performs a $2\pi /5$ anticlockwise rotation at the start point of this left hand edge and then follows a horizontal path rightwards stopping at the right-hand arrowed edge (so each patch on the path after the rotation, on the approximant, is a $2 \pi /5$ anticlockwise rotate of the actual patch of tiles on the diagram). As one can see, after the $2\pi /5$ rotation, the local patch at the end of the path agrees with that at the starting point and so this path is a loop. We thus define a homotopy from $a$, the loop around the exceptional sun fibre to the loop which, firstly, performs $a$ and then follows the horizontal path $p$ from the left-hand sun in the middle of the patch to the right-hand sun, that is, $a = a \cdot p$ in homotopy.

We can do something similar for the star vertex at the top of the diagram. So we consider the loop $b$ given by a $2 \pi /5$ anticlockwise rotation about the exceptional fibre at this star vertex. Analogously to above, we may homotope this loop to a loop which, firstly, travels horizontally from right to left between the sun vertices and then performs an anticlockwise rotation at this sun vertex, that is, $b=p^{-1} \cdot a$. So then we have that $a \cdot b = a \cdot p \cdot p^{-1} \cdot a = a \cdot a$ which means that $a = b$ in homotopy (alternatively, we see from $a=a \cdot p$ that $p$ and hence $p^{-1}$ is nullhomotopic so $b = p^{-1} \cdot a = a$). It follows that the exceptional fibres are homotopic in $K_R$ and their difference does not represent a torsion element of $H_1(K_R)$. The same argument can be used for approximants with larger values of $R$ by performing the above homotopy on the equivalent patches of super$^i$-tiles of sufficiently many subdivisions, so the two exceptional fibres are homotopic on each approximant. It is bewildering that the famous Penrose tilings can still hold tricks up their sleeves!

\chapter{Conclusion}
\section{Summary}

The work of this thesis may be considered as being separated into two parts. Firstly, we considered a good setting with which to study finite local complexity patterns. Our abstract notion of a pattern (Definition \ref{def: pattern}) neatly encodes the essential features of a given FLC tiling or point pattern of interest, discarding its underlying decorations and replacing it with the crucial structure of which partial isometries preserve portions of it. Many constructions depend only on these partial isometries, indeed, the MLD class of a tiling or point pattern only depends on them, and there are indications that they could be a useful abstract formalism with which to further investigate geometric invariants of tilings and point patterns (and perhaps more), as we shall discuss later. An added advantage is that, with this point of view, the study of patterns on non-homogeneous spaces, such as the pentagonal tilings of Bowers and Stephenson \cite{Pent}, presents no extra difficulties.

Many constructions, such as the pattern-equivariant cohomology or the tiling space, may in fact be defined in terms of a collection of induced equivalence relations on the ambient space of a tiling. We introduced the notion of a collage (Definition \ref{def: collage}) to capture this idea. The axioms defining a collage are simple and geometrically inspired and allow for many constructions which parallel those from tiling theory, such as the tiling metric (analogous to the collage uniformity), or the inverse limit of approximants. The expression of a tiling space as an inverse limit of approximants has been an important tool in the study of FLC tilings, in particular in describing topological invariants of tiling spaces. We showed in Theorem \ref{thm: inverse limit of approximants} how, at least under mild conditions, we have an inverse limit description of the hull for general collages.

In the second part of this thesis, we introduced geometric invariants for our abstract patterns. When the pattern is given a compatible cellular structure, we showed in Theorem \ref{thm: sing=cell} that there is an agreement between the singular and cellular pattern-equivariant homology groups. In the case of a pattern coming from an FLC tiling, these homology groups have an analogous description (see Section \ref{sect: Pattern-Equivariant Homology of Tilings}) to the pattern-equivariant cohomology groups \cite{Sadun1}, which are isomorphic to the \v{C}ech cohomology of the tiling space. In Chapter \ref{chap: Poincare Duality for Pattern-Equivariant Homology} we showed that the PE homology groups of an FLC tiling of $\mathbb{R}^d$ may be linked to the PE cohomology groups (and hence to the \v{C}ech cohomology groups of the tiling space) through Poincar\'{e} duality. In Chapter \ref{chap: Pattern-Equivariant Homology of Hierarchical Tilings} we outlined an efficient method for computation of the PE homology groups for hierarchical tilings.

Interestingly, unlike the PE cohomology groups, the PE homology groups of the tiling with respect to rigid motions are not a topological invariant of the tiling space and seem to be sensitive to the presence of rotationally invariant tilings in the tiling space. We showed how one may modify the PE homology groups so as to restore duality. This allowed us, in Chapter \ref{A Spectral Sequence for the Rigid Hull of a Two-Dimensional Tiling}, to define a spectral sequence for a general two-dimensional FLC (with respect to rigid motions) tiling converging to the \v{C}ech cohomology of the rigid hull incorporating these groups. There is also a simple description of the $d^2$ map for the spectral sequence, which is readily computable following computations for a hierarchical tiling. These methods combined allowed for an original and remarkably simple computation for the \v{C}ech cohomology of the rigid hull of the Penrose kite and dart tilings, which prompted the surprising discovery that the two exceptional fibres of the approximants are homotopic.

\section{Future Directions}

Our methods here have focused on thinking about invariants of a tiling as being represented concretely on the tiling itself, as opposed to considering topological invariants of its tiling space. The definition of pattern-equivariant homology is geometrically inspired as a simple analogue of the PE cohomology. One may now ask, though, if the PE homology groups are invariants of some space, or other structure, related to the tiling. Can the Poincar\'{e} duality here be seen in a more abstract setting?

We firstly note that the translational hull of an FLC tiling of $\mathbb{R}^d$ has the structure of a \emph{weak homology manifold} (see Bredon \cite{Bredon} pg.\ 329). The local structure of the tiling space is given as a product of $\mathbb{R}^d$ with some totally disconnected set. This means that the local homology of the tiling space is free and concentrated in degree $d$, making the tiling space a weak homology manifold. There are Poincar\'{e} duality results for such spaces between the sheaf cohomology (isomorphic here to the \v{C}ech cohomology) and the Borel-Moore homology (defined differently to here as used in the definition of the PE homology groups, and instead in sheaf-theoretic formalism better suited to spaces which are not CW-complexes, see \cite{BM}). The duality requires one to take coefficients for the sheaf cohomology in the orientation sheaf, which for an aperiodic tiling space will be complicated. However, it seems reasonable to conjecture that with some modification of the Borel-Moore homology of the tiling space (perhaps taking into account the structure of local transversals) the Poincar\'{e} duality seen here may be seen topologically on the translational hull. Another possible approach would be to consider Poincar\'{e} duality at the level of approximants. The approximants are not manifolds but are branched manifolds and there are duality results for such spaces by considering them as \emph{stratified spaces}, see \cite{Banagl}.

For invariants of $\mathcal{T}_0$ the PE-homology groups are \emph{not} topological invariants of the space $\Omega^{\mathcal{T}_0}$. A natural suggestion is that they may be invariants of the space $\Omega^{\mathcal{T}_{\text{rot}}}$ along with the group action of $SO(d)$ on it. Consider the classical periodic setting, with $T$ a periodic tiling of $\mathbb{R}^2$. One may consider the space given as the quotient of $\mathbb{R}^2$ by the space group $G_T$ of isometries preserving $T$. Unfortunately, this quotient loses lots of information, for example, for the square and triangle tilings, without reflections, these quotient spaces are homeomorphic to $S^2$. However, one may enrich the structure of this \emph{coarse quotient} $\mathbb{R}^2/G_T$ to that of an \emph{orbifold} $[\mathbb{R}^2/G_T]$, by keeping track of the local isotropy groups of points fixed by rotations or reflections. In some sense, the quotient space is deficient since the action is not free at points of local symmetry; the orbifold corrects this deficiency. And, indeed, the $17$ space groups of $\mathbb{R}^2$ (sometimes known as wallpaper groups) correspond precisely to the $17$ orbifolds given by the quotients of these group actions, so these orbifolds exactly classify the wallpaper groups.

The above shows the utility of preserving extra structure on these spaces of periodic tilings. It seems reasonable to suggest that similar structures keeping track of symmetries for aperiodic tilings could also be advantageous. Indeed, a large motivation for the study of aperiodic tilings comes from their use in modelling quasicrystals, which were discovered precisely because they may possess interesting symmetries, ones forbidden for periodic crystals by the crystallographic restriction theorem.

The analogue of the orbifold $[\mathbb{R}^2/G_T] \cong [(E^+(2)/G_T)/SO(d)] \cong [\Omega^{\mathcal{T}_{\text{rot}}}/SO(d)]$ for the aperiodic case should be the \emph{quotient stack} $[\Omega^{\mathcal{T}_{\text{rot}}}/SO(d)]$, let us call it the \emph{stacky tiling space}. Just as how the tiling space of an FLC tiling may be expressed as an inverse limit of approximant branched manifolds $\mathbb{R}^d/ R_{\mathcal{T}_0}$, the stacky tiling space should be expressible as an inverse limit of branched orbifolds $[\mathbb{R}^d/R_{\mathcal{T}_0}]$ over increasing values of patch radii $R$. The observation that these approximants may be considered as orbifolds is noted in \cite{BDHS,Sadun2}, but to the author's knowledge this structure has not yet been utilised in the study of aperiodic tilings.

We have a quotient map $f \colon \Omega^{\mathcal{T}_{\text{rot}}} \rightarrow \Omega^{\mathcal{T}_0}$ from the rigid hull of a tiling to the space of tilings modulo rotations. This is not a fibration due to the existence of rotationally symmetric tilings where the action of $SO(d)$ on $\Omega^{\mathcal{T}_{\text{rot}}}$ is not free. However, the map $f \colon \Omega^{\mathcal{T}_{\text{rot}}} \rightarrow [\Omega^{\mathcal{T}_{\text{rot}}} / SO(d)]$ to the \emph{stacky} tiling space will be a fibration. The associated Serre spectral sequence converging to the cohomology of $\Omega^{\mathcal{T}_{\text{rot}}}$ incorporates the cohomology of the stacky tiling space, which will agree with the cohomology of $\Omega^{\mathcal{T}_0}$ modulo torsion. The spectral sequence of Theorem \ref{thm: Spectral Sequence} also converges to the cohomology of $\Omega^{\mathcal{T}_{\text{rot}}}$ so it is potentially the case that the extra torsion found in the $\mathcal{T}_0$E homology could be related to invariants of the stacky tiling space. A more precise examination of these ideas could allow for a spectral sequence analogous to that of Theorem \ref{thm: Spectral Sequence} for tilings of dimension larger than $2$.

We finally consider the perspective of viewing patterns as inverse semigroups, a topic which we discussed in Subsection \ref{subsect: Patterns as Inverse Semigroups}. In spirit it seems that passing from the pattern of partial isometries to the collage of induced equivalence relations corresponds to forgetting the extra structure of local isotropy and passing to the coarse quotient. One would expect that, analogously to the PE cohomology groups, invariants of the stacky tiling space may be expressible as invariants of these abstract pseudogroups of partial isometries defined on the tiling, without reference to the abstract stacky tiling space itself. An alternative (and equivalent) approach \cite{MR} would be to consider the ambient space of the tiling along with induced germ groupoids which specify not only \emph{where} two points are equivalent to a sufficient radius but also \emph{how} they are equivalent, that is, by which local transformations. The theme is clear: there is a rich and, as of yet, unexplored setting in which one may study the geometry of aperiodic tilings using structures which keep track of their symmetries.

\countnotes

\end{document}